\documentclass[12pt,reqno,twoside,a4paper]{amsart}

\makeatletter
\@namedef{subjclassname@2020}{\textup{2020} Mathematics Subject Classification}
\makeatother

\usepackage{amsmath, amsthm, amssymb, amscd, amsfonts, tikz, float, enumitem, graphicx}
\usepackage{mathrsfs}

\usepackage{color}
\usepackage[all]{xy}             
 \usepackage{hyperref}          

\font \smallrm=cmr10 at 10truept
\font \ssmallrm=cmr10 at 9truept

\font \smallsl=cmsl10 at 9pt

\numberwithin{equation}{section}

\newcommand{\subu}[2]{{#1}_{\raise-2pt\hbox{$ \scriptstyle #2 $}}}
\newcommand{\subd}[3]{{#1}_{\raise-2pt\hbox{$ \scriptstyle #2 #3 $}}}

\newtheorem{lema}{Lemma}[subsection]
\newtheorem{theorem}[lema]{Theorem}
\newtheorem{cor}[lema]{Corollary}

\newtheorem{prop}[lema]{Proposition}

\theoremstyle{definition}
\newtheorem{definition}[lema]{Definition}

\newtheorem{exas}[lema]{Examples}

\newtheorem{rmk}[lema]{Remark}
\newtheorem{rmks}[lema]{Remarks}
\newtheorem{free text}[lema]{}
\newtheorem{obs}[lema]{Observation}
\newtheorem{obs's}[lema]{Observations}

\theoremstyle{remark}

\newcommand \id{\operatorname{id}}
\newcommand \ad{\operatorname{ad}}

\newcommand \rk{\operatorname{rk}}

\newcommand \End{\operatorname{End}}
\newcommand \Hom{\operatorname{Hom}}

\newcommand \Alt{\operatorname{Alt}}
\newcommand \Sym{\operatorname{Sym}}

\newcommand \Ker{\operatorname{Ker}}

\newcommand \Img{\operatorname{Im}}
\newcommand \op{\operatorname{op}}

\newcommand \smallast {{\operatorname{\raise1,5pt\hbox{$ \scriptscriptstyle \ast $}}}}
\newcommand \smallp {{\operatorname{\raise1pt\hbox{$ \scriptscriptstyle + $}}}}
\newcommand \smallm {{\operatorname{\raise1pt\hbox{$ \scriptscriptstyle - $}}}}
\newcommand \smallpm {{\operatorname{\raise1pt\hbox{$ \scriptscriptstyle \pm $}}}}
\newcommand \smallmp {{\operatorname{\raise1pt\hbox{$ \scriptscriptstyle \mp $}}}}

\newcommand{\eps}{\epsilon}
\newcommand{\ot}{\otimes}
\newcommand{\com}{\Delta}

\newcommand{\zero}{{\bar{0}}}
\newcommand{\one}{{\bar{1}}}

\newcommand{\R}{{\mathcal R}}

\newcommand{\F}{{\mathcal F}}

\newcommand{\G}{{\mathcal G}}
\newcommand{\Z}{{\mathcal Z}}
 \newcommand{\Lc}{{\mathcal L}}

\def \bq {\mathbf{q}}

\def \bx {\mathbf{x}}

\def \FF{\mathbb{F}}
\def \NN{\mathbb{N}}

\def \ZZ {\mathbb{Z}}
\def \k {\Bbbk}

\def \kh{{\Bbbk[[\hbar]]}}
\def \khp{{\Bbbk(\hskip-1,7pt(\hbar)\hskip-1,7pt)}}

\def \Zqqm {{\mathbb{Z}\big[\hskip1pt q,q^{-1}\big]}}

\def \Zabqpm {{\ZZ_A\big[\bq^{\pm 1}\big]}}

\def \dotRbq {{\dot{R}_{\,\bq}}}
\def \Rbqsq {{R_{\,\bq}^{\,\scriptscriptstyle {\sqrt{\phantom{I}}}}}}
\def \dotRbqsq {{\dot{R}_{\,\bq}^{\,\scriptscriptstyle {\sqrt{\phantom{I}}}}}}

\def\FF{\mathbb{F}}

\def \G {\mathcal{G}}

\def \K {\mathcal{K}}
\def \L {\mathcal{L}}

\def \SS {\mathcal{S}}

\def \JJ {\mathfrak{J}}

\def \Apicc {{\scriptscriptstyle A}}
\def \Dpicc {{\scriptscriptstyle D}}
\def \Ppicc {{\scriptscriptstyle P}}
\def \Rpicc {{\scriptscriptstyle \R}}
\def \Thetapicc {{\scriptscriptstyle \Theta}}
\def \Psipicc {{\scriptscriptstyle \Psi}}

\def \lieg{\mathfrak{g}}

\def \liegd{\mathfrak{g}^{\raise-1pt\hbox{$ \Dpicc $}}}

\def \liegdP{\mathfrak{g}^{\raise-1pt\hbox{$ \Dpicc $}}_\Ppicc}
\def \liegRpP{\mathfrak{g}^{\raise1pt\hbox{$ \scriptscriptstyle \mathcal{R},p $}}_\Ppicc}
\def \liegRpPa{\mathfrak{g}^{\raise1pt\hbox{$ \scriptscriptstyle \mathcal{R},p $}}_{\Ppicc_{\raise-3pt\hbox{$ \scriptscriptstyle \!\! A $}}}}
\def \liegPsi{\mathfrak{g}^{\raise-1pt\hbox{$ \scriptscriptstyle \Psi $}}}
\def \liegdPsi{\mathfrak{g}^{\raise-1pt\hbox{$ {\scriptscriptstyle D,\Psi} $}}}

\def \liegdotdq{\dot{\lieg}_{\raise-2pt\hbox{$ \Dpicc , \hskip-0,9pt{\scriptstyle \bq} $}}}
\def \liegtildq{\tilde{\lieg}_{\raise-2pt\hbox{$ \Dpicc , \hskip-0,9pt{\scriptstyle \bq} $}}}
\def \liegdq{\lieg_{\raise-2pt\hbox{$ \Dpicc , \hskip-0,9pt{\scriptstyle \bq} $}}}
\def \liegdqcheck{\lieg_{\raise-2pt\hbox{$ \Dpicc , \hskip-0,9pt{\scriptstyle \check{\bq}} $}}}
\def \Gtildeqstar {\widetilde{G}^{\,\raise2pt\hbox{$ \scriptstyle * $}}_{\!\Dpicc , \hskip+0,9pt{\scriptstyle \bq}}}

\def \lieb{\mathfrak{b}}

\def \liebd{\mathfrak{b}^{\raise-1pt\hbox{$ \Dpicc $}}}
\def \lieh{\mathfrak{h}}
\def \liehd{\mathfrak{h}^{\raise-1pt\hbox{$ \Dpicc $}}}
\def \liek{\mathfrak{k}}

\def \lies{\mathfrak{s}}
\def \lien{\mathfrak{n}}

\def \lieso{\mathfrak{so}}

\def \Phipicc {{\scriptscriptstyle \Phi}}

\def \uhg{U_\hbar(\hskip0,5pt\lieg)}

\def \uRpPhg{U^{\,\R\,,\,p}_{\!P,\hskip0,7pt\hbar}(\hskip0,5pt\lieg)}

\def \uRpPhh{U^{\,\R\,,\,p}_{\!P,\hskip0,7pt\hbar}(\hskip0,5pt\lieh)}

\def \uRpPhbp{U^{\,\R,p}_{\!P,\hskip0,7pt\hbar}(\hskip0,5pt\lieb_+)}

\def \uRpPhbm{U^{\,\R,p}_{\!P,\hskip0,7pt\hbar}(\hskip0,5pt\lieb_-)}

\def \uRpPhbpm{U^{\,\R,p}_{\!P,\hskip0,7pt\hbar}(\hskip0,5pt\lieb_\pm)}

\def \uRpPhnp{U^{\,\R,p}_{\!P,\hskip0,7pt\hbar}(\hskip0,5pt\lien_+)}
\def \uRpPhnm{U^{\,\R,p}_{\!P,\hskip0,7pt\hbar}(\hskip0,5pt\lien_-)}
\def \uRpPhnpm{U^{\,\R,p}_{\!P,\hskip0,7pt\hbar}(\hskip0,5pt\lien_\pm)}

\def \Ubqg {U_{\bq}(\hskip0,5pt\lieg)}
\def \Ubqh {U_{\bq}(\hskip0,5pt\lieh)}
\def \Ubqnm {U_{\bq}(\hskip0,5pt\lien_-)}
\def \Ubqnp {U_{\bq}(\hskip0,5pt\lien_+)}
\def \Ubqbm {U_{\bq}(\hskip0,5pt\lieb_-)}
\def \Ubqbp {U_{\bq}(\hskip0,5pt\lieb_+)}

\def \uqg{U_q(\hskip0,8pt\lieg)}

\newcommand{\sigmachi}{{\raise-1pt\hbox{$ \, \scriptstyle \dot\sigma_\chi $}}}

\def \pf{\begin{proof}}
\def \epf{\end{proof}}

\newcommand\fT{\mathsf{T}}
\newcommand\fH{\mathsf{H}}

\newcommand{\binomq}[2]{\genfrac{[}{]}{0pt}{}{#1}{#2}}

\theoremstyle{plain}

\begin{document}

\title[MULTIPARAMETER QUANTUM SUPERGROUPS, DEFORMATIONS, SPECIALIZATIONS]
 {Multiparameter Quantum Supergroups,  \\
  Deformations and Specializations}

\author[Gast{\'o}n Andr{\'e}s GARC{\'I}A \ , \ \  Fabio GAVARINI \ , \ \  Margherita PAOLINI]
{Gast{\'o}n Andr{\'e}s Garc{\'\i}a${}^\flat\,$,  \ \ Fabio Gavarini$\,{}^\sharp\,$,  \ \ Margherita Paolini${}^\natural$}

\address{
   \newline
 ${}^\flat$  Departamento de Matem\'atica, Facultad de Ciencias Exactas
   \newline
 Universidad Nacional de La Plata   ---   CMaLP-CIC-CONICET
   \newline
 1900 La Plata, ARGENTINA   ---   {\tt ggarcia@mate.unlp.edu.ar}
   \newline
 \phantom{A}  and
   \newline
 ${}^\flat$  Guangdong-Technion Israel Institute of Technology,
   \newline
 No.\ 241, Daxue Road, Shantou, Guangdong Province, China
 \vspace*{0.5cm}
   \newline
 ${}^\sharp$  Dipartimento di Matematica,
   \newline
 Universit\`a degli Studi di Roma ``Tor Vergata''   ---   INdAM\,/\,GNSAGA
   \newline
 Via della ricerca scientifica 1,  \; I\,-00133 Roma, ITALY   ---   {\tt gavarini@mat.uniroma2.it}
 \vspace*{0.5cm}
   \newline
  ${}^\natural$  Independent researcher   ---   {\tt margherita.paolini.mp@gmail.com} }

\thanks{\noindent 2020 \emph{MSC:}\,  17B37, 17B62   ---
}


\begin{abstract}
 \medskip
   In this paper we introduce a multiparameter version of the quantum universal enveloping superalgebras introduced by Yamane in  \cite{Ya1}.  For these objects we consider:
                                                                       \par
   \hskip-5pt   --- \textsl{(1)}\;  their deformations by twist and by 2--cocycle (both of ``toral type''); in particular, we prove that this family is stable under both types of deformations;
                                                                       \par
   \hskip-5pt   --- \textsl{(2)}\;  their semiclassical limits, which are multiparameter Lie superbialgebras;
                                                                       \par
   \hskip-5pt   --- \textsl{(3)}\;  the deformations by twist and by 2--cocycle (of ``toral type'') of these multiparameter Lie superbialgebras: in particular, we prove that this family is stable under these deformations, and that ``quantization commutes with deformation''.
%
 \vskip9pt
   \hfill   \textit{To the memory of Pierre Cartier,}   \qquad  {\ }
                                          \par
   \hfill   \textit{with our deepest admiration.}   \quad \qquad  {\ }
%
\end{abstract}

{\ } \vskip-65pt

   \centerline{ \smallrm  {\smallsl Communications in Contemporary Mathematics\/}  (to appear)}
%
                                                 \par
   \centerline{ \smallrm {\ } }
%
%
 \vskip1pt
   \centerline{\smallrm {\smallsl \ }}
%
 {\ }
\vskip1pt

\maketitle

\tableofcontents



\section{Introduction}  \label{sec: intro}

\vskip7pt

   In the last forty years, quantum groups
 --- mainly in the form of quantized universal enveloping algebras, hereafter shortened into QUEA's ---
 have been thoroughly studied, both in their  \textsl{formal\/}  version (\`a la Drinfeld, say),
 and in their  \textsl{polynomial\/}  formulation (following Jimbo and Lusztig).
 The semiclassical limit underlying them is given by Lie bialgebras (infinitesimally)
 and Poisson-Lie groups (globally): conversely, the ``quantization problem'' for
 a Lie bialgebra consists in the quest for some QUEA whose semiclassical
 limit yields back the given Lie bialgebra.
 \vskip5pt
   Quite soon,  \textit{multiparameter\/}  quantum groups have been introduced too, including those coming from the FRT-construction, as well as multiparameter QUEA's   --- or MpQUEA's, in short.
   A standard method for that uses the technique of  \textit{deformation\/}  of Hopf algebras,
   which comes in two versions: deformation by twist   --- where only the comultiplication
   and the antipode are modified ---   and deformation by 2--cocycle
   --- where one changes only the multiplication and the antipode.
   The twist deformation was applied by Reshetikhin to Drinfeld's formal QUEA's
   $ \uhg \, $,  thus finding some multiparameter QUEA's in which the new,
   ``additional'' parameters show up to describe the new coalgebra structure.
   The second method instead was applied by several people, all working with Jimbo \& Lusztig's
   polynomial QUEA  $ \uqg $   --- in a suitable ``quantum double version'' ---
   thus providing some other multiparameter QUEA's whose parameters enter
   in the description of the (new) algebra structure.  A common,  \textsl{unified\/}
   point of view on this matter was presented in  \cite{GG3}  (see also  \cite{GG1}):
   the key message there is that, although the two families of MpQUEA's arising
   from deformations by twist or by 2--cocycle seem to be entirely disconnected,
   they form instead one and the same family.  In addition, a second key result from  \cite{GG3}  is that
   deformations of either kind of any MpQUEA still provide some new MpQUEA's,
                                               \hbox{in this same family.}
                                                                          \par
   Looking at MpQUEA's, it was quickly realized that their many parameters could always be
   re-arranged so that only one of them kept the role of ``quantization parameter'',
   while all the other took instead a ``geometrical'' meaning.
   Namely, when one specializes the single ``quantum parameter''
   and thus gets the semiclassical limit, the remaining parameters enter in the description
   of the Poisson structure (in the form of a Lie cobracket) that is induced
 onto the underlying Lie algebra.  In this way, specialization
 of MpQUEA's yields  \textsl{multiparameter\/}
 structures onto the semiclassical limit, so one ends up with multiparameter
 Lie bialgebras (or MpLbA's, in short).  On the other hand, one can also introduce
 this family of MpLbA's via an independent construction, and prove that this family
 is stable under deformations (in the sense of Lie bialgebra theory) by twist or by
 2--cocycle; see  \cite{GG3}  for details.
 \vskip5pt
  The present paper extends the previous world-building process to the setup
  of quantum  \textsl{super\/}groups:  namely, we introduce  \textit{formal multiparameter quantized enveloping  \textsl{super}algebras\/}  (in short  \textit{formal MpQUESA}'s,  or just  \textit{FoMpQUESA}'s) and we study their deformations, by twist and by 2--cocycle.
                                                                            \par
   Both Lie superalgebras and their quantizations, the quantized universal enveloping superalgebras (QUESA),
can be regarded as Lie algebras and quantum universal enveloping algebras in the category of super vector spaces.
%
%
 As noted by Majid \cite{Mj}, both classical and super objects are related via the
process of the so-called \emph{bosonization} which yields a functor from one category to the
other. Despite this strong connection, Lie superalgebras and QUESA's are interesting on their own,
mainly because of their wide range of applications.
Simple Lie superalgebras over an algebraically closed field of characteristic zero
were classified by Kac in \cite{Ka2}:  like for simple Lie algebras, this classification is given in terms of Cartan matrices and (generalized) Dynkin diagrams.  Nevertheless, a single Lie superbialgebra may have non-isomorphic Dynkin diagrams.
 As the study of MpQUEA's (and their semiclassical limit) was mostly bounded
 to those whose semiclassical limits were Lie algebras of contragredient
type with symmetrizable Cartan matrices, we focus onto MpQUESA's whose semiclassical limit are simple contragredient Lie superalgebras of type  $ A $--$ G $
(w.r.t.\ the classification in  \cite{Ka2}).  These QUESA's were first introduced by
Yamane, in  \cite{Ya1},  who defined formal QUESA's of type  $ A $--$ G $:  for them, relations are not only commutation relations and super quantum Serre relations, as one also needs to add higher order relations.
Likewise, the MpLSbA's that we construct are, as Lie superalgebras, the contragredient ones of type  $ A $--$ G $,  again.
 \vskip5pt
   Specifically, we introduce multiparameter versions of Yamane's formal QUESA's  (in short, FoMpQUESA's), see  Definition \ref{def: FoMpQUESA},  by mimicking the strategy we developed in  \cite{GG3}  for the non-super case.  The very definition is based upon the connection between
the multiparameters and the action of a fixed commutative subsuperalgebra:
this is encoded in the notion of  \textit{realization\/}  of a multiparameter matrix  $ P $,
much like for Kac-Moody algebras. This allows us to relate the quantum objects
with their semiclassical limit, and also multiparameter objects with Yamane's standard ones:
the latter step goes through deformation(s).
%
%
 To do this, we consider deformations by toral twists of Yamane's uniparameter QUESA's,
just like Reshetikhin did with Drinfeld's uniparameter QUEA's, thus finding suitable new, deformed QUESA's that are now ``multiparameter'' ones, as they depend on the many parameters coming from the twist:
 these parameters, though, only show up in the description of the coalgebra structure,
 while the algebra structure is independent of them.
 As a consequence, one may endow these deformed QUESA's with a Hopf superalgebra
structure.
The second, key idea is to perform on these MpQUESA's a careful ``change of presentation'',
modifying the original generators (from Yamane's presentation) into new ones: the outcome is a
new presentation where the parameters appear in the defining relations, whereas the description
of the coalgebra structure now is entirely parameter-free.  Finally, we take this last presentation as
``the'' pattern for our definition of MpQUESA, given again by generators and relations,
mimicking the last presentation of the deformed MpQUEA's that we found after the ``change-the-generators'' step.
The arguments used in this analysis apply again to study, more in general, the deformations by twist of  \textsl{any\/}
one of our MpQUESA's: thus we find that any such deformation is yet another MpQUESA   --- this is proved in  Theorem
\ref{thm: uRpPhg=twist-uhg}.  Moreover, we prove in  Proposition \ref{prop: deform-Yamane=FoMpQUESA}  that any FoMpQUESA
can be obtained from a Yamane's uniparameter QUESA by a proper choice of the twist.  Then we draw our attention to an extension of the dual notion of twist,  \emph{polar--$2$--cocycles},  and show in  Theorem \ref{thm: double FoMpQUESAs_P-P'_mutual-deform.s}  that actually the family of FoMpQUESA is again stable by deformations of the superalgebra structure by a  polar--$2$--cocycle.  In conclusion, any FoMpQUESA can be obtained from Yamane's (uniparameter) QUESA either by a twist deformation
%
%
 or by a  polar--$2$--cocycle  deformation   --- the former fact is fairly natural, but the latter instead is quite surprising indeed.  So, for every FoMpQUESA one can decide
 to focus the dependence on the discrete multiparameters either on the supercoalgebra structure or on the superalgebra structure: in fact, we chose the second option.  As in the non-super setting  \cite{GG3},  everything is encoded in the notion of ``realization'' of multiparameters.
 \vskip5pt
   Quite naturally, the same strategy used in the quantum setup can be adopted as well in the semiclassical framework, even with some simplifications.  Therefore, we develop our construction of MpLSbA's, and the study of their deformations,  \textsl{independently\/}  of the quantum setup   --- though everything about that might be deduced by specialization from the quantum level (cf.\  Theorem \ref{thm: semicl-limit FoMpQUESA}).  We obtain then a self-contained ``chapter'' of Lie superbialgebra theory, that introduces a new family of multiparameter Lie superbialgebras
%
%
 and their deformations, the key fact being that this family is stable under deformations (by twist and by 2--cocycle), see  Theorem \ref{thm: twist-liegRP=new-liegR'P'}  and  Theorem \ref{thm: 2-cocycle-def-MpLSbA}.  These MpLSbA's come equipped with a presentation ``\`a la Serre'', in which the parameters rule the Lie superalgebra structure (cf. \S \ref{subsec: MpLSbA's}).  In addition, every such MpLSbA admits an alternative presentation, in which the Lie algebra structure stands fixed, while the Lie coalgebra structure varies according to the multiparameter matrix  $ P $.  Like in the quantum setup, the isomorphism between the two presentations is quite meaningful, as it boils down to a well-chosen change of generators.  In short, we might say that our multiparameters are encoded in realizations, FoMpQUESA's are quantizations of MpLSbA's, and multiparameter objects are given by deformations of either the (super)algebra or the (super)coalgebra structure of uniparameter objects   --- see a diagrammatic summary in  \S \ref{subsec: global synopsis}.
 In all this,  \textsl{the underlying Lie superalgebras are the (well-known) simple, contragredient ones of type  $ A $--$ G \, $  (in Kac' classification)\/}:  we are thus providing a new contribution to a classical topic of utmost importance.  Even more, indeed,  \textsl{all of our construction also extends,  \textit{verbatim},  to the setup of  \textit{affine}  Lie superalgebras and their quantum counterparts},  as introduced in  \cite{Ya2}  and  \cite{Ya3}.
%
 \vskip5pt
   Finally, we briefly introduce  \textit{polynomial\/}  MpQUESA's: mimicking the recipe of Lusztig \& Jimbo for polynomial QUEA's, these are certain subalgebras (up to modifying the ground ring) of our FoMpQUESA's.  We also study both types of deformations for these objects.  Now the  polar--$ 2 $--cocycles  become actually Hopf  $ 2 $--cocycles,  and the deformation of the algebraic structure is well-understood.  On the other hand, the twist deformations need to satisfy an arithmetic condition in order to be meaningful; nevertheless, these types of deformations cannot be seen,  \textit{strictu sensu},  as genuine ``twist deformations'' of our polynomial MpQUESA's.
 In fact, ``formal'' quantum groups are better suited to deal with deformations (up to resorting to the broader notion of ``polar--$ 2 $--cocycles''),  so we focused mainly onto them.
                                                                  \par
   One can also consider multiparameter versions of quantum formal series superalgebras on supergroups: in  \cite{GP}  this is done by constructing such an object as dual to the FoMpQUESA's of the present paper.  The same goal was pursued in  \cite{Ma}  via an FRT-construction: the link between  \cite{Ma}  and  \cite{GP}  is discussed in  \cite{GP}  as well.
%
 \vskip5pt
   A last word about the organization of the paper, and its detailed content.
                                                                  \par
   In  Section \ref{sec: Cartan-data_realiz's}  we introduce the ``combinatorial data'' underlying our constructions of MpLSbA's and FoMpQUESA's alike: the notion of multiparameters, their realizations and deformations.  Then in  Section \ref{sec: MpLSbAs}  we present our MpLSbA's, given by an explicit presentation (involving parameters) and bearing a Lie bialgebra structure which arise as deformation of one arising from Yamane's QUESA's.  We also study their deformations, both by (toral) twists and by (toral) 2--cocycles:
in particular, this proves that the family of these Lie superbialgebras is  \textsl{stable w.r.t.\ deformations}. In addition, before presenting our new results about multiparameter objects and their deformations, we quickly resume the general theory of Lie superbialgebras and their deformations, as the literature on the subject seems to be quite poor.
                                                                  \par
   Section \ref{sec: FoMpQUESA's}  is dedicated to our newly minted notion of FoMpQUESA's.  We begin by introducing them as superalgebras with a given presentation (by generators and relations), which closely mimics that of Yamane's QUESA's in  \cite{Ya1}  but for involving more general parameters.  Then we prove that they do bear a structure of Hopf superalgebra, which is obtained from that of a usual Yamane's QUESA through deformation by a suitable (toral) twist followed by a suitable ``change of variables''.  Finally, we prove the functoriality (w.r.t.\ the underlying, Cartan-like combinatorial data) of our construction of FoMpQUEA's, and their triangular decomposition.
                                                                  \par
   We spend Section  \ref{sec: deform-FoMpQUESA's}  to discuss deformations of FoMpQUEA's by (toral) twists and by (toral) polar--2--cocycles.  In short, we prove that these deformations turn FoMpQUEA's into new FoMpQUEA's again, hence the family of all FoMpQUESA's is  \textsl{stable by such deformations}.  Moreover, under mild assumptions things go the other way around too, so that any FoMpQUESA's can be realized as a suitable deformation --- both by (toral) twist and by (toral) polar--2--cocycle ---   of one of Yamane's uniparameter QUESA's.
                                                                  \par
   In Section  \ref{sec: FoMpQUESA's-vs-MpLSbA's}  we compare the classical and quantum setups.  Namely, we perform specializations of FoMpQUESA's, showing that their semiclassical limit is always a MpLSbA, with the same underlying ``Cartan datum''.  Conversely, any possible MpLSbA arises as such a limit: in other words, for any MpLSbA there is a suitable FoMpQUESA which is quantization of it.  As a second step, we compare deformations before and after specialization/quantization: the outcome is, in a nutshell, that ``specialization/quantization and deformation commute with each other''.
                                                                  \par
   Finally, in Section \ref{sec: polynomial MpQUESA's & deform.'s}  we introduce the notion of  \textit{polynomial MpQUESA's},  and we study their behavior under the processes of both types of deformations.

\vskip11pt

   \centerline{\ssmallrm ACKNOWLEDGEMENTS}
 \vskip3pt
   {\smallrm The authors thankfully acknowledge support by:
   \;\textit{(1)}\,  the  {\smallsl Research-in-Pairs\/}  program of  {\smallsl CIMPA-ICTP\/};
   \;\textit{(2)}\,  CONICET, ANPCyT, Secyt (Argentina),
   \;\textit{(3)}\,  the MIUR  {\smallsl Excellence Department Project
   MatMod@TOV (CUP E83C23000330006)\/}  awarded to the Department of Mathematics,
   University of Rome ``Tor Vergata'' (Italy),
   \;\textit{(4)}\,  the research branch  {\smallsl GNSAGA\/}  of INdAM (Italy),
   \;\textit{(5)}\,  the European research network  {\smallsl COST Action CaLISTA CA21109}.}

\bigskip
 \medskip

\section{Multiparameters and their realizations}  \label{sec: Cartan-data_realiz's}
%
 \vskip7pt
   In this section we fix the basic combinatorial data that we need.  The definition of our multiparameter Lie superbialgebras and formal multiparameter quantum super groups requires a full lot of related material that we now present.  In particular,  $ \, \NN = \{0, 1,\ldots\} \, $  and  $ \, \NN_+ := \NN \setminus \{0\} \, $,
   \,while  $ \k $  will be a field of characteristic zero.  This material is actually taken from  \cite{GG3}, \S 2,  up to minimal changes: we shall then skip many details (and proofs), as the reader can find them in  \cite{GG3}.

\medskip

\subsection{Cartan data, multiparameter matrices, and realizations}
  \label{MpMatrices, Cartan & realiz.'s}  {\ }
 \vskip7pt
   We introduce hereafter the ``multiparameters'', which we will use to construct (semi)classical and quantum objects as well.  The theory can be developed in wide generality, yet we restrict ourselves to the setup more closely related with (contragredient) Lie superalgebras of simple type   --- the so-called  \textit{basic\/}  ones.

\vskip9pt

\begin{free text}  \label{Cartan-data}
 {\bf Cartan data.}
 Let  $ \lieg $  be a complex, finite-dimensional simple Lie superalgebra.
 By the classification results due to Kac   --- cf.\ \cite{Ka2}  ---
 we know that the isomorphism class of  $ \lieg $  falls within a family in a finite list,
 namely  $ A_N $,  $ B_N $,  $ C_N $,  $ D_N $,  $ D(2\,,1;\alpha) $,  $ F_4 $  and
 $ G_3 \, $:  we discard hereafter the case  $ D(2,1;\alpha) $,  as we build upon
 \cite{Ya1},  which also drops it.  In short, we shall refer to these as
 ``\textit{basic\/}  Lie superalgebras''   --- of type  $ A \, $,  $ B \, $,  etc.\ up to
 $ G \, $,  in short  $ A-G \, $.
                                                                    \par
   Every basic Lie superalgebra  $ \lieg $  can be realized as a
   \textit{contragredient  \textsl{Lie superalgebra},  or a
   \textsl{suitable quotient}  of it},  following Kac' recipe in  \cite[\S 2.5.1]{Ka2}:
   this involves fixing a set of (positive) simple roots, say
   $ \, \alpha_1 , \dots , \alpha_n \, $,  with their parities
   --- given by a map  $ \, p : I \relbar\joinrel\relbar\joinrel\longrightarrow \ZZ_2 \, $  ---   and an associated
   ``Cartan matrix''  $ \, A = {\big(\, a_{ij} \big)}_{i, j \in I}
   \in \textsl{M}_{n}(\ZZ) \, $   --- where  $ \, n \in \NN_+ \, $
   is the  \textsl{rank\/}  of  $ \lieg $  and  $ \, I := \{1,\dots,n\} \, $  ---
   as the starting point of the construction.  Note that the matrix  $ A \, $,
   besides having integral coefficients, is also symmetrizable: indeed, we can (and will) also fix ``minimal'' elements  $ \, d_i \in \frac{\,1\,}{\,2\,} \ZZ \setminus \{0\} \, (\,i = 1, \dots, n \,) \, $  such that  $ \, DA \, $  is symmetric with integral entries, where  $ D $  is the diagonal matrix  $ \, D := {\big(\, d_i \, \delta_{ij} \big)}_{i, j \in I} \, $.
                                                                   \par
   We stress the fact that for the same basic Lie superalgebra
   $ \lieg $  one has different pairs  $ (A,p) $  that yield a realization of  $ \lieg $
   as the contragredient Lie superalgebra associated with that pair.
   In turn, every such pair corresponds (bijectively) to a suitable
   \textsl{Dynkin diagram},  after the recipe explained in  \cite[\S 2.5 (Tables IV--V)]{Ka2}:
   in particular, by default we ``incorporate'' the parity in the Dynkin diagram itself, by
   labelling each vertex as explained in  \cite{Ka2},
   so that the label itself yields the parity of the corresponding (vertex, hence) simple root.
   Namely, the label for a vertex in a Dynkin diagram is a ``colour'', which is ``white'',
   ``black'' or ``grey'' respectively for the vertex
   ``$ \,\circ \, $'',  ``$ \, \bullet \, $''  or  ``$ \, \otimes \, $'':
   then the parity is  $ \zero $  for white vertices and it is  $ \one $
   for black and for grey vertices. Therefore, we can construct our basic Lie superalgebras
   directly from these (labelled) Dynkin diagrams.  Still following  \cite{Ya1},
   we assume that in cases  $ F(4) $  and  $ G(3) $  the Dynkin is
   \textit{distinguished},  i.e.\ it has exactly one vertex which is odd.
     \par
   In the end, the list of all Dynkin diagrams we will deal with (arranged by type
   $ A $,  $ B $,  etc.) is the following (with the symbol  ``$ \, \times \, $''  standing
   for a white or grey vertex):
%
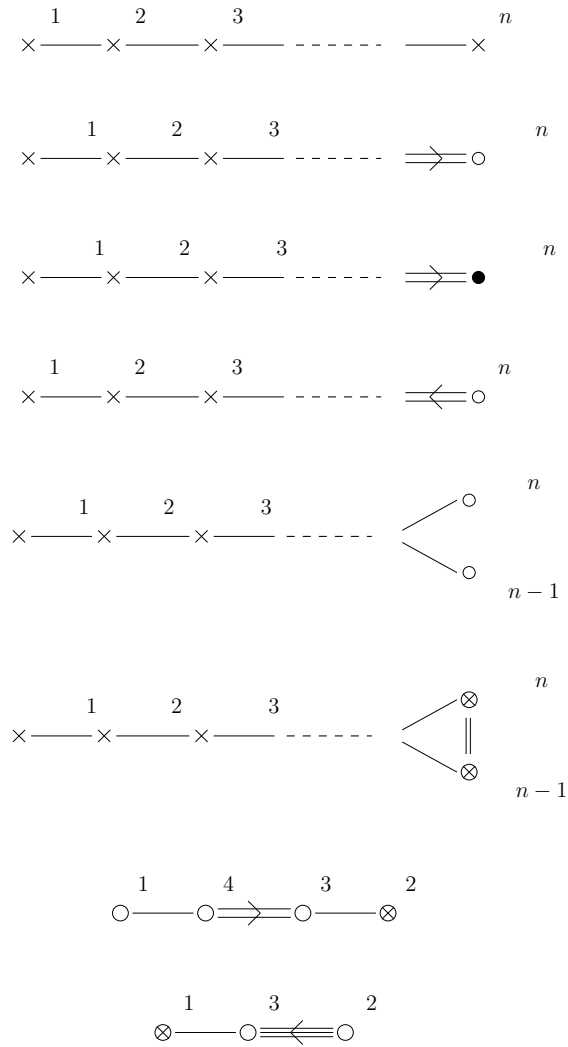
\begin{figure}[H]
    \centering
\begin{tikzpicture}[scale=0.8]   \quad  $ A_n $
\tikzset{cross/.style={cross out, draw,
         minimum size=2*(#1-\pgflinewidth),
         inner sep=0pt, outer sep=0pt}}
\draw[black, very thin] (2.2, 0) -- (3.2, 0);
\draw[black, very thin] (3.6, 0) -- (4.8, 0);
\draw[black, very thin] (5.2, 0) -- (6.2, 0);
\draw[black,dashed] (6.4, 0) -- (7.8, 0);
\draw[black, very thin] (8.2, 0) -- (9.2, 0);

\draw[black, very thin] (2-0.1, -0.1) -- (2+0.1, 0.1);
\draw[black, very thin] (2-0.1, 0.1) -- (2+0.1, -0.1);

\draw[black, very thin] (3.4-0.1, -0.1) -- (3.4+0.1, 0.1);
\draw[black, very thin] (3.4-0.1, 0.1) -- (3.4+0.1, -0.1);

\draw[black, very thin] (5-0.1, -0.1) -- (5+0.1, 0.1);
\draw[black, very thin] (5-0.1, 0.1) -- (5+0.1, -0.1);

\draw[black, very thin] (9.4-0.1, -0.1) -- (9.4+0.1, 0.1);
\draw[black, very thin] (9.4-0.1, 0.1) -- (9.4+0.1, -0.1);

\draw[black, very thin] (2,0)
node[above=0.2cm, scale=0.7] {1};

\draw[black, very thin] (3.4,0)
node[above=0.2cm, scale=0.7] {2};

\draw[black, very thin] (5,0)
node[above=0.2cm, scale=0.7] {3};

\draw[black, very thin] (9.4,0)
node[above=0.2cm, scale=0.7] {$n$};
\end{tikzpicture}
\end{figure}

\begin{figure}[H]
    \centering
\begin{tikzpicture}[scale=0.8]   $ B_n \phantom{i} (I) $   \qquad
\draw[black, very thin] (2.2, 0) -- (3.2, 0);
\draw[black, very thin] (3.6, 0) -- (4.8, 0);
\draw[black, very thin] (5.2, 0) -- (6.2, 0);
\draw[black,dashed] (6.4, 0) -- (7.8, 0);
\draw[black, very thin] (8.2, 0.07) -- (9.2, 0.07);
\draw[black, very thin] (8.2, -0.07) -- (9.2, -0.07);
\draw[black, very thin] (8.7-0.1, -0.1-0.1) -- (8.7+0.1, 0.1-0.1);
\draw[black, very thin] (8.7-0.1, 0.1+0.1) -- (8.7+0.1, -0.1+0.1);

\draw[black, very thin] (2-0.1, -0.1) -- (2+0.1, 0.1);
\draw[black, very thin] (2-0.1, 0.1) -- (2+0.1, -0.1);

\draw[black, very thin] (3.4-0.1, -0.1) -- (3.4+0.1, 0.1);
\draw[black, very thin] (3.4-0.1, 0.1) -- (3.4+0.1, -0.1);

\draw[black, very thin] (5-0.1, -0.1) -- (5+0.1, 0.1);
\draw[black, very thin] (5-0.1, 0.1) -- (5+0.1, -0.1);

\draw [black, very thin](9.4, 0) circle (0.1 cm);


\draw[black, very thin] (2,0)
node[above=0.2cm, scale=0.7] {1};

\draw[black, very thin] (3.4,0)
node[above=0.2cm, scale=0.7] {2};

\draw[black, very thin] (5,0)
node[above=0.2cm, scale=0.7] {3};

\draw[black, very thin] (9.4,0)
node[above=0.2cm, scale=0.7] {$n$};
\end{tikzpicture}
\end{figure}

\begin{figure}[H]
    \centering
\begin{tikzpicture}[scale=0.8]   $ B_n \phantom{i} (I\!I) $   \qquad
\draw[black, very thin] (2.2, 0) -- (3.2, 0);
\draw[black, very thin] (3.6, 0) -- (4.8, 0);
\draw[black, very thin] (5.2, 0) -- (6.2, 0);
\draw[black,dashed] (6.4, 0) -- (7.8, 0);
\draw[black, very thin] (8.2, 0.07) -- (9.2, 0.07);
\draw[black, very thin] (8.2, -0.07) -- (9.2, -0.07);
\draw[black, very thin] (8.7-0.1, -0.1-0.1) -- (8.7+0.1, 0.1-0.1);
\draw[black, very thin] (8.7-0.1, 0.1+0.1) -- (8.7+0.1, -0.1+0.1);

\draw[black, very thin] (2-0.1, -0.1) -- (2+0.1, 0.1);
\draw[black, very thin] (2-0.1, 0.1) -- (2+0.1, -0.1);

\draw[black, very thin] (3.4-0.1, -0.1) -- (3.4+0.1, 0.1);
\draw[black, very thin] (3.4-0.1, 0.1) -- (3.4+0.1, -0.1);

\draw[black, very thin] (5-0.1, -0.1) -- (5+0.1, 0.1);
\draw[black, very thin] (5-0.1, 0.1) -- (5+0.1, -0.1);

\draw [black, very thin, fill](9.4, 0) circle (0.1 cm);


\draw[black, very thin] (2,0)
node[above=0.2cm, scale=0.7] {1};

\draw[black, very thin] (3.4,0)
node[above=0.2cm, scale=0.7] {2};

\draw[black, very thin] (5,0)
node[above=0.2cm, scale=0.7] {3};

\draw[black, very thin] (9.4,0)
node[above=0.2cm, scale=0.7] {$n$};
\end{tikzpicture}
\end{figure}

\begin{figure}[H]
    \centering
\begin{tikzpicture}[scale=0.8]   $ C_n $   \qquad
\draw[black, very thin] (2.2, 0) -- (3.2, 0);
\draw[black, very thin] (3.6, 0) -- (4.8, 0);
\draw[black, very thin] (5.2, 0) -- (6.2, 0);
\draw[black,dashed] (6.4, 0) -- (7.8, 0);
\draw[black, very thin] (8.2, 0.07) -- (9.2, 0.07);
\draw[black, very thin] (8.2, -0.07) -- (9.2, -0.07);
\draw[black, very thin] (8.7+0.1, +0.1+0.1) -- (8.7-0.1, 0.1-0.1);
\draw[black, very thin] (8.7+0.1, -0.1-0.1) -- (8.7-0.1, -0.1+0.1);

\draw[black, very thin] (2-0.1, -0.1) -- (2+0.1, 0.1);
\draw[black, very thin] (2-0.1, 0.1) -- (2+0.1, -0.1);

\draw[black, very thin] (3.4-0.1, -0.1) -- (3.4+0.1, 0.1);
\draw[black, very thin] (3.4-0.1, 0.1) -- (3.4+0.1, -0.1);

\draw[black, very thin] (5-0.1, -0.1) -- (5+0.1, 0.1);
\draw[black, very thin] (5-0.1, 0.1) -- (5+0.1, -0.1);

\draw [black, very thin](9.4, 0) circle (0.1 cm);


\draw[black, very thin] (2,0)
node[above=0.2cm, scale=0.7] {1};

\draw[black, very thin] (3.4,0)
node[above=0.2cm, scale=0.7] {2};

\draw[black, very thin] (5,0)
node[above=0.2cm, scale=0.7] {3};

\draw[black, very thin] (9.4,0)
node[above=0.2cm, scale=0.7] {$n$};
\end{tikzpicture}
\end{figure}

\begin{figure}[H]
    \centering
\begin{tikzpicture}[scale=0.8]   $ D_n \phantom{i} (I) $   \quad
\draw[black, very thin] (2.2, 0) -- (3.2, 0);
\draw[black, very thin] (3.6, 0) -- (4.8, 0);
\draw[black, very thin] (5.2, 0) -- (6.2, 0);
\draw[black,dashed] (6.4, 0) -- (7.8, 0);
\draw[black, very thin] (8.3, 0.1) -- (9.2, 0.6);
\draw[black, very thin] (8.3, -0.1) -- (9.2, -0.6);

\draw[black, very thin] (2-0.1, -0.1) -- (2+0.1, 0.1);
\draw[black, very thin] (2-0.1, 0.1) -- (2+0.1, -0.1);

\draw[black, very thin] (3.4-0.1, -0.1) -- (3.4+0.1, 0.1);
\draw[black, very thin] (3.4-0.1, 0.1) -- (3.4+0.1, -0.1);

\draw[black, very thin] (5-0.1, -0.1) -- (5+0.1, 0.1);
\draw[black, very thin] (5-0.1, 0.1) -- (5+0.1, -0.1);

\draw [black, very thin](9.4, 0.6) circle (0.1 cm);
\draw [black, very thin](9.4, -0.6) circle (0.1 cm);


\draw[black, very thin] (2,0)
node[above=0.2cm, scale=0.7] {1};

\draw[black, very thin] (3.4,0)
node[above=0.2cm, scale=0.7] {2};

\draw[black, very thin] (5,0)
node[above=0.2cm, scale=0.7] {3};

\draw[black, very thin] (9.4,0.6)
node[above=0.05cm, scale=0.7] {$n$};

\draw[black, very thin] (9.4,-0.6)
node[below=0.05cm, scale=0.7] {$n-1$};
\end{tikzpicture}\\ 
\end{figure}

\begin{figure}[H]
    \centering
\begin{tikzpicture}[scale=0.8]   $ D_n \phantom{i} (I\!I) $   \quad
\draw[black, very thin] (2.2, 0) -- (3.2, 0);
\draw[black, very thin] (3.6, 0) -- (4.8, 0);
\draw[black, very thin] (5.2, 0) -- (6.2, 0);
\draw[black,dashed] (6.4, 0) -- (7.8, 0);
\draw[black, very thin] (8.3, 0.1) -- (9.2, 0.6);
\draw[black, very thin] (8.3, -0.1) -- (9.2, -0.6);

\draw[black, very thin] (9.35, 0.3) -- (9.35, -0.3);
\draw[black, very thin] (9.42, 0.3) -- (9.42, -0.3);

\draw[black, very thin] (2-0.1, -0.1) -- (2+0.1, 0.1);
\draw[black, very thin] (2-0.1, 0.1) -- (2+0.1, -0.1);

\draw[black, very thin] (3.4-0.1, -0.1) -- (3.4+0.1, 0.1);
\draw[black, very thin] (3.4-0.1, 0.1) -- (3.4+0.1, -0.1);

\draw[black, very thin] (5-0.1, -0.1) -- (5+0.1, 0.1);
\draw[black, very thin] (5-0.1, 0.1) -- (5+0.1, -0.1);

\draw [black, very thin](9.4, 0.6) circle (0.13 cm);
\draw [black, very thin](9.4, -0.6) circle (0.13 cm);

\draw[black, very thin] (9.4-0.1, 0.1+0.6) -- (9.4+0.1, -0.1+0.6);
\draw[black, very thin] (9.4-0.1, -0.1+0.6) -- (9.4+0.1,  0.1+0.6);

\draw[black, very thin] (9.4-0.1, -0.1-0.6) -- (9.4+0.1, 0.1-0.6);
\draw[black, very thin] (9.4-0.1, 0.1-0.6) -- (9.4+0.1, -0.1-0.6);

\draw[black, very thin] (2,0)
node[above=0.2cm, scale=0.7] {1};

\draw[black, very thin] (3.4,0)
node[above=0.2cm, scale=0.7] {2};

\draw[black, very thin] (5,0)
node[above=0.2cm, scale=0.7] {3};

\draw[black, very thin] (9.4,0.6)
node[above=0.07cm, scale=0.7] {$n$};

\draw[black, very thin] (9.4,-0.6)
node[below=0.04cm, scale=0.7] {$n-1$};
\end{tikzpicture}  \\
\end{figure}

\begin{figure}[H]
    \centering
\begin{tikzpicture}[scale=0.8]   $ F_4 $   \quad
\draw[black, very thin] (2.2, 0) -- (3.2, 0);

\draw[black, very thin] (3.6, 0.07) -- (4.8, 0.07);
\draw[black, very thin] (3.6, -0.07) -- (4.8, -0.07);
\draw[black, very thin] (4.2-0.1, -0.1-0.1) -- (4.2+0.1, 0.1-0.1);
\draw[black, very thin] (4.2-0.1, 0.1+0.1) -- (4.2+0.1, -0.1+0.1);

\draw[black, very thin] (5.2, 0) -- (6.2, 0);

\draw[black, very thin] (2,0)
node[above=0.2cm, scale=0.7] {1};

\draw[black, very thin] (3.4,0)
node[above=0.2cm, scale=0.7] {4};

\draw[black, very thin] (5,0)
node[above=0.2cm, scale=0.7] {3};

\draw[black, very thin] (6.4,0)
node[above=0.2cm, scale=0.7] {2};

\draw [black, very thin](2, 0) circle (0.13 cm);
\draw [black, very thin](3.4, 0) circle (0.13 cm);
\draw [black, very thin](5, 0) circle (0.13 cm);
\draw [black, very thin](6.4, 0) circle (0.13 cm);

\draw[black, very thin] (6.4-0.1, -0.1) -- (6.4+0.1, 0.1);
\draw[black, very thin] (6.4-0.1, 0.1) -- (6.4+0.1, -0.1);

\end{tikzpicture}
\end{figure}

\begin{figure}[H]
    \centering
\begin{tikzpicture}[scale=0.8]   $ G_3 $   \quad
\draw[black, very thin] (2.2, 0) -- (3.2, 0);
\draw[black, very thin] (3.6, 0.07) -- (4.8, 0.07);
\draw[black, very thin] (3.6, 0) -- (4.8, 0);
\draw[black, very thin] (3.6, -0.07) -- (4.8, -0.07);
\draw[black, very thin] (4.2+0.1, -0.1-0.1) -- (4.2-0.1, 0.1-0.1);
\draw[black, very thin] (4.2+0.1, 0.1+0.1) -- (4.2-0.1, -0.1+0.1);

\draw[black, very thin] (2,0)
node[above=0.2cm, scale=0.7] {1};

\draw[black, very thin] (3.4,0)
node[above=0.2cm, scale=0.7] {3};

\draw[black, very thin] (5,0)
node[above=0.2cm, scale=0.7] {2};

\draw [black, very thin](2, 0) circle (0.13 cm);
\draw [black, very thin](3.4, 0) circle (0.13 cm);
\draw [black, very thin](5, 0) circle (0.13 cm);

\draw[black, very thin] (2-0.1, -0.1) -- (2+0.1, 0.1);
\draw[black, very thin] (2-0.1, 0.1) -- (2+0.1, -0.1);
\end{tikzpicture}
\end{figure}
%
%
 \vskip5pt
   \textsl{Note\/}  that every Dynkin diagram as above defines a parity function  $ \, p : I \longrightarrow \ZZ_2 \, $,  given by  $ \, p(i) := \zero \, $  if the vertex  $ i $  ($ \in I \, $)  is white and  $ \, p(i) := \one \, $  if  $ i $  is grey or black.
\end{free text}

\vskip5pt

\begin{definition}  \label{def: Cartan-super-datum}
 We will call  \textit{Cartan super-datum\/}  any pair  $ (A,p) $  where  $ A $  is a Cartan matrix associated with one of the Dynkin diagrams listed in  \S \ref{Cartan-data}  above, and  $ \, p : I \longrightarrow \ZZ_2 \, $  is the parity function defined by that Dynkin diagram.   \hfill  $ \diamondsuit $
\end{definition}

\vskip7pt

  A basic tool in the subsequent construction is the following, taken from  \cite{GG3}, \S 2  (see also  \cite{Ka1}, Ch.\ 1):

\vskip11pt

\begin{definition}  \label{def: realization of P}
 Let  $ \hbar $  be a formal variable, and  $ \kh $  the ring of formal power series in  $ \hbar $
 with coefficients in  $ \k \, $.  Let  $ \lieh $  be a free  $ \kh $--module  of finite rank, and pick subsets
 $ \; \Pi^\vee := {\big\{ T^+_i , T^-_i \big\}}_{i \in I} \! \subseteq \lieh \; $,
 and
 $ \; \Pi := {\big\{ \alpha_i \big\}}_{i \in I} \subseteq \lieh^* := \Hom_\kh\!\big(\, \lieh \, , \kh \big) \; $.
 For later use, we also introduce the elements  $ \, S_i := 2^{-1} \big(\, T^+_i + T^-_i \big) \, $  and
 $ \,  \varLambda_i := 2^{-1} \big(\, T^+_i - T^-_i \big) \, $  (for  $ \, i \in I \, $)  and the sets
 $ \; \Sigma := {\big\{ S_i \big\}}_{i \in I} \! \subseteq \lieh \; $  and
 $ \; \Lambda := {\big\{ \varLambda_i \big\}}_{i \in I} \! \subseteq \lieh \; $.
 \vskip3pt
   Let  $ \, P \in M_n\big(\kh\big) \, $  be any  $ (n \times n) $--matrix  with entries in  $ \kh \, $.
 \vskip5pt
   \textit{(a)}\;  We call the triple  $ \; \R \, := \, \big(\, \lieh \, , \Pi \, , \Pi^\vee \,\big) \; $
   a  \textsl{realization\/}  of  $ P \, $ over $\kh$,
   with  \textsl{rank\/}  defined as  $ \, \rk(\R) := \rk_\kh(\lieh) \, $,  \,if:
 \vskip3pt
   \quad \textit{(a.1)} \;  $ \; \alpha_j\big(\,T^+_i\big) \, = \, p_{\,ij} \; $,
$ \;\;\; \alpha_j\big(\,T^-_i\big) \, = \, p_{j\,i} \; $,  \quad  for all  $ \, i, j \in I \, $;
 \vskip3pt
   \quad \textit{(a.2)} \;  the set  $ \, \overline{\Sigma} := {\big\{\, \overline{S}_i := S_i \; (\text{\,mod\ } \hbar\,\lieh \,) \big\}}_{i \in I} \, $
   is  $ \Bbbk $--linearly  independent in  $ \, \overline{\lieh} := \lieh \Big/ \hbar\,\lieh \, $   ---
   \textsl{N.B.:}\, this is equivalent to saying that  $ \Sigma $  itself can be completed to a  $ \kh $--basis  of  $ \lieh \, $,
   hence in particular  $ \Sigma $  is  $ \kh $--linearly  independent in  $ \, \lieh \, $.
 \vskip5pt
   \textit{(b)}\;  We call a realization  $ \; \R \, := \, \big(\, \lieh \, , \Pi \, , \Pi^\vee \,\big) \; $
   of the matrix  $ P $,  respectively,
 \vskip3pt
   \quad \textit{(b.1)} \;  \textsl{straight\/}  if the set  $ \, \overline{\Pi} := {\big\{\, \overline{\alpha}_i := \alpha_i \; (\text{\,mod\ } \hbar\,\lieh^* ) \big\}}_{i \in I}  \, $
   is  $ \Bbbk $--linearly  independent in  $ \, \overline{\lieh^*} := \lieh^* \Big/ \hbar\,\lieh^* \, $
   ---  \textsl{N.B.:}\, this is equivalent to saying that  $ \Pi $ can be completed to a  $ \kh $--basis  of  $ \lieh^* \, $,
   thus in particular  $ \Pi $  is  $ \kh $--linearly  independent in  $ \, \lieh^* \, $;
 \vskip3pt
   \quad \textit{(b.2)} \;  \textsl{small\/}  if
   $ \,\; \textsl{Span}_\kh\big( {\{ S_i \}}_{i \in I} \big) \; = \; \textsl{Span}_\kh\big( {\big\{ T_i^+ , T_i^- \big\}}_{i \in I} \big) \;\, $;
 \vskip3pt
   \quad \textit{(b.3)} \;  \textsl{split\/}  if the set
   $ \, \overline{\Pi^\vee} := {\big\{\, \overline{T^\pm}_i := T^\pm_i \; (\text{\,mod\ } \hbar\,\lieh \,) \big\}}_{i \in I} \, $
   is  $ \Bbbk $--linearly  independent in  $ \, \overline{\lieh} := \lieh \Big/ \hbar\,\lieh \, $   ---   \textsl{N.B.:}\,
   this is equivalent to saying that  $ \Pi^\vee $  can be completed to a  $ \kh $--basis  of  $ \lieh \, $,
   hence in particular it is  $ \kh $--linearly  independent in  $ \, \lieh \, $;
 \vskip3pt
   \quad \textit{(b.4)} \;  \textsl{minimal\/}  if
   $ \,\; \textsl{Span}_\kh\big( {\big\{ T_i^+ , T_i^- \big\}}_{i \in I} \big) \, = \; \lieh \;\, $
   ---   \textsl{N.B.:}\,  in particular,  $ \R $  is  \textsl{split\/  {\rm and}  minimal\/}  if and only if
   $ \, {\big\{\, T_i^+ , T_i^- \big\}}_{i \in I} \, $  is a  $ \kh $--basis  of  $ \lieh \, $.
 \vskip5pt
   \textit{(c)}\;  For any pair of realizations  $ \; \R \, := \, \big(\, \lieh \, , \Pi \, , \Pi^\vee \,\big) \; $  and  $ \; \dot{\R} \, := \, \big(\, \dot{\lieh} \, , \dot{\Pi} \, , {\dot{\Pi}}^\vee \,\big) \; $  of the same matrix  $ P $,
   a \textsl{(homo)morphism\/}  $ \, \underline{\phi} : \R \longrightarrow \dot{\R} \, $  is the datum of any  $ \kh $--module morphism  $ \, \phi : \lieh \longrightarrow \dot{\lieh} \, $  such that  $ \, \phi\big(T_i^\pm\big) = {\dot{T}}_{\sigma(i)}^\pm \, $  and  $ \, \phi^*\big(\dot{\alpha}_{\sigma(i)}\big) = \alpha_i \, $  (for all  $ \, i \in I \, $)  for some permutation  $ \, \sigma \in \mathbb{S}_{I} \, $   --- the symmetric group over  $ I $   ---   hence, in particular,  $ \, \phi\big(\Pi^\vee\big) = {\dot{\Pi}}^\vee \, $,  and also that  $ \, \phi^*\big(\dot{\Pi}\big) = \Pi \, $
   ---  \textsl{N.B.:}  realizations along with their morphisms form a category,
   in which the iso--/epi--/mono--morphisms are those morphisms  $ \phi $  as above that actually are  $ \kh $--module  iso--/epi--/mono--morphisms.
 \vskip5pt
   \textit{(d)}\;  Let  $ A $  be any Cartan matrix associated with a Dynkin diagram from the list given in  \S \ref{Cartan-data}, and let  $ \, D := {\big(\hskip0,7pt d_i \, \delta_{ij} \big)}_{i, j \in I} \, $  be the diagonal matrix associated with  $ A \, $,  as in  \S \ref{Cartan-data}.  We say that a square matrix  $ \, P \in M_n(\kh) \, $  is  \textsl{of Cartan type},  with associated Cartan matrix  $ A \, $,  if  $ \; P_s := 2^{-1} \big( P + P^{\,\scriptscriptstyle T} \big) = DA \; $.
 \vskip5pt
   \textit{N.B.:}\;  condition  \textit{(b.3)\/}  is equivalent  to requiring that
   $ \, \overline{\Sigma} \cup \overline{\Lambda} \, $  be  $ \Bbbk $--linearly  independent in
   $ \, \overline{\lieh} := \lieh \Big/ \hbar\,\lieh \, $;  \,in turn, this is equivalent to saying that
   $ \, \Sigma \cup \Lambda \, $  itself can be completed to a  $ \kh $--basis  of
   $ \lieh \, $,  hence in particular it is  $ \kh $--linearly  independent.  Similarly, the condition
   $ \, \phi\big(T_i^\pm\big) = {\dot{T}}_{\sigma(i)}^\pm \, $   ---  $ \, i \in I \, $,
   for some permutation  $ \, \sigma \in \mathbb{S}(I) \, $  ---   in  \textit{(c)\/}  can be replaced by
 $ \, \phi(S_i) = {\dot{S}}_{\sigma(i)} \, $  and  $ \, \phi(\varLambda_i) = {\dot{\varLambda}}_{\sigma(i)} \, $.
 \vskip5pt
   \textit{(e)}\;  In an entirely similar way, one may define realizations of a matrix
   $ \, P := {\big(\, p_{i,j} \big)}_{i, j \in I} \in M_n(\Bbbk) \, $  over a ground field  $ \Bbbk \, $.
   Such a realization  $ \, \R := \big(\, \lieh \, , \Pi \, , \Pi^\vee \,\big) \, $ consists of a  $ \Bbbk $--vector  space
   $ \lieh $  and distinguished subsets  $ \; \Pi^\vee := {\big\{ T^+_i , T^-_i \big\}}_{i \in I} \! \subseteq \lieh \; $
   and  $ \; \Pi := {\big\{ \alpha_i \big\}}_{i \in I} \subseteq \lieh^* := \Hom_\Bbbk\!\big(\, \lieh \, , \Bbbk \big) \; $:
   \,then condition  \textit{(a.1)\/}  reads the same, while  \textit{(a.2)\/}  instead says that the set
   $ \, {\big\{\, S_i = 2^{-1}\big(T^+_i + T^-_i\big) \,\big\}}_{i\in I} \, $  is linearly independent, and the
   \textsl{rank\/}  of the realization is  $ \, \rk(\R) := \dim_\Bbbk(\lieh) \, $.  Also,  $ \R $  is  \textsl{straight},
   resp.\  \textsl{split},  if  $ \Pi \, $,  resp.\  $ \Pi^\vee $,  is linearly independent.   \hfill  $ \diamondsuit $
%
%
\end{definition}

\vskip7pt

   The following consequence of the definitions yields a close link with Kac' notion:

\vskip11pt

\begin{lema}  \label{lem: realiz_P => realiz_A}  \textsl{(cf.\ \cite{GG3}, Lemma 2.1.4))}
 Let  $ \, P \in M_n(\kh) \, $  be a matrix as above.  If  $ \; \R := \big(\, \lieh \, , \Pi \, , \Pi^\vee \,\big) \, $
 is a straight realization of  $ \, P $,  then the triple  $ \, \big(\, \lieh \, , \Pi \, , \Pi^\vee_S \big) \, $
 --- with  $ \, \Pi^\vee_S \! := {\{S_i\}}_{i \in I} \, $  ---   is a realization of  $ \, P_s := 2^{-1} \big( P + P^{\,\scriptscriptstyle T} \big) \, $
 --- over the ring  $ \kh $  ---   in the sense of  \cite{Ka1}, Ch.\ 1, \S 1.1, but for condition (1.1.3).   \hfill  $ \square $
\end{lema}

\vskip7pt

   Note that condition (1.1.3) in  \cite[Ch.\ 1, \S 1.1]{Ka1},  is fulfilled whenever  $ \, \rk(\lieh) = 2\,n - \rk(P_s) \, $;
   \,in particular, we can always achieve that condition up to suitably enlarging or restricting  $ \lieh \, $.
   In any case, from now on with any straight realization of a matrix  $ P $  of Cartan type, for some Cartan matrix  $ A \, $,
   we shall always associate the realization of  $ \, P_s = DA \, $  given by  Lemma \ref{lem: realiz_P => realiz_A},
   hence also the corresponding realization of  $ A $  and then all the related data and machinery mentioned in  \S \ref{MpMatrices, Cartan & realiz.'s}.

\vskip9pt

   We collect now a few technical results:

\vskip9pt

\begin{prop}  \label{prop: exist-realiz's}  \textsl{(cf.\ \cite{GG3}, Proposition 2.1.5)}
 \vskip3pt
   (a)\;  For every  $ \, P \in M_n\big(\kh\big) \, $  and every  $ \, \ell \geq 3\,n-\rk\big(P+P^{\,\scriptscriptstyle T}\big) \, $,
   \,there exists a straight split realization of  $ P $  with  $ \, \rk(\lieh) = \ell \, $,
 which
 is unique up to isomorphisms.
 \vskip3pt
   (b)\;  Claim (a) still holds true if we drop the word ``straight'' and pick  $ \, \ell \geq 2\,n \, $.  $ \square $
\end{prop}

\vskip7pt

\begin{rmk}
 It follows from definitions that a  \textsl{necessary\/}  condition for a small realization of any  $ \, P \in M_n\big(\kh\big) \, $
 to exist is  $ \, \rk\!\big( P_s \,\big| P_a \big) = \rk(P_s) \, $.
 Conversely, with much the same arguments used in the proof of
 Proposition \ref{prop: exist-realiz's},  we can prove that such a condition is also  \textsl{sufficient},  as the following holds true, indeed:
\end{rmk}

\eject
   \vskip7pt

\begin{prop}  \label{prop: exist-realiz's_small}  \textsl{(cf.\ \cite{GG3}, Proposition 2.1.7)}
 \vskip3pt
 If  $ \, P \in M_n\big(\kh\big) \, $  is such that  $ \, \rk\!\big( P_s \,\big| P_a \big) = \rk(P_s) \, $,
 then, for all  $ \, \ell \geq 2\,n-\rk(P_s) \, $,
 there exists a straight small realization of  $ \, P $  with  $ \, \rk(\lieh) = \ell \, $,
\,and such a realization is unique up to isomorphisms.   \hfill  $ \square $
\end{prop}

\vskip7pt

   After this existence results concerning realizations of special type, we can achieve a more general result with two additional steps.  The first one tells in short that every realization can be ``lifted'' to a \textsl{split\/}  one:

\vskip9pt

\begin{lema}  \label{lemma: split/straight-lifting}  \textsl{(cf.\ \cite{GG3}, Lemma 2.1.8)}
 \vskip3pt
 Let  $ \; \R := \big(\, \lieh \, , \Pi \, , \Pi^\vee \,\big) \; $  be a realization of  $ \, P \in M_n\big(\kh\big) \, $.  Then:
 \vskip3pt
   \textit{(a)}\,  There exists a  \textsl{split}  realization  $ \; \dot{\R} := \big(\, \dot{\lieh} \, , \dot{\Pi} \, , {\dot{\Pi}}^\vee \,\big) \, $   of the same matrix  $ P $  and an epimorphism of realizations  $ \; \underline{\pi} : \dot{\R} \relbar\joinrel\twoheadrightarrow \R \; $  such that, if  $ \, \lieh_{{}_T} := \textsl{Span}\Big( {\big\{ T_i^\pm \big\}}_{i \in I} \Big) \, $  and  $ \, \dot{\lieh}_{{}_T} := \textsl{Span}\Big( {\big\{ \dot{T}_i^\pm \big\}}_{i \in I} \Big) \, $,  then  $ \, \underline{\pi} \, $  induces an isomorphism  $ \; \dot{\pi} : \dot{\lieh} \Big/ \dot{\lieh}_{{}_T} \cong \lieh \Big/ \lieh_{{}_T} \; $.
                                                               \par
   If in addition  $ \, \R $  is  \textsl{straight},  resp.\ \textsl{minimal},  then a  \textsl{split}  realization  $ \, \dot{\R} $
   as above can be found that is  \textsl{straight},  resp.\ \textsl{minimal},  as well.
 \vskip3pt
   \textit{(b)}\,  There exists a  \textsl{straight}  realization  $ \; \hat{\R} := \big(\, \hat{\lieh} \, , \hat{\Pi} \, , {\hat{\Pi}}^\vee \,\big) \, $   of the same matrix  $ P $  and an epimorphism of realizations  $ \; \underline{\pi} : \hat{\R} \relbar\joinrel\twoheadrightarrow \R \; $  such that, if  $ \, \lieh_{{}_T} := \textsl{Span}\Big( {\big\{ T_i^\pm \big\}}_{i \in I} \Big) \, $  and  $ \, \hat{\lieh}_{{}_T} := \textsl{Span}\Big( {\big\{ \hat{T}_i^\pm \big\}}_{i \in I} \Big) \, $,  then  $ \, \underline{\pi} \, $  induces an isomorphism  $ \; \hat{\pi} : \hat{\lieh}_{{}_T} \cong \lieh_{{}_T} \; $.
                                                               \par
   If in addition  $ \, \R $  is  \textsl{split},  resp.\ \textsl{minimal},  then a  \textsl{straight}  realization  $ \, \hat{\R} $
   as above can be found that is  \textsl{split},  resp.\ \textsl{minimal},  as well.
\end{lema}

\begin{proof}
   \textit{(a)}\,  This is nothing but Lemma 2.1.8 in  \cite{GG3}.
 \vskip5pt
   \textit{(b)}\,  Take  $ \, \hat{\lieh} := \lieh \oplus \lieh' \, $  for some  $ \kh $--module  $ \lieh' $  such that the  $ \alpha_i $'s  extend from  $ \lieh $  to all of  $ \hat{\lieh} $  so to give linearly independent elements  $ \hat{\alpha}_i $'s  of  $ \hat{\lieh} \, $:  e.g., choose  $ \lieh' $  to be a free  $ \kh $--module  of rank  $ |I| $  with  $ \kh $--basis  $ \, {\{e_j\}}_{j \in I} \, $  and set  $ \, \hat{\alpha}_i{\big|}_{\lieh} := \alpha_i \, $  and  $ \, \hat{\alpha}_i(e_j) := \delta_{i,j} \, $  for all  $ \, i, j \in I \, $.  Then  $ \; \hat{\R} := \big(\, \hat{\lieh} \, , \hat{\Pi} := {\big\{\, \hat{\alpha}_i \,\big\}}_{i \in I} \, , {\hat{\Pi}}^\vee := \Pi^\vee \,\big) \, $  is a  \textsl{straight\/}  realization of  $ P $  as required, which is even split, respectively minimal, if so is  $ \R \, $;  moreover, an epimorphism  $ \; \underline{\pi} : \hat{\R} \relbar\joinrel\twoheadrightarrow \R \; $  is given by  the natural projection  $ \, \pi : \hat{\lieh} := \lieh \oplus \lieh' \relbar\joinrel\twoheadrightarrow \lieh \, $,  which induces an isomorphism  $ \; \hat{\pi} : \hat{\lieh}_{{}_T} \cong \lieh_{{}_T} \; $  as required.
\end{proof}

\vskip7pt

   A last, significant result concerns morphisms between realizations.

\vskip11pt

\begin{lema}  \label{lemma: ker-morph's_realiz's}
  \textsl{(cf.\ \cite{GG3}, Lemma 2.1.8)}
 \vskip3pt
 Let  $ \; \hat{\R} \, := \, \big(\, \hat{\lieh} \, , \hat{\Pi} \, , \hat{\Pi}^\vee \,\big) \; $  and
 $ \; \check{\R} \, := \, \big(\, \check{\lieh} \, , \check{\Pi} \, , \check{\Pi}^\vee \,\big) \; $
 be two realizations of a same  $ \, P \in M_n\big(\kh\big) \, $,  and let
 $ \; \underline{\phi} : \hat{\R} \longrightarrow \check{\R} \; $
 be a morphism between them.  Then
 $ \; \Ker\big(\, \phi : \hat{\lieh} \longrightarrow \check{\lieh} \,\big) \, \subseteq \,
 \bigcap\limits_{j \in I} \Ker(\hat{\alpha}_j) \;\, $.   \hfill  $ \square $
\end{lema}

\vskip9pt

\begin{rmk}
 Working with matrices in  $ \, M_n(\Bbbk) \, $  and realizations of them over  $ \Bbbk $  (for some field  $ \Bbbk \, $),  all the previous constructions still make sense; some results (e.g.,  Proposition \ref{prop: exist-realiz's}  and  Lemma \ref{lemma: split/straight-lifting})  even get stronger and/or have simpler proofs.
\end{rmk}

\medskip
  \eject

\subsection{Twist deformations of multiparameters and realizations}
  \label{subsec: twisted-realiz}  {\ }
 \vskip7pt

   In this subsection we introduce the notion of  \textit{deformation by twist\/}  of realizations, which will be needed later when dealing with (quantum) super-Lie theoretical objects.  Like before, we are essentially recollecting results from  \cite{GG3}, \S 2.2,  so we skip those details that can be found there.

\vskip9pt

\begin{free text}  \label{twist-deform.'s x mpmatr.'s & realiz.'s}
 {\bf Deforming realizations (and matrices) by twist.}  Fix  $ \, P := {\big(\, p_{i,j} \big)}_{i, j \in I} \in M_n\big(\kh\big) \, $
 and a realization  $ \; \R \, := \, \big(\, \lieh \, , \Pi \, , \Pi^\vee \,\big) \, $,
 \,possibly   --- up to changing minimal details in what follows ---   over  $ \Bbbk $  if  $ \, P \in M_n(\Bbbk) \, $.
                                                                       \par
   Recall that  $ \lieh $  is, by assumption, a free  $ \kh $--module  of finite rank  $ \, t := \rk(\lieh) \, $.
   We fix in  $ \lieh $  any  $ \kh $--basis  $ \, {\big\{ H_g \big\}}_{g \in \G} \, $,
   where  $ \G $  is an index set with  $ \, |\G| = \rk(\lieh) = t \, $.
                                                                       \par
   Fix an antisymmetric square matrix  $ \; \Phi = \big( \phi_{gk} \big)_{g, k \in \G} \in \lieso_t\big(\kh\big) \; $
   --- in fact, we might work with any  $ \, \Phi \in M_t\big(\kh\big) \, $,
   \,but at some point we should single out its antisymmetric part
   $ \, \Phi_a := 2^{-1} \big( \Phi - \Phi^{\scriptscriptstyle T} \,\big) \, $
   which would be all that matters.  We define the  \textsl{twisted ``distinguished toral elements'' (or ``coroots'')}
\begin{equation}  \label{eq: T-phi}
  T^\pm_{{\scriptscriptstyle \Phi},\ell} \; := \;
  T^\pm_\ell \, \pm {\textstyle \sum_{g,k=1}^t} \alpha_\ell(H_g) \, \phi_{kg} \, H_k
\end{equation}
 As a matter of notation, let
 $ \, \fT := \bigg(\, {\displaystyle {\underline{T}^+ \atop \underline{T}^-}} \bigg) \, $
 be the  $ (2n \times 1) $--matrix  given by the column vectors
 $ \, \underline{T}^\pm = {\big(\, T_i^\pm \big)}_{i \in I} \, $.  Similarly, let
 $ \; \fT_{{\scriptscriptstyle \Phi}} :=
 \bigg(\, {\displaystyle {\underline{T}^+_{\,\scriptscriptstyle \Phi} \atop \underline{T}^-_{\,\scriptscriptstyle \Phi}} }\bigg) \; $
 be the  $ (2n \times 1) $--matrix  given by the (superposed) two column vectors
 $ \, \underline{T}^\pm_{\,\scriptscriptstyle \Phi} = {\big(\, T_{{\scriptscriptstyle \Phi}, i}^\pm \big)}_{i \in I} \, $,
 \,and  $ \fH $  the column vector  $ \, \fH := {\big(\, H_g \big)}_{g \in \G} \, $.
 Moreover, denote by  $ \mathfrak{A} $  the  $ (n \times t) $--matrix  with entries in  $ \kh $  given by
 $ \; \mathfrak{A} := {\big(\, \alpha_\ell(H_g) \big)}_{\ell \in I}^{g \in \G} \; $,  \,and set
 $ \; \mathfrak{A}_\bullet  :=  \begin{pmatrix}
        +\mathfrak{A} \,  \\
        -\mathfrak{A} \,
                        \end{pmatrix} \; $
 --- a matrix of size  $ (2n \times t) \, $.
 \vskip3pt
   Now, using matrix notation we have
  $ \; \fT_{\scriptscriptstyle \Phi} \, = \, \fT - \, \mathfrak{A}_\bullet \, \Phi \, \fH \; $.
 Eventually, define also
\begin{equation}  \label{def-P_Phi}
  P_{\scriptscriptstyle \Phi} \, =
   \, {\big(\, p^{\scriptscriptstyle \Phi}_{i,j} \big)}_{i, j \in I} \, :=
   \, {\Big(\, p_{i,j} + {\textstyle \sum_{g,k=1}^t} \alpha_j(H_g) \, \phi_{g,k} \, \alpha_i(H_k) \Big)}_{i, j \in I} \, =
   \, P - \mathfrak{A} \, \Phi \, \mathfrak{A}^{\,\scriptscriptstyle T}
\end{equation}

\vskip5pt

   Now, using the above notation, a direct computation yields
  $$  \displaylines{
   S_{{\scriptscriptstyle \Phi},i} := \, 2^{-1} \big(\, T^+_{{\scriptscriptstyle \Phi},i} + T^-_{{\scriptscriptstyle \Phi},i} \big) \,
   = \, 2^{-1} \big(\, T^+_i + T^-_i \big) = S_i   \qquad \qquad  \forall \;\; i \in I  \cr
   \qquad   \alpha_j\big(\,T^+_{{\scriptscriptstyle \Phi},i}\,\big) \, = \, p^{\scriptscriptstyle \Phi}_{i,j}  \quad ,
   \qquad  \alpha_j\big(\,T^-_{{\scriptscriptstyle \Phi},i}\,\big) \, = \, p^{\scriptscriptstyle \Phi}_{j,i}
   \qquad \qquad \qquad  \forall \;\; i, j \in I  }  $$
 so that the triple  $ \; \R_\Phi := \big(\, \lieh_\Phi \, , \Pi_\Phi \, , \Pi^\vee_{\scriptscriptstyle \Phi} \,\big) \; $
 with  $ \, \lieh_\Phi := \lieh \, $,  $ \, \Pi_\Phi := \Pi \, $  and
 $ \, \Pi^\vee_{\scriptscriptstyle \Phi} := \big\{\, T^+_{{\scriptscriptstyle \Phi},i} \, ,
 T^-_{{\scriptscriptstyle \Phi},i} \,\big|\, i \in I \,\big\} \, $,  \,is a realization of the matrix
 $ \, P_{\scriptscriptstyle \Phi} = {\big(\, p^{\scriptscriptstyle \Phi}_{i,j} \big)}_{i, j \in I} \, $
 --- as in  Definition \ref{def: realization of P};  also, by construction
 $ \R_{\scriptscriptstyle \Phi} $  is also \textsl{straight},  resp.\  \textsl{small},
 if and only if such is  $ \R \, $.  Moreover,  $ P_{\scriptscriptstyle \Phi} $
 is the sum of  $ P $  plus an antisymmetric matrix, so
 \textsl{the symmetric part of  $ P_{\scriptscriptstyle \Phi} $  is the same as  $ P $},
 i.e.\  $ \, {\big( P_{\scriptscriptstyle \Phi} \big)}_s = P_s \, $.
 In particular, if  $ P $  is of Cartan type, then so is  $ P_{\scriptscriptstyle \Phi} \, $,
 and they are associated with the same Cartan matrix. In short, we get:
\end{free text}

\vskip7pt

\begin{prop}  \label{prop: twist-realizations}
  \textsl{(cf.\ \cite{GG3}, Proposition 2.2.2)}
 \vskip1pt
 With notation as above, the following holds true:
 \vskip3pt
   (a)\,  the matrix  $ \; P_{\scriptscriptstyle \Phi} \, := \, P -  \, \mathfrak{A} \, \Phi \, \mathfrak{A}^{\,\scriptscriptstyle T} \; $
obeys  $ \, {\big( P_{\scriptscriptstyle \Phi} \big)}_s = P_s \, $;  \,in particular, if
$ P $  is of Cartan type, then so is  $ P_{\scriptscriptstyle \Phi} \, $,
and they are associated with the same Cartan matrix.
 \vskip3pt
   (b)\,  the triple  $ \, \R_{\scriptscriptstyle \Phi} := \big(\, \lieh_\Phi \! := \lieh \, ,
   \, \Pi_\Phi \! := \Pi \, , \, \Pi^\vee_{\scriptscriptstyle \Phi} \! :=
   {\big\{ T^+_{{\scriptscriptstyle \Phi},i} \, , T^-_{{\scriptscriptstyle \Phi},i} \big\}}_{i \in I} \,\big) \, $
   is a realization of  $ \, P_{\scriptscriptstyle \Phi} \, $;  \,moreover,  $ \R_\Phi $  is  \textsl{straight},
   resp.\  \textsl{small},  if and only if such is  $ \R \, $.   \hfill  $\square $
\end{prop}

\vskip7pt

\begin{definition}  \label{def: twisted-realization}
 The realization  $ \, \R_\Phi := \big(\, \lieh \, , \Pi \, , \Pi^\vee_{\scriptscriptstyle \Phi} \,\big) \, $
 of the matrix  $ \, P_{\scriptscriptstyle \Phi} = {\big(\, p^{\scriptscriptstyle \Phi}_{i,j} \big)}_{i, j \in I} $  \,
 is called a  \textit{twist deformation\/}  (via  $ \Phi \, $)  of the realization
 $ \; \R = \big(\, \lieh \, , \Pi \, , \Pi^\vee \,\big) \, $  of  $ P \, $.
                                                                    \par
   Similarly, the matrix  $ P_{\scriptscriptstyle \Phi} $  is called a  \textit{twist deformation\/}
   of the matrix  $ P $.   \hfill  $ \diamondsuit $
\end{definition}

\vskip7pt

\begin{rmks}  \label{rmks: transitivity-twist}
 \textit{(a)}\,  Observe that, by the very definition of twisting one has that
  $$  \big(P_{\scriptscriptstyle \Phi}\big)_{\scriptscriptstyle \Phi'} = P_{\scriptscriptstyle \Phi + \Phi'}  \qquad \text{ and }
  \qquad  \big(\R_{\scriptscriptstyle \Phi}\big)_{\scriptscriptstyle \Phi'} = \R_{\scriptscriptstyle \Phi + \Phi'}
  \qquad  \text{ for all }  \quad  \Phi \, , \, \Phi' \in \lieso_t\big(\kh\big)  $$
 Therefore, the additive group  $ \, \lieso_t\big(\kh\big) \, $  acts on the set of
 (multiparameter) matrices of size  $ \, n := |I| \, $  with fixed symmetric part,
 as well as on the set of their realizations of (any) fixed rank.  \textsl{When two matrices
 $ \, P , P' \in M_n\big(\kh\big) \, $  belong to the same orbit of this  $ \lieso_t\big(\kh\big) $--action  we say that
 \textit{$ P $  and  $ P' $  are twist equivalent}}.
 \vskip5pt
   \textit{(b)}\,  It follows from  Proposition \ref{prop: twist-realizations}\textit{(a)\/}
   that if two multiparameter matrices  $ P $  and  $ P' $  are twist equivalent,
   then their symmetric part is the same, that is  $ \, P_s = P'_s \, $.
   Next result shows the converse holds true as well,  \textit{up to taking the group
   $ \lieso_t\big(\kh\big) $  big enough, namely with  $ \, t \geq 3\,n-\rk(P_s) \, $}.
\end{rmks}

\vskip3pt

\begin{lema}  \label{lemma: twist=sym}
  \textsl{(cf.\ \cite{GG3}, Lemma 2.2.5)}
 \vskip3pt
 With notation as above, let  $ \, P , P' \in M_n\big(\kh\big) \, $,
 \,and consider the aforementioned action by twist on  $ M_n\big(\kh\big) $
 by any additive group  $ \lieso_t\big(\kh\big) $  with  $ \, t \geq 3\,n-\rk(P_s) \, $.
 Then  $ P $  and  $ P' \, $  are twist equivalent if and only if  $ \, P_s = P'_s \, $.   \hfill  $ \square $
\end{lema}

\vskip7pt

   Next proposition ``upgrades'' the previous result to the level of realizations.

\vskip11pt

\begin{prop}  \label{prop: realiz=twist-standard}
  \textsl{(cf.\ \cite{GG3}, Proposition 2.2.7)}
 Let  $ P $  and  $ P' $  be two matrices in  $ M_n\big(\kh\big) $  with the same symmetric part, i.e.\ such that  $ \, P_s = P'_s \; $.
 \vskip5pt
   (a)\,  Let  $ \, \R $  be a  \textsl{straight}  realization of  $ P $ of rank  $ \, t \geq 3\,n-\rk(P_s) \, $.  Then there is  $ \, \Phi \in \lieso_t\big(\kh\big) \, $  such that  $ \, P' = P_{\scriptscriptstyle \Phi} \, $  and the corresponding realization  $ \, \R_{\scriptscriptstyle \Phi} \, $  is straight.
                                                            \par
   In short, if  $ \, P'_s = P_s \, $  then from any  \textsl{straight}  realization of  $ P $  we can obtain by twist deformation a straight realization (of the same rank) for  $ P' $,  and viceversa.
 \vskip5pt
   (b)\,  Let  $ \, \R $  and  $ \R' $  be  \textsl{straight small}  realizations of  $ \, P $  and  $ P' \, $,  such that  $ \, \rk(\R) = \rk(\R') =: t \geq 3\,n-\rk(P_s) \, $.
   Then there exists
 $ \, \Phi \in \lieso_t\big(\kh\big)
 \, $  such that  $ \, \R' \cong \R_{\scriptscriptstyle \Phi} \, $.
                                                            \par
   In short, if  $ \, P'_s = P_s \, $  then any straight small realization of  $ P' $  is isomorphic to a twist deformation of a
   straight small realization of  $ P $  of same rank, and viceversa.
 \vskip5pt
   (c)\,  Every  \textsl{straight small}  realization  $ \, \R $  of  $ P $  is isomorphic to some
   twist deformation of the standard realization of  $ P_s $  of the same rank as  $ \, \R \, $,
   \hbox{as in  Lemma \ref{lem: realiz_P => realiz_A}.   \hfill  $ \square $}
\end{prop}

\vskip11pt

\begin{free text}  \label{further_stability}
 {\bf Stability issues for twisted realizations.}  Applying deformation by twist upon a given realization provides a new realization (for a different matrix, in general).  If the initial realization is  \textsl{split},  however, the new realization achieved through deformation is  \textsl{not\/}  necessarily split, in general.  This issue is discussed at length in  \cite{GG3}, \S\S 2.2.8--9,  where examples and counterexamples are given as well.
\end{free text}

\medskip

\subsection{2-cocycle deformations of multiparameters and realizations}
  \label{subsec: 2-cocycle-realiz}  {\ }
 \vskip7pt
   In this subsection we introduce the notion of  \textit{deformation by 2--cocycles\/}  of realizations (and of multiparameters), which is dual to deformation by twist.  Here again, we directly report material from  \cite{GG3}, \S 2.3,  so we skip all secondary details that can be found there.

\vskip7pt

\begin{free text}  \label{2-cocycle-deform.'s x mpmatr.'s & realiz.'s}
 {\bf Deforming realizations (and matrices) by  $ 2 $--cocycles.}  Fix again
 $ \, P := {\big(\, p_{i,j} \big)}_{i, j \in I} \in M_n\big( \kh \big) \, $  and a realization
 $ \, \R := \big(\, \lieh \, , \Pi \, , \Pi^\vee \,\big) \, $  of it.
 We consider special deformations of realizations, called ``2-cocycle deformations''.  To this end, like in  \S \ref{subsec: twisted-realiz},  we fix in  $ \lieh $  a  $ \kh $--basis
 $ \, {\big\{ H_g \big\}}_{g \in \G} \, $,  with  $ \G $  an index set,  $ \, |\G| = \rk(\lieh) = t \, $.
 \vskip5pt
   Let  $ \, \chi : \lieh \times \lieh \relbar\joinrel\longrightarrow \kh \, $  be any  $ \kh $--bilinear  map:
   note that it bijectively corresponds to some
   $ \, X = {\big(\, \chi_{g{}\gamma} \big)}_{g , \gamma \in \G} \in M_t\big(\kh\big) \, $  via
   $ \, \chi_{g{}\gamma} = \chi(H_g\,,H_\gamma) \, $.  We assume that  $ \chi $ is antisymmetric, which means
   $ \, \chi^{\scriptscriptstyle T}(x,y) = -\chi(x,y) \, $  where
   $ \, \chi^{\scriptscriptstyle T}(x,y) := \chi(y,x) \, $,  \,for all  $ \, x, y \in \lieh \, $;
   \,this is equivalent to saying that  $ X $  is antisymmetric, i.e.\  $ \, X \in \lieso_t\big(\kh\big) \, $.
   We denote by  $ \, \text{\sl Alt}_{\,\kh}\big(\, \lieh \times \lieh \, , \, \kh \,\big) \, $  the set of all antisymmetric,
   $ \kh $--bilinear  maps from  $\, \lieh \times \lieh \, $  to  $ \, \kh \, $.
   We assume also that  $ \chi $  obeys
\begin{equation}  \label{eq: condition-chi}
  \chi(S_i \, ,\,-\,)  \, = \,  0  \, = \,  \chi(\,-\,,S_i)   \qquad \qquad \forall \;\; i \in I
\end{equation}
 where  $ \, S_i := 2^{-1} \big(\, T^+_i + T^-_i \big) \, $  for all  $ \, i \in I \, $.  In particular,
 this implies  (for  $ \, i \in I \, $,  $ \, T \in \lieh \, $)  that
  $ \; \chi\big( +T_i^+ , \, T \,\big)  \, = \,  \chi\big( -T_i^- , \, T \,\big) \; $,  $ \; \chi\big(\, T \, , +T_i^+ \big)  \, = \,  \chi\big(\, T \, , -T_i^- \big) \; $,
 \,hence
\begin{equation}  \label{eq: condition-chi-chit}
 +\chi\big(\, \text{--} \, , T_i^+ \big) = -\chi\big(\, \text{--} \, , T_i^- \big)   \qquad \qquad  \forall \;\; i \in I
\end{equation}
 For later use,  \textit{we introduce also the notation}
\begin{equation}  \label{eq: def-Alt_S}
 \text{\sl Alt}_{\,\kh}^{\,S}(\lieh)  \, := \,
 \Big\{\, \chi \in \text{\sl Alt}_{\,\kh}\big(\, \lieh \times \lieh \, , \, \kh \,\big) \,\Big|\; \chi \text{\ obeys\ }  (\ref{eq: condition-chi}) \,\Big\}
\end{equation}
 and to each  $ \, \chi \in  \text{\sl Alt}_{\,\kh}^{\,S}(\lieh)  \, $  we associate
  $ \, \mathring{X} := {\Big( \mathring{\chi}_{i{}j} = \chi\big(\,T_i^+,T_j^+\big) \!\Big)}_{i, j \in I} \! \in \lieso_n\big(\kh\big) \, $.
 \vskip5pt
  Basing on the above, we define
  $$
  P_{(\chi)}  := \,  P \, + \, \mathring{X}  \, = \,  {\Big(\, p^{(\chi)}_{i{}j} := \,  p_{ij} + \mathring{\chi}_{i{}j} \Big)}_{\! i, j \in I}  \;\, ,  \!\quad
   \Pi_{(\chi)}  := \,  {\Big\{\, \alpha_i^{(\chi)}  := \,
   \alpha_i \pm \chi\big(\, \text{--} \, , T_i^\pm \big)  \Big\}}_{i \in I}
   $$
\end{free text}

\vskip7pt

   We are now ready for our key result on  $ 2 $--cocycle  deformations.

\vskip13pt

\begin{prop}  \label{prop: 2cocdef-realiz}
  \textsl{(cf.\ \cite{GG3}, Proposition 2.3.2)}
 Keep notation as above. Then:
 \vskip5pt
   (a)\;  $ \, P_{(\chi)} \, := \, P \, + \, \mathring{X} \; $  obeys  $ \, {\big( P_{(\chi)} \big)}_s = P_s \, $;
   \,in particular, if  $ P $  is of Cartan type, then so is  $ P_{(\chi)} \, $,
   and they are associated with the same Cartan matrix.
 \vskip5pt
   (b)\;  the triple  $ \, \R_{(\chi)} \, = \, \big(\, \lieh \, , \Pi_{(\chi)} \, , \Pi^\vee \,\big) \, $
   is a realization of the matrix  $ \, P_{(\chi)} \, $,  \,which is minimal, resp.\ split, if so is  $ \, \R \, $.   \hfill  $ \square $
\end{prop}

\vskip7pt

\begin{definition}\label{def:2-cocycle-realization}
 The realization  $ \, \R_{(\chi)} = \big(\, \lieh \, , \Pi_{(\chi)} \, , \Pi^\vee \,\big) \, $  of
 $ \, P_{(\chi)} = {\big(\, p^{(\chi)}_{i{}j} \big)}_{\!i, j \in I} \, $  is called a
 \textit{2--cocycle deformation\/}  of the realization
 $ \; \R = \big(\, \lieh \, , \Pi \, , \Pi^\vee \,\big) \, $  of  $ P \, $.
                                                                    \par
   Similarly, the matrix  $ P_{(\chi)} $  is called a  \textit{2--cocycle deformation\/}  of the matrix  $ P $.   \hfill  $ \diamondsuit $
\end{definition}

\vskip7pt

\begin{rmks}  \label{rmks: cocycle-deform-realiz}
 \textit{(a)}\,  The very definitions give
  $$  \big(P_{(\chi)}\big)_{(\chi)'} = P_{(\chi + \chi')}  \quad \text{ and }  \quad
  \big(\R_{(\chi)}\big)_{(\chi')} = \R_{(\chi + \chi')}  \qquad  \text{ for all }
  \;\;  \chi \, , \, \chi' \in \text{\sl Alt}_{\,\kh}^{\,S}(\lieh)  $$
 Thus, the additive group  $ \, \text{\sl Alt}_{\,\kh}^{\,S}(\lieh) \, $  acts on the set of
 (multiparameter) matrices of size  $ \, n := |I| \, $
 with fixed symmetric part, as well as on the set of their realizations of (any) fixed rank.
 \textsl{When two matrices  $ P $  and  $ P' $  in  $ \, M_n\big(\kh\big) \, $
 belong to the same orbit of this  $ \text{\sl Alt}_{\,\kh}^{\,S}(\lieh) $--action,
 we say that  \textit{$ P $  and  $ P' $  are 2--cocycle equivalent}}.
 \vskip5pt
   \textit{(b)}\,  It follows from  Proposition \ref{prop: 2cocdef-realiz}\textit{(a)\/}
   that if two multiparameter matrices  $ P $  and  $ P' $  are  $ 2 $--cocycle  equivalent,
   then their symmetric part is the same, i.e.\  $ \, P_s = P'_s \, $.
   As a consequence of the next result, the converse holds true as well
   (cf.\  Lemma \ref{lemma: 2-cocycle=sym}  below),  \textit{under mild, additional assumptions}.
\end{rmks}

\vskip7pt

   Next result concerns the aforementioned  $ \text{\sl Alt}_{\,\kh}^{\,S}(\lieh) $--action
   on realizations; indeed, up to minor details it can be seen as the
   ``$ 2 $--cocycle  analogue'' of  Proposition \ref{prop: realiz=twist-standard}:

\vskip13pt

\begin{prop}  \label{prop: mutual-2-cocycle-def}
  \textsl{(cf.\ \cite{GG3}, Proposition 2.3.2)}
 \vskip3pt
 Let  $ \, P, P' \in M_n\big(\kh\big) \, $  be two matrices  with the same symmetric part,
 i.e.\ such that  $ \, P_s = P'_s \, $.  Moreover, let  $ \, \R $  be a  \textsl{split}  realization of  $ P $.
 \vskip5pt
   (a)\,  There exists a map $ \, \chi \in \text{\sl Alt}_{\,\kh}^{\,S}(\lieh) \, $
   such that  $ \, P' = P_{(\chi)} \, $  and the realization
   $ \, \R_{(\chi)}\,= \, \big(\, \lieh \, , \Pi_{(\chi)} \, , \Pi^\vee \,\big) \, $  of  $ \, P' = P_{(\chi)} \, $
   is  \textsl{split}.  In a nutshell, if  $ \, P'_s = P_s \, $  then from any  \textsl{split}  realization of  $ P $
   we can obtain a split realization (of the same rank) of  $ P' $  by 2--cocycle deformation, and viceversa.
 \vskip5pt
 (b)\,  Assume in addition that  $ \, \R $  be  \textsl{minimal}.
 Then  $ \, \R $  is isomorphic to a 2--cocycle deformation of the split minimal realization of  $ \, P_s \, $.   \hfill  $ \square $
\end{prop}

\vskip9pt

   As a byproduct, we find the following ``2-cocycle counterpart'' of  Lemma \ref{lemma: twist=sym}:

\vskip11pt

\begin{lema}  \label{lemma: 2-cocycle=sym}
  \textsl{(cf.\ \cite[Lemma 2.3.6]{GG3} )}
 With notation as above, let  $ \, P , P' \in M_n\big(\kh\big) \, $.  Then  $ P $  and  $ P' $  are 2-cocycle equivalent (for the above 2-cocycle action on  $ M_n\big(\kh\big) $  of an additive group  $ \lieso_t\big(\kh\big) \, $)  if and only if  $ \, P_s = P'_s \, $.   \hfill  $ \square $
\end{lema}

\vskip13pt

\begin{obs}  \label{obs: two deform's x P & R}
 To sum up, we wish to stress the following, remarkable fact.  Consider two matrices
 $ \, P , P' \in M_n\big(\kh\big) \, $
 with the same symmetric part  $ \, P_s = P'_s \, $,  and a realization
 $ \, \R = \big(\, \lieh \, , \Pi \, , \Pi^\vee \,\big) \, $  of  $ P $
 that is  \textsl{split\/}  and  \textsl{straight}.  Then, by
 Proposition \ref{prop: realiz=twist-standard}  and  Proposition \ref{prop: mutual-2-cocycle-def},
 one can construct  \textsl{two\/}  realizations  $ \R_{\scriptscriptstyle \Phi} $  and
 $ \R_{(\chi)} $  of  $ P' $  by a twist deformation,
 respectively a 2-cocycle deformation, of  $ \R $  that affects only the coroot set  $ \Pi^\vee $
 or the root set  $ \Pi \, $,  respectively; in particular,
 $ \R_{\scriptscriptstyle \Phi} $  is still straight (yet possibly not split) and  $ \R_{(\chi)} $
 is still split (yet possibly not straight),
 while both have the same rank as  $ \R \, $.
\end{obs}

\bigskip
 \medskip

\section{Multiparameter Lie superbialgebras and their deformations}  \label{sec: MpLSbAs}
 \vskip7pt
   In this section we first recall the notion of Lie superbialgebra, and we explain their theory of deformations, both by twist and by 2--cocycles.  Then we introduce our  \textsl{multiparameter Lie superbialgebras\/}  (MpLSbAs in short).  Finally, we study deformations of MpLSbAs, both by toral twists and by toral 2--cocycles, proving in particular that our family of MpLSbAs is  \textsl{stable}  for deformations of either type.

\medskip

\subsection{Lie superbialgebras and their deformations}  \label{subsec: Lie-supbialg's & deform's}  {\ }
 \vskip7pt
 We recall hereafter a few notions   --- including some partially new, to the best of our knowledge ---   concerning Lie superbialgebras and their deformations.  For basic definitions and results see  \cite{An,Ka2};  as to notation, we adopt the same used for Lie bialgebras in  \cite{GG3}.  We work over any field of characteristic different from 2.

\medskip

\begin{free text}  \label{gen's on LSbA's}
 {\bf Generalities on Lie superbialgebras.}
We begin by giving some basic and (more or less) well-known definitions.
 \vskip3pt
   A  \emph{super vector space\/}  (or ``vector superspace'') is a  $ \ZZ_2 $--graded  vector space  $ \, V = V_\zero \oplus V_\one \, $;  the homogeneous component  $ V_\zero $  is called the even part and  $ V_\one $  the odd part.  For any homogeneous element  $ \, v \in V \, $,  \,we denote by  $ \, |v| \in \{\zero,\one\} =: \ZZ_2 \, $  the degree of  $ v \,$,  which is defined by  $ \, |v| = \overline{a} \iff v \in V_{\overline{a}} \; $.  By a common abuse of notation, in the sequel whenever a formula involves parity of some elements, one assumes that the concerned elements be homogeneous, and then extends the validity of that formula to all elements by linearity.
                                                           \par
   Any super vector space has a  \textsl{braiding\/}  map $ \, \sigma: V \otimes V \relbar\joinrel\relbar\joinrel\longrightarrow V \otimes V \, $  defined by  $ \, \sigma(v \otimes w) := {(-1)}^{|v||w|} w \otimes v \, $  for all  $ \, v, w \in V \, $,  \,which satisfies the braid equation  $ \, (\sigma \ot \id)(\id \ot\sigma)(\sigma \ot \id) = (\id \ot \sigma)(\sigma \ot \id)(\id \ot\sigma) \, $.  As  $ \, \sigma^2 = \id \, $,  the map  $ \sigma $  induces a representation of the symmetric group  $ \mathbb{S}_n $  on the  $ n $--th tensor  power  $ T^n(V) $  of  $ V $.  In particular, one has the super symmetrizer operator  $ \, \Sym_n : T^n(V) \longrightarrow T^n(V) \, $  given by  $ \, \Sym_n := \sum_{\mu \in \mathbb{S}_n} \mu \, $  and the super antisymmetrizer operator  $ \, \Alt_n : T^n(V) \longrightarrow T^n(V) \, $  given by  $ \, \Alt_n := \sum_{\mu \in \mathbb{S}_n} {(-1)}^{\ell(\mu)} \mu \, $;  \,note that, in order to compute the action of the operators, one has to write each permutation as a product of transpositions.  With this action in mind, the exterior powers  $ \, \Lambda^n(V) \, $  of  $ V $  are defined, as usual, as the subspace of  $ T^n(V) $  given by the image of  $ \Alt_n \, $,  \,or more explicitly as
  $$  \Lambda^n(V)  \,\; := \;\,  \big\{\, v_{(n)} \in T^n(V) \,\big|\, \mu\,.\,v_{(n)} = {(-1)}^{\ell(\mu)} v_{(n)} \,\; \forall \; \mu \in \mathbb{S}_n \,\big\}  $$
   In particular,  $ \Lambda^2(V) = \big\{\, v_{(2)} \in V \otimes V \,\big|\, \sigma\big(v_{(2)}\big) = -v_{(2)} \,\big\} = \Ker(\id + \sigma) \, $.
 \vskip1pt
   Similarly one defines the symmetric powers  $ \, S^n(V) := \Img(\Sym_n) \, $.
 \vskip3pt
   A  \emph{Lie superalgebra\/}  is a pair  $ \, \big(\, \lieg \, , [\,\ ,\ ] \,\big) \, $  where  $ \, \lieg = \lieg_{\zero} \oplus \lieg_{\one} \, $  is a super vector space and  $ \, [\,\ ,\ ] : \lieg \times \lieg \relbar\joinrel\longrightarrow \lieg \, $  is a  $ \ZZ_2 $--graded  bilinear map, i.e.\  $ \, [\lieg_\alpha,\lieg_\beta] \subseteq \lieg_{\alpha+\beta} \, $  for all  $ \, \alpha , \beta \in \ZZ_2 \, $,  \,called a (Lie) superbracket, satisfying the following conditions for any choice of homogeneous elements  $ \, x , y , z \in \lieg \, $:
\begin{equation}   \label{eq: super-Lie axioms}
  \begin{aligned}
    0  &  \, = \,  [x,y] + {(-1)}^{|x||y|} [y,x]  \\
    0  &  \, = \,  {(-1)}^{|x||z|} [x,[y,z]] + {(-1)}^{|y||x|} [y,[z,x]] + {(-1)}^{|z||y|} [z,[x,y]]
  \end{aligned}
\end{equation}
   \indent   The first condition in  \eqref{eq: super-Lie axioms}  is called  \emph{super
   antisymmetry\/}:  it can be expressed in element-free form by saying that  $ \, [\,\ ,\ ] \, $  is a map from  $ \, \lieg \otimes \lieg \, $  to  $ \lieg $  such that
  $$  0  \; = \;  [\,\ ,\ ] \circ (\id + \,\sigma)  \; = \;  [\,\ ,\ ] \circ \Sym_2  $$
   \indent   Note that in general one has that  $ \; (\id - (1,2)) \circ \Sym_{\mathbb{A}_n} \, = \, \Alt_n \; $,  \,where  $ \Sym_{\mathbb{A}_n} $ is the symmetrizer of the alternating subgroup  $ \mathbb{A}_n \, $.  In case  $ \, n = 3 \, $  this reads
  $$  \Alt_3  \,\; = \;\,  (\id - (1,2)) \circ \Sym_{\mathbb{A}_3}  \,\; = \;\,  (\id - (1,2)) \circ \big(\id + (1,2,3) + (3,2,1) \big)  $$
 The action of the operator  $ \Sym_{\mathbb{A}_3} $  on $ T(\lieg) $  is also usually called  \emph{``(action by) cyclic permutations''\/}  and written as \emph{``c.p.''}
                                                    \par
   The second condition in  \eqref{eq: super-Lie axioms}  is called the  \emph{super Jacobi\/}  identity.  One can express the latter by stating that the adjoint action of any element is a  \emph{super derivation},  i.e.
  $$  [x,[y,z]]   \,\; = \;\,  [[x,y],z] \, + \, {(-1)}^{|x||y|}[y,[x,z]]  $$
 or by using the antisymmetrizer operator
  $$  [[\,\ ,\ ]\,,\ ] \circ \Alt_3  \; = \;  0   \qquad  \text{which is equivalent to}  \qquad   [[\,\ ,\ ]\,,\ ] \circ \Sym_{\mathbb{A}_3}  \; = \;  0  $$
   \indent   It is worth noting that the operators  $\Sym_{\mathbb{A}_n} $  and  $ \Alt_n $  actually commute with the action of  $ \lieg $  on tensor powers  $ T^n(\lieg) \, $.  Indeed,  $ \lieg $  acts on itself by the adjoint action  $ \, x\,.\,y \, := \, \ad(x)(y) = [x,y] \, $  for all  $ \, x , y \in \lieg \, $.  This action extends to the super vector space  $ T^n(\lieg) $  by
\begin{equation}  \label{eq: action-g-tensor}
  x\,.\,(y_1 \ot y_2 \ot \cdots \ot y_n)  \,\; := \;\,
{\textstyle \sum\limits_{i=1}^n} \, {(-1)}^{|x| \sum_{k=1}^{i-1} |y_k|} y_1 \ot \cdots \ot [x,y_i] \ot \cdots \ot y_n
\end{equation}
 for all elementary tensors  $ \, y_1 \ot y_2 \ot \cdots \ot y_n \in T^n(\lieg) \, $  with homogeneous factors.  For example, the adjoint action on  $ \, \lieg \ot \lieg \, $  reads
  $$  x\,.\,(y\ot z)  \,\; = \;\,  [x,y] \ot z \, + \, {(-1)}^{|x| |y|} y \ot [x,z]  $$
 One can show that this action commutes with the action of the elementary transposition, whence with the operators  $\Sym_{\mathbb{A}_n} $ and  $ \Alt_n \, $.
 \vskip3pt
   A  \emph{Lie supercoalgebra\/}  is a pair  $ \, (\lieg\,,\delta\,) $  where  $ \, \lieg = \lieg_{\zero} \oplus \lieg_\one \, $  is a super vector space  and  $ \delta : \lieg \longrightarrow \Lambda^2(\lieg) \, $  is a morphism of super vector spaces, called a  \emph{super (Lie) cobracket},  that satisfies the  \emph{super co-Jacobi\/}  identity, namely
  $$  \Alt_3 \circ\, \big( (\id\ot\delta) \circ \delta \big)  \; = \;  0   \!\!\qquad  \text{which is equivalent to}  \qquad\!\!   \Sym_{\mathbb{A}_3} \circ\, \big( (\id \ot \delta) \circ \delta \big)  \; = \;  0  $$
 Equivalently again, this amounts to be the same as requiring that the pair  $ \, \big( \lieg^* , \delta^* \big) \, $  be a Lie superalgebra.  Indeed, let  $ \, \langle - \,, -\rangle : \lieg^* \times \lieg \longrightarrow \Bbbk \, $  be the canonical evaluation map.  If we identify  $ \, \lieg^* \otimes \lieg^* \, $  with a subspace of  $ \, {\big( \lieg \otimes \lieg \big)}^* \, $  via
  $$  \big\langle \alpha \ot \beta , x \ot y \big\rangle  \; = \;  {(-1)}^{|\beta| |x|} \langle \alpha , x \rangle \langle \beta , y \rangle  $$
 for all (homogeneous)  $ \, x , y \in \lieg \, $,  $ \, \alpha , \beta \in \lieg^* \, $,  we have  $ \, \big\langle \alpha \ot \beta ,\delta(x) \big\rangle = \big\langle \delta^*(\alpha \ot \beta) , x \big\rangle \, $.  Then, the image of  $ \delta $  is contained in  $ \Lambda^2(\lieg) $ if and only if  $ \, (\id+\sigma) \circ \delta = 0 \, $  if and only if  $ \, \delta^* \circ (\id + \sigma) = 0 \, $  if and only if  $ \delta^* $  satisfies the super
 antisymmetry condition. Moreover,  $ \delta $  satisfies the super co-Jacobi if and only if  $ \delta^* $  satisfies the super Jacobi.
                                                \par
   To express the equations above on elements, it is better to consider  $ \Lambda^2(\lieg) $  as the subspace of  $ \, V \ot V \, $  given by the image of  $ \Alt_2 \, $.  Then, for all (homogeneous)  $ \, x , y \in \lieg \, $,  \,we set
  $$  x \wedge y  \,\; := \;\,  2^{-1} \Alt_2(x \ot y)  \, = \, 2^{-1} \big( x \otimes y - {(-1)}^{|x||y|} y \otimes x \big)  $$
   \indent   Finally, as a matter of notation, all along the paper we shall use a Sweedler's-like notation  $ \, \delta(x) = x_{[1]} \wedge x_{[2]} \, $  (summation being understood) for any  $ \, x \in \lieg \, $.
 \vskip3pt
   A  \emph{Lie superbialgebra\/}  is a triple  $ \, \big(\, \lieg \, , [\,\ ,\ ] \, , \, \delta \,\big) \, $  where
 $ \, \big(\, \lieg \, , [\,\ ,\ ] \,\big) \, $  is a Lie superalgebra,  $ \, (\,\lieg \, , \delta\,) \, $  is a Lie supercoalgebra, and  $ \delta $  is a  $ 1 $--cocycle  for the Lie superalgebra cohomology,  cf.\ \cite{Le}.  This is a compatibility condition that, for any homogeneous  $ \, x , y \in \lieg \, $,  \,reads as follows
\begin{equation*}
  \delta\big([x,y]\big)  \; = \;  x\,.\,\delta(y) - {(-1)}^{|x||y|} y\,.\,\delta(x)
\end{equation*}
 where  $ \, x\,.\,\delta(y) \, $  denotes the action of  $ \, \lieg \, $  onto  $ \Lambda^2(\lieg) \, $  induced by the adjoint action on  $ \, \lieg \ot \lieg \, $  as in  \eqref{eq: action-g-tensor}.  Taking this into account, the compatibility condition above reads
\begin{align*}
  \delta\big([x,y]\big)  &  \; = \;  \big[x,y_{[1]}\big] \wedge y_{[2]} \, + \, {(-1)}^{|x||y_1|} y_{[1]} \wedge \big[x,y_{[2]}\big]  \, -  \\
  &  \quad - \, {(-1)}^{|x||y|} \big[y,x_{[1]}\big] \wedge x_{[2]} \, - \, {(-1)}^{|x_{2}||y|} x_{[1]} \wedge \big[y,x_{[2]}\big]
\end{align*}
 If we write the action above by  $ \, x\,.\,(y \ot z) = \big[ \Delta(x) , y \ot z \big] \, $,  \,we have
  $$  \delta\big([x,y]\big)  \; = \;  \big[ \Delta(x) ,\delta(y) \big] \, - \, {(-1)}^{|x||y|} \big[ \Delta(y) , \delta(x) \big]  \; = \;  \big[ \Delta(x) , \delta(y) \big] \, + \, \big[ \delta(x) , \Delta(y) \big]  $$
 So, the element-free formulation of the compatibility condition reads
  $$  \delta \circ [\,\ ,\ ]  \; = \;  [\,\ ,\ ] \circ \big( \Delta \ot \delta \big) \circ \Alt_2   \qquad \text{ inside }  \End(\lieg\ot\lieg)  $$
 \vskip3pt
   When  $ \, \big(\, \lieg \, , [\,\ ,\ ] \, , \delta \,\big) \, $  is a Lie superbialgebra, the same holds for  $ \, \big(\, \lieg^* , \delta^* , {[\,\ ,\ ]}^* \,\big) \, $   --- up to topological technicalities, if  $ \lieg $  is infinite-dimensional ---   which is called the  \textit{dual\/}  Lie superbialgebra to  $ \, \big(\, \lieg \, , [\,\ ,\ ] \, , \delta \,\big) \, $.  We simply write  $ \lieg $  for a Lie superbialgebra if the structure maps are clear from the context.
 \vskip3pt
   As claimed before, the super Lie cobracket in a Lie superbialgebra satisfies a  $ 1 $--cocycle  condition. Thus, as in the non-super case, one may define super cobrackets by taking a  $ 1 $--cochain:  one then says that  $ \, \big(\, \lieg \, , [\,\ ,\ ] \, , \delta \,\big) \, $  is a  \textit{coboundary\/}  Lie superbialgebra if there exists  $ \, c \in \lieg \ot \lieg \, $  such that  $ \, \delta = \partial{}c \, $, \,that is, for all (homogeneous)  $ \, x \in \lieg \, $  one has
  $$  \delta(x)  \; = \;  (\partial{}c)(x)  \; = \;  x\,.\,c  \; = \;  [x,c_1] \ot c_2 + {(-1)}^{|x||c_1|} c_1 \ot [x,c_2]  $$
 \vskip3pt
   Dually,  $ \, \big(\, \lieg \, , [\,\ ,\ ] \, , \delta \,\big) \, $  is a  \textit{boundary\/}  Lie superbialgebra if there exists  $ \, \chi \in {( \lieg \ot \lieg)}^* \, $  such that  $ \, [\,\ ,\ ] = \partial_*{}\chi \, $,  \,that is, for all (homogeneous)  $ \, x , y \in \lieg \, $,  \,one has
  $$  [x,y]  \; = \;  (\partial_*{}\chi)(x \ot y)  \; = \;  x_{[1]} \, \chi(x_{[2]},y) \, + \, {(-1)}^{|x| |y_{[1]}|} \, y_{[1]} \, \chi(x,y_{[2]})  $$
 where  $ \, \partial_* \, $  is the coboundary map for the Lie algebra  $ \lieg^* \, $.
                                                                      \par
   Finally, given  $ \, r = r_1 \otimes r_2 \, $  (summation understood) and  $ \, s = s_1 \otimes s_2 \, $  in  $ \, \lieg \otimes \lieg \, $,  \,set
  $$  [\![\,r\,,s\,]\!]  \, := \,  {(-1)}^{|s_1||r_2|} \big[\,r_1\,,s_1\big] \otimes r_2 \otimes s_2 + r_1 \otimes \big[\,r_2\,,s_1\big] \otimes s_2 + {(-1)}^{|s_1||r_2|} \, r_1 \otimes s_1 \otimes \big[\,r_2\,,s_2\big]  $$
 which in compact form reads
  $$  [\![\,r\,,s\,]\!]  \, := \,  (-1)^{|s_1||r_2|} \, [\,r_{1,2}\,,s_{1,3}\,] \, + \, [\,r_{1,2}\,,s_{2,3}\,] \, + \, {(-1)}^{|s_1||r_2|} \, [\,r_{1,3}\,,s_{2,3}\,]  $$
 The following proposition is a  \textit{super}--version  of  \cite[Proposition 8.1.3]{Mj}:
\end{free text}

\vskip5pt

\begin{prop}  \label{prop:coboundary-liesuperbialgebra}  {\ }
\vskip3pt
   (a)\;  Let  $ \, \big(\, \lieg \, , [\,\ ,\ ] \,\big) \, $  be a Lie superalgebra and  $ \, c \in \lieg \ot \lieg \, $.  Then  $ \, \big( \lieg \, , [\,\ ,\ ] \, , \partial c \,\big) \, $  is a Lie superbialgebra if and only if
for all  $ \, x \in \lieg \, $  one has
  $$  x\,.\,\big((\id + \sigma)(c)\big) \, = \, 0   \qquad  \text{ and } \qquad
x\,.\,[\![\, c \, , c \,]\!] \, = \, 0  $$
 \vskip3pt
   (b)\;  Let  $ \, (\, \lieg \, , \delta \,) \, $  be a Lie supercoalgebra and
$ \, \chi \in {(\, \lieg \ot \lieg )}^* \, $.  Then  $ \, (\, \lieg \, , \partial_* \chi \, , \delta \,) \, $  is a Lie superbialgebra if and only if for all  $ \, \alpha \in \lieg^* \, $  one has
  $$  \alpha\,.\,\big(\chi (\id + \sigma)\big) \, = \, 0   \qquad  \text{ and }  \qquad   \alpha\,.\,{[\![\, \chi \, , \chi \,]\!]}_* \, = \, 0  $$
 where  $ \, {[\![\, \alpha \, , \beta \,]\!]}_* \, $  stands for the corresponding object defined above in  $ \, {(\, \lieg \ot \lieg \ot \lieg \,)}^* \, $  associated with any pair  $ \, \alpha , \beta \in \lieg^* \, $.
\end{prop}

%
%
 \textit{Proof.}
 We prove only item  \textit{(a)},  and we leave instead  \textit{(b)\/}  to the reader as it follows quite directly from duality.
                                                       \par
   The first equality is the one that guarantees that the image of  $ \partial c $
is in  $ \, \Lambda^2(\lieg) \, $:  \,for any  $ \, x \in \lieg \, $  we have
  $$  (\id+\sigma)\big((\partial c )(x)\big)  \, = \,  \Sym_2(x.c)  \, = \,  x.\Sym_2(c)  \, = \,  x.\big((\id + \sigma)(c)\big)  $$
 Hence,  $ \, (\partial c )(x) \in \Ker (\id+\sigma) = \Lambda^2(\lieg) \, $  if and
only if  $ \, x\,.\,\big((\id + \sigma)(c)\big) = 0 \, $.
                                                       \par
   To check the super co-Jacobi identity, we proceed as follows.  Write  $ \, c = c^+ + c^- \, $  with  $ \, c^+ \in S^2(\lieg) \, $  being the symmetric part of $ c $  and  $ \, c^- \in \Lambda^2(\lieg) \, $  the antisymmetric part.  Then
one has that  $ \, (\partial c )(x) = (\partial c^-)(x) \, $  for all  $ \, x \in \lieg \, $.  Indeed, as  $ \, \Sym_2(c^+) = 2\,c^+ \, $  and  $ \, \Sym_2(c^-) = 0 \, $,  \,the equality above tells us that  $ \, x\,.\,c^+ = 0 \, $.  Thus we have that
 $ \, (\partial c )(x)  = x\,.\,c = x\,.\,c^- = (\partial c^-)(x) \, $.  Now we show that
  $$  \big( \Sym_{\mathbb{A}_3} \circ\, \big( \id \ot (\partial c ) \big) \circ (\partial c) \big)(x)  \,\; = \;\,  x\,.\,[\![\, c^- , c^- \,]\!]  $$
 for all  $ \, x \in \lieg \, $;  \,the claim then will follow from this.
                                            \par
   Take  $ \, x \in \lieg \, $.  We begin by checking the following:
\begin{align*}
  &  \big( \big( \id \ot (\partial c^-) \big) \circ (\partial c^-) \big)(x)  \; = \;  \big( \id \ot (\partial c^-) \big)(x\,.\,c^-)  \; =  \\
  &  \qquad   = \;  \big( \id \ot (\partial c^-) \big)\Big( [\,x\,,c^-_1] \ot c^-_2 \, + \, {(-1)}^{|x||c^-_1|} \, c^-_1 \ot [\,x\,,c^-_2] \Big)  \; =  \\
%
%
  &  \qquad   = \;  [\,x\,,c^-_1] \ot (c^-_2.\,c^-) + {(-1)}^{|x||c^-_1|} \, c^-_1 \ot \big(\, [\,x\,,c^-_2]\,.\,c^- \big)  \; =  \\
  &  \qquad  = \;  [\,x\,,c^-_1] \ot \big[\,c^-_2, {c'}^-_1\big] \ot {c'}^-_2 \, + \, {(-1)}^{|c^-_2||{c'}^-_1|} [\,x\,,c^-_1] \ot {c'}^-_1 \ot \big[\,c^-_2,{c'}^-_2\big] \, +  \\
  &  \qquad \qquad   + \, {(-1)}^{|x||c^-_1|} \, c^-_1 \ot \big[ [\,x\,,c^-_2] \, , {c'}^-_1 \big] \ot {c'}^-_2 \, +  \\
  &  \qquad \qquad \qquad   + \, {(-1)}^{|x||c^-_1|+|x||{c'}^-_1|+|c^-_2||{c'}^-_1|} \, c^-_1 \ot {c'}^-_1 \ot \big[ [\,x\,,c^-_2] \, , {c'}^-_2 \big]
\end{align*}
 Then we get
  $$  \displaylines{
   \big( \Sym_{\mathbb{A}_3} \circ\, \big( \id \ot (\partial c^-) \big) \circ (\partial c^-) \big)(x)  \; =   \hfill  \cr
   \quad   = \;  \Sym_{\mathbb{A}_3}\Big(\, [\,x\,,c^-_1] \ot \big[\,c^-_2 , {c'}^-_1 \big] \ot {c'}^-_2 \, + \, {(-1)}^{|c^-_2||{c'}^-_1|} \, [\,x\,,c^-_1] \ot {c'}^-_1 \ot \big[\,c^-_2 , {c'}^-_2 \big] \, +   \hfill  \cr
   \quad \qquad   + \, {(-1)}^{|x||c^-_1|} \, c^-_1 \ot \big[ [\,x\,,c^-_2] \, , {c'}^-_1 \big] \ot {c'}^-_2 \, +   \hfill  \cr
  \quad \qquad \qquad   + \, {(-1)}^{|x||c^-_1|+|x||{c'}^-_1|+|c^-_2||{c'}^-_1|} \, c^-_1 \ot {c'}^-_1 \ot \big[ [\,x\,,c^-_2]\,,{c'}^-_2\big] \,\Big)  \; =   \hfill  \cr
  \quad   = \;  \Sym_{\mathbb{A}_3}\Big(\, [\,x\,,c^-_1] \ot \big[\,c^-_2 , {c'}^-_1\big] \ot {c'}^-_2 \, + \, {(-1)}^{|c^-_2||{c'}^-_1|} \, [\,x\,,c^-_1] \ot {c'}^-_1 \ot \big[\,c^-_2,{c'}^-_2\big] \, +   \hfill  \cr
  \quad \;\;   + {(-1)}^{|x||c^-_1|+|c^-_1|(|x| + |c^-_2| + |{c'}^-_1|+|{c'}^-_2|)} \, \big[ [\,x\,,c^-_2] \, , {c'}^-_1 \big] \ot {c'}^-_2 \ot c^-_1 \, +   \hfill  \cr
  \quad \;\;\;\;   + {(-1)}^{|x|(|c^-_1|+|{c'}^-_1|)+|c^-_2||{c'}^-_1|+(|x|+|c^-_2|+|{c'}^-_2|)(|{c'}^-_1|+|c^-_1|)} \, \big[ [\,x\,,c^-_2] \, , {c'}^-_2 \big] \ot c^-_1 \ot {c'}^-_1 \,\Big)  \, =   \hfill  \cr
  \quad   = \;  \Sym_{\mathbb{A}_3}\Big(\, [\,x\,,c^-_1] \ot \big[\,c^-_2 , {c'}^-_1\big] \ot {c'}^-_2 \, + \, {(-1)}^{|c^-_2||{c'}^-_1|} \, [\,x\,,c^-_1] \ot {c'}^-_1 \ot \big[\,c^-_2,{c'}^-_2\big] \, +   \hfill  \cr
  \quad \qquad   + \, {(-1)}^{|c^-_1| |c^-_2| } \, \big[ [\,x\,,c^-_2] \, , {c'}^-_1 \big] \ot {c'}^-_2 \ot c^-_1 \, +   \hfill  \cr
  \quad \qquad \qquad   + \, {(-1)}^{|{c'}^-_2||{c'}^-_1|
+|c^-_2||c^-_1|+|{c'}^-_2| |c^-_1|} \, \big[ [\,x\,,c^-_2] \, , {c'}^-_2 \big] \ot c^-_1 \ot {c'}^-_1 \,\Big)  \; =   \hfill  \cr
  \quad   = \;  \Sym_{\mathbb{A}_3}\Big(\, [\,x\,,c^-_1] \ot \big[\,c^-_2 , {c'}^-_1\big] \ot {c'}^-_2 \, + \, {(-1)}^{|c^-_2||{c'}^-_1|} \, [\,x\,,c^-_1] \ot {c'}^-_1 \ot \big[\,c^-_2,{c'}^-_2\big] \, +   \hfill  \cr
  \quad \qquad   + \, {(-1)}^{||c^-_1| |c^-_2|} \, \big[\,x\,,\big[\,c^-_2,{c'}^-_1\big]\big] \ot {c'}^-_2 \ot c^-_1 \, -   \hfill  \cr
  \quad \qquad \qquad   - \, {(-1)}^{|c^-_1| |c^-_2| + |x| |c^-_2|} \, \big[\, c^-_2 , \big[\,x\,,{c'}^-_1\big] \big] \ot {c'}^-_2 \ot c^-_1 \, +   \hfill  \cr
  \quad \qquad \qquad \qquad   + \, {(-1)}^{|x||{c'}^-_2|+|{c'}^-_1||{c'}^-_2|} \, \big[\, {c'}^-_2 , [\,x\,,c^-_1] \big] \ot c^-_2 \ot {c'}^-_1 \,\Big)  \; =   \hfill  \cr
  \quad   = \;  \Sym_{\mathbb{A}_3}\Big(\, [\,x\,,c^-_1] \ot \big[\,c^-_2 , {c'}^-_1\big] \ot {c'}^-_2 \, +   \hfill  \cr
  \quad \qquad \qquad   + \, {(-1)}^{|c^-_2||{c'}^-_1|} \, [\,x\,,c^-_1] \ot {c'}^-_1 \ot \big[\,c^-_2,{c'}^-_2\big] \, +   \hfill  \cr
  \quad \qquad \qquad \qquad \qquad   + \, {(-1)}^{|c^-_1||c^-_2|} \, \big[\,x\,,\big[\,c^-_2,{c'}^-_1\big]\big] \ot {c'}^-_2 \ot c^-_1 \,\Big)  \; =   \hfill  \cr
  \quad   = \;  \Big( \Sym_{\mathbb{A}_3} \circ\, \big(\, [\,x\,,-\,] \ot \id \ot \id \big) \Big) \big(\, c^-_1 \ot \big[\,c^-_2,{c'}^-_1\big] \ot {c'}^-_2 \, +   \hfill  \cr
  \quad \qquad   + \, {(-1)}^{|c^-_2||{c'}^-_1|} \, c^-_1 \ot {c'}^-_1 \ot \big[\,c^-_2,{c'}^-_2\big] \, + \, {(-1)}^{|c^-_1||{c'}^-_1|} \, \big[\,{c'}^-_1,c^-_1\big] \ot {c'}^-_2 \ot c^-_2 \,\big)  \; =   \hfill  \cr
   }  $$
  $$  \displaylines{
  \quad   = \;  \Big( \Sym_{\mathbb{A}_3} \circ\, \big(\, [\,x\,,-\,] \ot \id \ot \id \big) \Big) \big(\, {(-1)}^{|c^-_1||{c'}^-_2|} \, \big[\,{c'}^-_1,c^-_1\big] \ot {c'}^-_2 \ot c^-_2 \,\big)  \, +   \hfill  \cr
  \quad \qquad   + \, c^-_1 \ot \big[\, c^-_2 , {c'}^-_1 \big] \ot {c'}^-_2 \, + \, {(-1)}^{|c^-_2||{c'}^-_1|} \, c^-_1 \ot {c'}^-_1 \ot \big[\,c^-_2,{c'}^-_2\big]  \; =   \hfill  \cr
  \quad \hskip101pt   = \;  \Big( \Sym_{\mathbb{A}_3} \circ\, \big(\, [\,x\,,-\,] \ot \id \ot \id \big) \Big) \big([\![\, c^- , c^- \,]\!]\big)  \; = \;  x\,.\,[\![\, c^- , c^- \,]\!]   }  $$
 Note that, although it is written in a compact form, the last equality involves a long computation where a lot of signs are involved.
                                                                    \par
   Finally, the cocycle condition follows from the fact that $\partial c$ is a
coboundary. Indeed, by using the co-Jacobi identity, for all  $ \, x , y \in \lieg \, $  we have
  $$  \displaylines{
   (\partial c)\big([x,y]\big)  \; = \;  [x,y]\,.\,c  \; = \;
\big[[x,y]\,,c_1\big] \ot c_2 \, + \, {(-1)}^{(|x|+|y|)|c_1|} \, c_1 \ot \big[[x,y]\,,c_2\big]  \; =   \hfill  \cr
   \quad   = \;  \big[x,[y,c_1]\big] \ot c_2 \, - \, {(-1)}^{|x||y|} \, \big[y,[x,c_1]\big] \ot c_2 \; +   \hfill  \cr
   \hfill   + \; {(-1)}^{(|x|+|y|)|c_1|} \, c_1 \ot \big[x,[y,c_2]\big] \,
- \, {(-1)}^{(|x|+|y|)|c_1|+|x||y|} \, c_1 \ot \big[y,[x,c_2]\big]  \; =   \quad \qquad  \cr
   \hfill   = \;  x\,.\big((\partial c)(y)\big) \, - \, {(-1)}^{|x||y|} \, y\,.\big((\partial c)(x)\big)   \quad  \qed  }  $$
%
%
%

\vskip9pt

\begin{free text}  \label{deformations of LSbA's}
{\bf Deformations of Lie superbialgebras.}
 To the best of the our knowledge, there is a lack of reference and bibliography about
 a general theory of  deformations  for Lie superbialgebras. The reason may lie in the fact that
 Lie superbialgebras may be related to Lie bialgebras through the process of bosonization. Moreover,
 the theory of deformations of Lie bialgebras can be considered as a sub-theory of that of Lie algebras. For
 the latter, we refer to
 \cite{CG},  \cite{MW}, and references therein.
 In fact, as in the present work we are mainly interested in two special kinds of deformations performed by
 \emph{even} objects,
 where either the Lie cobracket or the Lie bracket alone is deformed, leaving the ``other side'' of the overall structure untouched,
 we rely heavily on definitions and basic results on Lie bialgebra cohomology.

 \vskip3pt
  We study first a deformation of the super cobracket.  Let  $ (\lieg, [\,\ ,\ ],\, \delta)$
   be a Lie superbialgebra.  A homogeneous even element $ \, c \in \lieg \otimes \lieg \, $  is
   called a \emph{twist} if
\begin{equation}  \label{eq: twist-cond_Lie-superbialg}
   x\,.\,\big((\id +\sigma )(c)\big)  \; = \;  0
     \quad  \text{ and }  \quad
   x\,.\,\big( \Sym_{\mathbb{A}_{3}}(\id \otimes \delta)(c)  - [\![\, c \, , c \,]\!] \,\big)  \; = \;  0
\end{equation}
for all  $ \, x \in \lieg \, $;  \,here the (adjoint) action of  $ \, x \in \lieg \, $  is onto  $ \, \lieg \otimes \lieg \otimes \lieg \, $   --- cf.\ the action onto  $ \lieg \otimes \lieg \, $  above.  As a side note, observe that this definition coincides with that of twist for Lie bialgebras, but with respect to a different symmetric braiding.

\vskip11pt

\begin{prop}  \label{prop:twist-superliebial}
 Let  $ \, \big(\, \lieg \, , [\,\ ,\ ] \, , \, \delta \,\big) \, $  be a Lie superbialgebra and  $ \, c \in \lieg \otimes \lieg \, $   a twist. Then the map
$ \, \delta^{\,c} : \lieg \relbar\joinrel\longrightarrow \lieg \wedge \lieg \, $  defined by
\begin{equation}  \label{eq: def_twist-delta}
  \delta^{\,c}(x)  \; := \;  \delta(x) - (\partial c)(x)  \; = \;  \delta(x) - x\,.\,c   \qquad  \forall \;\; x \in \lieg
\end{equation}
 is a new super cobracket on the Lie superalgebra  $ \, \big(\, \lieg \, ; \, [\,\ ,\,\ ] \,\big) \, $  making  $ \, \big(\, \lieg \, ; \, [\,\ ,\ ] \, , \, \delta^{c} \,\big) \, $  into a new Lie superbialgebra.
\end{prop}

\pf
 The proof follows  \cite[Theorem 8.1.7]{Mj},  mutatis mutandis, taking into account the sign changes.
                                                              \par
   As the image of  $ \delta $  lies in  $ \Lambda^2(\lieg) \, $,  the first condition in  \eqref{eq: twist-cond_Lie-superbialg}  (which is the same as the first in Proposition \ref{prop:coboundary-liesuperbialgebra})  ensures that also the image of   $ \delta^c $  is contained in  $ \Lambda^2(\lieg) \, $.
                                                              \par
   In order to prove that  $ \delta^c $  is a Lie super-cobracket, we have to check the identity  $ \, \Sym_{\mathbb{A}_3}\!\big( (\id \ot \delta^c) \circ \delta^c \,\big)(x) = 0 \, $  for all  $ \, x \in \lieg \, $.  But  $ \, \Sym_{\mathbb{A}_3}\!\big( (\id \ot \delta) \circ \delta \,\big)(x) = 0 \, $,  \,hence
\begin{align*}
   (\id \ot \delta^c) \circ \delta^c  &  \; = \;  \big(\id \ot (\delta - \partial c) \big) \circ (\delta - \partial c)  \; =  \\
   &  = \;  (\id \ot \delta\,) \circ \delta \, - \, (\id \ot \delta\,) \circ (\partial c) \, - \, (\id \ot \partial c) \circ \delta \, + \, (\id \ot \partial c) \circ (\partial c)
\end{align*}
 and we have
 $ \, \Sym_{\mathbb{A}_3}\!\big((\id \ot \partial c) \circ (\partial c)\big)(x) =
 x.[\![\, c \, , c \,]\!] \, $,
 by the proof of  Proposition \ref{prop:coboundary-liesuperbialgebra}:
 \,thus it is enough to prove that
  $$  \Sym_{\mathbb{A}_{3}}\!\big((\id\ot\delta)\circ(\partial c) + (\id\ot\partial c) \circ \delta \,\big)(x)  \; = \;  x\,.\,\big( \Sym_{\mathbb{A}_{3}}(\id \otimes \delta)(c)\big)  $$
for all  $ \, x \in \lieg \, $.  Set  $ \, \delta(x) = x_{[1]} \ot x_{[2]} \in \Lambda^2(\lieg) \, $:  \,then by the cocycle condition we have
\begin{align*}
& \big((\id \ot \delta) \circ (\partial c) + (\id \ot \partial c) \circ \delta \,\big) (x)  \; =  \\
& \qquad  = \;
[x,c_{1}]\ot \delta(c_{2}) + (-1)^{|x||c_{1}|}c_{1}\ot \delta([x,c_{2}]) +
x_{[1]}\ot [x_{[2]},c_{1}]\ot c_2 \, +  \\
&  \quad \qquad  + (-1)^{|x_{[2]}||c_{1}|} x_{[1]}\ot c_{1}\ot [x_{[2]},c_{2}]  \; =  \\
&  \qquad  = \;
[x,c_{1}]\ot \delta(c_{2}) + (-1)^{|x||c_{1}|}c_{1}\ot \big( x \,.\, \delta(c_{2})\big) - c_1 \ot \big( c_2 \,.\, \delta(x) \big) \, +  \\
&  \qquad \qquad
+ x_{[1]}\ot [x_{[2]},c_{1}]\ot c_2
  + (-1)^{|x_{[2]}||c_{1}|} x_{[1]}\ot c_{1}\ot [x_{[2]},c_{2}]  \; =  \\
&  \qquad  = \;
x \,.\, \big(c_{1}\ot \delta(c_{2})\big) -
c_1 \ot \big( c_2 \,.\, \delta(x) \big) +x_{[1]}\ot [x_{[2]},c_1]\ot c_2
\, +  \\
&   \qquad\qquad
  + (-1)^{|x_{[2]}||c_{1}|} x_{[1]}\ot c_{1}\ot [x_{[2]},c_{2}]  \; =  \\
&  \qquad  = \;
x \,.\, \big((\id\ot \delta)(c)\big) -
c_{1} \ot \big( c_2 \,.\, \delta(x) \big) +x_{[1]}\ot [x_{[2]},c_{1}] \ot c_2
\, +  \\
&   \qquad\qquad
  + (-1)^{|x_{[2]}||c_{1}|} x_{[1]}\ot c_{1}\ot [x_{[2]},c_{2}]
\end{align*}
 Since  $ \, \Sym_{\mathbb{A}_3}\!\big( x \,.\, \big((\id\ot \delta)(c)\big) \big) = x \,.\, \big( \Sym_{\mathbb{A}_{3}}(\id \otimes \delta)(c)\big) \, $,  \,to finish this calculation it is enough to prove that
  $$  \Sym_{\mathbb{A}_3}\!\big(-c_1 \ot (c_2\,.\,\delta(x)) +
  x_{[1]}\ot [x_{[2]},c_{1}] \ot c_2 + (-1)^{|x_{[2]}||c_{1}|} x_{[1]}\ot c_{1}\ot [x_{[2]},c_{2}]\big)  \; = \;  0  $$
which follows directly from a straightforward computation.
To finish the proof we just note that the super cobracket $\delta^{c} = \delta - \partial c$
satisfies the cocycle condition since it is a difference of two functions that they do,
$\delta$ because it is a super cobracket for a Lie superbialgebra and $\partial c$
because it is a coboundary.
\epf

\end{free text}
\vskip9pt

\begin{definition}  \label{def: twist-deform_Lie-bialg's}
 Let $ (\lieg, [\,\ ,\ ],\, \delta)$ be a Lie superbialgebra and
 $ \, c \in \lieg \otimes \lieg \, $ a twist. The associated Lie superbialgebra
 $ \; \lieg^{\,c} \, := \, \big(\, \lieg \, ; \, [\,\ ,\,\ ] \, , \, \delta^{\,c} \,\big) \; $
  is called a \textit{deformation by twist}
 (or  ``\textit{twist deformation\/}'')  of the original Lie superbialgebra
 $ \lieg \, $.   \hfill  $ \diamondsuit $
\end{definition}

\vskip7pt

   We introduce now the dual recipe of  \textit{deformation by  $ 2 $--cocycle},
   where one deforms the Lie superbracket while keeping the Lie supercobracket untouched.
   Let again  $ \big(\, \lieg \, ; \, [\,\ ,\,\ ] \, , \, \delta \,\big) \, $  be a Lie superbialgebra.
                                                           \par
   A homogeneous even element  $ \, \chi \in {(\lieg \otimes \lieg)}^* \, $  is called a  \textit{$ 2 $--cocycle\/}  if
\begin{equation}  \label{eq: 2-cocycle-cond_Lie-superbialg}
   \alpha\,.\,\big(\, \chi \circ (\id + \, \sigma )\big)  \, = \,  0
     \quad  \text{and}  \quad
   \alpha\,.\,\big(\, \chi \circ \big( \id \otimes [\,\ ,\,\ ] \big) \circ \Sym_{\mathbb{A}_{3}} - \, [\![\, \chi \, , \chi \,]\!]_* \,\big)  \; = \;  0  \;\;
\end{equation}
 for all  $ \, \alpha \in \lieg^* \, $.  The dual version of
 Proposition \ref{prop:twist-superliebial}  reads as follows.

\vskip11pt

\begin{prop}  \label{prop:cocycle-superliebial}
 Let  $ \, (\lieg, [\,\ ,\ ],\, \delta) \, $  be a Lie superbialgebra and
 $ \, \chi \in (\lieg \otimes \lieg)^{*} \, $  a  $ 2 $--cocycle.  Then the map
 $ \, [\,\ ,\ ]_{\chi} : \lieg \ot \lieg \relbar\joinrel\longrightarrow  \lieg \, $
 defined for all  $ \, x , y \in \lieg \, $  by
\begin{equation}  \label{eq: def_cocycle-bracket}
  [x,y ]_{\chi} := \, [x,y] - (\partial_{*} \chi)(x\ot y)\, =
  \, [x,y] - x_{[1]} \, \chi(x_{[2]},y) \, - \, {(-1)}^{|x| |y_{[1]}|} \, y_{[1]} \, \chi(x,y_{[2]})
\end{equation}
 is a new Lie superbracket on the Lie supercoalgebra
 $ \big(\, \lieg \, ; \, \delta\,\big) \, $  making
 $ \, \big(\, \lieg \, ; \, [\,\ ,\ ]_{\chi} \, , \, \delta \,\big) \, $  into a (new) Lie superbialgebra.
\qed
\end{prop}

\vskip7pt

   The proof is a sheer dualization of that for  Proposition \ref{prop:twist-superliebial},
   hence is left to the reader.  The relevant terminology is fixed as follows:

\vskip11pt

\begin{definition}  \label{def: cocyc-deform_Lie-superbialg's}
 Let  $ \, (\lieg, [\,\ ,\ ],\, \delta) \, $  be a Lie superbialgebra and
 $ \, \chi \in (\lieg \otimes \lieg)^* \, $  be a  $ 2 $--cocycle for it.
 Then the Lie superbialgebra  $ \, \lieg_\chi \, := \, \big(\, \lieg \, ; \, {[\,\ ,\ ]}_\chi \, , \, \delta \,\big) \, $
 is called a  \textit{defor\-mation by 2--cocycle}  (or  ``\textit{2--cocycle deformation\/}'')
    \hbox{of the Lie superbialgebra  $ \lieg \, $.   \hskip17pt \hfill  $ \diamondsuit $}
\end{definition}

\vskip5pt

   At last, we point out that the two notions of ``twist'' and of
   ``$ 2 $--cocycle''  for Lie superbialgebras,
   as well as the associated deformations, are so devised as to be dual to each other.
   The following result then holds, whose proof is left to the reader:

\vskip7pt

\begin{prop}  \label{prop: duality-deforms x LbA's}
 Let  $ \lieg $  be a Lie superbialgebra,  $ \lieg^* $  the dual Lie superbialgebra.
 \vskip2pt
   {\it (a)}\,  Let  $ c $  be a twist for  $ \lieg \, $,  and  $ \chi_c $  the image of
   $ c $  in  $ {\big( \lieg^* \otimes \lieg^* \big)}^* $
   for the natural composed embedding
   $ \, \lieg \otimes \lieg \lhook\joinrel\relbar\joinrel\longrightarrow \lieg^{**}
   \otimes \lieg^{**} \lhook\joinrel\relbar\joinrel\longrightarrow {\big( \lieg^* \otimes \lieg^* \big)}^* \, $.
   Then  $ \chi_c $  is a 2--cocycle  for  $ \lieg^* \, $,  and there exists a canonical isomorphism
   $ \, {\big( \lieg^* \big)}_{\chi_c} \cong {\big( \lieg^c \big)}^* \, $.
 \vskip2pt
   {\it (b)}\,  Let  $ \chi $  be a 2--cocycle  for  $ \lieg \, $;  assume that
   $ \lieg $  is finite-dimensional, and let  $ c_{\,\chi} $  be the image of  $ \, \chi $
   in the natural identification  $ \, {(\lieg \otimes \lieg)}^* = \lieg^* \otimes \lieg^* \, $.
   Then  $ c_{\,\chi} $  is a twist for  $ \lieg^* \, $,
   and there exists a canonical isomorphism
   $ \, {\big( \lieg^* \big)}^{c_{\,\chi}} \cong {\big( \lieg_\chi \big)}^* \, $.
\qed
\end{prop}

\vskip17pt

\subsection{Multiparameter Lie superbialgebras (=MpLSbA's)}  \label{subsec: MpLSbA's}  {\ }
 \vskip5pt
   Let  $ \, (A \, , p) \, $  be a fixed Cartan super-datum (cf.\  Definition \ref{def: Cartan-super-datum}):  we fix a matrix of Cartan type  $ \, P := {\big(\, p_{i,j} \big)}_{i, j \in I} \in M_n(\Bbbk) \, $  associated with  $ A \, $,  i.e.\  $ \, P_s := 2^{-1} \big( P + P^{\,\scriptscriptstyle T} \big) = DA =: P_{\!\Apicc} \, $,  and we let  $ \, \R := \big(\, \lieh \, , \Pi \, , \Pi^\vee \,\big) \, $  be a realization of  $ P $.
                                                                     \par
   With these data, we are going to associate some Lie superbialgebra  $ \liegRpP \, $,
   to be called ``multiparameter Lie superbialgebra'' (or MpLSbA's in short): these will depend explicitly on the parameters in  $ P $
   (i.e.\ its entries), but also on the chosen realization  $ \R \, $.
   Up to technicalities, these MpLSbA's arise as deformations by twist of Lie superbialgebras of a special family,
   which comes from quantum counterparts introduced by Yamane   --- hence we call them ``of Yamane's type''.

\vskip7pt

\begin{free text}  \label{Yamane's LieSBial's}
 {\bf Yamane's Lie superbialgebras.}  In the present subsection, we refer to the notions in  Definition \ref{def: realization of P}
 and all what follows in  \S \ref{sec: Cartan-data_realiz's},  \textsl{but\/}
 working now with  $ \Bbbk $  as ground ring instead of  $ \kh \, $;  then, since we are dealing with  $ \Bbbk $  rather than  $ \kh \, $,
 we identify the space  $ \lieh $  with  $ \overline{\lieh} \, $,  the roots  $ \alpha_j $  with  $ \overline{\alpha}_j \, $,  \,etc.
 (cf.\  \S \ref{def: realization of P}).  Following  \S \ref{MpMatrices, Cartan & realiz.'s},
 let  $ \, (A \, , p) \, $  be a Cartan super-datum, let  $ \, P_{\!\Apicc} = DA \, $  be the symmetrization of  $ A \, $,
 and fix a realization  $ \, \R := \big(\, \lieh \, , \Pi \, , \Pi^\vee \,\big) \, $  of  $ P_{\!\scriptscriptstyle A} \, $
 which is  \textsl{split},  \textsl{straight\/}  and  \textsl{minimal}.
 \vskip1pt
   We define the Lie superalgebra  $ \liegRpPa $  as being the one generated over
   $ \Bbbk $  by the  $ \Bbbk $--subspace  $ \lieh $  together with elements  $ E_i $
   and  $ F_i $  ($ \, i \in I \, $),  whose parity is fixed by
   $ \, |T| := \zero \, $  and  $ \, \big|E_i\big| := p(i) =:  \big|F_i\big| \, $
   (for all  $ \, T \in \lieh \, $  and  $ \, i \in I \, $),
   with relations as follows (for all  $ \; T \, , T' , T'' \in \lieh \, $,
   $ \; i , j, \ell, k \in I \, $,  $ \; i \neq \ell \, $):
  $$  \displaylines{
   \big[\, T' , T'' \,\big] \, = \, 0  \;\; ,
   \qquad   \big[\, T , E_j \,\big] \, - \, \alpha_j(\,T) \, E_j \, = \, 0  \;\; ,
   \qquad   \big[\, T , F_j \,\big] \, + \, \alpha_j(\,T) \, F_j \, = \, 0  \cr
   \big[\, E_i \, , F_\ell \,\big] \, = \, \delta_{i\ell} \, {{\;T_i^+ \! + T_i^-\;} \over {\;2\,d_i\;}}  \;\; ,
   \;\qquad  \big[ E_i \, , E_j \big] \, = \, 0  \;\; ,  \quad  \big[ F_i \, , F_j \big] \, = \, 0  \!\qquad   \text{if} \quad  a_{i,j} = 0  \cr
   {\big(\ad\,(E_i)\big)}^{1-a_{ij}}(E_j) \, = \, 0  \;\; ,   \;\quad
   {\big(\ad\,(F_i)\big)}^{1-a_{ij}}(F_j) \, = \, 0   \qquad  \text{with \ }
   p(i) = \zero \; , \;\; i \not= j  \cr
%
    \text{and then also all of the following, with  $ \, X \in \{E\,,F\} \, $:}_{\phantom{|}}^{\phantom{\big|}}
   \hfill  \cr
%
%
\Big[ \big[ [ X_i \, , X_j ] \, , X_k \big] \, , X_j \Big] \, = \, 0  \quad
\text{in any of the following cases:}  \cr
\begin{tikzpicture}
    \draw[black, very thin] (2-0.1, -0.1) -- (2+0.1, 0.1);
    \draw[black, very thin] (2-0.1, 0.1) -- (2+0.1, -0.1);
    \draw[black, very thin] (2.2, 0) -- (3.2, 0);
    \draw[black, very thin] (3.4-0.1, -0.1) -- (3.4+0.1, 0.1);
    \draw[black, very thin] (3.4-0.1, 0.1) -- (3.4+0.1, -0.1);
   \draw [black, very thin](3.4, 0) circle (0.15 cm);
    \draw[black, very thin] (3.6, 0) -- (4.6, 0);
    \draw[black, very thin] (4.8-0.1, -0.1) -- (4.8+0.1, 0.1);
    \draw[black, very thin] (4.8-0.1, 0.1) -- (4.8+0.1, -0.1);
\draw[black, very thin] (2,0)
node[above=0.2cm, scale=0.7] {$i$};
\draw[black, very thin] (3.4,0)
node[above=0.2cm, scale=0.7] {$j$};
\draw[black, very thin] (4.8,0)
node[above=0.2cm, scale=0.7] {$k$};
\end{tikzpicture}
  \quad \qquad
\begin{tikzpicture}
    \draw[black, very thin] (2-0.1, -0.1) -- (2+0.1, 0.1);
    \draw[black, very thin] (2-0.1, 0.1) -- (2+0.1, -0.1);
    \draw[black, very thin] (2.2, 0) -- (3.2, 0);
    \draw[black, very thin] (3.4-0.1, -0.1) -- (3.4+0.1, 0.1);
    \draw[black, very thin] (3.4-0.1, 0.1) -- (3.4+0.1, -0.1);
   \draw [black, very thin](3.4, 0) circle (0.15 cm);
    \draw[black, very thin] (3.6, +0.1) -- (4.6, +0.1);
    \draw[black, very thin] (3.6, -0.1) -- (4.6, -0.1);
    \draw[black, very thin] (4.1-0.1, -0.1-0.1) -- (4.1+0.1, 0.1-0.1);
    \draw[black, very thin] (4.1-0.1, 0.1+0.1) -- (4.1+0.1, -0.1+0.1);
    \draw [black, very thin, fill=black](4.8, 0) circle (0.15 cm);
\draw[black, very thin] (2,0)
node[above=0.2cm, scale=0.7] {$i$};
\draw[black, very thin] (3.4,0)
node[above=0.2cm, scale=0.7] {$j$};
\draw[black, very thin] (4.8,0)
node[above=0.2cm, scale=0.7] {$k$};
\end{tikzpicture}
  \qquad \quad
\begin{tikzpicture}
    \draw[black, very thin] (2-0.1, -0.1) -- (2+0.1, 0.1);
    \draw[black, very thin] (2-0.1, 0.1) -- (2+0.1, -0.1);
    \draw[black, very thin] (2.2, 0) -- (3.2, 0);
    \draw[black, very thin] (3.4-0.1, -0.1) -- (3.4+0.1, 0.1);
    \draw[black, very thin] (3.4-0.1, 0.1) -- (3.4+0.1, -0.1);
   \draw [black, very thin](3.4, 0) circle (0.15 cm);
    \draw[black, very thin] (3.6, +0.1) -- (4.6, +0.1);
    \draw[black, very thin] (3.6, -0.1) -- (4.6, -0.1);
    \draw[black, very thin] (4.1-0.1, -0.1-0.1) -- (4.1+0.1, 0.1-0.1);
    \draw[black, very thin] (4.1-0.1, 0.1+0.1) -- (4.1+0.1, -0.1+0.1);
    \draw[black, very thin] (4.8-0.1, -0.1) -- (4.8+0.1, 0.1);
    \draw[black, very thin] (4.8-0.1, 0.1) -- (4.8+0.1, -0.1);
\draw[black, very thin] (2,0)
node[above=0.2cm, scale=0.7] {$i$};
\draw[black, very thin] (3.4,0)
node[above=0.2cm, scale=0.7] {$j$};
\draw[black, very thin] (4.8,0)
node[above=0.2cm, scale=0.7] {$k$};
\end{tikzpicture}
   \cr
 }  $$
  $$  \displaylines{
   \Big[ \big[ \big[ X_{n-1} \, , X_n \big] \, , X_n \big] \, , X_n \Big]^{\phantom{\Big|}} = \; 0  \cr
\begin{tikzpicture}[scale=0.8]
\draw[black, very thin] (2-0.1, -0.1) -- (2+0.1, 0.1);
    \draw[black, very thin] (2-0.1, 0.1) -- (2+0.1, -0.1);
    \draw[black, very thin] (2.2, 0.08) -- (3.3, 0.08);
    \draw[black, very thin] (2.2, -0.08) -- (3.3, -0.08);
\draw [black, very thin, fill=black](3.7, 0) circle (0.15 cm);
\draw[black, very thin] (2.8-0.1, -0.3) -- (2.8+0.1, 0);
\draw[black, very thin] (2.8-0.1, 0.3) -- (2.8+0.1, 0);
\draw[black, very thin] (2,0)
node[above=0.2cm, scale=0.7] {$n\!-\!1$};
\draw[black, very thin] (3.5,0)
node[above=0.2cm, scale=0.7] {\hskip15pt $n$};
\draw[black, very thin] (0,0)
node { with };
\draw[black, very thin] (7,0)
node { and $A$ of type $B_n$ };
\end{tikzpicture}
  \cr
%
%
   \Big[ \big[ X_{n-2} \, ,X_{n-1} \big] \, , X_n \Big] \; - \; \Big[ \big[ X_{n-2} \, ,
   X_n \big] \, , X_{n-1} \Big]^{\phantom{\Big|}}  \; = \;\,  0  \cr
\begin{tikzpicture}    \quad
\draw[black, very thin] (4.3, 0.1) -- (5.2, 0.6);
\draw[black, very thin] (4.3, -0.1) -- (5.2, -0.6);
\draw[black, very thin] (5.35, 0.3) -- (5.35, -0.3);
\draw[black, very thin] (5.42, 0.3) -- (5.42, -0.3);
\draw [black, very thin](5.4, 0.6) circle (0.13 cm);
\draw [black, very thin](5.4, -0.6) circle (0.13 cm);
\draw[black, very thin] (4-0.1, 0.1) -- (4+0.1, -0.1);
\draw[black, very thin] (4-0.1, -0.1) -- (4+0.1, 0.1);
\draw[black, very thin] (5.4-0.1, 0.1+0.6) -- (5.4+0.1, -0.1+0.6);
\draw[black, very thin] (5.4-0.1, -0.1+0.6) -- (5.4+0.1, 0.1+0.6);
\draw[black, very thin] (5.4-0.1, 0.1-0.6) -- (5.4+0.1, -0.1-0.6);
\draw[black, very thin] (5.4-0.1, -0.1-0.6) -- (5.4+0.1, 0.1-0.6);
\draw[black, very thin] (5.4,0.6)
node[above=0.1cm, scale=0.7] {$n$};
\draw[black, very thin] (4,0)
node[above=0.1cm, scale=0.7] {$n\!-\!2$};
\draw[black, very thin] (5.4,-0.6)
node[below=0.1cm, scale=0.7] {$n\!-\!1$};
\draw[black, very thin] (2,0.3)
node[below=0.1cm] {with};
\draw[black, very thin] (8.5,0.3)
node[below=0.1cm] { and  $ A $  of type  $ \, D_n(I\!I) $};
\end{tikzpicture}
%
  \cr
   \bigg[ \Big[ \big[\, [X_{n-2} \, , X_{n-1}] \, , X_n \big] \, , [X_{n-2}, X_{n-1}] \Big] \, ,
   \, X_{n-1} \,\bigg]  \; = \;\,  0  \cr
   \text{for}  \quad
\begin{tikzpicture}[scale=0.8]
\draw[black, very thin] (2.2, 0) -- (3.2, 0);
\draw[black, very thin] (3.6, 0.07) -- (4.8, 0.07);
\draw[black, very thin] (3.6, -0.07) -- (4.8, -0.07);
\draw[black, very thin] (4.2-0.1, -0.1+0.1) -- (4.2+0.1, 0.1+0.1);
\draw[black, very thin] (4.2-0.1, 0.1-0.1) -- (4.2+0.1, -0.1-0.1);
\draw[black, very thin] (2-0.1, -0.1) -- (2+0.1, 0.1);
\draw[black, very thin] (2-0.1, 0.1) -- (2+0.1, -0.1);
\draw[black, very thin] (3.4-0.1, -0.1) -- (3.4+0.1, 0.1);
\draw[black, very thin] (3.4-0.1, 0.1) -- (3.4+0.1, -0.1);
\draw [black, very thin](2, 0) circle (0.15 cm);
\draw [black, very thin](3.4, 0) circle (0.15 cm);
\draw [black, very thin](5, 0) circle (0.15 cm);
\draw[black, very thin] (2,0) node[above=0.2cm, scale=0.7] {$n\!-\!2$};
\draw[black, very thin] (3.4,0) node[above=0.2cm, scale=0.7] {$n\!-\!1$};
\draw[black, very thin] (5,0) node[above=0.2cm, scale=0.7] {$n$};
\end{tikzpicture}
 \quad   \text{and  $ \, A \, $  of type  $ \, C_n $}  \cr
  \cr
   \Bigg[ \Bigg[ \bigg[ \Big[ \big[\,  [X_{n-3} \, , \, X_{n-2}] \, , \, X_{n-1} \big] \, ,
   X_n \,\Big] \, , \, X_{n-1} \,\bigg] \, , X_{n-2} \,\Bigg] \, , \, X_{n-1} \,\Bigg]
   \,\; = \;\;  0   \quad  \cr
%
    \centering   \text{for}  \quad
\begin{tikzpicture}[scale=0.8]
\draw[black, very thin] (0.8, 0) -- (1.8, 0);
\draw[black, very thin] (2.2, 0) -- (3.2, 0);
\draw[black, very thin] (3.6, 0.07) -- (4.8, 0.07);
\draw[black, very thin] (3.6, -0.07) -- (4.8, -0.07);
\draw[black, very thin] (4.2-0.1, -0.1+0.1) -- (4.2+0.1, 0.1+0.1);
\draw[black, very thin] (4.2-0.1, 0.1-0.1) -- (4.2+0.1, -0.1-0.1);
\draw[black, very thin] (0.6-0.1, -0.1) -- (0.6+0.1, 0.1);
\draw[black, very thin] (0.6-0.1, 0.1) -- (0.6+0.1, -0.1);
\draw[black, very thin] (2-0.1, -0.1) -- (2+0.1, 0.1);
\draw[black, very thin] (2-0.1, 0.1) -- (2+0.1, -0.1);
\draw[black, very thin] (3.4-0.1, -0.1) -- (3.4+0.1, 0.1);
\draw[black, very thin] (3.4-0.1, 0.1) -- (3.4+0.1, -0.1);
\draw [black, very thin](2, 0) circle (0.15 cm);
\draw [black, very thin](3.4, 0) circle (0.15 cm);
\draw [black, very thin](5, 0) circle (0.15 cm);
\draw[black, very thin] (0.6,0)
node[above=0.2cm, scale=0.7] {$n\!-\!3$};
\draw[black, very thin] (1.97,0)
node[above=0.2cm, scale=0.7] {$n\!-\!2$};
\draw[black, very thin] (3.4,0)
node[above=0.2cm, scale=0.7] {$n\!-\!1$};
\draw[black, very thin] (5,0)
node[above=0.2cm, scale=0.7] {$n$};
%
\end{tikzpicture}
 \quad  \text{and  $ \, A \, $  of type  $ \, C_n $}  }  $$
 \vskip5pt
   This  $ \liegRpPa $  is nothing but the simple Lie superalgebra associated with
   $ (A\,,p) \, $:  its presentation given above comes from the presentation of a suitable
   quantization of  $ U\big(\liegRpPa\big) $   --- in the sense of  \S \ref{QUESA's}  ---
   introduced in  \cite{Ya1}   --- in a ``quantum double version'', denoted in
   [\textit{loc.\ cit.}]  by  $ \mathcal{D}' $.  \textit{Moreover}, out of that quantization
   \textit{the Lie superalgebra  $ \liegRpPa $  inherits an additional Lie cobracket, which
   makes it into a Lie superbialgebra},  given on generators by  (for  $ \, T \in \lieh \, $
   and  $ \, i \in I \, $)
\begin{equation}  \label{eq: Lie-cobracket x gRPa}
  \delta\big(\,T\big) \, = \, 0  \quad ,   \qquad  \delta\big(E_i\big) \, = \, 2 \; T^+_i
  \!\wedge E_i  \quad ,   \qquad  \delta\big(F_i\big) \, = \,  2 \; T^-_i \!\wedge F_i
\end{equation}
   \indent   Let us now pick a realization  $ \R' $  of  $ P_\Apicc $  that is still
   straight and minimal but not necessarily split; then we can still define a Lie superalgebra
   $ \lieg^{\raise1pt\hbox{$ \scriptscriptstyle \R',p $}}_{\Ppicc_{\raise-3pt\hbox{$ \!\!
   \Apicc $}}} $
   with the same presentation as  $ \liegRpPa $  above (with a ``Cartan subalgebra''
   $ \lieh' $  replacing  $ \lieh \, $): indeed, the former will just be a quotient of the
   latter.  Moreover, the Lie cobracket of  $ \liegRpPa $  will clearly be pushed forward onto
   $ \lieg^{\raise1pt\hbox{$ \scriptscriptstyle \R',p $}}_{\Ppicc_{\raise-3pt\hbox{$ \!\!
   \Apicc $}}} $,
   so that the latter is indeed a Lie superbialgebra with the same description as
   $ \liegRpPa \, $,  but for changing  $ \R $  into  $ \R' $.  Finally, consider a realization
   $ \R'' $  of  $ P_\Apicc $  that is still straight but not necessarily split nor minimal.
   Then we can again define a Lie superalgebra
   $ \lieg^{\raise1pt\hbox{$ \scriptscriptstyle \R'',p $}}_{\Ppicc_{\raise-3pt\hbox{$ \!\!
   \Apicc $}}} $  with the same presentation as
   $ \lieg^{\raise1pt\hbox{$ \scriptscriptstyle \R',p $}}_{\Ppicc_{\raise-3pt\hbox{$ \!\!
   \Apicc $}}} $  (with  $ \lieh'' $  replacing  $ \lieh' \, $),
   the former being an extension of the latter, i.e.\
   $ \lieg^{\raise1pt\hbox{$ \scriptscriptstyle \R',p $}}_{\Ppicc_{\raise-3pt\hbox{$ \!\!
   \Apicc $}}} $
   is a Lie sub-superalgebra of  $ \lieg^{\raise1pt\hbox{$ \scriptscriptstyle \R'',p $}}_{\Ppicc_{\raise-3pt\hbox{$ \!\! \Apicc $}}} \, $;
   moreover, the Lie cobracket of  $ \lieg^{\raise1pt\hbox{$ \scriptscriptstyle \R',p $}}_{\Ppicc_{\raise-3pt\hbox{$ \!\! \Apicc $}}} $
   uniquely extends to all of
   $ \lieg^{\raise1pt\hbox{$ \scriptscriptstyle \R'',p $}}_{\Ppicc_{\raise-3pt\hbox{$ \!\!
   \Apicc $}}} $,
   so the latter is also a Lie  super\textsl{bi\/}algebra with Lie cobracket as in
   \eqref{eq: Lie-cobracket x gRPa}  again.
                                                          \par
   In the following, we will refer to any of the Lie superalgebras  $ \liegRpPa $
   just described   --- for every straight realization  $ \R $  of  $ P_\Apicc $  ---
   as \textit{``Yamane's Lie superbialgebras''}.
\end{free text}

\vskip11pt

\begin{free text}  \label{Constr-MpLieSBial's}
 {\bf Definition of multiparameter Lie superbialgebras.}  Let  $ \, (A \, , p) \, $  be a fixed
 Cartan super-datum (as in  Definition \ref{def: Cartan-super-datum}),  and let
 $ \, P := {\big(\, p_{i,j} \big)}_{i, j \in I} \in M_n(\Bbbk) \, $
 be a matrix of Cartan type associated with
 $ A \, $,  i.e.\  $ \, P_s := 2^{-1} \big( P + P^{\,\scriptscriptstyle T} \big) =
 DA =: P_{\!\Apicc} \, $.
%
                                                                      \par
    With these assumptions, let  $ \, \R := \big(\, \lieh \, , \Pi \, , \Pi^\vee \,\big) \, $
    be a realization of  $ P $.
 \vskip5pt
   \textit{We define the Lie  $ \k $--superalgebra  $ \liegRpP $  as the one having the very
   same presentation as  $ \liegRpPa $  above,  \textsl{but\/}  with respect to the
   realization  $ \R \, $}.  Then we prove that this  $ \liegRpP $  actually bears a
   structure of Lie super\textsl{bi\/}algebra,  through a series of steps.
 \vskip5pt
   First of all,  \textsl{let  $ \R $  be a straight realization of
   $ \, P_{\scriptscriptstyle A} := D\,A \, $};  \,then Yamane's Lie superbialgebra
   $ \liegRpPa $  is defined (see  \S \ref{Yamane's LieSBial's}  above).
                                                                  \par
   Second, for any antisymmetric matrix
   $ \; \Theta = \big( \theta_{gk} \big)_{g, k \in \G} \in \lieso_t(\Bbbk) \; $
   we consider
\begin{equation}  \label{eq: def_twist_Lie-bialg}
  j_\Thetapicc  \; := \;  {\textstyle \sum_{g,k=1}^t} \theta_{gk} \, H_g \otimes H_k  \;
  \in \; \lieh \otimes \lieh  \; \subseteq \;  \lieg \otimes \lieg
\end{equation}
 \textit{This is indeed a  \textsl{twist\/}  element\/}  (cf.\
 \eqref{eq: twist-cond_Lie-superbialg} and Definition \ref{def: twist-deform_Lie-bialg's})
 for Yamane's Lie superbialgebra  $ \liegRpPa \, $:  \,we call this  $ \, j_\Thetapicc \, $
 \textit{the toral twist\/}  (or  \textit{``twist of toral type''\/})  associated with
 $ \, \Theta\, $.  Then out of  $ \, \Theta \, $  we can consider the deformed Lie superalgebra
 $ \, {\big( \liegRpPa \big)}^{j_\Thetapicc} \, $   --- as in  \S \ref{deformations of LSbA's}
 ---   the ``deformed'' multiparameter matrix
 $ \; P_\Thetapicc = \, {\big(\, p^\Thetapicc_{i,j} \big)}_{i, j \in I} \, := \,
 P_\Apicc - \mathfrak{A} \, \Theta \, \mathfrak{A}^{\,\scriptscriptstyle T} \; $
 as in  \eqref{def-P_Phi}   --- again of Cartan type  $ A \, $, the same as  $ P $  ---
 and its corresponding ``deformed'' straight realization
 $ \, \R_\Thetapicc := \big(\, \lieh \, , \Pi \, , \Pi^\vee_\Thetapicc \! :=
 {\big\{ T^+_{\Thetapicc,i} \, , T^-_{\Thetapicc,i} \big\}}_{i \in I} \,\big) \, $   ---
 see  Definition \ref{def: twisted-realization}.  Accordingly, we write  $ \lieg_{P_\Theta}^{\R_\Thetapicc,p} $  for the Lie superalgebra associated with the realization  $ \R_\Thetapicc $  of  $ P_\Theta \, $.
 \vskip5pt
   Third, we find the following, strict link between all these ``deformed objects'':
\end{free text}

\vskip11pt

\begin{prop}  \label{prop: Yamane's twist-liegRPa = liegRP-twist}
   With assumptions as above, the following hold:
 \vskip5pt
   (a)\;  there exists a Lie superalgebra isomorphism
  $$  \phi^\Thetapicc_\Ppicc : \lieg_{P_\Theta}^{\R_\Thetapicc,p} \, {\buildrel \cong \over
 {\lhook\joinrel\relbar\joinrel\relbar\joinrel\relbar\joinrel\twoheadrightarrow}} \,
{\Big( \liegRpPa \Big)}^{j_\Thetapicc} \; ,  \qquad
  E_i \, \mapsto \, E_i \, , \; T \, \mapsto \, T \, , \; F_i \, \mapsto \, F_i \quad
  (\; i \in I \, , \, T \in \lieh \,)  $$
 \vskip1pt
   (b)\;  there exists a unique structure of Lie super\textsl{bi}algebra onto  $ \, \lieg_{P_\Theta}^{\R_\Thetapicc,p} \, $  such that the map  $ \phi^\Thetapicc_\Ppicc $
   in (a) above is an isomorphism of Lie super\textsl{bi}algebras;
 \vskip5pt
   (c)\;  the Lie cobracket of the Lie super\textsl{bi}algebra structure of  $ \, \lieg_{P_\Theta}^{\R_\Thetapicc,p} \, $  mentioned in (b) above is given by the formulas
   (for all  $ \, T \in \lieh \, $,  $ \, i \in I $)
  $$  \delta\big(\,T\big) \, = \, 0  \quad ,   \qquad  \delta\big(E_i\big) \, = \, 2 \; T^+_{\Thetapicc,i} \wedge E_i  \quad ,   \qquad  \delta\big(F_i\big) \, = \,  2 \; T^-_{\Thetapicc,i} \wedge F_i  $$
\end{prop}

\begin{proof}
 \textit{(a)}\,  By construction, the Lie algebra structure in  $ {\big( \liegRpPa \big)}^{j_\Thetapicc} $  is the same of  $ \liegRpPa \, $,
 and the latter only depends on the $ \alpha_j $'s  and the sums
 $ \, S_j := 2^{-1}\big(T_j^+ \! + T_j^-\big) \, $  ($ \, j \in I \, $).
 Now, both the  $ \alpha_j $'s  and the  $ S_j $'s  \textsl{do not change\/}
 (by construction)  when we pass from  $ \liegRpPa $  to  $ \lieg_{P_\Theta}^{\R_\Thetapicc,p} $
 or viceversa; therefore, the formulas in the claim   --- mapping each generator of  $
 \lieg_{P_\Theta}^{\R_\Thetapicc,p} $  onto the same name generator of
 $ \, \liegRpPa = {\big( \liegRpPa \big)}^{j_\Thetapicc} \, $  ---
 actually do provide an isomorphism  $ \phi^\Thetapicc_\Ppicc $  of Lie superalgebras,
 as claimed.
 \vskip5pt
   \textit{(b)}\,  By pull-back through  $ \phi^\Thetapicc_\Ppicc $  one gets onto
   $ \lieg_{P_\Theta}^{\R_\Thetapicc,p} $  the unique Lie cobracket which makes the former
   into a Lie superbialgebra such that the map  $ f^\Thetapicc_\Ppicc $
   is also an isomorphism of Lie super\textsl{bi\/}algebras.
 \vskip5pt
   \textit{(c)}\,  For the toral twist
  $ \; j_\Thetapicc := {\textstyle \sum_{g,k=1}^t} \theta_{gk} \, H_g \otimes H_k \; $
 in  \eqref{eq: def_twist_Lie-bialg},  formula  \eqref{eq: def_twist-delta}  yields
  $$  \delta^{j_\Thetapicc}(x)  \; := \;  \delta(x) - x.j_\Thetapicc  \; = \;
  \delta(x) - {\textstyle \sum_{g,k=1}^t} \theta_{gk} \,
  \big( \big[x,H_g\big] \otimes H_k + H_g \otimes \big[x,H_k\big] \big)  $$
 for all  $ \, x \in \liegRpPa \, $.  Now take  $ \, x := E_\ell \, $  ($ \, \ell \in I \, $):  \,then computations give
  $$  \displaylines{
   \delta^{j_\Thetapicc}(E_\ell)  \; := \;  \delta(E_\ell) - {\textstyle \sum_{g,k=1}^t} \, \theta_{gk} \,
   \big( \big[E_\ell \, , H_g\big] \otimes H_k + H_g \otimes \big[E_\ell \, , H_k\big] \big)  \; =   \hfill  \cr
   \quad   = \;  \delta(E_\ell) - {\textstyle \sum_{g,k=1}^t} \, \theta_{gk} \,
   \big(\! -\alpha_\ell(H_g) E_\ell \otimes H_k - H_g \otimes \alpha_\ell(H_k) E_\ell \,\big)  \; =   \hfill  \cr
   = \;  T^+_\ell \otimes E_\ell - E_\ell \otimes T^+_\ell + {\textstyle \sum_{g,k=1}^t} \, \theta_{gk} \,
   \big(\, \alpha_\ell(H_g) E_\ell \otimes H_k + \alpha_\ell(H_k) H_g \otimes E_\ell \,\big)  \; =   \quad  \cr
   \hfill   = \;  \Big(\, T^+_\ell + {\textstyle \sum_{g,k=1}^t} \, \theta_{kg} \, \alpha_\ell(H_g) H_k \Big) \otimes E_\ell \, -
   \, E_\ell \otimes \Big(\, T^+_\ell + {\textstyle \sum_{g,k=1}^t} \, \theta_{kg} \, \alpha_\ell(H_g) H_k \Big)  \; =  \cr
   \hfill   = \;  T^+_{\Thetapicc,\ell} \otimes E_\ell \, - \, E_\ell \otimes T^+_{\Thetapicc,\ell}  \; = \; 2 \;
   T^+_{\Thetapicc,\ell} \wedge \! E_\ell  }  $$
 hence in short we get  $ \; \delta^{j_\Thetapicc}(E_\ell) \, = \, 2 \; T^+_{\Thetapicc,\ell} \wedge \! E_\ell  \; $.
 Similar computations give
  $$
 \delta^{j_\Thetapicc}(T) = 0  \; ,   \!\quad
 \delta^{j_\Thetapicc}(F_\ell) = 2 \, T^-_{\Thetapicc,\ell} \wedge \! F_\ell  \; ,   \quad  \forall \; \ell \in I \, , \, T \!\in \lieh  $$
 Therefore, pulling back through  $ \phi_\Ppicc^\Thetapicc \, $  onto  $ \liegRpP $  the Lie cobracket of  $ {\Big( \liegRpPa \Big)}^{j_\Thetapicc} $,  \,we find that this Lie cobracket is described by the formulas in claim  \textit{(c)},  q.e.d.
\end{proof}

\vskip11pt

  With some extra effort we can get the following, key result:

\vskip11pt

\begin{theorem}  \label{thm: Lie-sBIalg x liegRP}
 Every Lie superalgebra  $ \liegRpP $  is indeed a Lie superbialgebra,
 whose Lie cobracket is given on generators by the formulas
 (for all  $ \, T \in \lieh \, $,  $ \, i \in I $)
\begin{equation}  \label{eq: Lie-cobracket x gRP}
 \delta\big(\,T\big) \, = \, 0  \quad ,   \qquad  \delta\big(E_i\big) \, =
 \, 2 \; T^+_i \!\wedge E_i  \quad ,   \qquad  \delta\big(F_i\big) \, = \,
 2 \; T^-_i \!\wedge F_i
\end{equation}
\end{theorem}

\begin{proof}
 The result follows from the discussion above when the realization  $ \R $  of  $ P $  is straight.  To prove the statement for general realizations  $ \R \, $,  we proceed as follows.  By assumption, we have a fixed Cartan super-datum  $ \, (A,p) \, $;
%
%
 furthermore,  $ \, P_s := 2^{-1} \big( P + P^{\,\scriptscriptstyle T} \big) = DA =:
 P_{\!\Apicc} \, $,  \,and finally
 $ \, \R := \big(\, \lieh \, , \Pi \, , \Pi^\vee \,\big) \, $  a realization of  $ P $.
                                                                  \par
   By  Lemma \ref{lemma: split/straight-lifting},  there is a  \textsl{straight\/}  realization  $ \; \tilde{\R} := \big(\,\tilde{\lieh} \, , \tilde{\Pi} \, , {\tilde{\Pi}}^\vee \,\big) \, $   of  $ P $  and an epimorphism of realizations  $ \; \underline{\pi} : \tilde{\R} \relbar\joinrel\relbar\joinrel\twoheadrightarrow \R \; $;  \,moreover, up to possibly enlarging  $ \tilde{\lieh} \, $,  we can assume that  $ \, t := \text{rk}\big(\tilde{\lieh}\big) \geq 3n - \text{rk}\big(P_{\!\Apicc}\big) \, $.  Then, by  Proposition \ref{prop: realiz=twist-standard}\textit{(a)}   --- with  $ \, P' = P_{\!\Apicc} := DA \, $  ---   there exists some  $ \, \Phi \in \lieso_t\big(\kh\big) \, $  such that  $ \, P_{\!\Apicc} = P_{\scriptscriptstyle \Phi} \, $  while  $ \, \R_\Apicc := \tilde{\R}_{\,\scriptscriptstyle \Phi} \, $  is a straight realization of  $ \, P_{\!\Apicc} = P_{\scriptscriptstyle \Phi} \, $.  Hence, conversely, we have also  $ \, P = {\big( P_{\!\Apicc} \big)}_{\scriptscriptstyle -\Phi} \, $  and  $ \, \tilde{\R} = {\big( \R_\Apicc \big)}_{\scriptscriptstyle -\Phi} \, $.  Then there exists an isomorphism of Lie superbialgebras
  $ \; \phi^{\scriptscriptstyle -\Phi}_\Ppicc : \lieg_\Ppicc^{\tilde{\Rpicc},p} = \lieg_{(\Ppicc_\Apicc)_{-\Phipicc}}^{(\Rpicc_\Apicc)_{-\Phipicc},p} \, {\buildrel \cong \over
 {\lhook\joinrel\relbar\joinrel\relbar\joinrel\relbar\joinrel\twoheadrightarrow}} \,
{\Big( \lieg_{\Ppicc_{\!\Apicc}}^{\Rpicc_\Apicc,p} \Big)}^{j_{-\Phipicc}} \, $,
 \;explicitly described as in  Proposition \ref{prop: Yamane's twist-liegRPa = liegRP-twist},  up to switching  $ \Phi $  to  $ -\Phi \, $;  \,moreover, the Lie cobracket in  $ \, \lieg_\Ppicc^{\tilde{\Rpicc},p} = \lieg_{(\Ppicc_\Apicc)_{-\Phipicc}}^{(\Rpicc_\Apicc)_{-\Phipicc},p} \, $  is described as in  Proposition \ref{prop: Yamane's twist-liegRPa = liegRP-twist}\textit{(c)},  which in terms of the defining generators of  $ \lieg_\Ppicc^{\tilde{\Rpicc},p} $  reads exactly like  \eqref{eq: Lie-cobracket x gRP}  in the claim.
                                                                  \par
   Now, the epimorphism of realizations  $ \; \underline{\pi} : \tilde{\R} \relbar\joinrel\relbar\joinrel\twoheadrightarrow \R \; $  ``include'' an epimorphism  $ \, \pi : \tilde{\lieh} \relbar\joinrel\relbar\joinrel\twoheadrightarrow \lieh \, $  of (Abelian) Lie  $ \kh $--superalgebras; moreover, it follows at once from the presentation of both  $ \, \lieg_\Ppicc^{\tilde{\Rpicc},p} \, $  and  $ \, \liegRpP \, $  that  $ \pi $  extends to an epimorphism of Lie superalgebras  $ \; \L_{\underline{\pi}} : \lieg_\Ppicc^{\tilde{\Rpicc},p} \relbar\joinrel\relbar\joinrel\relbar\joinrel\twoheadrightarrow \liegRpP $,  \;and that
   $ \Ker\big(\L_{\underline{\pi}}\big) $  is generated by  $ \, \Ker(\pi) \, $.  Furthermore,
   $ \, \Ker(\pi) \, $  lies in the center of  $ \lieg_\Ppicc^{\tilde{\Rpicc},p} \, $,  \,by definitions and by  Lemma \ref{lemma: ker-morph's_realiz's};  moreover, the Lie cobracket of
   $ \lieg_\Ppicc^{\tilde{\Rpicc},p} $  kills  $ \Ker(\pi) \, $,  \,as it is trivial on  $
   \tilde{\lieh} \, $,  so the latter is a Lie  super\textsl{bi\/}ideal  in the Lie superbialgebra  $ \lieg_\Ppicc^{\tilde{\Rpicc},p} \, $.  Thus  $ \liegRpP $  inherits via
   $ \L_{\underline{\pi}} $  a  \textsl{quotient Lie superbialgebra structure\/}  from  $ \lieg_\Ppicc^{\tilde{\Rpicc},p} \, $,  again described by  \eqref{eq: Lie-cobracket x gRP},
   q.e.d.
\end{proof}

\vskip9pt

\begin{definition}  \label{def: MpLSbA's}
   Every such Lie superbialgebra  $ \liegRpP $  considered above will be called  \textit{multiparameter Lie superbialgebra},  in short  \textit{MpLSbA}.
   In addition, we say that the MpLSbA  $ \liegRpP $  is  \textsl{straight},
   or \textsl{small},  or \textsl{minimal},  or  \textsl{split},  if such is  $ \R \, $.
 \vskip3pt
   Finally, we define the  \textsl{rank\/}  of  $ \liegRpP $  as  $ \; \rk\!\big( \liegRpP \big) := \rk(\R) = \rk_\Bbbk(\lieh) \; $.   \hfill  $ \diamondsuit $
\end{definition}

\vskip9pt

\begin{rmks}  \label{rmks: sub-MpLSbA's & triang-decomp.'s}
 Let  $ \lien^\Rpicc_+ \, $,  resp.\  $ \lien^\Rpicc_- \, $,  be the Lie sub-superalgebra of
 $ \liegRpP $
 generated by all the  $ E_i $'s,  resp.\ all the  $ F_i $'s; and let  $ \, \lieb^{\raise1pt\hbox{$ \scriptscriptstyle \R,p $}}_{\Ppicc,\pm} := \lieh \oplus \lien^\Rpicc_\pm = \lien^\Rpicc_\pm \oplus \lieh \, $.  Then:
 \vskip3pt
   \textit{(a)}\;  $ \lieb^{\raise1pt\hbox{$ \scriptscriptstyle \R,p $}}_{\Ppicc,\pm} $
   is a Lie sub-superbialgebra of  $ \liegRpP \, $,
 \vskip3pt
   \textit{(b)}\;  there exist obvious  \textsl{triangular decompositions}
\begin{equation}  \label{eq: triang-decomp's_Lie-bialg's}
  \liegRpP  \; = \;  \lien_+ \oplus \lieh \oplus \lien_-  \quad ,  \qquad \qquad  \liegRpP  \;
  = \;  \lien_- \oplus \lieh \oplus \lien_+
\end{equation}
%
%
\end{rmks}

\vskip9pt

   The following, last two result stress the fact that the dependence of MpLSbA's
   on realizations   --- for a common, fixed multiparameter matrix  $ P $  ---
   is  \textsl{functorial\/}:

\vskip13pt

\begin{prop}  \label{prop: functor_R->liegRP}
 Let  $ \, (A \, , p) \, $  be a fixed Cartan super-datum and let
 $ \, P \in M_n(\Bbbk) \, $ be of Cartan type $A$.
 If both  $ \R' $  and  $ \, \R'' $  are realizations of  $ P $  and
 $ \, \underline{\phi} : \R' \relbar\joinrel\relbar\joinrel\longrightarrow \R'' \, $  is a morphism between them, then there exists a unique morphism of Lie superbialgebras
 $ \; \Lc_{\underline{\phi}} : \lieg^{\scriptscriptstyle \R',p}_\Ppicc \!\relbar\joinrel\longrightarrow \lieg^{\scriptscriptstyle \R'',p}_\Ppicc \; $
 that extends the morphism
 $ \, \phi : \lieh' \!\relbar\joinrel\longrightarrow \lieh'' \, $
 given by  $ \underline{\phi} \, $;  \,moreover,
 $ \, \Lc_{\underline{\id}_\R} = \id_{\liegRpP} \, $  and
 $ \; \Lc_{\underline{\phi}' \circ\, \underline{\phi}} = \Lc_{\underline{\phi}'} \circ \Lc_{\underline{\phi}} \; $
 (whenever  $ \, \underline{\phi}' \circ\, \underline{\phi} \, $  is defined).
 In a nutshell, the construction  $ \, \R \mapsto \liegRpP \, $  (for fixed  $ P $)  is functorial in  $ \R \, $.
                                                                              \par
   Further, if  $ \underline{\phi} $  is an epimorphism, resp.\ a monomorphism, then
  $ \Lc_{\underline{\phi}} $  is an epimorphism, resp.\ a monomorphism, as well.
  Finally, for any morphism
  $ \, \underline{\phi} : \R' \relbar\joinrel\longrightarrow \R'' \, $,
  \,the kernel  $ \, \Ker\big(\Lc_{\underline{\phi}}\,\big) $  of
  $ \, \Lc_{\underline{\phi}} $  coincides with  $ \, \Ker(\phi) \, $,
  and the latter is central in
  $ \, \lieg^{\scriptscriptstyle \R',p}_\Ppicc \, $.
                                                                                \par
   In particular, when  $ \underline{\phi} \, $,  and hence also
   $ \Lc_{\underline{\phi}} \, $,  is an epimorphism, we have   ---
   setting  $ \; \liek := \Ker(\phi) \, $  ---
   a central exact sequence of Lie superbialgebras
  $$  0 \,\relbar\joinrel\relbar\joinrel\longrightarrow\, \liek
  \,\relbar\joinrel\relbar\joinrel\longrightarrow\, \lieg^{\scriptscriptstyle \R',p}_\Ppicc \,\;{\buildrel {\Lc_{\underline{\phi}}} \over {\relbar\joinrel\relbar\joinrel\longrightarrow}}\;\,
  \lieg^{\scriptscriptstyle \R'',p}_\Ppicc \,\relbar\joinrel\relbar\joinrel\longrightarrow\, 0  $$
\end{prop}

\pf
 The existence of  $ \Lc_{\underline{\phi}} $  is obvious, as well as all the other claims; we only spend a moment on  $ \Ker(\phi) $  being central.
 Lemma \ref{lemma: ker-morph's_realiz's}  implies
 $ \; \Ker(\phi) \, \subseteq \, \bigcap\limits_{j \in I} \Ker(\alpha'_j) \; $;
 \,then from the relations among the generators of
 $ \lieg^{\scriptscriptstyle \R'}_\Ppicc $  we get that each element in  $ \Ker(\phi) $  commutes with all generators of  $ \lieg^{\scriptscriptstyle \R'}_\Ppicc \, $,  thus  $ \Ker(\phi) $  is central.
\epf

\vskip9pt

\begin{cor}  \label{cor: isom_R -> isom_liegRP}
 With notation as above, if  $ \; \R' \cong \R'' \, $  then
 $ \; \lieg^{\scriptscriptstyle \R',p}_\Ppicc \cong \lieg^{\scriptscriptstyle \R'',p}_\Ppicc \; $.
 \vskip3pt
   In particular, all MpLSbA's built upon split realizations,
   respectively small realizations, of the same matrix  $ P $
   and sharing the same rank of  $ \, \lieh \, $,  are isomorphic to each other,
   hence they are independent (up to isomorphisms) of the specific realization,
   but only depend on  $ P $  and on the rank of  $ \, \lieh \, $.
\end{cor}

\pf
 This follows at once from  Proposition \ref{prop: functor_R->liegRP}
 together with the uniqueness result in  Proposition \ref{prop: exist-realiz's}
 and  Proposition \ref{prop: exist-realiz's_small}.
\epf

\medskip

\subsection{Deformations of MpLSbA's by toral twists}  \label{subsec: tor-twist def's_mp-Lie_supbialg's}  {\ }

\vskip7pt

   We will now go and study deformations of MpLSbA's, following
   \S \ref{deformations of LSbA's}.
   We begin with deformations by twist, choosing the latter to be of a special type.
 \vskip5pt
 Let  $ \liegRpP $  be a MpLSbA as in  \S \ref{subsec: MpLSbA's}  above;
 then  $ \lieh $  is a free  $ \k $--module  of finite rank  $ \, t := \rk(\lieh) \, $:
 \,we fix in it a  $ \k $--basis  $ \, {\big\{ H_g \big\}}_{g \in \G} \, $,  where
 $ \G $  is an index set with  $ \, |\G| = \rk(\lieh) =: t \, $.
 We begin introducing the so-called ``toral'' twists for  $ \liegRpP \, $.

\vskip9pt

\begin{definition}  \label{def: toral twists x MpLbA's}
 For any  \textsl{antisymmetric\/}  matrix
 $ \; \Theta = \big( \theta_{gk} \big)_{g, k \in \G} \in \lieso_t(\Bbbk) \; $,  \;we set
\begin{equation}  \label{eq: def_twist_Lie-bialg_2nd-time}
  j_\Thetapicc  \; := \;  {\textstyle \sum_{g,k=1}^t} \theta_{gk} \, H_g \otimes H_k  \; \in \;
  \lieh \otimes \lieh  \; \subseteq \;  \lieg \otimes \lieg
\end{equation}
 and we call this  $ \, j_\Thetapicc \, $  \textit{the toral twist\/}
 (or  \textit{``twist of toral type''\/})
 associated with  $ \, \Theta\, $.   \hfill  $ \diamondsuit $
\end{definition}

\vskip7pt

   Next result   --- which explains our use of terminology ---   follows by construction;
   in particular, it makes use of the antisymmetry condition on  $ \Theta \, $.

\vskip11pt

\begin{lema}  \label{lemma: toral twist x MpLbA's}
 For any  $ \; \Theta = \big( \theta_{gk} \big)_{g, k \in \G} \in \lieso_t(\Bbbk) \; $,  \,the element  $ \, j_\Thetapicc \, $  in  Defini\-tion \ref{def: toral twists x MpLbA's}  is a  \textsl{twist\/}  element for the Lie superbialgebra  $ \liegRpP \, $,  \,in the sense of  \eqref{eq: twist-cond_Lie-superbialg}.   \qed
\end{lema}

\vskip9pt

   Concerning deformations of MpLSbA's by toral twists, our main result is the next one.
   To settle its content, let  $ \, P\in M_t(\Bbbk) \, $  be a multiparameter matrix of Cartan
   type with associated Cartan super-datum  $ (A,p) \, $,
   let  $ \R $  be a realization of it, and let  $ \liegRpP $
   be the associated multiparameter Lie bialgebra; then, for any given antisymmetric matrix
   $ \; \Theta = \big( \theta_{gk} \big)_{g, k \in \G} \in \lieso_t(\Bbbk) \; $,  \,let
 $ \; j_\Thetapicc \, := \, {\textstyle \sum\limits_{g,k=1}^t} \theta_{gk} \, H_g \otimes H_k \; $
 be the associated twist as in  \eqref{eq: def_twist_Lie-bialg_2nd-time}.
 Moreover, we consider the ``deformed'' multiparameter matrix
 $ \; P_\Thetapicc = \, {\big(\, p^\Thetapicc_{i,j} \big)}_{i, j \in I} \, := \,
 P - \mathfrak{A} \, \Theta \, \mathfrak{A}^{\,\scriptscriptstyle T} \; $
 as in  \eqref{def-P_Phi}   --- again of Cartan type, the same as  $ P $  ---
 and its corresponding ``deformed'' realization
 $ \, \R_\Thetapicc := \big(\, \lieh \, , \Pi \, , \Pi^\vee_\Thetapicc \! := {\big\{ T^+_{\Thetapicc,i} \, , T^-_{\Thetapicc,i} \big\}}_{i \in I} \,\big) \, $.

\vskip11pt

\begin{theorem}  \label{thm: twist-liegRP=new-liegR'P'}
 There is a Lie superbialgebra isomorphism
  $ \; f^\Thetapicc_{\Ppicc} : \lieg_{P_\Thetapicc}^{\scriptscriptstyle \Rpicc_\Thetapicc,p} \, {\buildrel \cong \over
 {\lhook\joinrel\relbar\joinrel\twoheadrightarrow}} \,
{\big( \liegRpP \big)}^{j_\Thetapicc} $
 \; (notation as above, with in right-hand side the twist deformation
 $ \, {\big( \liegRpP \big)}^{j_\Thetapicc} $  of  $ \, \liegRpP $
 by  $ j_\Thetapicc $  occurs) given by
  $ \; E_i \, \mapsto \, E_i \, , \; T \, \mapsto \, T \, $  and
  $ \; F_i \, \mapsto \, F_i \; $
 for all  $ \, i \in I \, $,  $ \, T \in \lieh \, $.
 \vskip3pt
   In particular, the class of all MpLSbA's of any fixed Cartan type and
   of fixed rank is stable by toral twist deformations.
   Moreover, inside it the subclass of all such MpLSbA's associated with
   \textsl{straight},  resp.\  \textsl{small},  realizations is stable as well.
\end{theorem}

\pf
 By definition   --- cf.\ \S \ref{deformations of LSbA's}  ---
 the Lie superalgebra structure in  $ {\big( \liegRpP \big)}^{j_\Thetapicc} $
 is the same as in  $ \liegRpP \, $,  and the latter only depends on the
 $ \alpha_j $'s  and the sums  $ \, S_j := 2^{-1}\big(T_j^+ \! + T_j^-\big) \, $
 ($ \, j \in I \, $).  Now, both the  $ \alpha_j $'s  and the  $ S_j $'s
 \textsl{do not change\/}  (again by construction)  when we pass from
 $ \liegRpP $  to  $ \lieg_{\Ppicc_\Theta}^{\Rpicc_\Thetapicc,p} $  or viceversa;
 therefore, the formulas in the claim  (mapping each generator of
 $ \lieg_{\Ppicc_\Theta}^{\Rpicc_\Thetapicc,p} $  onto the same name generator of
 $ \, \liegRpP = {\big( \liegRpP \big)}^{j_\Thetapicc} \, $)  provide an isomorphism of Lie superalgebras.
                                                                  \par
   Now consider the toral twist
  $ \; j_\Thetapicc := {\textstyle \sum_{g,k=1}^t} \theta_{gk} \, H_g \otimes H_k \; $
 given in  \eqref{eq: def_twist_Lie-bialg_2nd-time}.  By  \eqref{eq: def_twist-delta}
  $$  \delta^{j_\Thetapicc}(x)  \; := \;  \delta(x) - x\,.\,j_\Thetapicc  \; = \;
  \delta(x) - {\textstyle \sum_{g,k=1}^t} \theta_{gk} \,
  \big( \big[x,H_g\big] \otimes H_k + H_g \otimes \big[x,H_k\big] \big)  $$
 for all  $ \, x \in \lieg \, $.  Now take  $ \, x := E_\ell \, $  ($ \, \ell \in I \, $):  \,then the previous formula reads
  $$  \displaylines{
   \delta^{j_\Thetapicc}(E_\ell)  \; := \;  \delta(E_\ell) - {\textstyle \sum_{g,k=1}^t} \, \theta_{gk} \,
   \big( \big[E_\ell \, , H_g\big] \otimes H_k + H_g \otimes \big[E_\ell \, , H_k\big] \big)  \; =   \hfill  \cr
   \quad   = \;  \delta(E_\ell) - {\textstyle \sum_{g,k=1}^t} \, \theta_{gk} \,
   \big(\! -\alpha_\ell(H_g) E_\ell \otimes H_k - H_g \otimes \alpha_\ell(H_k) E_\ell \,\big)  \; =   \hfill  \cr
  = \;  T^+_\ell \otimes E_\ell - E_\ell \otimes T^+_\ell + {\textstyle \sum_{g,k=1}^t} \, \theta_{gk} \,
   \big(\, \alpha_\ell(H_g) E_\ell \otimes H_k + \alpha_\ell(H_k) H_g \otimes E_\ell \,\big)  \; =   \quad
 }  $$
   $$  \displaylines{
   \hfill   = \;  \Big(\, T^+_\ell + {\textstyle \sum_{g,k=1}^t} \, \theta_{kg} \, \alpha_\ell(H_g) H_k \Big) \otimes E_\ell \, -
   \, E_\ell \otimes \Big(\, T^+_\ell + {\textstyle \sum_{g,k=1}^t} \, \theta_{kg} \, \alpha_\ell(H_g) H_k \Big)  \; =  \cr
   \hfill   = \;  T^+_{\Thetapicc,\ell} \otimes E_\ell \, - \, E_\ell \otimes T^+_{\Thetapicc,\ell}  \; = \; 2 \;
   T^+_{\Thetapicc,\ell} \wedge \! E_\ell  }  $$
 hence in short we get  $ \; \delta^{j_\Thetapicc}(E_\ell) \, = \, 2 \; T^+_{\Thetapicc,\ell} \wedge \! E_\ell  \; $.
 Similar computations give
  $$
 \delta^{j_\Thetapicc}(T) = 0  \; ,   \!\quad
 \delta^{j_\Thetapicc}(F_\ell) = 2 \, T^-_{\Thetapicc,\ell} \wedge \! F_i  \; ,   \quad  \forall \; \ell \in I \, , \, T \!\in \lieh  $$
 This means that,  \textsl{through the Lie superalgebra isomorphism
 $ f_\Ppicc^\Thetapicc \, $,    \textit{the Lie coalgebra structure of
 $ {\big( \liegRpP \big)}^{j_\Thetapicc} $  is described on generators exactly like that of
 $ \lieg_{\Ppicc_\Theta}^{\Rpicc_\Thetapicc,p} \, $},  with the new ``coroots''
 $T^\pm_{\Thetapicc,\ell} $  ($ \, \ell \in I \, $)  in
 $ {\big( \liegRpP \big)}^{j_\Thetapicc} $  playing the role of the coroots
 $ T^\pm_\ell $  ($ \, \ell \in I \, $)  in
 $ \lieg_{\Ppicc_\Theta}^{\Rpicc_\Thetapicc,p} \, $}.
 Thus
 $ \; f^\Thetapicc_{\scriptscriptstyle P} : \lieg_{\Ppicc_\Theta}^{\Rpicc_\Thetapicc}
 \longrightarrow {\big( \liegRpP \big)}^{j_\Thetapicc} \; $
 is an isomorphism of Lie super\textsl{bi\/}algebras.
\epf

\vskip7pt

   In fact,  \textsl{the previous result can be reversed},  somehow.
   Namely, our next result shows, in particular, that
   \textit{every straight small MpLSbA can be realized as a toral twist
   deformation of the Yamane's MpLSbA  $ \liegRpPa $}
   (as in  \S \ref{Yamane's LieSBial's}),  cf.\ claim  \textit{(c)\/}  below.

\vskip7pt

\begin{theorem}  \label{thm: MpLSbA=twist-Yamane's}
 With assumptions as above, let  $ P $  and  $ P' $  be two matrices of Cartan type with the
 same associated Cartan super-datum $ (A,p) \, $,  \,i.e.\  $ \, P_s = P'_s \, $.
 \vskip3pt
   (a)\,  Let  $ \, \R $  be a  \textsl{straight}  realization of  $ P $  and let
   $ \lieg_\Ppicc^{\scriptscriptstyle \R,p} $
   be the associated MpLSbA.  Then there exists a matrix  $ \, \Theta \in \lieso_t(\Bbbk) \, $
   such that  $ \, P' = P_\Thetapicc \, $,  the corresponding $ \R_\Thetapicc $
   is a  \textsl{straight}  realization of  $ \, P' = P_\Thetapicc \, $,
   and for the twist element
   $ j_\Thetapicc $  as in  \eqref{eq: def_twist_Lie-bialg_2nd-time}  we have
  $$
  \lieg_{\scriptscriptstyle P'}^{\scriptscriptstyle \R_\Thetapicc,p}  \; \cong \;
  {\big(\, \lieg_\Ppicc^{\scriptscriptstyle \R,p} \big)}^{j_\Theta}
  $$
   \indent   In a nutshell, if  $ \, P'_s = P_s \, $  then from any straight MpLSbA over  $ P $
   we can obtain by toral twist deformation a straight MpLSbA (of the same rank) over  $ P' \, $.
 \vskip3pt
   (b)\,  Let  $ \R $  and  $ \R' $  be  \textsl{straight small}  realizations of
   $ P $  and  $ P' $
respectively, with  $ \, \rk(\R) = \rk(\R') =: t \, $,  and let  $ \liegRpP $  and
$ \lieg_{\scriptscriptstyle P'}^{\scriptscriptstyle \R'\!,p} $  be the associated MpLSbA's.
Then there exists a matrix  $ \, \Theta \in \lieso_t(\Bbbk) \, $
such that for the twist element
$ j_\Thetapicc $  as in  \eqref{eq: def_twist_Lie-bialg_2nd-time}  we have
  $$
  \lieg_{\scriptscriptstyle P'}^{\scriptscriptstyle \R'\!,p}  \; \cong \;
  {\big(\, \liegRpP \big)}^{j_\Theta}
  $$
   \indent   In a nutshell, if  $ \, P'_s = P_s \, $  then any straight small MpLSbA over
   $ P' $  is isomorphic to a toral twist deformation of any straight small MpLSbA over  $ P $  of the same rank.
 \vskip3pt
   (c)\,  Every straight small MpLSbA is isomorphic to some toral twist deformation of
   Yamane's MpLSbA  $ \liegRpPa $
   (over  $ \, P_{\scriptscriptstyle A} = DA = P_s \, $,  cf.\  \S \ref{Yamane's LieSBial's})
   of the same rank.
\end{theorem}

\pf
 \textit{(a)}\,  By  Theorem \ref{thm: twist-liegRP=new-liegR'P'}  it is enough to find
 $ \, \Theta \in \lieso_t(\Bbbk)  \, $  such that  $ \, P' = P_\Thetapicc \, $,  \,that is
 $ \; P' = P - \mathfrak{A} \, \Theta \, \mathfrak{A}^{\,\scriptscriptstyle T} \, $;
 \,but this is guaranteed by  Lemma \ref{lemma: twist=sym},  so we are done.
 \vskip3pt
   \textit{(b)}\,  This follows from claim  \textit{(a)},  along with the uniqueness of straight small realizations   --- cf.\ Proposition \ref{prop: exist-realiz's}\textit{(b)}  ---   and  Proposition \ref{prop: functor_R->liegRP}.
 \vskip3pt
   \textit{(c)}\,  This follows from  \textit{(b)},  once we take as
   $ \lieg_{\scriptscriptstyle P'}^{\scriptscriptstyle \R'\!,p} $
   the given straight small MpLSbA and as  $ \liegRpP $  Yamane's MpLSbA  $ \lieg_{\scriptscriptstyle P'}^{\scriptscriptstyle \R'\!,p} $  over
   $ \, P := P_{\scriptscriptstyle A} = DA \, $.
\epf

\vskip7pt

\begin{obs's}  \label{obs: parameter_(P,Phi) x MpLbA's}  {\ }
 \vskip3pt
   \textit{(a)}\,  Theorems \ref{thm: twist-liegRP=new-liegR'P'}  and
   \ref{thm: MpLSbA=twist-Yamane's}  allow the following interpretation.
   Our MpLSbA's  $ \, \liegRpP \, $  depend on the multiparameter  $ P \, $;
   \,at a further level, once we perform onto  $ \, \liegRpP \, $  a deformation by
   toral twist, the outcome
   $ \; \lieg^{\scriptscriptstyle \R,p}_{\Ppicc,\Thetapicc} := {\big(\, \liegRpP \big)}^{j_\Theta} \; $  depends on  \textsl{two\/}  multiparameters, namely
   $ P $  \textsl{and\/}  $ \Theta \, $.  So all the
   $ \lieg^{\scriptscriptstyle \R,p}_{\Ppicc,\Thetapicc} $'s  form a
   (seemingly) richer family of
   ``twice-multiparametric'' Lie superbialgebras.
   In spite of this,  Theorem \ref{thm: twist-liegRP=new-liegR'P'}
   proves that this family still  \textsl{coincides\/}
   with our initial family of MpLSbA's.
                                                              \par
   In short,  Theorems \ref{thm: twist-liegRP=new-liegR'P'}  and
   \ref{thm: MpLSbA=twist-Yamane's}  show the following.  The dependence of the Lie
   superbialgebra structure of  $ \lieg^{\scriptscriptstyle \R , p}_{\Ppicc,\Thetapicc} $
   on the ``double parameter''  $ (P\,,\Theta) $  is ``split'' in the algebraic structure
   (ruled by  $ P \, $)  and in the coalgebraic structure (ruled by  $ \Theta $).
   Now,  Theorems \ref{thm: twist-liegRP=new-liegR'P'}  and  \ref{thm: MpLSbA=twist-Yamane's}
   enable us to ``polarize'' this dependence so to codify it either entirely
   within the algebraic structure (while the coalgebraic one is reduced to a ``canonical form'')
   --- so the single multiparameter  $ P_\Theta $  is enough to describe it ---
   or entirely within the coalgebraic structure (with the algebraic one being reduced
   to the ``standard'' one)   --- so that one multiparameter  $ \Theta_P $
   is enough to describe it.
 \vskip3pt
   \textit{(b)}\,  As we saw at the end of  \S \ref{further_stability},
   the (sub)class of  \textsl{split\/}
   realizations is  \textsl{not closed\/}  under twist deformation; as a consequence,
   the subclass of all MpLSbA's that are ``split'' is not closed either
   under twist deformation.
\end{obs's}

\medskip

\subsection{Deformations of MpLSbA's by toral  $ 2 $--cocycles}  \label{subsec: tor-2cocyc def's_mp-Lie_supbialg's}  {\ }

\vskip7pt

   Let $ \liegRpP $  be a MpLSbA as in  \S \ref{subsec: MpLSbA's},
   and keep notation as above.  Dually to what we did before in
   \S \ref{subsec: tor-twist def's_mp-Lie_supbialg's},  we consider now the so-called
   ``toral'' 2--cocycles for  $ \liegRpP $  and the corresponding deformations by them.

\vskip9pt

\begin{definition}  \label{def: toral_2-cocyc x MpLbA's}
 Fix an antisymmetric  $ \Bbbk $--linear  map  $ \, \chi : \lieh \otimes \lieh \relbar\joinrel\longrightarrow \Bbbk \, $
 such that
\begin{equation}  \label{eq: condition-eta}
  \chi(S_i \, ,\,-\,)  \, = \,  0  \, = \,  \chi(\,-\,,S_i)
  \qquad \qquad \forall \;\; i \in I
\end{equation}
 where  $ \, S_i := 2^{-1} \big(\, T^+_i + T^-_i \big) \, $
 for all  $ \, i \in I \, $  (cf.\ Definition \ref{def: realization of P});
 in other words, we have  $ \, \chi \in \text{\sl Alt}_{\,\kh}^{\,S}(\lieh) \, $,  cf.\ \eqref{eq: def-Alt_S}.
 Moreover, let
 $ \,\; \pi^{\liegRpP}_\lieh : \liegRpP \relbar\joinrel\twoheadrightarrow \lieh \;\, $
 be the canonical linear projection induced by any one of the triangular
 decompositions in
 \eqref{eq: triang-decomp's_Lie-bialg's}.  We define
  $$
  \chi_\lieg := \chi  \circ \Big(\, \pi^{\liegRpP}_\lieh \otimes \pi^{\liegRpP}_\lieh \Big) :
  \liegRpP \otimes \liegRpP \relbar\joinrel\relbar\joinrel\twoheadrightarrow
  \lieh \otimes \lieh
  \relbar\joinrel\longrightarrow \Bbbk   \qquad
  $$
 and we call it  \textit{the toral 2--cocycle\/}
 (or  \textit{``2--cocycle  of toral type''\/})
 \hbox{associated with  $ \, \chi \, $.   \hfill  $ \diamondsuit $}
\end{definition}

\vskip9pt

   The following lemma explains our use of terminology:

\vskip11pt

\begin{lema}  \label{lemma: toral 2-cocyc x MpLbA's}
 For any  antisymmetric
 $ \Bbbk $--linear  map  $ \, \chi : \lieh \otimes \lieh \longrightarrow \Bbbk \, $  obeying  \eqref{eq: condition-eta},
  $$  \chi_\lieg := \chi \circ \Big( \pi^{\liegRpP}_\lieh \otimes \pi^{\liegRpP}_\lieh \,\Big) : \liegRpP
  \otimes \liegRpP \relbar\joinrel\relbar\joinrel\twoheadrightarrow \lieh \otimes \lieh \relbar\joinrel\longrightarrow \Bbbk  $$
 is a  \textsl{2--cocycle\/}  map for the Lie bialgebra  $ \liegRpP \, $,  \,in the sense of  \eqref{eq: 2-cocycle-cond_Lie-superbialg}.
\end{lema}

\pf
Follows from \cite[Lemma 3.4.2]{GG4} \emph{mutatis mutandis}.
\epf

\vskip7pt

   A dual analog of  Theorems \ref{thm: twist-liegRP=new-liegR'P'} and
   \ref{thm: MpLSbA=twist-Yamane's} asserts that one may obtain any (straight,
   minimal, split) MpLSbA
   as a $2$-cocycle deformation of some Yamane's MpLSbA.
   To state properly the result, we need some notation.  Let $(A,p)$ be a Cartan super-datum,
   $ \, P\in M_t(\Bbbk) \, $ a multiparameter matrix of Cartan type with Cartan matrix
   $ A \, $,
   $ \R $  be a realization of it, and
   $ \liegRpP $ the associated multiparameter Lie superbialgebra.
   For any  $ \, \chi \in \text{\sl Alt}_{\,\kh}^{\,S}(\lieh) \, $
   as in  \eqref{eq: def-Alt_S},  let
   $ \, \chi_\lieg : \liegRpP \!\otimes \liegRpP \longrightarrow \Bbbk \, $
   be the  $ 2 $--cocycle  as in Lemma \ref{lemma: toral 2-cocyc x MpLbA's}.

\vskip5pt

   Consider the antisymmetric matrix
   $ \, \mathring{X} := {\Big(\, \mathring{\chi}_{i{}j} = \chi\big(\,T_i^+\,,T_j^+\big) \Big)}_{i, j \in I} \! \in \lieso_n(\Bbbk) \, $.
   By  Proposition \ref{prop: 2cocdef-realiz},  we have a matrix  $ P_{(\chi)} $  and a realization  $ \R_{(\chi)} $  of it,
   given by
  $$
  P_{(\chi)}  := \,  P \, + \, \mathring{X}  \, = \,  {\Big(\, p^{(\chi)}_{i{}j} := \,  p_{ij} + \mathring{\chi}_{i{}j} \Big)}_{\! i, j \in I}  \;\; ,  \quad
   \Pi_{(\chi)}  := \,  {\Big\{\, \alpha_i^{(\chi)}  := \,
   \alpha_i \pm \chi\big(\, \text{--} \, , T_i^\pm \big)  \Big\}}_{i \in I}
   $$
 Note that $P_{(\chi)} $ is also a multiparameter matrix of Cartan type with Cartan
 matrix $A$, for we are adding just an antisymmetric part.

\vskip13pt

\begin{theorem}  \label{thm: 2-cocycle-def-MpLSbA}
 Keep notation as above.
 \vskip2pt
   (a)\,  There exists a Lie superbialgebra isomorphism
  $ \lieg_{P_{(\chi)}}^{\R_{(\chi)},p} \simeq  {\big( \liegRpP \big)}_{\chi_\lieg} \; $.
 \vskip3pt
   (b)\,  Let  $ P $  and  $ P' $  be two matrices of Cartan type with the same associated Cartan matrix $ A \, $.  Then the following holds:
 \vskip3pt
   \, (b.1)\,  Let  $ \, \R $  be a  \textsl{split}  realization of  $ P $  and let
   $ \liegRpP $
   be the associated MpLSbA.  Then there is  $ \, \chi \in \text{\sl Alt}_{\,\Bbbk}^{\,S}(\lieh) \, $ such that
   $ \, P' = P_{(\chi)} \, $,  the corresponding $ \R_{(\chi)} $  is a  \textsl{split}  realization of  $ \, P' \, $,
   and for the 2--cocycle  $ \chi_\lieg $  as in Definition \ref{def: toral_2-cocyc x MpLbA's}  we have
  $$
  \lieg_{\scriptscriptstyle P'}^{\scriptscriptstyle \R_{(\chi)},p}  \; \cong \;
  {\big(\, \lieg_\Ppicc^{\scriptscriptstyle \R} \big)}_{\chi_{{}_\lieg}}  $$
 \vskip3pt
   \, (b.2)\,  Let  $ \R $  and  $ \R' $  be  \textsl{split minimal}
   realizations of  $ P $  and  $ P' $
respectively, and let  $ \liegRpP $  and
$ \lieg_{\scriptscriptstyle P'}^{\scriptscriptstyle \R',p} $
be the associated MpLSbA's.
Then there exists  $ \, \chi \in \text{\sl Alt}_{\,\Bbbk}^{\,S}(\lieh) \, $
such that for the 2--cocycle  $ \chi_\lieg $
as in Definition \ref{def: toral_2-cocyc x MpLbA's}  we have
  $$
  \lieg_{\scriptscriptstyle P'}^{\scriptscriptstyle \R',p}  \; \cong \;
  {\big(\, \liegRpP \big)}_{\chi_{\lieg}}
  $$
 \vskip3pt
   (b.3)\,  Every split minimal MpLSbA is isomorphic to some
   toral 2--cocycle deformation of Yamane's Lie superbialgebra.
\end{theorem}

\smallskip

\pf
 \textit{(a)}\, Since the Lie coalgebra structure in
 $ {\big( \liegRpP \big)}_{\chi_\lieg} $  is the same as in  $ \liegRpP \, $,
 and the latter coincides with the one in
 $ \lieg_{P_{(\chi)}}^{\R_{(\chi)},p} $, we have a Lie coalgebra isomorphism among them.
 Let us show that it is a Lie algebra isomorphism as well.
 The Lie bracket in  $ {\big( \liegRpP \big)}_{\chi_\lieg} $
 is a deformation of that of  $ \liegRpP $ by the formula
 \eqref{eq: def_cocycle-bracket}:
$$
  [x,y ]_{\chi_\lieg} := \, [x,y] - (\partial_{*} \chi_\lieg)(x\ot y)\, =
  \, [x,y] - x_{[1]} \, \chi_\lieg(x_{[2]},y) \, - \, {(-1)}^{|x| |y_{[1]}|} \, y_{[1]} \, \chi_\lieg(x,y_{[2]})
$$
 Let us check that the modified defining relations in
 $ {\big( \liegRpP \big)}_{\chi_\lieg} $  coincide with the ones in
 $ \lieg_{P_{(\chi)}}^{\R_{(\chi)},p} \, $. First note that by the formula above,
 $[x,y ]_{\chi_\lieg} = [x,y ]$ whenever $x,y$ are elements in the radical of
 $\chi_\lieg$; in particular for all elements $E_i, F_i$ for all $i\in I$.
 Hence, we need to check only those relations involving elements in
 the Cartan subalgebra  $ \lieh $  of  $ \lieg_{P_{(\chi)}}^{\R_{(\chi)},p} $.
 As the Lie cobracket on this subalgebra is trivial, we get that
  $$  {[T',T'']}_{\chi_{{}_\lieg}}  \, = \;  [T',T''] \, - \, T'_{[1]} \, \chi_\lieg\big(T'_{[2]},T''\big) \, +
  \, T''_{[1]} \, \chi_\lieg\big(T''_{[2]},T'\big)  \, = \,  [T',T'']  \, = \,  0  $$
 for all  $ \, T', T'' \in \lieh \, $.  Take now  $ \, i \in I \, $:  since
 $ \, \delta(T) = 0 \, $,  $ \, \delta(E_i) = 2 \, T_i^+ \!\wedge E_i \, $  and
 $ \, \delta(F_i) = 2 \, T_i^- \!\wedge F_i \, $,  \,we get that
\begin{align*}
   {[T,E_i]}_{\chi_\lieg}  &  \, = \,  [T,E_i] - E_i \, \chi_\lieg\big(T_i^+,T\big)  \, = \,
   \alpha_i(T) E_i + \chi\big(T, T_i^+\big) E_i  \, = \, \alpha_i^{(\chi)}(T)  \, E_i  \\
   {[T,F_i]}_{\chi_\lieg}  &  \, = \,  [T,F_i] - F_i \, \chi_\lieg\big(T_i^-,T\big)  \, = \,
   -\alpha_i(T) F_i + \chi_{\lieg}\big(T, T_i^-\big) \, F_i  \, =  \\
   &  = \,  -\Big(\alpha_i(T) F_i - \chi\big(T,T_i^-\big)\Big) \, F_i  \, =
   \,  -\alpha_i^{(\chi)}(T)\, F_i
\end{align*}

 \vskip7pt
   \textit{(b.1)}\, By  \textit{(a)},  it is enough to find an antisymmetric
   $ \Bbbk $--linear
  $ \, \chi : \lieh \otimes \lieh \!\longrightarrow\! \Bbbk \, $  obeying
  \eqref{eq: condition-eta} such that
  $ \, P' = P_{(\chi)} \, $;
  this is guaranteed by  Proposition \ref{prop: mutual-2-cocycle-def}.  Now
\textit{(b.2)}\, follows from  \textit{(b.1)},
and the uniqueness of split realizations
   --- cf.\ Proposition \ref{prop: exist-realiz's}\textit{(b)}, and
   \textit{(b.3)}\,  follows from \textit{(b.2)},  taking as
   $ \lieg_{\scriptscriptstyle P'}^{\scriptscriptstyle \R',p} $
   the given split minimal MpLSbA and for  $ \liegRpP $  Yamane's  MpLSbA
   which by definition is straight and split minimal.
\epf

\vskip9pt

\begin{rmk}
 The theorem above tell us that the class of all MpLSbA's of any fixed Cartan type and of fixed rank is stable by toral 2--cocycle deformations, and inside it, the subclass of all such MpLSbA's associated with \textsl{split},  resp.\  \textsl{minimal},  realizations is stable as well.  Also, for
$ P $  and  $ P' $  two matrices of Cartan type with the same associated Cartan matrix $ A \, $,  one can construct by a toral 2--cocycle deformation a split MpLSbA over one of them from any split MpLSbA over the other.  Besides, any split minimal MpLSbA over  $ P' $ is isomorphic to a toral 2--cocycle deformation of any split minimal MpLSbA over  $ P $.
\end{rmk}

\vskip3pt

\begin{free text}  \label{gener_GaGa-alg.'s}
 \textbf{Generalizations.}  The entire analysis and constructions devised in this section can be applied  \textit{verbatim\/}  to the setup of  \textsl{affine\/}  Lie superalgebras, starting form their presentation as provided in \cite{Ya2}  and  \cite{Ya3}.  One has to take care of even more additional relations and their transformations, but this is indeed a matter of bookkeeping, nothing more.  What is more relevant, though, is that the key ideas underlying the whole recipe can indeed be applied to a much larger framework (and specific examples): we expand on this in  \cite{GG5}.
\end{free text}

\bigskip
 \vskip13pt

\section{Formal Multiparameter QUE Superalgebras (=FoMpQUESA's)}  \label{sec: FoMpQUESA's}
 \vskip7pt

   This section is devoted to introduce formal multiparameter quantized universal enveloping superalgebras (in short, FoMpQUESA's) and to study their deformations.
 \vskip15pt

\subsection{The Hopf (super)algebra setup and QUESA's}  \label{sec: Hopf-setup & QUEA's}  {\ }

\vskip7pt

   We will deal in the following with Hopf superalgebras and Hopf algebras alike, for which we
   will use standard notation.  Most of the cases, the ground ring will be the ring  $ \kh $  of
   formal power series in the indeterminate  $ \hbar $  with coefficients in the field  $ \k $
   (which is always assumed to have characteristic zero).  For this ring, we will consider
   $ \hbar $--adically  complete  $ \kh $--modules:  these form a tensor category, whose tensor
   product is nothing but the  $ \hbar $--adic  completion of the standard (algebraic) tensor
   product.  When we deal with ``super'' objects (i.e., they are  $ \ZZ_2 $--graded),  the tensor
   product is graded as well, so that the relevant category is indeed a (non-trivially) braided
   symmetric category.  In particular, in this category a ``Hopf (super)algebra''
   will be a Hopf algebra in the generalized, categorical sense
   (i.e., with respect to the tensor category)
   --- in particular, it will be topological and  $ \hbar $--adically  complete:
     for further details, see for instance  \cite{CP},  \cite{KS},  and  \cite{Ge1}  for the super case.

\vskip9pt

\begin{free text}  \label{Topol-issues-topol-Hopf-super-alg}
{\bf Topological issues and topological Hopf (super)algebras.}  The ring  $ \kh $  mentioned
above carries a natural topology, called the  \textit{$ \hbar $--adic  topology},  coming from
the so-called  \textit{$ \hbar $--adic  norm\/}  with respect to which it is complete, namely
  $$  \big\|\, a_n \hbar^n + a_{n+1} \hbar^{n+1} + \cdots \big\|  \; := \;  C^{-n}
  \qquad  \big(\, a_n \not= 0 \,\big)  $$
 where $ \, C > 1 \, $  is any fixed constant in  $ \mathbb{R} \, $.
 In this sense, we shall consider \textit{topological\/}  $ \kh $--modules  and the  \textit{completed\/}  tensor products among them, which we denote by
 $ \, \widehat{\otimes}_\kh \, $  or simply by  $ \, \widehat{\otimes} \, $.
 For any  $ \Bbbk $--vector  space  $ V \, $,  set
  $$  V[[\hbar]]  \; := \;  \Big\{\, {\textstyle \sum_{n \geq 0}}
  v_n \hbar^n \,\Big|\, v_n \in V \, , \; \forall \, n \geq 0 \,\Big\}  $$
 then  $ V[[\hbar]] $  is a complete  $ \kh $--module.  We call a topological  $ \kh $--module
 \textit{topologically free\/}  if it is isomorphic to  $ V[[\hbar]] $  for some
 $ \Bbbk $--vector  space  $ V \, $.  For two topologically free modules
 $ V[[\hbar]] $  and  $ W[[\hbar]] $  one has that
 $ \; V[[\hbar]] \,\widehat{\otimes}\, W[[\hbar]] \, \cong \,
 \big(V \otimes W\big)[[\hbar]] \, $.
 %
 %
   \eject

  All completed tensor products between  $ \kh $--modules
 will be denoted simply by  $ \otimes \, $,
 unless we intend to stress the topological aspect.  In particular,
 we make no distinction on the notation
 between Hopf superalgebras and topological
 Hopf superalgebras; we assume it is well-understood from the context.
\end{free text}

\vskip11pt

\begin{free text}  \label{QUESA's}
 {\bf Quantized Universal Enveloping Superalgebras.}\,
 Recall that a  \textit{deformation superalgebra\/}  is a topological, unital, associative
 $ \kh $--superalgebra  $ A $  which is topologically free as a  $ \kh $--module.
 Conversely, a deformation of a (unital, associative)  $ \Bbbk $--superalgebra  $ A_0 $  is
 by definition a deformation algebra  $ A $  such that
 $ \; A_0 \, \cong \, A \big/ \hbar \, A \; $.  The same criteria apply to the notion of
 ``deformation Hopf superalgebra'', just replacing ``topological, unital, associative
 $ \kh $--superalgebra''  with ``topological Hopf  $ \kh $--superalgebra''.
 Following Geer  (cf.\ \cite{Ge1},  \S 2.3, and references therein),
 we say that a deformation Hopf superalgebra  $ H $  is a
 \textit{quantized universal enveloping superalgebra\/}  (or  \textsl{QUESA\/}  in short) if
 $ \; H \big/ \hbar \, H \, \cong \, U(\lieg) \; $  for some Lie superalgebra  $ \lieg \, $.
 In particular, the Lie superbracket in  $ \lieg $  comes from the multiplication in
 $ \; U(\lieg) \, \cong \, H \big/ \hbar H \; $.
   Moreover, this  $ \lieg $  inherits a Lie supercoalgebra structure from the QUESA,
   making it into a Lie \textsl{superbialgebra},  thanks to the following result:
\end{free text}

\vskip9pt

\begin{theorem}  (cf.\ \cite[\S 2.3]{Ge1})  \label{thm:Lie-coalgebra-limit}
 Let  $ H $  be a quantized universal enveloping superalgebra with
 $ \; H \big/ \hbar \, H \, \cong \, U(\lieg) \; $.
 Then the Lie superalgebra  $ \lieg $  is naturally equipped with a
 Lie superbialgebra structure, whose Lie cobracket is defined by
\begin{equation}  \label{def: cocorchete-semiclasico}
 \delta(\text{\rm x})  \, := \,  \frac{\,\Delta(x) - \Delta^{\op}(x)\,}{\,\hbar\,}   \qquad  \big(\hskip-10pt \mod \hbar \,\big)
\end{equation}
where  $ \, x \in H \, $  is any lifting of  $ \; \text{\rm x} \in \lieg \subseteq U(\lieg) \, \cong \, H \big/ \hbar H \, $.   \qed
\end{theorem}

\vskip11pt

\begin{definition}  \cite{CP,Ge1}
 The semiclassical limit of a quantized universal enveloping superalgebra
 $ H $ is the Lie superbialgebra  $ \, \big(\, \lieg \, , \, [-,-] \, , \, \delta \,\big) \, $
 where  $ \lieg $  is the Lie superalgebra such that
 $ \, H \big/ \hbar \, H \cong U(\lieg) \, $  and  $ \delta $  is defined as above.
 Conversely, we say that  $ H $  is a quantization of the Lie superbialgebra
 $ \big(\, \lieg \, , \, [-,-] \, , \, \delta \,\big) \, $.   \hfill  $ \diamondsuit $
\end{definition}

\vskip11pt

\begin{free text}  \label{defs_Hopf-algs}
 {\bf Hopf superalgebra deformations.}
 There exist two standard methods to deform Hopf algebras in braided (symmetric) categories,
 thus in particular Hopf superalgebras
 (and even quasi-Hopf superalgebras, more in general),
 leading to so-called ``twist deformations'' and to  ``$ 2 $--cocycle
 deformations'': hereafter we recall both procedures   --- which are naturally
   dual to each other,
%
%
 in a very precise sense ---   adapting them to the setup of  \textsl{topological\/}
 Hopf superalgebras, then later on we apply them to formal quantum supergroups.
 We will be somewhat sketchy, since these objects for a Hopf algebra in any symmetric tensor category are defined in the same way as in the category of vector spaces; further details can be found in literature.

\vskip11pt

 \textsl{$ \underline{\text{Twist deformations}} $:}\,
 Let  $ H $  be a Hopf superalgebra (possibly in topological sense).  We call
 ``{\it twist element}'',  or just a  ``{\it twist}'',  of  $ H $  any even (hence homogeneous) , invertible element  $ \, \F \in H \otimes H \, $  such that
  $$  \F_{1,2} \, \big( \Delta \otimes \text{id} \big)(\F)  \, =
  \,  \F_{2,3} \, \big( \text{id} \otimes \Delta \big)(\F)  \quad ,
  \qquad  \big( \epsilon \otimes \text{id} \big)(\F)  =  1
  =  \big( \text{id} \otimes \epsilon \big)(\F)  $$
 Then  $ H $  bears a second Hopf superalgebra structure, denoted  $ H^\F $
 and called  \textit{twist deformation\/}  of the old one, with the old product,
 unit and counit, but with new ``twisted'' coproduct  $ \Delta^\F $ and antipode
 $ \SS^\F $  given by
  $$  \qquad   \Delta^{\!\F}(x)  \, := \,  \F \, \Delta(x) \, \F^{-1}  \quad ,
    \qquad  \SS^\F(x) \, := \, v\,\SS(x)\,v^{-1}
    \quad \qquad  \forall \;\; x \in H  $$
 where  $ \, v := \sum_\F \SS(f'_1)\,f'_2 \, $   --- with
 $ \, \sum_\F f'_1 \otimes f'_2 = \F^{-1} \, $  ---
 is invertible in  $ H \, $.
 %
 %
  \textsl{When  $ H $  is in fact a  \textit{topological}  Hopf superalgebra}
  --- meaning that, in particular, its coproduct  $ \Delta $
 takes values into  $ \, H \otimes H \, $  where now  ``$ \, \otimes \, $''
 stands for a suitable topological tensor product ---
 then the same notions still make sense, and the related results apply again,
 up to properly reading them.  Moreover, if  $ H $  is only a superbialgebra
 (possibly topological), then the same recipe makes  $ H^\F $
 into a superbialgebra again.

\vskip11pt

 \textsl{$ \underline{\text{Cocycle deformations}} $:}\,
 Let $ \, \big( H, m\,, 1\,, \Delta\,, \epsilon \big) \, $
 be a superbialgebra over a ring  $ \k \, $.  We call ``(normalized) {\it Hopf 2-cocycle\/}''
(or just a  ``$ 2 $--cocycle''  if no confusion arises) any convolution invertible,
parity-preserving map  $ \sigma $  in  $ \, \Hom_{\Bbbk}(H \otimes H, \k \,) \, $
such that
  $$  {(-1)}^{|b_{(2)}| |c_{(1)}|} \, \sigma\big(b_{(1)},c_{(1)}\big) \, \sigma\big(a,b_{(2)}c_{(2)}\big)  \,\; = \;\,  {(-1)}^{|a_{(2)}| |b_{(1)}|} \, \sigma\big(a_{(1)},b_{(1)}\big) \, \sigma\big(a_{(2)}b_{(2)},c\big)  $$
and  $ \, \sigma (a,1) = \eps(a) = \sigma(1,a) \, $  for all  $ \, a, b, c \in H \, $.
                                                                     \par
   Using a  $ 2 $--cocycle  $ \sigma $  it is possible to define a new
   superalgebra structure on  $ H $  by deforming the multiplication.
   Indeed, define
   $ \, m_{\sigma} = \sigma * m * \sigma^{-1} : H \otimes H \longrightarrow H \, $  by
  $$  m_{\sigma}(a,b)  \, = \,  a \cdot_{\sigma} b  \, =  \,  {(-1)}^{|a_{(2)}| |b_{(1)}| + |a_{(3)}| |b_{(1)}| + |a_{(3)}| |b_{(2)}|} \, \sigma\big(a_{(1)},b_{(1)}\big) \, a_{(2)} \, b_{(2)} \, \sigma^{-1}\big(a_{(3)},b_{(3)}\big)  $$
for all  $ \, a , b \in H \, $.  If in addition  $ H $  is a Hopf superalgebra with antipode  $ \SS \, $,  then define also the ``deformed antipode''  $ \, \SS_\sigma : H \longrightarrow H \, $  by
  $$  \SS_{\sigma}(a)  \, = \,  \sigma\big(a_{(1)},\SS(a_{(2)})\big) \, \SS(a_{(3)}) \, \sigma^{-1}\big(\SS(a_{(4)}),a_{(5)}\big)   \eqno \forall \;\, a \in H  \quad  $$
 It is known  that  $ \, \big( H, m_\sigma, 1, \Delta, \eps \big) \, $
 is in turn a superbialgebra, and  $ \, \big( H, m_\sigma, 1, \Delta, \eps, \SS_\sigma \big) $  \, is a Hopf superalgebra: we shall call such a new structure on  $ H $  a
 \textit{cocycle deforma\-tion\/}  of the old one, and we shall graphically denote it by
 $ H_\sigma \, $.
 \vskip7pt
   \textsl{N.B.:}\,  in  \S \ref{subsec: q2cocyc-def_FoMpQUESAs}  later on, we will actually apply a suitable generalization of this procedure, namely the process of ``deformation by  \textsl{polar--2--cocycle\/}'':  this is introduced in  \cite{GG4}  for Hopf algebras that are quantum groups, we shall presently extend it (in a specific example) to the context of quantum  \textsl{super\/}groups.
\end{free text}

\medskip

\begin{free text}  \label{q-numbers}
 {\bf Some  $ q $--numbers.}\,
%
%
 Let  $ \Zqqm $  be the ring of Laurent polynomials with integral coefficients in the indeterminate  $ q \, $.  For every  $ \, n \in \NN_{+} \, $ we define
  $$  \displaylines{
   {(0)}_q  \, := \;  1 \;\; ,  \quad
    {(n)}_q  \, := \;  \frac{\,q^n -1\,}{\,q-1\,}  \; = \;  1 + q + \cdots + q^{n-1}  \; =
    \; {\textstyle \sum\limits_{s=0}^{n-1}} \, q^s  \qquad  \big(\, \in \, \ZZ[q] \,\big)  \cr
   {(n)}_q!  \, := \;  {(0)}_q {(1)}_q \cdots {(n)}_q  := \,  {\textstyle \prod\limits_{s=0}^n}
{(s)}_q  \;\; ,  \quad
    {\binom{n}{k}}_{\!q}  \, := \;  \frac{\,{(n)}_q!\,}{\;{(k)}_q! {(n-k)}_q! \,}  \qquad
\big(\, \in \, \ZZ[q] \,\big)  \cr
   {[0]}_q  := \,  1  \; ,  \;\;
    {[n]}_q  := \,  \frac{\,q^n -q^{-n}\,}{\,q-q^{-1}\,}  =
\, q^{-(n-1)} + \cdots + q^{n-1}  =  {\textstyle \sum\limits_{s=0}^{n-1}}
\, q^{2\,s - n + 1}  \!\quad  \big( \in \Zqqm \,\big)  \cr
   {[n]}_q!  \, := \;  {[0]}_q {[1]}_q \cdots {[n]}_q  = \,
{\textstyle \prod\limits_{s=0}^n} {[s]}_q  \;\; ,  \quad
    {\bigg[\, {n \atop k} \,\bigg]}_q  \, := \;  \frac{\,{[n]}_q!\,}{\;{[k]}_q! {[n-k]}_q!\,}
\qquad  \big(\, \in \, \Zqqm \,\big)  }  $$
 \vskip5pt
\noindent
   In particular, we have the identities
  $$
  {(n)}_{q^2} = q^{n-1} {[n]}_q \;\; ,  \qquad  {(n)}_{q^2}! =
q^{\frac{n(n-1)}{2}} {[n]}_q ! \;\; ,
\qquad  {\binom{n}{k}}_{\!\!q^{2}} = q^{k(n-k)} {\bigg[\, {n \atop k} \,\bigg]}_q  \;\; .
$$
 Moreover, for any field  $ \FF $  we can think of Laurent polynomials as functions on  $ \FF^\times $,
 hence for any  $ \, q \in \FF^\times \, $  we shall read every symbol above as a suitable element in  $ \, \FF \, $.
\end{free text}

\medskip

\subsection{Formal multiparameter QUESA's (=FoMpQUESA's)}  \label{subsec: FoMpQUESA's}
 {\ }
 \vskip11pt
We introduce now the notion of  {\sl formal multiparameter quantum universal enveloping
superalgebra}   --- or just ``FoMpQUESA'', in short ---
in particular proving, somewhat indirectly, that it carries a structure of Hopf superalgebra.
 \vskip3pt
   Hereafter,  $ \Bbbk $  is again a field of characteristic zero,
   $ \kh $  the ring of formal power series in  $ \hbar $  with coefficients in  $ \Bbbk \, $.
   In any topological,  $ \kh $--adically  complete  $ \kh $--algebra  $ \mathcal{A} \, $,
   if  $ \, X \in \mathcal{A} \, $  we set
   $ \; e^{\hbar\,X} := \exp(\hbar\,X) \, = \, {\textstyle \sum\limits_{n=0}^{+\infty}} \, \hbar^n X^n \big/ n! \; \in \, \mathcal{A} \; $.

\vskip9pt

\begin{definition}  \label{def: params-x-Uphgd}
 For later use, we fix some more notation:
 \vskip5pt
   \textit{(a)}\,  Let  $ \, A := {\big(\, a_{i,j} \big)}_{i, j \in I} \, $
   be a fixed Cartan matrix as in  Definition \ref{def: Cartan-super-datum},
   with  $ \, I := \{1,\dots,n\} \, $,  and let
   $ \, P := {\big(\, p_{i,j} \big)}_{i, j \in I} \in M_n\big( \kh \big) \, $
   be a square matrix of Cartan type associated with  $ A $
   (cf.\  Definition \ref{def: realization of P}\textit{(d)\/}),  so
 $ \, P + P^{\,\scriptscriptstyle T} = 2\,D\,A \, $,
 i.e.\  $ \, p_{ij} + p_{ji} = 2\,d_i\,a_{ij} \, $  for all
 $ \, i, j \in I \, $,  which implies
 $ \, p_{ii} = 2\,d_i \, $  for all  $ \, i \in I \, $
 (notation of  \S \ref{MpMatrices, Cartan & realiz.'s}).
                                             \par
   We define the following elements in  $ \kh \, $,  \,for all  $ \, i, j \in I \, $:
  $$  \displaylines{
   q := e^\hbar = \exp(\hbar) \;\; ,  \quad  q_i := e^{\hbar\,d_i} = q^{d_i} \;\; ,
   \quad  q_{ij} := e^{\hbar\,p_{ij}} = q^{p_{ij}} \;\; ,  \quad  q_{ij}^{1/2} := e^{\hbar\,p_{ij}/2}  \cr
   k_{ij} \, := \, e^{\hbar\,p_{i,j}/2} e^{-\hbar\,p_{j,i}/2} \, = \, q^{+p_{ij}/2} \, q^{-p_{ji}/2} \, =
   \, q_{ij}^{1/2} \, q_{ji}^{-1/2}  }  $$
 \textsl{In particular we have  $ \, q_{ii}^{1/2}= e^{\hbar{}d_i} = q_i \, $  and
 $ \; q_{ij} \, q_{ji} = q_{ii}^{\,a_{ij}} \; $  for all  $ \, i, j \in I \, $}.
                                             \par
   For any  $ \, i \in I \, $,  we set  $ \, \nu_i := q^{\varepsilon_i} \, $,
   \,where the  $ \varepsilon_i $'s  are given as follows:

%
%
%
\begin{equation}
  \begin{aligned}
   (\varepsilon_1\,,\varepsilon_2\,,\varepsilon_3\,,\varepsilon_4)  &  \, := \;  (6,-2,-2,-2)
   \qquad  \text{ if } \  A = F_4 \;\; ,  \\
   (\varepsilon_1\,,\varepsilon_2\,,\varepsilon_3)  &  \, := \;  (-2,2,6)  \qquad   \text{ if } \  A = G_3 \;\; ,  \\
   \varepsilon_i \, := \; \pm 1  &  \quad \forall \;\; i \quad  \text{otherwise,  \textsl{but\/}  for the conditions}  \\
   \varepsilon_n \,= \; +\varepsilon_{n-1}  &  \text{\ \ if } \  A = D_n(I) \; ,  \quad  \varepsilon_n \,= \; -\varepsilon_{n-1}  \text{\ \ if } \  A = D_n(I\!I)
  \end{aligned}
\end{equation}

In fact this corresponds to choosing an orthogonal bilinear form
 --- in some vector space containing the roots of the Lie superalgebra built out of  $ A $  and some parity function ---   so that
 $ \,\varepsilon_i = (\epsilon_i \, , \epsilon_i) \, $,  \,cf.\ \cite{Ya2}.
\vskip5pt
   \textit{(b)}\,  Given any superalgebra  $ A $  over some ring  $ R \, $,
   we define the ``supercommutator''
   $ \: {[X,Y]}_c := X \, Y - c \, {(-1)}^{p(X){}p(Y)} \, Y \, X \; $
   for any  $ \, c \in R \, $  and for any homogeneous  $ X $,  $ Y \in A \, $,
   and then extend by linearity as usual.  Moreover, for  $ \, c = 1 \, $
   we adopt the simpler notation  $ \; [X,Y] := {[X,Y]}_1 \; $.   \hfill
   $ \diamondsuit $
\end{definition}

\vskip7pt

   We can now define our FoMpQUESA's, using notation as in
   Definition \ref{def: params-x-Uphgd}  above.

\vskip11pt

\begin{definition}  \label{def: FoMpQUESA}
%
 Let  $ \, \big(\, A := {\big(\, a_{i,j} \big)}_{i, j \in I} \, , \, p \,\big) \, $
 be a Cartan super-datum, as in  Definition \ref{def: Cartan-super-datum},  and
 $ \, P := {\big(\, p_{i,j} \big)}_{i, j \in I} \in M_n\big( \kh \big) \, $
 a matrix of Cartan type associated with  $ A $.
 We fix a realization  $ \, \R := \big(\, \lieh \, , \Pi \, , \Pi^\vee \,\big) \, $
 of  $ P $  as in  Definition \ref{def: realization of P}.
 \vskip3pt
   {\it (a)}\,  We define the  {\sl formal multiparameter quantum universal super
   enveloping algebra}   --- in short  {\sl formal MpQUESA, or simply FoMpQUESA}  ---
   with multiparameter  $ P $  and realization  $ \R $  as follows.
   It is the unital, associative, topological,  $ \hbar $--adically  complete
   $ \kh $--superalgebra  $ \uRpPhg $  generated (in topological sense) by the
   $ \kh $--submodule  $ \lieh $  together with elements
 $ \, E_i \, $,  $ F_i \, $  (for all  $ \, i \in I \, $),
 with the  $ \ZZ_2 $--grading  given by setting parity on generators
  $$  |T| := \zero \quad  \forall \;\; T \in \lieh \;\; ,
  \qquad   |E_i| := p(i) =: |F_i|  \quad \forall \;\; i \in I  $$
 and with relations (for all  $ \, T , T' , T'' \in \lieh \, $,  $ \, i \, , j \in I \, $)
 \vskip-5pt
 \begin{align}
 \label{eq: rel T-E / T-T / T-F}
   T \, E_j - E_j \, T \, = \,  +\alpha_j(T) \, E_j  \; ,  \quad
   T' \, T'' \, = \, T'' \, T'  \; ,
   \quad T \, F_j - F_j \, T \, = \, -\alpha_j(T) \, F_j \\
 \label{eq: rel E-F}
   [E_i \, , F_j]  \,\; = \;\,
   \delta_{i,j} \, {{\; e^{+\hbar \, T_i^+} -
   \, e^{-\hbar \, T_i^-} \;} \over {\; q_i^{+1} - \, q_i^{-1} \;}}  \hskip93pt
\end{align}

\begin{align}
   \sum\limits_{s=0}^{1 - a_{ij}} {(-1)}^s {\left[ {{1 - a_{ij}} \atop s}
\right]}_{\!q_i}  &
 k_{ij}^{\,s}
 \, E_i^{1 - a_{ij} - s} E_j E_i^s  \; = \;  0   \quad\;\;  \big(\, i \neq j \, , \, p(i) = \zero \,\big)   \hskip3pt   \label{eq: q-Serre (E)}  \\
%
%
%
 \sum\limits_{s=0}^{1 - a_{ij}} {(-1)}^s {\left[{{1 - a_{ij}} \atop s}
\right]}_{\!q_i}  &
 k_{ji}^{\,s}
 \, F_i^{1 - a_{ij} - s} F_j F_i^{\,s}  \; = \;  0   \quad\;\;  \big(\, i \neq j \, , \, p(i) = \zero \,\big)   \hskip5pt   \label{eq: q-Serre (F)}
\end{align}

\begin{align}  \label{eq: rel1}
   \big[ E_i \, , E_j \big]_{k_{ij}} = \, 0  \;\; ,  \quad
   \big[ F_i \, , F_j \big]_{k_{ji}} = \, 0
     \qquad \qquad   \text{if} \quad  a_{i,j} = 0
\end{align}

\begin{align}  \label{eq: rel2}
   {\Big[ \big[ {[ E_i \, , E_j ]}_{\nu_j k_{ij}} \, ,
   E_k \big]_{\nu_{j+1} k_{ik} k_{jk}} \, , E_j \Big]}_{k_{ij} k_{kj}} \! = \, 0  \;\;\;  \text{in any of the following cases:}
\end{align}

\centerline{
\begin{tikzpicture}
    \draw[black, very thin] (2-0.1, -0.1) -- (2+0.1, 0.1);
    \draw[black, very thin] (2-0.1, 0.1) -- (2+0.1, -0.1);
    \draw[black, very thin] (2.2, 0) -- (3.2, 0);
    \draw[black, very thin] (3.4-0.1, -0.1) -- (3.4+0.1, 0.1);
    \draw[black, very thin] (3.4-0.1, 0.1) -- (3.4+0.1, -0.1);
   \draw [black, very thin](3.4, 0) circle (0.15 cm);
    \draw[black, very thin] (3.6, 0) -- (4.6, 0);
    \draw[black, very thin] (4.8-0.1, -0.1) -- (4.8+0.1, 0.1);
    \draw[black, very thin] (4.8-0.1, 0.1) -- (4.8+0.1, -0.1);
\draw[black, very thin] (2,0)
node[above=0.2cm, scale=0.7] {$i$};
\draw[black, very thin] (3.4,0)
node[above=0.2cm, scale=0.7] {$j$};
\draw[black, very thin] (4.8,0)
node[above=0.2cm, scale=0.7] {$k$};
\end{tikzpicture}
  \quad \qquad
\begin{tikzpicture}
    \draw[black, very thin] (2-0.1, -0.1) -- (2+0.1, 0.1);
    \draw[black, very thin] (2-0.1, 0.1) -- (2+0.1, -0.1);
    \draw[black, very thin] (2.2, 0) -- (3.2, 0);
    \draw[black, very thin] (3.4-0.1, -0.1) -- (3.4+0.1, 0.1);
    \draw[black, very thin] (3.4-0.1, 0.1) -- (3.4+0.1, -0.1);
   \draw [black, very thin](3.4, 0) circle (0.15 cm);
    \draw[black, very thin] (3.6, +0.1) -- (4.6, +0.1);
    \draw[black, very thin] (3.6, -0.1) -- (4.6, -0.1);
    \draw[black, very thin] (4.1-0.1, -0.1-0.1) -- (4.1+0.1, 0.1-0.1);
    \draw[black, very thin] (4.1-0.1, 0.1+0.1) -- (4.1+0.1, -0.1+0.1);
    \draw [black, very thin, fill=black](4.8, 0) circle (0.15 cm);
\draw[black, very thin] (2,0)
node[above=0.2cm, scale=0.7] {$i$};
\draw[black, very thin] (3.4,0)
node[above=0.2cm, scale=0.7] {$j$};
\draw[black, very thin] (4.8,0)
node[above=0.2cm, scale=0.7] {$k$};
\end{tikzpicture}
  \qquad \quad
\begin{tikzpicture}
    \draw[black, very thin] (2-0.1, -0.1) -- (2+0.1, 0.1);
    \draw[black, very thin] (2-0.1, 0.1) -- (2+0.1, -0.1);
    \draw[black, very thin] (2.2, 0) -- (3.2, 0);
    \draw[black, very thin] (3.4-0.1, -0.1) -- (3.4+0.1, 0.1);
    \draw[black, very thin] (3.4-0.1, 0.1) -- (3.4+0.1, -0.1);
   \draw [black, very thin](3.4, 0) circle (0.15 cm);
    \draw[black, very thin] (3.6, +0.1) -- (4.6, +0.1);
    \draw[black, very thin] (3.6, -0.1) -- (4.6, -0.1);
    \draw[black, very thin] (4.1-0.1, -0.1-0.1) -- (4.1+0.1, 0.1-0.1);
    \draw[black, very thin] (4.1-0.1, 0.1+0.1) -- (4.1+0.1, -0.1+0.1);
    \draw[black, very thin] (4.8-0.1, -0.1) -- (4.8+0.1, 0.1);
    \draw[black, very thin] (4.8-0.1, 0.1) -- (4.8+0.1, -0.1);
\draw[black, very thin] (2,0)
node[above=0.2cm, scale=0.7] {$i$};
\draw[black, very thin] (3.4,0)
node[above=0.2cm, scale=0.7] {$j$};
\draw[black, very thin] (4.8,0)
node[above=0.2cm, scale=0.7] {$k$};
\end{tikzpicture} }

\begin{align}  \label{eq: rel3}
   \Big[ \big[ {\big[ E_{n-1} \, , E_n \big]}_{\nu_n k_{n-1,n}} \, ,
   E_n \big]_{k_{n-1,n}} \, , E_n \Big]_{\nu_n^{-1} k_{n-1,n}} \; = \; 0
\end{align}

\begin{figure}[H]
    \centering
\begin{tikzpicture}[scale=0.8]
\draw[black, very thin] (2-0.1, -0.1) -- (2+0.1, 0.1);
    \draw[black, very thin] (2-0.1, 0.1) -- (2+0.1, -0.1);
    \draw[black, very thin] (2.2, 0.08) -- (3.3, 0.08);
    \draw[black, very thin] (2.2, -0.08) -- (3.3, -0.08);
\draw [black, very thin, fill=black](3.7, 0) circle (0.15 cm);

\draw[black, very thin] (2.8-0.1, -0.3) -- (2.8+0.1, 0);
\draw[black, very thin] (2.8-0.1, 0.3) -- (2.8+0.1, 0);
\draw[black, very thin] (2,0)
node[above=0.2cm, scale=0.7] {$n\!-\!1$};
\draw[black, very thin] (3.5,0)
node[above=0.2cm, scale=0.7] {\hskip15pt $n$};

\draw[black, very thin] (0,0)
node { with };
\draw[black, very thin] (7,0)
node { and $A$ of type $B_n$ };
\end{tikzpicture}
\end{figure}

\begin{equation}  \label{eq: rel4}
\begin{aligned}
    &  \Big[ \big[ E_{n-2} \, ,E_{n-1} \big]_{\nu_{n-1}k_{n-2,n-1}} \, ,
    E_n \Big]_{\nu_n k_{n-2,n} k_{n-1,n}} \; -  \\
    &  \qquad   - \; \Big[ \big[ E_{n-2} \, , E_n \big]_{\nu_{n-1} k_{n-2,n}} \, ,
    E_{n-1} \Big]_{\nu_n k_{n-2,n-1} k_{n,n-1}}  \; = \;  0
\end{aligned}
\end{equation}

\begin{figure}[H]
\begin{tikzpicture}    \quad
\draw[black, very thin] (4.3, 0.1) -- (5.2, 0.6);
\draw[black, very thin] (4.3, -0.1) -- (5.2, -0.6);

\draw[black, very thin] (5.35, 0.3) -- (5.35, -0.3);
\draw[black, very thin] (5.42, 0.3) -- (5.42, -0.3);

\draw [black, very thin](5.4, 0.6) circle (0.13 cm);
\draw [black, very thin](5.4, -0.6) circle (0.13 cm);

\draw[black, very thin] (4-0.1, 0.1) -- (4+0.1, -0.1);
\draw[black, very thin] (4-0.1, -0.1) -- (4+0.1, 0.1);

\draw[black, very thin] (5.4-0.1, 0.1+0.6) -- (5.4+0.1, -0.1+0.6);
\draw[black, very thin] (5.4-0.1, -0.1+0.6) -- (5.4+0.1, 0.1+0.6);

\draw[black, very thin] (5.4-0.1, 0.1-0.6) -- (5.4+0.1, -0.1-0.6);
\draw[black, very thin] (5.4-0.1, -0.1-0.6) -- (5.4+0.1, 0.1-0.6);

\draw[black, very thin] (5.4,0.6)
node[above=0.1cm, scale=0.7] {$n$};

\draw[black, very thin] (4,0)
node[above=0.1cm, scale=0.7] {$n\!-\!2$};

\draw[black, very thin] (5.4,-0.6)
node[below=0.1cm, scale=0.7] {$n\!-\!1$};

\draw[black, very thin] (2,0.3)
node[below=0.1cm] {with};
\draw[black, very thin] (8.5,0.3)
node[below=0.1cm] { and  $ A $  of type  $ \, D_n(I\!I) $};

\end{tikzpicture}
\end{figure}

\begin{equation}  \label{eq: rel5}
\begin{aligned}
  &  \bigg[ \Big[ \big[\, [E_{n-2} \, , E_{n-1}]_{\nu_{n-1} k_{n-2,n-1}} \, ,
  E_n \big]_{\nu_n k_{n-2,n-1} k_{n-1,n}} \, ,   \hfill  \\
  &  \hskip85pt   [E_{n-2}, E_{n-1}]_{\nu_{n-1} k_{n-2,n-1}} \Big] \, , \, E_{n-1} \,\bigg]_{\nu_{n-1}k_{n-2,n-1}}  \; = \;\,  0
\end{aligned}
\end{equation}

\begin{figure}[H]
    \centering  for  \quad
\begin{tikzpicture}[scale=0.8]
\draw[black, very thin] (2.2, 0) -- (3.2, 0);
\draw[black, very thin] (3.6, 0.07) -- (4.8, 0.07);
\draw[black, very thin] (3.6, -0.07) -- (4.8, -0.07);
\draw[black, very thin] (4.2-0.1, -0.1+0.1) -- (4.2+0.1, 0.1+0.1);
\draw[black, very thin] (4.2-0.1, 0.1-0.1) -- (4.2+0.1, -0.1-0.1);

\draw[black, very thin] (2-0.1, -0.1) -- (2+0.1, 0.1);
\draw[black, very thin] (2-0.1, 0.1) -- (2+0.1, -0.1);

\draw[black, very thin] (3.4-0.1, -0.1) -- (3.4+0.1, 0.1);
\draw[black, very thin] (3.4-0.1, 0.1) -- (3.4+0.1, -0.1);

\draw [black, very thin](2, 0) circle (0.15 cm);
\draw [black, very thin](3.4, 0) circle (0.15 cm);
\draw [black, very thin](5, 0) circle (0.15 cm);

\draw[black, very thin] (2,0)
node[above=0.2cm, scale=0.7] {$n\!-\!2$};

\draw[black, very thin] (3.4,0)
node[above=0.2cm, scale=0.7] {$n\!-\!1$};

\draw[black, very thin] (5,0)
node[above=0.2cm, scale=0.7] {$n$};

\end{tikzpicture}
 \quad  and  $ \, A \, $  of type  $ \, C_n $
\end{figure}

\begin{equation}
  \begin{aligned}  \label{eq: rel6}
  &  \Bigg[ \Bigg[ \bigg[ \Big[ {\big[\,  {[E_{n-3} \, ,
  \, E_{n-2}]}_{\nu_{n-2} \, k_{n-3,n-2}} \, , \, E_{n-1} \big]}_{\nu_{n-1} \, k_{n-3,n-1} \, k_{n-2,n-1}} \, ,   \hskip25pt  \\
  &  \quad  E_n \,{\Big]}_{2 \, \nu_n \, k_{n-3,n} \, k_{n-2,n} \, k_{n-1,n}} \, , \, E_{n-1}\,{\bigg]}_{\nu_n \, k_{n-3,n-1} \, k_{n-2,n-1} \, k_{n,n-1}} \, ,  \\
  &  \qquad  E_{n-2} \,{\Bigg]}_{\nu_{n-1} \, k_{n-3,n-2} \, k_{n-1,n-2}^{\,2} k_{n,n-2}} \, , \, E_{n-1} \,{\Bigg]}_{k_{n-3,n-1} \, k_{n-2,n-1}^{\,2} \,  k_{n,n-1}}  \,\; = \;\;  0
  \end{aligned}
\end{equation}

\begin{figure}[H]
    \centering  for  \quad
  \begin{tikzpicture}[scale=0.8]
\draw[black, very thin] (0.8, 0) -- (1.8, 0);
\draw[black, very thin] (2.2, 0) -- (3.2, 0);
\draw[black, very thin] (3.6, 0.07) -- (4.8, 0.07);
\draw[black, very thin] (3.6, -0.07) -- (4.8, -0.07);
\draw[black, very thin] (4.2-0.1, -0.1+0.1) -- (4.2+0.1, 0.1+0.1);
\draw[black, very thin] (4.2-0.1, 0.1-0.1) -- (4.2+0.1, -0.1-0.1);

\draw[black, very thin] (0.6-0.1, -0.1) -- (0.6+0.1, 0.1);
\draw[black, very thin] (0.6-0.1, 0.1) -- (0.6+0.1, -0.1);
\draw[black, very thin] (2-0.1, -0.1) -- (2+0.1, 0.1);
\draw[black, very thin] (2-0.1, 0.1) -- (2+0.1, -0.1);
\draw[black, very thin] (3.4-0.1, -0.1) -- (3.4+0.1, 0.1);
\draw[black, very thin] (3.4-0.1, 0.1) -- (3.4+0.1, -0.1);
\draw [black, very thin](2, 0) circle (0.15 cm);
\draw [black, very thin](3.4, 0) circle (0.15 cm);
\draw [black, very thin](5, 0) circle (0.15 cm);

  \draw[black, very thin] (0.6,0)
node[above=0.2cm, scale=0.7] {$n\!-\!3$};
  \draw[black, very thin] (2,0)
node[above=0.2cm, scale=0.7] {$n\!-\!2$};
  \draw[black, very thin] (3.4,0)
node[above=0.2cm, scale=0.7] {$n\!-\!1$};
  \draw[black, very thin] (5,0)
node[above=0.2cm, scale=0.7] {$n$};
%
\end{tikzpicture}
 \quad  and  $ \, A \, $ of type  $ \, C_n $
\end{figure}
\noindent
 \textsl{and also the like of relations  \eqref{eq: rel2},  \eqref{eq: rel3}, \eqref{eq: rel4}, \eqref{eq: rel5}, \eqref{eq: rel6}  with the  $ E_i $'s  replaced by  $ F_i $'s,  the  $ q_{ij} $'s  by  $ q_{ji} $'s  and the  $ k_{ij} $'s  by  $ k_{ji} $'s.}
 \vskip1pt
   \textit{$ \underline{\text{N.B.}} $:\/}  In  \eqref{eq: rel6}  we have removed a subscript ``2'' which appears in Yamane's original formula in  \cite{Ya1},  seemingly as a misprint.  Also, the like of formulas  \eqref{eq: q-Serre (E)}  and  \eqref{eq: q-Serre (F)}  in  [\textit{loc.\ cit.}]  have  ``$ \, +\big|a_{ij}\big| \, $''  instead of  ``$ \, -a_{ij} \, $'',  \,yet the two do coincide because of the assumption  $ \, p(i) = \zero \, $.
 \vskip3pt
   {\it (b)}\,  We say that  $ \uRpPhg $  is  \textsl{straight},  or  \textsl{small},
   or \textsl{minimal},  or  \textsl{split},
   if such is  $ \R \, $;  also, we define the  \textsl{rank\/}  of  $ \uRpPhg $  as
   $ \; \rk\!\big( \uRpPhg \big) := \rk(\R) = \rk_\kh(\lieh) \, $.
 \vskip3pt
   {\it (c)}\,  We define the  {\sl Cartan subalgebra\/}
   $ \uRpPhh \, $,  or just  $ U_\hbar(\lieh) \, $,  of a FoMpQUESA  $ \uRpPhg $
   as being the unital,  $ \hbar $--adically  complete topological
   $ \kh $--subsuperalgebra  of  $ \uRpPhg $
   generated by the  $ \kh $--submodule  $ \lieh \, $.
 \vskip3pt
   {\it (d)}\,  We define the  {\sl positive, resp.\ negative, Borel subsuperalgebra\/}
   $ \uRpPhbp $,  resp.\  $ \uRpPhbm $,
 of  $ \, \uRpPhg $  to be
 the unital,  $ \hbar $--adically  complete topological  $ \kh $--subsu\-peralgebra of
 $ \uRpPhg $  generated by  $ \lieh $  and the  $ E_i $'s,  resp.\  by  $ \lieh $  and the
 $ F_i $'s  ($ \, i \in I \, $).
 \vskip5pt
   {\it (e)}\,  We define the  {\sl positive, resp.\ negative, nilpotent subsuperalgebra\/}
   $ \uRpPhnp $,  resp.\  $ \uRpPhnm $,
   of a FoMpQUESA  $ \uRpPhg $  to be the unital,  $ \hbar $--adically  complete topological
   $ \kh $--subsuperalgebra
   of  $ \uRpPhg $  generated by all the  $ E_i $'s,  resp.\ all the  $ F_i $'s
   ($ \, i \in I \, $).   \hfill  $ \diamondsuit $
\end{definition}

\vskip9pt

\begin{rmk}  \label{rmk: q-Serre = iterated q-commutator}
 In  Definition \ref{def: FoMpQUESA}\textit{(a)},  all relations  \eqref{eq: rel1}
 through  \eqref{eq: rel6},  as well as their counterparts with the  $ F_\ell $'s
 replacing the  $ E_\ell $'s,  are all given in terms of ``iterated
 $ q $--supercommutators''  (either identifying such a  $ q $--supercommutator  to zero,
 or identifying two of them).  Similarly, it is important to stress that also
 \textsl{the quantum Serre's relations, namely  \eqref{eq: q-Serre (E)}  and
 \eqref{eq: q-Serre (F)},  can be expressed in that way}.
 We prove this for  \eqref{eq: q-Serre (E)},  the case of  \eqref{eq: q-Serre (F)}  being entirely similar, up to switching every  $ q_{ij} \, $,  resp.\  every  $ k_{ij} \, $,  to  $ q_{ji} \, $,  resp.\ to $ k_{ji} \, $.
                                                        \par
   Indeed, let  $ E_{i,j} $  denote the left-hand side of  \eqref{eq: q-Serre (E)}.  Then with notation as in  Definition \ref{def: params-x-Uphgd}\textit{(b)},  straightforward computations yield
\begin{align}
\label{eq: q-S = q-bra (0)}
   a_{i,j} = 0  &  \quad \Longrightarrow \quad
   E_{i,j} \; = \; {[E_i \, , E_j]}_{q_{ij}} \qquad
   \\
\label{eq: q-S = q-bra (1)}
   a_{i,j} = -1  &  \quad \Longrightarrow \quad  E_{i,j} \; = \; \big[ E_i \, , {[E_i \, , E_j]}_{q_{ij}} \big]_{q_{ii} \, q_{ij}}  \\
\label{eq: q-S = q-bra (2)}
   a_{i,j} = -2  &  \quad \Longrightarrow \quad  E_{i,j} \; = \; \Big[ E_i \, ,
   \big[ E_i \, , {[E_i \, , E_j]}_{q_{ij}} \big]_{q_{ii} \, q_{ij}} \Big]_{q_{ii}^{\,2} \, q_{ij}}  \\
\label{eq: q-S = q-bra (3)}
   a_{i,j} = -3  &  \quad \Longrightarrow \quad  E_{i,j} \; = \; \bigg[ E_i \, , \Big[ E_i \, ,
   \big[ E_i \, , {[E_i \, , E_j]}_{q_{ij}} \big]_{q_{ii} \, q_{ij}}
   \Big]_{q_{ii}^{\,2} \, q_{ij}} \bigg]_{q_{ii}^{\,3} \, q_{ij}}
\end{align}

                                                                 \par
   In more general terms, the iterated  $ q $--supercommutators
   arise from a  \emph{braided} adjoint action.
   This braiding is the canonical one given in the category of super
   Yetter-Drindfeld modules over the Cartan subalgebra.
\end{rmk}

\vskip9pt

   The following two results underscore that the dependence of FoMpQUSEA's
   on realizations (which includes that on the multiparameter matrix) is functorial:

\vskip11pt

\begin{prop}  \label{prop: functor_R->uRPhg}
 Let  $ \, P \in M_{n}\big( \kh \big) \, $,  and
 $ \, p : I \longrightarrow \ZZ_2 \, $  a parity function.
 If both  $ \R' $  and  $ \, \R'' $  are realizations of  $ P $  and
 $ \, \underline{\phi} : \R' \relbar\joinrel\relbar\joinrel\longrightarrow \R'' \, $
 is a morphism between them,  then there exists a unique morphism
 $ \; U_{\underline{\phi}} : U^{\,\R'\!,\,p}_{\!P,\,\hbar}(\lieg) \relbar\joinrel\relbar\joinrel\longrightarrow U^{\,\R''\!,\,p}_{\!P,\,\hbar}(\lieg) \; $ of unital topological
 $ \, \kh $--algebras
 that extends the morphism  $ \, \phi : \lieh' \relbar\joinrel\relbar\joinrel\longrightarrow \lieh'' \, $
 of  $ \, \kh $--modules  given by
 $ \underline{\phi} \, $;  \,moreover,
 $ \, U_{\underline{\id}_\R} = \id_{\uRpPhg} \, $  and
 $ \; U_{\underline{\phi}' \circ\, \underline{\phi}} = U_{\underline{\phi}'} \circ U_{\underline{\phi}} \; $
 (whenever  $ \, \underline{\phi}' \circ\, \underline{\phi} \, $  is defined).
 Therefore, the construction  $ \, \R \mapsto \uRpPhg \, $
 --- for any fixed  $ P $  and  $ p $  ---   is functorial in  $ \R \, $.
                                                                              \par
   Moreover, if  $ \underline{\phi} $  is an epimorphism,
   resp.\ a monomorphism, then  $ U_{\underline{\phi}} $
   is an epimorphism, resp.\ a monomorphism, as well.
                                                                              \par
   Finally, for any morphism  $ \, \underline{\phi} : \R' \relbar\joinrel\longrightarrow \R'' \, $,
   \,the kernel  $ \, \Ker\big(U_{\underline{\phi}}\big) $
   of  $ \, U_{\underline{\phi}} $  is the two-sided ideal in
   $ \, U^{\,\R'\!,\,p}_{\!P,\hbar}(\lieg) $  generated by
   $ \, \Ker(\phi) \, $,  and the latter is central in  $ \, U^{\,\R'\!,\,p}_{\!P,\hbar}(\lieg) \, $.
\end{prop}

\pf
 This is just  Proposition 4.2.3 in \cite{GG3},  but for the
 addition of the parity function, which nevertheless is harmless.
\epf

\vskip7pt

\begin{cor}  \label{cor: isom_R -> isom_uRPhg}
 With notation as above, if  $ \; \R' \cong \R'' \, $  then  $ \; U^{\,\R'\!,\,p}_{P,\hbar}(\lieg) \cong U^{\,\R''\!,\,p}_{P,\hbar}(\lieg) \; $.
                                                     \par
   In particular, all FoMpQUSEA's built upon split realizations, respectively small realizations, of the same matrix  $ P $  and sharing the same rank of  $ \, \lieh $  and the same parity function  $ p \, $,  are isomorphic to each other   --- hence they are independent (up to isomorphisms)  of the specific realization, but only depend on
 $ P $,  on the rank of  $ \, \lieh \, $,  and on the parity function.
\end{cor}

\pf
 Again, this is essentially Corollary 4.2.4 in \cite{GG3}.
\epf

\vskip9pt

\begin{exas}  \label{example-Yamane's_FoQUEASA's}
 \textsl{(Yamane's QUESA's as examples of FoMpQUESA's)}
 \vskip3pt
   Fix a Cartan super-datum  $ (A,p) $  as in  Definition \ref{def: Cartan-super-datum},  with
   $ D $  the diagonal matrix such that  $ \, D A \, $  is symmetric, and a realization
   $ \; \R \, := \, \big(\, \lieh \, , \Pi \, , \Pi^\vee \,\big) \; $  of  $ \, P := D A \, $.
 \vskip3pt
   \textit{(a)}\,  Assume that  $ \R $  is straight and such that  $ \, T_i^+ = T_i^- \, $
   for all  $ \, i \in I \, $.  Then  $ \, \uRpPhg \, $  is nothing but the quantized universal enveloping superalgebra
   (=QUESA) introduced by Yamane in  \cite{Ya1},  and denoted by  $ \, U_\hbar(\mathcal{G}) = U_\hbar(\Pi,p) \, $.
   See also  \cite{Ge1}  for a further study of these QUESA's.
 \vskip5pt
   \textit{(b)}\,  Assume that  $ \R $  is straight and split.  Then  $ \, \uRpPhg \, $
   is nothing but the ``quantum double version'' of Yamane's  $ \, U_\hbar(\mathcal{G}) = U_\hbar(\Pi,p) \, $,
   again introduced in  \cite{Ya1}  where it is denoted by  $ \mathcal{D}' \, $.
 \vskip5pt
   \textit{(c)}\,  Assume that  $ \R $  is just straight.  Then  $ \, \uRpPhg \, $
   is a suitable quotient (by an ideal generated by a submodule of the ``Cartan subalgebra'') of
   the ``quantum double version'' of Yamane's QUESA, i.e.\ the  $ \mathcal{D}' $  mentioned in  \textit{(b)\/}  above.
   We refer to such an example as a ``Yamane's QUESA''.
\end{exas}

\vskip9pt

\begin{free text}  \label{FoMpQUESA from Yamane's QUESA}
 \textbf{Hopf superalgebra structure on FoMpQUESA's.}
 The examples of FoMpQUESA's from Yamane's work can be modified so to yield more examples.

   Let us begin with  $ (A\,,p\,) $,  $ \, P_A := DA \, $  and a realization
   $ \; \R := \big(\, \lieh \, , \Pi \, , \Pi^\vee \,\big) \; $  of  $ P_A \, $,  as in  Definition \ref{def: FoMpQUESA};
   we assume in addition that  $ \, \R $  be straight, so that  $ \, U_{P_A,\,\hbar}^{\,\R,\,p}(\lieg) \, $
   is just a ``Yamane's QUESA'', as in  Example \ref{example-Yamane's_FoQUEASA's}\textit{(c)\/}  above.
   From Yamane's work  \cite{Ya1}  we know that  $ \, U_{P_A,\,\hbar}^{\,\R,\,p}(\lieg) \, $
   bears indeed a  \textsl{Hopf structure\/}  which is given on (super)algebra generators by the formulas
\begin{equation}  \label{eq: coprod_x_uRpPhg}
 \begin{aligned}
   \Delta \big(E_\ell\big)  \, = \,  E_\ell \otimes 1 \, + \, e^{+\hbar \, T_\ell^+} \otimes E_\ell  \\
   \Delta\big(T\big)  \, = \,  T \otimes 1 \, + \, 1 \otimes T   \hskip25pt  \\
   \Delta\big(F_\ell\big)  \, = \, F_\ell \otimes e^{-\hbar \, T_\ell^-} \, + \, 1 \otimes F_\ell
 \end{aligned}   \qquad
\end{equation}
 (for all  $ \, T \in \lieh \, $,  $ \, \ell \in I $)  for the coproduct,
 and for the counit and antipode by
 \vskip-9pt
\begin{eqnarray}
   \epsilon\big(E_\ell\big) \, = \, 0  \;\;\; ,  \qquad  &   \;\quad
   \epsilon\big(T\big) \, = \, 0  \;\;\; ,  &
   \;\qquad  \epsilon\big(F_\ell\big) \, = \, 0  \;\;\;\;  \label{eq: counit_x_uPhg}  \\
   \SS\big(E_\ell\big)  \, = \,  - e^{-\hbar \, T_\ell^+} E_\ell  \;\; ,  &   \;\;\;
   \SS\big(T\big)  \, = \,  - T   \;\; ,  &
   \;\;\;  \SS\big(F_\ell\big)  \, = \,  - F_\ell \, e^{+\hbar \, T_\ell^-}   \qquad
  \label{eq: antipode_x_uPhg}
\end{eqnarray}
\par
   Now we fix in  $ \lieh $  any  $ \kh $--basis  $ \, {\big\{ H_g \big\}}_{g \in \G} \, $,
   where  $ \, |\G| = \rk(\lieh) = t \, $.  Pick an antisymmetric  $ (t \times t) $--matrix
   $ \; \Phi = \big( \phi_{gk} \big)_{g, k \in \G} \in \lieso_t\big(\kh\big)  \; $,  \,and set
  $$
  \JJ_\Phi  \; := \;  {\textstyle \sum\limits_{g,k=1}^t} \phi_{gk} \, H_g \otimes H_k  \; \in \;
  \lieh \otimes \lieh  \; \subseteq \; U^\R_{\!P_A,\hskip0,7pt\hbar}(\lieh) \otimes U^\R_{\!P_A,\hskip0,7pt\hbar}(\lieh)
  $$
 By direct check, we see that the element
\begin{equation}  \label{eq: Resh-twist_F-uPhgd - 1}
  \F_\Phi  \,\; := \;\,
e^{\,\hbar \, 2^{-1} \JJ_\Phi}
 \,\; = \;\,  \exp\Big(\hskip1pt \hbar \,  \, 2^{-1} \,
 {\textstyle \sum_{g,k=1}^t} \phi_{gk} \, H_g \otimes H_k \Big)
\end{equation}
 in  $ \, U^\R_{\!P_A,\,\hbar}(\lieh) \,\widehat{\otimes}\, U^\R_{\!P_A,\,\hbar}(\lieh) \, $
 is actually a  {\sl twist\/}  for  $ U_{P_A,\,\hbar}^{\,\R,\,p}(\lieg) \, $,  in the sense of  \S \ref{defs_Hopf-algs}.
 Using it, we construct a new (topological) Hopf algebra  $ \, {\big(\, U_{P_A,\,\hbar}^{\,\R,\,p}(\lieg) \big)}^{\F_\Phi} \, $,
 isomorphic to  $ U_{P_A,\,\hbar}^{\,\R,\,p}(\lieg) $  as an algebra but with a new, twisted coalgebra structure,
 as in  \S \ref{defs_Hopf-algs}.  A direct calculation yields formulas for the new coproduct on generators, namely
  $$  \displaylines{
   \qquad   \Delta^{\!\F_\Phi}\big(E_\ell\big)  \; = \;
 E_\ell \otimes \L_{\Phi, \ell}^{+1} \, + \, e^{+\hbar \, T^+_\ell} \K_{\Phi, \ell}^{+1} \otimes
E_\ell   \quad \qquad  \big(\, \forall \;\; \ell \in I \,\big)  \cr
   \qquad \qquad \quad   \Delta^{\!\F_\Phi}\big(T\big)  \; = \;  T \otimes 1 \, + \, 1 \otimes T
   \quad \quad \qquad \qquad  \big(\, \forall \;\; T \in \lieh \,\big)  \cr
   \qquad   \Delta^{\!\F_\Phi}\big(F_\ell\big)  \; = \;
   F_\ell \otimes \L_{\Phi, \ell}^{-1} \, e^{-\hbar \, T^-_\ell } + \, \K_{\Phi, \ell}^{-1} \otimes F_\ell
   \quad \qquad  \big(\, \forall \;\; \ell \in I \,\big)  }  $$
with
  $ \; \L_{\Phi, \ell} := e^{+ \hbar \, 2^{-1} \sum_{g,k=1}^t \alpha_\ell(H_g) \, \phi_{gk} H_k} \; $,
  $ \; \K_{\Phi,\ell} := e^{+ \hbar \, 2^{-1} \sum_{g,k=1}^t \alpha_\ell(H_g) \, \phi_{kg} H_k} = \L_{\Phi, \ell}^{-1} \; $  for  $ \, \ell \in I \, $.
 Similarly, the ``twisted'' antipode  $ \SS^{\F_\Phi} $  and the
%
%
 counit  $ \epsilon^{\F_\Phi} \! := \epsilon $   are given by
  $$  \begin{matrix}
   \qquad   \SS^{\F_\Phi}\big(E_\ell\big)  \, = \,
   - e^{-\hbar \, T^+_\ell} \K_{\Phi, \ell}^{-1} \, E_\ell \, \L_{\Phi, \ell}^{-1}  \quad ,
   &  \qquad  \epsilon^{\F_\Phi}\big(E_\ell\big) \, = \, 0   \qquad \qquad  \big(\, \forall \;\; \ell \in I \,\big)  \\
   \qquad
 \SS^{\F_\Phi}\big(T\big)  \, = \,  -T  \quad ,  \phantom{\Big|^|}
  &  \qquad
 \epsilon^{\F_\Phi}\big(T\big) \, = \, 0   \qquad \qquad  \big(\, \forall \;\; T \in \lieh \,\big)
 \\
   \qquad   \SS^{\F_\Phi}\big(F_\ell\big)  \, = \,
   - \K_{\Phi, \ell}^{+1} \, F_\ell \, \L_{\Phi, \ell}^{+1} \, e^{+\hbar \, T^-_\ell}  \quad ,
   &  \qquad  \epsilon^{\F_\Phi}\big(F_\ell\big) \, = \, 0   \qquad \qquad  \big(\, \forall \;\; \ell \in I \,\big)
\end{matrix}  $$
             \par
   In short, starting from Yamane's QUESA  $ U_{P_A,\,\hbar}^{\,\R,\,p}(\lieg) $  we get,
   via a twist deformation process, a new Hopf superalgebra  $ {\big( U_{P_A,\,\hbar}^{\,\R,\,p}(\lieg) \big)}^{\F_\Phi} \, $;
   moreover, one immediately see that, by construction, this  $ {\big( U_{P_A,\,\hbar}^{\,\R,\,p}(\lieg) \big)}^{\F_\Phi} \, $
   is yet another QUESA (in the sense of  \S \ref{QUESA's}).
   Again by construction, this QUESA is of ``multiparametric nature'', in that it depends on the entries of
   $ \Phi $  as ``parameters''.  Indeed, by means of a suitable change of generators
   (hence a change of presentation), the following result shows that
   $ {\big( U_{P_A,\,\hbar}^{\,\R,\,p}(\lieg) \big)}^{\F_\Phi} $
   is in fact a FoMpQUESA in the sense of  Definition \ref{def: FoMpQUESA}:
\end{free text}

\vskip9pt

\begin{prop}  \label{prop: deform-Yamane=FoMpQUESA}
 Let  $ (A,p) $  be a Cartan super-datum,  and  $ \R $  a realization of  $ \, P_A := DA \, $;
%
%
we assume that  $ \, \R $  be straight, so
   $ \, U_{P_A,\hbar}^{\,\R,\,p}(\lieg) \, $  is just a ``Yamane's QUESA'', as in
 Example \ref{example-Yamane's_FoQUEASA's}\textit{(c)}.  We fix a twist
 $ \;  \F_\Phi \, := \, \exp\Big(\hskip1pt \hbar \,  \, 2^{-1} \,
 {\textstyle \sum_{g,k=1}^t} \phi_{gk} \, H_g \otimes H_k \Big) \; $
 as above, and we let  $ \, {\big( U_{P_A,\,\hbar}^{\,\R,\,p}(\lieg) \big)}^{\F_\Phi} \, $  be the corresponding deformed
 Hopf superalgebra.
                                                       \par
  Then there exists an isomorphism of topological superalgebras
  $$  f^{\scriptscriptstyle \Phi}_{\scriptscriptstyle P} \; :
   \; U_{\!{(P_A)}_\Phi,\,\hbar}^{\R_{\scriptscriptstyle \Phi},\,p}(\lieg)  \,\; {\buildrel \cong \over {\lhook\joinrel\relbar\joinrel\relbar\joinrel\relbar\joinrel\relbar\joinrel\relbar\joinrel\twoheadrightarrow}} \;\,
   {\big(\, U_{\!P_A,\,\hbar}^{\R,\,p}(\lieg)\big)}^{\F_\Phi}  $$
 given by
  $ \; E_i \, \mapsto \, E^{\scriptscriptstyle \Phi}_i := \, \L_{\Phi,i}^{-1} \, E_i \; , \,\; T \, \mapsto \, T \, $  and  $\; F_i \, \mapsto \, F^{\scriptscriptstyle \Phi}_i := \, F_i \, \K_{\Phi,i}^{+1} \; $
 \,with
 $ \; \L_{\Phi,i} := e^{+ \hbar \, 2^{-1} \sum_{g,k=1}^t \alpha_i(H_g) \, \phi_{gk} H_k} \; $
 and
 $ \; \K_{\Phi,i} := e^{+ \hbar \, 2^{-1} \sum_{g,k=1}^t \alpha_i(H_g) \, \phi_{kg} H_k} \, \big(= \L_{\Phi, \ell}^{-1} \,\big) \; $
 for all  $ \, i \in I \, $,  $ \, T \in \lieh \, $,  \,where  $ \, \R_\Phi $  and  $ P_\Phi $  are the twist deformations (via  $ \Phi $)  of  $ \, \R $  and  $ P \, $,  as in  Definition \ref{def: twisted-realization}.
\end{prop}

\begin{proof}
 It is clear that the  $ E^{\scriptscriptstyle \Phi}_i $'s,  the  $ T $'s  and the  $ F^{\scriptscriptstyle \Phi}_i $'s  (for  $ \, i \in I \, $  and  $ \, T \in \lieh \, $)  altoge\-ther generate  $ {\big(\, U_{\!P_A,\,\hbar}^{\R,\,p}(\lieg) \big)}^{\F_\Phi} $,  as the latter as an algebra is just  $ U_{\!P_A,\,\hbar}^{\R,\,p}(\lieg) \, $.  So we find a complete presentation of this algebra just by rewriting with respect to these new generators the old relations among the old ones, i.e.\ the  $ E_i $'s,  the  $ T $'s  and the  $ F_i $'s.  Hereafter we present the explicit computations for some of these relations, the other ones can be easily filled in by the interested reader.
 \vskip5pt
   First of all, from the relations in  \eqref{eq: rel T-E / T-T / T-F}  one gets at once identical relations with the  $ E_j^\Phipicc $'s,  resp.\  the  $ F_j^\Phipicc $'s,  replacing the  $ E_j $'s,  resp.\ the  $ F_j $'s.
 \vskip5pt
   Second, direct computations give
  $$  \displaylines{
   [E_i\,,F_j] \, - \, \delta_{i,j} \, \frac{\,e^{+\hbar\,T_i^+} - e^{-\hbar\,T_i^-} \,}{\,q_i^{+1} - q_i^{-1}\,}  \; = \;  \big[ \L_{\Phipicc,i}^{+1} \, E_i^\Phipicc \, , F_j^\Phipicc \,\K_{\Phipicc,j}^{-1} \big] \, - \, \delta_{i,j} \, \frac{\,e^{+\hbar\,T_i^+} - e^{-\hbar\,T_i^-} \,}{\,q_i^{+1} - q_i^{-1}\,}  \; =   \hfill  \cr
   = \;  \L_{\Phipicc,i}^{+1} \, E_i^\Phipicc \, F_j^\Phipicc \, \K_{\Phipicc,j}^{-1} \, - \,  {(-1)}^{p(i)\,p(j)} \, F_j^\Phipicc \, \K_{\Phipicc,j}^{-1} \, \L_{\Phipicc,i}^{+1} \, E_i^\Phipicc \, - \, \delta_{i,j} \, \frac{\,e^{+\hbar\,T_i^+} - e^{-\hbar\,T_i^-} \,}{\,q_i^{+1} - q_i^{-1}\,}  \; =   \qquad  \cr
   = \;  \L_{\Phipicc,i}^{+1} \, E_i^\Phipicc \, F_j^\Phipicc \, \K_{\Phipicc,j}^{-1} \, - \,  {(-1)}^{p(i)\,p(j)} \, \L_{\Phipicc,i}^{+1} \, F_j^\Phipicc \, E_i^\Phipicc \, \K_{\Phipicc,j}^{-1} \, - \, \delta_{i,j} \, \frac{\,e^{+\hbar\,T_i^+} - e^{-\hbar\,T_i^-} \,}{\,q_i^{+1} - q_i^{-1}\,}  \; =   \quad
   }  $$
  $$  \displaylines{
%
   \quad   = \;  \L_{\Phipicc,i}^{+1} \, \big( E_i^\Phipicc \, F_j^\Phipicc \, - \,  {(-1)}^{p(i)\,p(j)} \, F_j^\Phipicc \, E_i^\Phipicc \big) \, \K_{\Phipicc,j}^{-1} \, - \, \delta_{i,j} \, \L_{\Phipicc,i}^{+1} \, \frac{\,e^{+\hbar\,T_{\Phipicc,i}^+} - e^{-\hbar\,T_{\Phipicc,i}^-} \,}{\,q_i^{+1} - q_i^{-1}\,} \, \K_{\Phipicc,j}^{-1}  \; =  \cr
   \hfill   = \;  \L_{\Phipicc,i}^{+1} \, \left( E_i^\Phipicc \, F_j^\Phipicc \, - \,  {(-1)}^{p(i)\,p(j)} \, F_j^\Phipicc \, E_i^\Phipicc \, - \, \delta_{i,j} \, \frac{\,e^{+\hbar\,T_{\Phipicc,i}^+} - e^{-\hbar\,T_{\Phipicc,i}^-} \,}{\,q_i^{+1} - q_i^{-1}\,} \right) \, \K_{\Phipicc,j}^{-1}  \; =  \cr
   \hfill   = \;  \L_{\Phipicc,i}^{+1} \, \left( \big[ E_i^\Phipicc , F_j^\Phipicc \big] \, - \, \delta_{i,j} \, \frac{\,e^{+\hbar\,T_{\Phipicc,i}^+} - e^{-\hbar\,T_{\Phipicc,i}^-} \,}{\,q_i^{+1} - q_i^{-1}\,} \right) \, \K_{\Phipicc,j}^{-1}  }  $$
 from which we see that the ``old'' relation
  $ \,\; \displaystyle{ [E_i\,,F_j] \, - \, \delta_{i,j} \, \frac{\,e^{+\hbar\,T_i^+} - e^{-\hbar\,T_i^-} \,}{\,q_i^{+1} - q_i^{-1}\,}  \; = \;  0 } \;\, $
 is equivalent to the new one, that is
  $ \,\; \displaystyle{ \big[ E_i^\Phipicc , F_j^\Phipicc \big] \, - \, \delta_{i,j} \, \frac{\,e^{+\hbar\,T_{\Phipicc,i}^+} - e^{-\hbar\,T_{\Phipicc,i}^-} \,}{\,q_i^{+1} - q_i^{-1}\,}  \;\, = \;\,  0 } \;\, $.
 \vskip5pt
   Third, recall that all relations from  \eqref{eq: q-Serre (E)}  through  \eqref{eq: rel6}  can be written in terms of iterated  $ q $--supercommutators  (cf.\ \S \ref{rmk: q-Serre = iterated q-commutator}).  Now, for a single such  $ q $--supercommutator we have, for all  $ \, i, j \in I \, $  and  $ \, \nu \in \kh \, $,
  $$  \displaylines{
   {\big[ E_i \, , E_j \big]}_\nu  = \,  {\big[ \L_{\Phipicc,i}^{+1} \, E_i^\Phipicc , \L_{\Phipicc,j}^{+1} \, E_j^\Phipicc \big]}_\nu  = \,  \L_{\Phipicc,i}^{+1} \, E_i^\Phipicc \, \L_{\Phipicc,j}^{+1} \, E_j^\Phipicc - {(-1)}^{p(i)\,p(j)} \nu \, \L_{\Phipicc,j}^{+1} \, E_j^\Phipicc \, \L_{\Phipicc,i}^{+1} \, E_i^\Phipicc  =   \hfill  \cr
   \qquad   = \;  e^{-2^{-1} \hbar \sum_{g,k=1}^t \alpha_j(H_g) \, \phi_{gk} \, \alpha_i(H_k)} \, \L_{\Phipicc,i}^{+1} \, \L_{\Phipicc,j}^{+1} \, E_i^\Phipicc \, E_j^\Phipicc \, -   \hfill  \cr
   \hfill   - \, {(-1)}^{p(i)\,p(j)} \nu \, e^{-2^{-1} \hbar \sum_{g,k=1}^t \alpha_i(H_g) \, \phi_{gk} \, \alpha_j(H_k)} \, \L_{\Phipicc,j}^{+1} \, \L_{\Phipicc,i}^{+1} \, E_j^\Phipicc \, E_i^\Phipicc  \; =   \qquad  \cr
   \qquad   = \;  e^{-2^{-1} \hbar \sum_{g,k=1}^t \alpha_j(H_g) \, \phi_{gk} \, \alpha_i(H_k)} \, \L_{\Phipicc,i}^{+1} \, \L_{\Phipicc,j}^{+1} \, \times   \hfill  \cr
   \hfill   \times \, \Big( E_i^\Phipicc \, E_j^\Phipicc \, - \, {(-1)}^{p(i)\,p(j)} \nu \, e^{+\hbar \sum_{g,k=1}^t \alpha_j(H_g) \, \phi_{gk} \, \alpha_i(H_k)} \, E_j^\Phipicc \, E_i^\Phipicc \Big)  \; =   \qquad  \cr
   \hfill   = \;  e^{-2^{-1} \hbar \sum_{g,k=1}^t \alpha_j(H_g) \, \phi_{gk} \, \alpha_i(H_k)} \, \L_{\Phipicc,i}^{+1} \, \L_{\Phipicc,j}^{+1} \, {\big[ E_i^\Phipicc , E_j^\Phipicc \big]}_{\nu \, e^{+\hbar \sum_{g,k=1}^t \alpha_j(H_g) \, \phi_{gk} \, \alpha_i(H_k)}}  }  $$
 i.e., using notation  $ \; k_{i{}j}^\Phipicc := e^{+\hbar \sum_{g,k=1}^t \alpha_j(H_g) \, \phi_{gk} \, \alpha_i(H_k)} \; $,  \,we find
\begin{equation}  \label{eq: Phi change q-bracket x E}
  {\big[ E_i \, , E_j \big]}_\nu  \,\; = \;\,  {\big( k_{i{}j}^\Phipicc \big)}^{-1/2} \, \L_{\Phipicc,i}^{+1} \, \L_{\Phipicc,j}^{+1} \, {\big[ E_i^\Phipicc , E_j^\Phipicc \big]}_{\nu \, k_{i{}j}^\Phipicc}
\end{equation}
%
%
                                                  \par
   Similarly, a parallel calculation gives
\begin{equation}  \label{eq: Phi change q-bracket x F}
  {\big[ F_i \, , F_j \big]}_\nu  \,\; = \;\,  {\big[ F_i^\Phipicc , F_j^\Phipicc \big]}_{\nu \, k_{j{}i}^\Phipicc} \, \K_{\Phipicc,i}^{-1} \, \K_{\Phipicc,j}^{-1} \, {\big( k_{j{}i}^\Phipicc \big)}^{-1/2}
\end{equation}
                                                  \par
   More in general, a similar computation yields
\begin{equation}  \label{eq: prod-E => prod-EPhi}
  E^\Phipicc_{i_1} \cdots E^\Phipicc_{i_\ell}  \; = \;
  \left({\textstyle \prod\limits_{1 \leq r \leq s \leq \ell}}
  q^{-2^{-1} \sum_{g,k=1}^t \alpha_{i_s}(H_g) \, \phi_{gk} \, \alpha_{i_r}(H_k)}\right)
  \L_{\Phipicc,i_1}^{+1} \cdots \L_{\Phipicc,i_\ell}^{+1} \,
 E_{i_1} \cdots E_{i_\ell}
\end{equation}
 and then from this one can deduce formulas generalizing  \eqref{eq: Phi change q-bracket x E}
 that express any  $ q $--supercommutator  of products of the  $ E^\Phipicc_i $'s
 in terms of a similar supercommutator of products of the  $ E_i $'s.  A similar outcome holds with  ``$ F $''  replacing  ``$ E $''  throughout.
                                                  \par
   Applying these arguments   --- in particular, applying the like of
   \eqref{eq: Phi change q-bracket x E}  ---   one can deduce all relations from
   \eqref{eq: q-Serre (E)}  through  \eqref{eq: rel6}  for the  $ E_j^\Phipicc $'s,
   resp.\ the  $ F_j^\Phipicc $'s,  from the parallel relations for the  $ E_j $'s,
   resp.\ the  $ F_j $'s.  Hereafter we show that in a few cases.
 \vskip5pt
   For instance, the original relation  \eqref{eq: rel1}  in
   $ \, U_{\!P_A,\,\hbar}^{\R\,,\,p}(\lieg) \, $  reads
   $ \, {\big[ E_i \, , E_j \big]} = 0 \, $  because  $ \, k_{i,j} = 1 \, $  when
   $ \, a_{i,j} = 0 \, $  for the matrix  $ P_{\scriptscriptstyle A} $  (as it is symmetric).
   But then this relation is  \textsl{equivalent},  by  \eqref{eq: Phi change q-bracket x E},
   to  $ \, {\big[ E_i^\Phipicc , E_j^\Phipicc \big]}_{k_{i,j}^\Phipicc} = 0 \, $,  \,q.e.d.
                                                  \par
   Similarly, consider the case of the quantum Serre relation  \eqref{eq: q-Serre (E)}
   for  $ \, a_{i,j} = -2 \, $.  Direct computations   --- taking into account
   \eqref{eq: q-S = q-bra (2)},  \eqref{eq: Phi change q-bracket x E}  and  \eqref{def-P_Phi},
   that in particular imply  $ \, q_{i,j} \, k_{i,j}^\Phipicc = e^{+\hbar \, p_{i,j}^\Phipicc} =: q_{i,j}^\Phipicc \, $  ---   give us
  $$  \displaylines{
   E_{i,j}  \; = \;  \Big[ E_i \, , \big[ E_i \, , {[E_i \, , E_j]}_{q_{ij}}
   \big]_{q_{ii} \, q_{ij}} \Big]_{q_{ii}^{\,2} \, q_{ij}}  \; =   \hfill  \cr
   = \;  {\big( k_{i,j}^\Phipicc \big)}^{-1/2} \,
   \Big[ E_i \, , \big[ E_i \, , \L_{\Phipicc,i}^{+1} \, \L_{\Phipicc,j}^{+1} \,
   {[E_i^\Phipicc \, , E_j^\Phipicc]}_{q_{i,j}^\Phipicc} \big]_{q_{ii} \, q_{ij}} \Big]_{q_{ii}^{\,2} \, q_{ij}}  \; =   \hskip45pt   \cr
   = \;  {\big( k_{i,j}^\Phipicc \big)}^{-1} \,
   \Big[ E_i \, , \L_{\Phipicc,i}^{+2} \, \L_{\Phipicc,j}^{+1} \, \big[ E_i^\Phipicc \, ,
   {[E_i^\Phipicc \, , E_j^\Phipicc]}_{q_{i,j}^\Phipicc} \,\big]_{q_{ii}^\Phipicc \, q_{ij}^\Phipicc} \,\Big]_{q_{ii}^{\,2} \, q_{ij}}  \; =   \cr
   \hfill   = \;  {\big( k_{i,j}^\Phipicc \big)}^{-3/2} \,
   \L_{\Phipicc,i}^{+3} \, \L_{\Phipicc,j}^{+1} \, \Big[ E_i^\Phipicc \, , \big[ E_i^\Phipicc \, ,
   {[E_i^\Phipicc \, , E_j^\Phipicc]}_{q_{i,j}^\Phipicc} \,\big]_{q_{ii}^\Phipicc \, q_{ij}^\Phipicc}
   \,\Big]_{{(q_{ii}^\Phipicc)}^{\,2} \, q_{ij}^\Phipicc}  \; =   \qquad  \cr
   \hfill   = \;  {\big( k_{i,j}^\Phipicc \big)}^{-3/2} \,
   \L_{\Phipicc,i}^{+3} \, \L_{\Phipicc,j}^{+1} \, E_{i,j}^\Phipicc  }  $$
 from which we infer that the quantum Serre relation  $ \, E_{i,j} = 0 \, $  for the  $ E_\ell $'s
 is equivalent to the similar quantum Serre relation  $ \, E_{i,j}^\Phipicc = 0 \, $  for the  $ E_\ell^\Phipicc $'s.
                                                  \par
   The other quantum Serre relations are disposed of in a similar way.
                                                  \par
   Finally, as to the other relations we consider for instance the case of  \eqref{eq: rel2}.
   In case  $ \, \, P = P_{\scriptscriptstyle A} \, $  (which is symmetric, hence the $ k_{i,j} $'s  are one)
   relation  \eqref{eq: rel2}  reads just
   $ \; \Big[ {\big[ {[E_i\,,E_j]}_{\nu_j} , E_k \big]}_{\nu_{j+1}} , E_j \Big] \, = \, 0 \; $.
   Now, direct computations give
  $$  \displaylines{
   \Big[ \big[ {[E_i\,,E_j]}_{\nu_j} , E_k {\big]}_{\nu_{j+1}} , E_j \Big]  \,\; = \;\, q^{-1}_{ij} \,
   \Big[ {\big[ \L_{\Phi,i}^{+1} \, \L_{\Phi,j}^{+1} \, {\big[ E_i^\Phipicc , E_j^\Phipicc \big]}_{\nu_i k_{ij}} ,
   E_k \big]}_{\nu_{j+1}} , E_j \,\Big]  \,\; =   \hfill  \cr
   = \;\,  q^{-1}_{ij} \, q^{-1}_{ik} \, q^{-1}_{jk} \, \Big[ \L_{\Phi,k}^{+1} \, \L_{\Phi,i}^{+1} \,
   \L_{\Phi,j}^{+1} \, {\big[ {\big[ E_i^\Phipicc , E_j^\Phipicc \big]}_{\nu_j k_{ij}} ,
   E_k^\Phipicc \big]}_{\nu_{j+1} k_{ik} k_{jk}} , E_j \,\Big]  \,\; =   \qquad  \cr
   \hfill   = \;\,  q^{-2}_{ij} \, q^{-1}_{ik} \, q^{-1}_{jk} \, q_{kj}^{-1} \, q^{-1}_{jj} \,
   \L_{\Phi,k}^{+1} \, \L_{\Phi,i}^{+1} \, \L_{\Phi,j}^{+2} \, {\Big[ {\big[ {\big[ E_i^\Phipicc ,
   E_j^\Phipicc \big]}_{\nu_j k_{ij}} , E_k^\Phipicc \big]}_{\nu_{j+1} k_{ik}k_{jk}} ,
   E_j^\Phipicc \,\Big]}_{k_{ij} k_{jj} k_{kj}}  \, =  \cr
   \hfill   = \;\,  q^{-2}_{ij} \, q^{-1}_{ik} \, q^{-1}_{jk} \, q_{kj}^{-1} \, q^{-1}_{jj} \,
   \L_{\Phi,k}^{+1} \, \L_{\Phi,i}^{+1} \, \L_{\Phi,j}^{+2} \, {\Big[ {\big[ {\big[ E_i^\Phipicc ,
   E_j^\Phipicc \big]}_{\nu_j k_{ij}} , E_k^\Phipicc \big]}_{\nu_{j+1} k_{ik} k_{jk}} ,
   E_j^\Phipicc \,\Big]}_{k_{ij} k_{kj}}  }  $$
%
%
 from which we eventually see that relation  \eqref{eq: rel2}  for the  $ E_\ell $'s
 is equivalent to relation  \eqref{eq: rel2}  for the  $ E_\ell^\Phipicc $'s.
 Similarly one deals with relations  \eqref{eq: rel3}  through  \eqref{eq: rel6}.
                                                             \par
   As a different method, one can adopt the following.  For any  $ \, b, d \in \NN \, $  a special instance of  \eqref{eq: prod-E => prod-EPhi}  reads just
  $ \; {\big( E^\Phipicc_i \big)}^b \, E^\Phipicc_j \, {\big( E^\Phipicc_i \big)}^d  \, = \,
  q^{\,\epsilon(i,j)} \L_{\Phipicc,\,i}^{d+b} \, \L_{\Phipicc,\,j}^{+1} \, E_i^{\,b} \, E_j \, E_i^{\,d} \; $
 with
  $$  \displaylines{
   \quad   \epsilon(i,j)  \, := \,  -2^{-1} {{d+b} \choose 2} {\textstyle \sum\limits_{g,k=1}^t}
   \alpha_i(H_g) \, \phi_{gk} \, \alpha_i(H_k) \, - \, 2^{-1} \,
   d {\textstyle \sum\limits_{g,k=1}^t} \alpha_i(H_g) \, \phi_{gk} \, \alpha_j(H_k) \, -   \hfill  \cr
   \hfill   - \, 2^{-1} \, b {\textstyle \sum\limits_{g,k=1}^t} \alpha_j(H_g) \, \phi_{gk} \, \alpha_i(H_k)  \,
   = \,  2^{-1} \, (d-b) {\textstyle \sum\limits_{g,k=1}^t} \alpha_j(H_g) \, \phi_{gk} \, \alpha_i(H_k)  }  $$
%
 so that  $ \; q^{\,\epsilon(i,j)} = e^{\hbar \, 2^{-1} (d-b) \sum_{g,k=1}^t \alpha_j(H_g) \,
 \phi_{gk} \, \alpha_i(H_k)} = {\big( k_{ij}^\Phipicc \big)}^{(d-b)/2} \, $.  Reversing that identity, we get
  $ \; E_i^{\,b} \, E_j \, E_i^{\,d}  =  {\big( k_{ij}^\Phipicc \big)}^{(d-b)/2}
  \L_{\Phipicc,\,i}^{-(d+b)} \, \L_{\Phipicc,\,j}^{-1} \, {\big( E^\Phipicc_i \big)}^b \,
  E^\Phipicc_j \, {\big( E^\Phipicc_i \big)}^d \, $,
 \;which for the pair  $ \, (b\,,d\,) = \big( 1\!-\!a_{ij}\!-\!s\,,s \big) \, $  reads
\begin{equation*}  \label{eq: EiEjEi ===> EiEjEk(Phi)}
  E_i^{\,1-a_{ij}-s} \, E_j \, E_i^{\,s}  \; = \,  {\big( k_{ij}^\Phipicc \big)}^{(a_{ij}-1)/2} \,
  {\big( k_{ij}^\Phipicc \big)}^s \, \L_{\Phipicc,\,i}^{\,a_{ij}-1} \, \L_{\Phipicc,\,j}^{\,-1} \,
  {\big( E^\Phipicc_i \big)}^{1-a_{ij}-s} \, E^\Phipicc_j \, {\big( E^\Phipicc_i \big)}^s
\end{equation*}
 Finally, replacing this last identity in the left-hand side of  \eqref{eq: q-Serre (E)}  we get
  $$  \displaylines{
   \sum\limits_{s=0}^{1 - a_{ij}} {(-1)}^s {\left[ {{1 \! - \! a_{ij}} \atop s} \right]}_{\!q_i}
 k_{ij}^{\,s}
 \, E_i^{1 - a_{ij} - s} E_j E_i^s  \; = \;
   {\big( k_{ij}^\Phipicc \big)}^{(a_{ij}-1)/2} \, \L_{\Phipicc,\,i}^{\,a_{ij}-1} \,
   \L_{\Phipicc,\,j}^{\,-1} \,\; \times   \hfill  \cr
   \hfill   \times \, \sum\limits_{s=0}^{1 - a_{ij}} {(-1)}^s {\left[ {{1 \! - \! a_{ij}} \atop s} \right]}_{\!q_i} q_{ij}^{+s/2\,} q_{ji}^{-s/2} \, {\big( k_{ij}^\Phipicc \big)}^s \,
   {\big( E^\Phipicc_i \big)}^{1-a_{ij}-s} \, E^\Phipicc_j \, {\big( E^\Phipicc_i \big)}^s  \; =   \quad
   }  $$
  $$  \displaylines{
%
   \qquad   = \;  {\big( k_{ij}^\Phipicc \big)}^{(a_{ij}-1)/2} \,
   \L_{\Phipicc,\,i}^{\,a_{ij}-1} \, \L_{\Phipicc,\,j}^{\,-1} \,\; \times   \hfill  \cr
   \hfill   \times \, \sum\limits_{s=0}^{1 - a_{ij}} {(-1)}^s {\left[ {{1 \! - \! a_{ij}} \atop s} \right]}_{\!q_i}
   {\big( q^\Phipicc_{ij} \big)}^{+s/2\,} {\big( q^\Phipicc_{ji} \big)}^{-s/2} \,
   {\big( E^\Phipicc_i \big)}^{1-a_{ij}-s} \, E^\Phipicc_j \, {\big( E^\Phipicc_i \big)}^s  }  $$
 --- since  $ \, k_{ij}^\Phipicc = {\big( k_{ji}^\Phipicc \big)}^{-1} \, $  and
 $ \, q_{ij} \, k_{ij}^\Phipicc = q_{i,j}^\Phipicc \, $  ---   and then we notice that the last factor
 (i.e., the last line above) is just the left-hand side of the quantum Serre relation  \eqref{eq: q-Serre (E)}
 for the  $ E^\Phipicc_\ell $'s:  so we deduce that  \eqref{eq: q-Serre (E)}  for the  $ E_\ell $'s
 --- inside  $ U_{\!P_A,\,\hbar}^{\R\,,\,p}(\lieg) $  ---   is indeed equivalent to \eqref{eq: q-Serre (E)}
 for the  $ E^\Phipicc_\ell $'s  inside  $ {\big( U_{\!P_A,\,\hbar}^{\R\,,\,p}(\lieg) \big)}^{\F_\Phi} $,
 and we are done.
 \vskip3pt
   In the same way, replacing the  $ E_k $'s  with the  $ F_k $'s,  one disposes of all relations involving the
   $ F_k $'s  (alone) turning into relations involving the  $ F^\Phipicc_k $'s  (again alone).
 \vskip7pt
   To sum up, the previous analysis prove that the  $ E_i^\Phipicc $'s,  the  $ T $'s  and the
   $ F_i^\Phipicc $'s  altogether do generate the algebra  $ {\big( U_{\!P_A,\,\hbar}^{\R\,,\,p}(\lieg) \big)}^{\F_\Phi} $
   \textsl{and\/}  the original ideal of defining relations of  $ U_{\!P_A,\,\hbar}^{\R\,,\,p}(\lieg) $
   --- described in terms of the original generators ---   is also generated by the same relations  \textsl{but\/}
   written  \textsl{in terms of the  \textit{new}  generators and of the deformed multiparameter matrix}  $ P_\Phipicc \, $.
   In a nutshell, using  $ \big\{\, E_i^\Phipicc \, , T , F_i^\Phipicc \,\big|\, i \in I , T \in\lieh \,\big\} \, $
   as set of generators, we get for  $ {\big( U_{\!P_A,\,\hbar}^{\R\,,\,p}(\lieg) \big)}^{\F_\Phi} $
   nothing but the defining presentation of  $ \, U_{\!P_\Phi,\,\hbar}^{\R_{\scriptscriptstyle \Phi},\,p}(\lieg) \, $.
   In a different wording, this means exactly that there exists a well-defined isomorphism
   $ \, f_{\scriptscriptstyle P}^\Phipicc \, $  of topological superalgebras as given in the claim.
\end{proof}

\vskip9pt

   A direct consequence of  Proposition \ref{prop: deform-Yamane=FoMpQUESA}  is that the FoMpQUESA
   $ \, U_{\!P_\Phi,\,\hbar}^{\R_{\scriptscriptstyle \Phi},\,p}(\lieg) \, $
   inherits, via pull-back through the isomorphism  $ \, f^{\scriptscriptstyle \Phi}_{\scriptscriptstyle P} \, $,
   a full structure of Hopf superalgebra, more precisely one of QUESA.  In fact, by a ``reverse--engineering argument'',
   we see that this result extends to  \textsl{all\/}  FoMpQUESA's
   (hence the very use of the termi-{}\break{}nology ``FoMpQUESA'' is indeed correct!).
   This is the content of our next result:

\vskip11pt

\begin{theorem}  \label{thm: FoMpQUEA's}
 Every FoMpQUESA  $ \uRpPhg $  is indeed a (topological) Hopf superalgebra, whose Hopf structure
 is given on (super)algebra generators by the formulas
\begin{equation}  \label{eq: coprod_x_uRpPhg_2nd-time}
 \begin{aligned}
   \Delta \big(E_\ell\big)  \, = \,  E_\ell \otimes 1 \, + \, e^{+\hbar \, T_\ell^+} \otimes E_\ell,  \\
   \Delta\big(T\big)  \, = \,  T \otimes 1 \, + \, 1 \otimes T,   \hskip25pt  \\
   \Delta\big(F_\ell\big)  \, = \, F_\ell \otimes e^{-\hbar \, T_\ell^-} \, + \, 1 \otimes F_\ell,
 \end{aligned}   \qquad
\end{equation}
 (for all  $ \, T \in \lieh \, $,  $ \, \ell \in I $)  for the coproduct, and for the counit and antipode by
 \vskip-9pt
\begin{eqnarray}
   \epsilon\big(E_\ell\big) \, = \, 0  \;\;\; ,  \qquad  &  \;\quad
   \epsilon\big(T\big) \, = \, 0  \;\;\; ,  &  \;\qquad  \epsilon\big(F_\ell\big) \, = \, 0   \;\;\;\;\;\;\;   \label{eq: counit_x_uRpPhg}  \\
   \SS\big(E_\ell\big)  \, = \,  - e^{-\hbar \, T_\ell^+} E_\ell  \;\; ,  &  \;\;\;
   \SS\big(T\big)  \, = \,  - T   \;\; ,  &  \;\;\;  \SS\big(F_\ell\big)  \, = \,  - F_\ell \, e^{+\hbar \, T_\ell^-} \;\;\;   \label{eq: antipode_x_uRpPhg}
\end{eqnarray}
\end{theorem}

\begin{proof}
 By assumption, we have a fixed Cartan super-datum  $ \, (A,p) \, $,  as in  Definition \ref{def: Cartan-super-datum},  such that the symmetric part of  $ P $  is  $ \, P_s := 2^{-1} \big( P + P^{\,\scriptscriptstyle T} \big) = D\,A =: P_A \, $,  with  $ \, P = {\big(\, p_{i,j} \,\big)}_{i,j \in I} \, $  and  $ \, \R := \big(\, \lieh \, , \Pi \, , \Pi^\vee \,\big) \, $  a realization of  $ P $.
                                                                  \par
   By  Lemma \ref{lemma: split/straight-lifting},  there is a  \textsl{straight\/}  realization  $ \; \tilde{\R} := \big(\,\tilde{\lieh} \, , \tilde{\Pi} \, , {\tilde{\Pi}}^\vee \,\big) \, $   of the matrix  $ P $  and an epimorphism of realizations  $ \; \underline{\pi} : \tilde{\R} \relbar\joinrel\twoheadrightarrow \R \; $.  Up to possibly enlarging  $ \tilde{\lieh} \, $,  we can assume that  $ \, t := \text{rk}\big(\tilde{\lieh}\big) \geq 3n - \text{rk}\big(P_A\big) \, $.  Then, by  Proposition \ref{prop: realiz=twist-standard}\textit{(a)}   --- with  $ \, P' = P_A := DA \, $  ---   we get that there exists some  $ \, \Phi \in \lieso_t\big(\kh\big) \, $  such that  $ \, P_A = P_{\scriptscriptstyle \Phi} \, $  while  $ \, \R_A := \tilde{\R}_{\,\scriptscriptstyle \Phi} \, $  is a straight realization of  $ \, P_A = P_{\scriptscriptstyle \Phi} \, $.  Hence, conversely, we have also  $ \, P = {\big( P_A \big)}_{\scriptscriptstyle -\Phi} \, $  and  $ \, \tilde{\R} = {\big( \R_A \big)}_{\scriptscriptstyle -\Phi} \, $.  Then by  Proposition \ref{prop: deform-Yamane=FoMpQUESA}  there exists an isomorphism of topological superalgebras
  $$  f^{\scriptscriptstyle -\Phi}_{\scriptscriptstyle P_A} \; :
   \; U_{\!P,\,\hbar}^{\tilde{R},\,p}(\lieg) = U_{\!{(P_A)}_{-\Phi},\,\hbar}^{{(\R_A)}_{\scriptscriptstyle -\Phi},\,p}(\lieg)  \,\; {\buildrel \cong \over
{\lhook\joinrel\relbar\joinrel\relbar\joinrel\relbar\joinrel\relbar\joinrel\twoheadrightarrow}} \;\,   {\big(\, U_{\!P_A,\,\hbar}^{\R_A,\,p}(\lieg)\big)}^{\F_{-\Phi}}  $$
 which is described by explicit formulas as in  Proposition \ref{prop: deform-Yamane=FoMpQUESA}    --- up to paying attention to switching  $ \Phi $  to  $ -\Phi \, $.  Eventually, we can pull back onto  $ U_{\!P,\,\hbar}^{\tilde{R},\,p}(\lieg) $   --- via  $ f^{\scriptscriptstyle -\Phi}_{\scriptscriptstyle P_A} $  ---   the Hopf structure of  $ {\big(\, U_{\!P_A,\,\hbar}^{\R_A,\,p}(\lieg)\big)}^{\F_{-\Phi}} \, $:  \, then, when tracking the whole construction, the explicit description of this Hopf structure on the generators of  $ U_{\!P,\,\hbar}^{\tilde{R},\,p}(\lieg) $  is given exactly by the formulas in the claim.
                                                                   \par
   Finally,  Proposition \ref{prop: functor_R->uRPhg}  gives an epimorphism  $ \; U_{\underline{\pi}} : U^{\,\tilde{\R},\,p}_{\!P,\,\hbar}(\lieg) \relbar\joinrel\relbar\joinrel\relbar\joinrel\twoheadrightarrow \uRpPhg \; $  of unital topological  $ \, \kh $--algebras  that extends the epimorphism  $ \, \pi : \tilde{\lieh} \relbar\joinrel\relbar\joinrel\twoheadrightarrow \lieh \, $
 of  $ \, \kh $--modules  given by  $ \underline{\pi} \, $;  \,moreover,  $ \, \Ker\!\big(\, \underline{\pi} \,\big) \, $  is the two--sided ideal in  $ U^{\,\tilde{\R},\,p}_{\!P,\,\hbar}(\lieg) $  generated by  $ \Ker(\pi) \, $,  \,and the latter is central in  $ U^{\,\tilde{\R},\,p}_{\!P,\,\hbar}(\lieg) \, $.  But then, since  $ \Ker(\pi) $  is made of primitive elements,  $ \, \Ker\!\big(\, \underline{\pi} \,\big) \, $  is automatically a Hopf ideal: therefore, the quotient algebra  $ \, \uRpPhg = U^{\,\tilde{\R},\,p}_{\!P,\,\hbar}(\lieg) \Big/ \Ker\!\big(\, \underline{\pi} \,\big) \, $  inherits via the algebra epimorphism  $ U_{\underline{\pi}} \, $  the whole Hopf structure of  $ U^{\,\tilde{\R},\,p}_{\!P,\,\hbar}(\lieg) \, $,  \,hence we are done.
\end{proof}

\vskip11pt

\begin{cor}  \label{cor: Hopf-struct_Cartan-&-Borel}
%
%
%
 The Cartan subalgebra  $ U_\hbar(\lieh) $  and the Borel subsuperalgebras\break  $ \uRpPhbpm $  are actually (topological)  \textsl{Hopf}  subsuperalgebras of  $ \, \uRpPhg \, $,  their Hopf structure being described again via formulas  \eqref{eq: coprod_x_uRpPhg_2nd-time},  \eqref{eq: counit_x_uRpPhg}  and  \eqref{eq: antipode_x_uRpPhg}.   \hfill  \qed
\end{cor}

\vskip7pt

   The next result is a straightforward ``upgrading''  Proposition \ref{prop: functor_R->uRPhg}:

\vskip11pt

\begin{prop}  \label{prop: central-Hopf-extens_FoMpQUEAs}
 Let  $ \, \underline{\phi} : \R' \!\relbar\joinrel\relbar\joinrel\longrightarrow\! \R'' \, $
 be a morphism between realizations of a same matrix  $ \, P \in M_n\big( \kh \big) \, $.
 Then the morphism of unital topological  $ \, \kh $--superalgebras
 $ \, U_{\underline{\phi}} : U^{\,\R',\,p}_{\!P,\,\hbar}(\lieg) \relbar\joinrel\relbar\joinrel\longrightarrow U^{\,\R'',\,p}_{\!P,\,\hbar}(\lieg) \, $
 granted by  Proposition \ref{prop: functor_R->uRPhg}  is indeed a morphism of (unital topological)  \textsl{Hopf}  $ \, \kh $--superalgebras.
                                                                          \par
   If we set  $ \; \liek := \Ker(\phi) \, $, then  $ \, U_\hbar(\liek) \, $  is a unital,  $ \hbar $--adically  complete  $ \, \kh $--subsuper-\break{}algebra  of  $ U^{\,\R',\,p}_{\!P,\,\hbar}(\lieg) $  which is a central Hopf subsuperalgebra, isomorphic to a quantum Cartan (in the sense of  Definition \ref{def: FoMpQUESA}\textit{(c)\/}),  and  $ \; \Ker(U_{\underline{\phi}}) = U^{\,\R',p}_{\!P,\hbar}(\lieg) \, {U_\hbar(\liek)}^+ \, $   --- where  $ \, {U_\hbar(\liek)}^+ \, $  is the augmentation ideal of  $ \, U_\hbar(\liek) \, $.  In particular, if  $ \, U_{\underline{\phi}} $  is an  \textsl{epi}morphism,  then  $ \; U^{\,\R''}_{\!P,\hbar}(\lieg) \, \cong \, U^{\,\R',\,p}_{\!P,\,\hbar}(\lieg) \Big/ U^{\,\R',\,p}_{\!P,\,\hbar}(\lieg) \, {U_\hbar(\liek)}^+ \; $.   \hfill  $ \square $
\end{prop}

\vskip9pt

   We conclude this subsection proving that every FoMpQUESA is indeed a QUESA in the sense of  \S \ref{QUESA's}:  later on, in  \S \ref{sec: FoMpQUESA's-vs-MpLSbA's},  we will study in detail their semiclassical limit.

\vskip13pt

\begin{theorem}  \label{thm: FoMpQUESAs-are-QUESAs}
 Every FoMpQUESA is a QUESA.
\end{theorem}

\begin{proof}
 Let  $ \, U_\hbar := \uRpPhg \, $  be any given FoMpQUESA, and let  $ \; U_0 := U_\hbar \Big/ \hbar\,U_\hbar \; $  denote its semiclassical limit; also, for any  $ \, X  \in U_\hbar \, $  we set  $ \, \overline{X} := X \; \big(\,\text{mod} \; \hbar\,U_\hbar \,\big) \, $  for its coset modulo  $ \, \hbar\,U_\hbar \, $  in  $ U_0 \, $.
                                                        \par
   First of all,  $ U_\hbar $  is topologically free (as a  $ \kh $--module).  Indeed, assume first that  $ \,\R $  is straight and  $ \lieh $  is ``large enough'', namely  $ \, \text{rk}(\lieh) \geq 3\,n - \text{rk}(DA) \ $:  \,then   --- see the proof of  Theorem \ref{thm: FoMpQUEA's}  ---   $ \uRpPhg $  is isomorphic (as a topological algebra) to some Yamane's QUESAs, and the latter is known to be topologically free, by  \cite{Ya1},  hence the same holds true for  $ \uRpPhg \, $.  The general case (for any  $ \lieh $  and any  $ \,\R \, $)  then is obtained like in the proof of  Theorem \ref{thm: FoMpQUEA's}  by reducing to that of some straight realization with large enough Cartan (sub)algebra (for more details, see the proof of the non-super case in  \cite{GG3}, Theorem 6.1.4   ---   the present case is similar).
                                                                    \par
   Second, as  $ U_\hbar $  is generated by the  $ E_i $'s,  $ F_i $'s  and  $ T $'s  ($ \, i \in I \, $,  $ \, T \in \lieh \, $),  \,it follows that  $ U_0 $  is generated by the  $ \overline{E_i} $'s,  $ \overline{F_i} $'s  and  $ T $'s.  Moreover, the formulas for the coproduct in  $ U_\hbar $  imply that the latter generators are  \textit{primitive\/}  in  $ U_0 \, $.  Therefore, the Hopf algebra  $ U_0 $  is generated by its subspace  $ P(U_0) $  of primitive elements, hence   --- by Milnor--Moore Theorem for Hopf superalgebras  cf.~\cite[Proposition 3.2]{Ko}  ---   we conclude that  $ P(U_0) $  is a Lie superalgebra and  $ U_0 $  is its universal enveloping superalgebra, hence we are done.
\end{proof}

\vskip9pt

   We finish with a typical ``triangular decomposition result'' for our FoMpQUESA's:

\vskip15pt

\begin{theorem}  \label{thm: triang-decomp}  {\ }
 \vskip3pt
   Every FoMpQUESA  $ \uRpPhg $  admits  \textsl{triangular decompositions}  such as
  $$  \uRpPhnp \otimes \uRpPhh \otimes \uRpPhnm  \, \cong \,  \uRpPhg  \, \cong \,  \uRpPhnm \otimes \uRpPhh \otimes \uRpPhnp  $$
 and similarly for Borel sub-superalgebras we have
  $$  \uRpPhnpm \otimes \uRpPhh  \,\; \cong \;\,  \uRpPhbpm  \, \cong \,  \uRpPhh \otimes \uRpPhnpm  $$
\end{theorem}

\begin{proof}
 The claim follows from a standard, fairly ``classical'' argument, which is detailed in the  \textsl{Third Proof\/}  in  \cite[\S 4.4.9]{GG3};  that one is written there for the non-super setup, but it applies  \textit{verbatim\/}  to the present context as well.
\end{proof}

\bigskip
 \medskip

\section{Deformations of FoMpQUESA's}  \label{sec: deform-FoMpQUESA's}
 \vskip11pt
   After introducing formal MpQUEAs in the previous section, now we go and study their deformations, either by twist or by  polar--$ 2 $--cocycle  (both of ``toral type'', say).

\vskip13pt

\subsection{Deformations of FoMpQUESAs by toral twists}  \label{subsec: tw-def_FoMpQUESAs}
 {\ }
 \vskip9pt
   We discuss now suitable twist deformations (of ``toral type'') of FoMpQUESAs,
   proving that they are again FoMpQUESAs.
   In fact, this is but a generalization of the argument that we already used in
   \S \ref{subsec: FoMpQUESA's}  to  \textsl{define\/}  our FoMpQUESAs,
   as they are indeed introduced as deformations by toral twists of Yamane's QUESAs.

\vskip9pt

\begin{free text}  \label{twist-uPhgd}
 {\bf Toral twist deformations of  $ \uRpPhg \, $.}
 We fix a Cartan super-datum  $ \, (A\,,p) \, $  as in  Definition  \ref{def: Cartan-super-datum},
 a multiparameter matrix  $ \, P := {\big(\, p_{i,j} \big)}_{i, j \in I} \in M_n\big(\kh\big) \, $
 associated with  $ A \, $, and a realization  $ \, \R \, := \, \big(\, \lieh \, , \Pi \, , \Pi^\vee \,\big) \, $  of it,
 as in  \S \ref{sec: Cartan-data_realiz's}.  In particular, we have  $ \, d_i := p_{ii}/2 \, $  ($ \, i \in I \, $),
 and  $ \lieh $  is a free  $ \kh $--module  of finite rank  $ \, t := \rk(\lieh) \, $.  We fix in
 $ \lieh $  any  $ \kh $--basis  $ \, {\big\{ H_g \big\}}_{g \in \G} \, $,  where  $ \, |\G| = \rk(\lieh) = t \, $.
                                                              \par
   Now let  $ \uRpPhg $  be the FoMpQUESA given in  Definition \S \ref{def: FoMpQUESA}.
   We pick  $ \; \Phi = \big( \phi_{gk} \big)_{g, k \in \G} \in \lieso_t\big(\kh\big)  \; $,  \,and set
  $$
  \JJ_\Phi  \; := \;  {\textstyle \sum\limits_{g,k=1}^t} \phi_{gk} \, H_g \otimes H_k  \; \in \;
  \lieh \otimes \lieh  \; \subseteq \; U^\R_{\!P,\hskip0,7pt\hbar}(\lieh) \otimes U^\R_{\!P,\hskip0,7pt\hbar}(\lieh)\;.
  $$
 By direct check, we see that the element
\begin{equation}  \label{eq: Resh-twist_F-uPhgd}
  \F_\Phi  \,\; := \;\,
e^{\,\hbar \, 2^{-1} \JJ_\Phi}
 \,\; = \;\,  \exp\Big(\hskip1pt \hbar \,  \, 2^{-1} \,
 {\textstyle \sum_{g,k=1}^t} \phi_{gk} \, H_g \otimes H_k \Big)
\end{equation}
 in  $ \, U^{\R,p}_{\!P,\hskip0,7pt\hbar}(\lieh) \,\widehat{\otimes}\, U^{\R,p}_{\!P,\hskip0,7pt\hbar}(\lieh) \, $  is actually a  {\sl twist\/}  for  $ \uRpPhg $  in the sense of  \S \ref{defs_Hopf-algs}.  Using it, we construct a new (topological) Hopf algebra  $ \, {\big(\, \uRpPhg \big)}^{\F_\Phi} \, $,  isomorphic to  $ \uRpPhg $  as an algebra but with a new, twisted coalgebra structure, as in  \S \ref{defs_Hopf-algs}.
 \vskip5pt
   Our key result now is that this new Hopf algebra  $ \, {\big(\, \uRpPhg \big)}^{\F_\Phi} \, $  is yet another FoMpQUESA, and even more we know exactly which one it is:
\end{free text}

\vskip11pt

\begin{theorem}  \label{thm: twist-uRpPhg=new-uRpPhg}
 There exists an isomomorphism of topological Hopf algebras
  $$  f^{\scriptscriptstyle \Phi}_{\scriptscriptstyle P} \; :
   \; U_{\!P_\Phi,\,\hbar}^{\R_{\scriptscriptstyle \Phi},\,p}(\lieg)  \,\; {\buildrel \cong \over
   {\lhook\joinrel\relbar\joinrel\relbar\joinrel\relbar\joinrel\twoheadrightarrow}} \;\,
   {\big(\, \uRpPhg\big)}^{\F_\Phi}  $$
 given by
  $ \; E_\ell \, \mapsto \, E^{\scriptscriptstyle \Phi}_\ell := \, \L_{\Phi, \ell}^{-1} \, E_\ell \, , \; T \, \mapsto \, T \, $,  $ \; F_\ell \, \mapsto \, F^{\scriptscriptstyle \Phi}_\ell := \, F_\ell \, \K_{\Phi, \ell}^{+1}  \;\;\;  (\, \forall \; \ell \in I \, , \, T \in \lieh \,) \, $,
 \;with
 $ \;\; \L_{\Phi, \ell} := e^{+ \hbar \, 2^{-1} \sum_{g,k=1}^t \alpha_\ell(H_g) \, \phi_{gk} H_k} \; $,
 $ \;\; \K_{\Phi,\ell} := e^{+ \hbar \, 2^{-1} \sum_{g,k=1}^t \alpha_\ell(H_g) \, \phi_{kg} H_k} \; \big(\! = \L_{\Phi, \ell}^{\,-1} \,\big) \; $.
 \vskip3pt
   In particular, the class of all FoMpQUESAs of any fixed Cartan type and of fixed rank
%
%
 is stable by toral twist deformations.  Moreover, inside it the subclass of all such
 FoMpQUESAs associated with  \textsl{straight},  resp.\  \textsl{small},  realizations is stable as well.
 \vskip3pt
   Similar, parallel statements hold true for the Borel FoMpQUEAs, i.e.\ there exist isomorphisms
 $ \; f^{\scriptscriptstyle \Phi}_{\!\scriptscriptstyle P,\pm} :
   U_{\!P_\Phi,\,\hbar}^{\R_{\scriptscriptstyle \Phi},\,p}(\lieb_\pm) \, {\buildrel \cong \over
   {\lhook\joinrel\relbar\joinrel\relbar\joinrel\twoheadrightarrow}} \,
   {\big(\, U_{\!P,\,\hbar}^{\R,\,p}(\lieb_\pm)\big)}^{\F_\Phi} \; $
 given by formulas as above.
\end{theorem}

\begin{proof}
 The proof follows closely in the footprints of  Proposition \ref{prop: deform-Yamane=FoMpQUESA},  so we can be somewhat quicker now.  First of all,  $ {\big(\, \uRpPhg\big)}^{\F_\Phi} $  coincides with  $ \uRpPhg $  as an algebra, so we can take the generators of the latter as generators of the former too.
                                                        \par
%
%
   Direct calculation yields explicit formulas for the new coproduct  $ \Delta^{\scriptscriptstyle \F_\Phi} $,
%
%
 namely
  $$  \displaylines{
   \qquad   \Delta^{\scriptscriptstyle \F_\Phi}\big(E_\ell\big)  \; = \;
 E_\ell \otimes \L_{\Phi, \ell}^{+1} \, + \, e^{+\hbar \, T^+_\ell} \K_{\Phi, \ell}^{+1} \otimes
E_\ell   \quad \qquad  \big(\, \forall \;\; \ell \in I \,\big)  \cr
   \qquad \qquad \quad   \Delta^{\scriptscriptstyle \F_\Phi}\big(T\big)  \; = \;  T \otimes 1 \, + \, 1 \otimes T
   \quad \quad \qquad \qquad  \big(\, \forall \;\; T \in \lieh \,\big)  \cr
   \qquad   \Delta^{\scriptscriptstyle \F_\Phi}\big(F_\ell\big)  \; = \;
   F_\ell \otimes \L_{\Phi, \ell}^{-1} \, e^{-\hbar \, T^-_\ell } + \, \K_{\Phi, \ell}^{-1} \otimes F_\ell
   \quad \qquad  \big(\, \forall \;\; \ell \in I \,\big)  }  $$
with
%
%
  $ \,\; \L_{\Phi, \ell} := e^{+ \hbar \, 2^{-1} \sum_{g,k=1}^t \alpha_\ell(H_g) \, \phi_{gk} H_k} \; $,
  $ \,\; \K_{\Phi,\ell} := e^{+ \hbar \, 2^{-1} \sum_{g,k=1}^t \alpha_\ell(H_g) \, \phi_{kg} H_k} \; \big(\! = \L_{\Phi, \ell}^{\,-1} \,\big) $  \; for  $ \ell \!\in\! I $.
 Similarly, the ``twisted'' antipode  $ \SS^{\scriptscriptstyle \F_\Phi} $  and
 the counit  $ \epsilon^{\scriptscriptstyle \F_\Phi} \! := \epsilon $   are given by
  $$  \begin{matrix}
   \qquad   \SS^{\scriptscriptstyle \F_\Phi}\big(E_\ell\big)  \, = \,
   - e^{-\hbar \, T^+_\ell} \K_{\Phi, \ell}^{-1} \, E_\ell \, \L_{\Phi, \ell}^{-1}  \quad ,
   &  \qquad  \epsilon^{\scriptscriptstyle \F_\Phi}\big(E_\ell\big) \, = \, 0   \qquad \qquad  \big(\, \forall \;\; \ell \in I \,\big) \quad , \\
   \qquad
 \SS^{\scriptscriptstyle \F_\Phi}\big(T\big)  \, = \,  -T  \quad ,  \phantom{\Big|^|}
  &  \qquad
 \epsilon^{\scriptscriptstyle \F_\Phi}\big(T\big) \, = \, 0   \qquad \qquad  \big(\, \forall \;\; T \in \lieh \,\big)\quad ,
 \\
   \qquad   \SS^{\scriptscriptstyle \F_\Phi}\big(F_\ell\big)  \, = \,
   - \K_{\Phi, \ell}^{+1} \, F_\ell \, \L_{\Phi, \ell}^{+1} \, e^{+\hbar \, T^-_\ell}  \quad ,
   &  \qquad  \epsilon^{\scriptscriptstyle \F_\Phi}\big(F_\ell\big) \, = \, 0   \qquad \qquad  \big(\, \forall \;\; \ell \in I \,\big)\quad .
\end{matrix}  $$
   \indent   Tiding up, from the explicit description of the coproduct  $ \Delta^{\scriptscriptstyle \F_\Phi} \, $,  it follows that  $ \, {\big(\, \uRpPhg \big)}^{\F_\Phi} \, $  is generated by group-likes and skew-primitive elements; hence, it is a  \textsl{pointed\/}  Hopf superalgebra.  Moreover, both Hopf superalgebras  $ \,  \uRpPhg  \, $  and  $ \, {\big(\, \uRpPhg \big)}^{\F_\Phi} \, $  have the same coradical and the same space of skew-primitive elements.  As the coproduct is changed by twist, one sees that the skew-primitive generators of  $ \, \uRpPhg  \, $,  which are  $ (1,g) $--primitive  or $ (g,1) $--primitive  with respect to  $ \Delta \, $,  for some  $ \, g \in G\big(\, \uRpPhg \big) \, $,  become  $ (h,k) $--primitive  for $ \Delta^{\scriptscriptstyle \F_\Phi} \, $.  Looking at the coradical filtration, and the associated graded Hopf superalgebra,  one finds from that set of generators some new  $ (1,\ell) $--primitive  or  $ (\ell,1) $--primitives  for  $ {\big(\, \uRpPhg \big)}^{\F_\Phi} $.  This leads to devise (new)  \textit{twisted generators\/}  and a corresponding presentation for  $ {\big(\, \uRpPhg \big)}^{\F_\Phi} $,  which yields a Hopf superalgebra isormorphism between  $ {\big(\, \uRpPhg \big)}^{\F_\Phi} $  and a new FoMpQUESA with suitable multiparameter matrix and realization.
                                                                   \par
   Motivated by the above analysis, we introduce now in  $ {\big(\, \uRpPhg \big)}^{\F_\phi} $  the ``twisted'' generators
  $ \; E^{\scriptscriptstyle \Phi}_\ell := \, \L_{\Phi, \ell}^{-1} \, E_\ell \; $,
 $ \; F^{\scriptscriptstyle \Phi}_\ell := \, F_\ell \, \K_{\Phi, \ell}^{+1} \; $
  (for all  $ \, \ell \in I \, $);  clearly, these along with all the  $ T $'s  in  $ \lieh $  still generate  $ \,  {\big(\, \uRpPhg \big)}^{\F_\Phi} = \uRpPhg \, $.  Moreover, we introduce the twisted ``distinguished toral elements'' (or ``coroots'') that were already defined in  \eqref{eq: T-phi},  i.e.\
 $ \; T^\pm_{{\scriptscriptstyle \Phi},\ell} := \, T^\pm_\ell \, \pm {\textstyle \sum\limits_{g,k=1}^t} \alpha_\ell(H_g) \, \phi_{kg} \, H_k \; $.
                                                                    \par
   Finally, again from  \S \ref{twist-deform.'s x mpmatr.'s & realiz.'s},  we recall also
  $ \, P_{\scriptscriptstyle \Phi} := {\big(\, p^{\scriptscriptstyle \Phi}_{i,j} \big)}_{i, j \in I} \, $  and
  $ \, \R_\Phi := \big(\, \lieh \, , \Pi \, , \Pi^\vee_{\scriptscriptstyle \Phi} \,\big) \, $,
  \,with the latter being a realization of the former indeed.
 \vskip3pt
   As a second step, the defining relations in the superalgebra  $ \,{\big(\, \uRpPhg \big)}^{\F_\Phi} \, $  give new relations between twisted generators.  Namely, by straightforward computations   --- for instance using that  $ \; \K_{\Phi,j} \, E_i = e^{\hbar\, 2^{-1} \, \sum_{g,k=1}^t \alpha_j(H_g) \, \phi_{kg} \alpha_i(H_k)} E_i \, \K_{\Phi,j} \; $  and that  $ \; e^{+\hbar \, T^{\pm}_{{\scriptscriptstyle \Phi},i}} = e^{+\hbar \, T_i^\pm} \big(\, \K_{\Phi,i} \, \L_{\Phi,i}^{-1} \,\big)^{\pm 1} \; $  ---   one proves that inside  $ \,{\big(\, \uRpPhg \big)}^{\F_\Phi} \, $  the following identities hold true (for all $ \, T , T' , T'' \in \lieh \, $,  $ \, i \, , j \in I \, $):
  $$  \begin{aligned}
   T \, E^{\scriptscriptstyle \Phi}_j \, - \, E^{\scriptscriptstyle \Phi}_j \, T  \; = \;  +\alpha_j(T) \, E^{\scriptscriptstyle \Phi}_j  \quad ,
   \qquad  T \, F^{\scriptscriptstyle \Phi}_j \, - \, F^{\scriptscriptstyle \Phi}_j \, T  \; = \;  -\alpha_j(T) \, F^{\scriptscriptstyle \Phi}_j   \hskip43pt  \\
   T' \, T''  \; = \;  T'' \, T'  \quad ,  \quad \qquad
   \big[ E^{\scriptscriptstyle \Phi}_i , F^{\scriptscriptstyle \Phi}_j \big]  \; = \; \delta_{i,j} \, {{\; e^{+\hbar \, T^+_{{\scriptscriptstyle \Phi},i}} \, -
   \, e^{-\hbar \, T^-_{{\scriptscriptstyle \Phi},i}} \;} \over {\; q_i^{+1} - \, q_i^{-1} \;}},   \hskip65pt  \\
   \sum\limits_{s=0}^{1-a_{ij}} {(-1)}^s {\left[ {{1-a_{ij}} \atop s}
\right]}_{\!q_i} \!
 {\big( k_{ij}^{\scriptscriptstyle \Phi} \big)}^{\!s}
 {\big( E^{\scriptscriptstyle \Phi}_i \big)}^{1-a_{ij}-s} E^{\scriptscriptstyle \Phi}_j {\big( E^{\scriptscriptstyle \Phi}_i \big)}^s
\; = \;  0   \quad   \Bigg( {{i \neq j} \atop {\,p(i) \!=\! \zero}} \Bigg)  \\
   \sum\limits_{s=0}^{1-a_{ij}} {(-1)}^s {\left[ {{1-a_{ij}} \atop s}
\right]}_{\!q_i} \!
 {\big( k_{ji}^{\scriptscriptstyle \Phi} \big)}^{\!s}
 {\big( F^{\scriptscriptstyle \Phi}_i \big)}^{1-a_{ij}-s} F^{\scriptscriptstyle \Phi}_j {\big( F^{\scriptscriptstyle \Phi}_i \big)}^s
\; = \;  0   \quad   \Bigg( {{i \neq j} \atop {\,p(i) \!=\! \zero}} \Bigg)
 \end{aligned}  $$
 --- with  $ \; q^{\scriptscriptstyle \Phi}_{i,j} := e^{\hbar \, p^{\scriptscriptstyle \Phi}_{i,j}} \; (\, i , j \in I \,) \, $,  so that  $ \, q^{\scriptscriptstyle \Phi}_{i,i} = e^{\hbar \, p^{\scriptscriptstyle \Phi}_{i,i}} = e^{\hbar \, p_{i,i}} = e^{\hbar \, 2 d_i} = q_i^{\,2} \, $,  \,as well as  $ \; k^{\scriptscriptstyle \Phi}_{i,j} := q^{\scriptscriptstyle \Phi}_{i,j} \, q^{\scriptscriptstyle \Phi}_{j,i} = e^{\hbar \, (p^{\scriptscriptstyle \Phi}_{i,j} - p^{\scriptscriptstyle \Phi}_{j,i})} \; (\, i , j \in I \,) \, $  ---   and also
  $$  \big[ E^{\scriptscriptstyle \Phi}_i , E^{\scriptscriptstyle \Phi}_j \big]_{k^{\scriptscriptstyle \Phi}_{ij}} \, = \, 0  \;\; ,  \qquad
      \big[ F^{\scriptscriptstyle \Phi}_i , F^{\scriptscriptstyle \Phi}_j \big]_{k^{\scriptscriptstyle \Phi}_{ji}} \, = \, 0
   \qquad \qquad   \text{if} \quad  a_{i,j} = 0  $$
 as well as all other relations listed in  Definition \ref{def: FoMpQUESA},  namely  \eqref{eq: rel2}  through  \eqref{eq: rel6}   --- including their ``counterparts'' with the  $ F_\ell $'s  replacing the  $ E_\ell $'s  throughout.  Indeed, apart from the first two lines (which follow from direct approach), all the other formulas follows from iterated applications of  \eqref{eq: Phi change q-bracket x E}   --- that is
  $ \; {\big[ E_i \, , E_j \big]}_\nu  = \,  {\big( k_{i,j}^\Phipicc \big)}^{-1/2} \, \L_{\Phipicc,i}^{+1} \, \L_{\Phipicc,j}^{+1} \, {\big[ E_i^\Phipicc , E_j^\Phipicc \big]}_{\nu \, k_{i{}j}^\Phipicc} \; $
 where the left-hand side now is to be read inside  $ \uRpPhg $  ---   and  \eqref{eq: Phi change q-bracket x F}  alike, just sticking to the very same argument that was followed in the proof of  Proposition \ref{prop: deform-Yamane=FoMpQUESA}.
 \vskip5pt
   Third, the Hopf operators on the ``twisted'' generators read  (for  $ \, \ell \in I \, $,  $ T \in \lieh \, $)
  $$  \begin{matrix}
   \Delta^{\scriptscriptstyle \F_\Phi}\big(E^{\scriptscriptstyle \Phi}_\ell\big)
\, = \,  E^{\scriptscriptstyle \Phi}_\ell \otimes 1 \, + \, e^{+\hbar\,T^+_{{\scriptscriptstyle \Phi},\ell}} \otimes E^{\scriptscriptstyle \Phi}_\ell  \; ,   &   \hskip1pt
   \SS^{\scriptscriptstyle \F_\Phi}\big(E^{\scriptscriptstyle \Phi}_\ell\big)  \, = \,  - e^{-\hbar\,T^+_{{\scriptscriptstyle \Phi},\ell}} E^{\scriptscriptstyle \Phi}_\ell  \; ,   &   \hskip1pt
  \epsilon^{\scriptscriptstyle \F_\Phi}\big(E^{\scriptscriptstyle \Phi}_\ell\big) \, = \, 0  \\
   \Delta^{\scriptscriptstyle \F_\Phi}\big(T\big)  \, = \,  T \otimes 1 \, + \, 1 \otimes T  \; ,   &   \hskip1pt
   \SS^{\scriptscriptstyle \F_\Phi}\big(T\big)  \, = \,  - T  \; ,   &   \hskip1pt
   \epsilon^{\scriptscriptstyle \F_\Phi}\big(T\big)  \, = \,  0  \\
   \Delta^{\scriptscriptstyle \F_\Phi}\big(F^{\scriptscriptstyle \Phi}_\ell\big)  \, = \,  F^{\scriptscriptstyle \Phi}_\ell \otimes e^{-\hbar\,T^-_{{\scriptscriptstyle \Phi},\ell}} \, + \, 1 \otimes F^{\scriptscriptstyle \Phi}_\ell  \; ,   &   \hskip1pt
   \SS^{\scriptscriptstyle \F_\Phi}\big(F^{\scriptscriptstyle \Phi}_\ell\big)  \, = \, - F^{\scriptscriptstyle \Phi}_\ell \, e^{+\hbar \, T^-_{{\scriptscriptstyle \Phi},\ell}}  \; ,   &   \hskip1pt
   \epsilon^{\scriptscriptstyle \F_\Phi}\big(F^{\scriptscriptstyle \Phi}_\ell\big)  \, = \,  0
 \end{matrix}  $$
 \vskip5pt
   At last, it is now immediate to check that mapping
  $$  E_\ell \, \mapsto \, E^{\scriptscriptstyle \Phi}_\ell := \, \L_{\Phi, \ell}^{-1} \, E_\ell \; ,  \quad  T \, \mapsto \, T \; ,  \quad  F_\ell \, \mapsto \, F^{\scriptscriptstyle \Phi}_\ell := \, F_\ell \, \K_{\Phi, \ell}^{+1}   \eqno  (\; \forall \;\; \ell \in I \, , \, T \in \lieh \,)  $$
 does indeed provide a well-defined isomorphism of (topological)
 superalgebras that is in fact also one of Hopf superalgebras, q.e.d.
\end{proof}

\vskip9pt

   As a matter of fact,  \textsl{the previous result can be somehow reversed},
   as the following shows: in particular, loosely speaking, we end up
   finding that  \textit{every straight small FoMpQUESA can be realized as a
   toral twist deformation of the ``standard'' Yamane's QUESA}
   ---  cf.\ claim  \textit{(c)\/}  in  Theorem \ref{thm: uRpPhg=twist-uhg}  here below.

\vskip13pt

\begin{theorem}  \label{thm: uRpPhg=twist-uhg}
 With assumptions as above, let  $ P $  and  $ P' $  be two matrices
 of Cartan type with the same associated Cartan matrix $ A \, $,
 \,namely  $ \, P_s = P'_s \, $.
 \vskip3pt
   (a)\,  Let  $ \R $  be a  \textsl{straight}  realization of  $ P $  and let
   $ \uRpPhg $  be the associated FoMpQUESA.  Then there is a  \textsl{straight}  realization  $ \check{\R}' $  of
   $ P' $  and a matrix
   $ \, \Phi \in \! \lieso_t\big(\kh\big) \, $  such that for the
   associated twist element  $ \F_\Phi $  as in
   \eqref{eq: Resh-twist_F-uPhgd}  we have
  $$
  U_{\!P'\!,\,\hbar}^{\check{\R}'\!,\,p}(\lieg)  \; \cong \;  {\big(\, \uRpPhg \big)}^{\F_\Phi}
  $$
 In a nutshell, if  $ \, P'_s = P_s \, $  then from any
 \textsl{straight}  FoMpQUESA over  $ P $  we can obtain by toral twist deformation a
 \textsl{straight}  FoMpQUESA (of the same rank) over  $ P' $.
                                                                  \par
   Conversely, if  $ \, \R' $  is any  \textsl{straight}
   realization of  $ P' $  and  $  U_{\!P'\!,\,\hbar}^{\R'\!,\,p}(\lieg) $
   is the associated FoMpQUESA, then there exists a  \textsl{straight}
   realization  $ \hat{\R} $  of  $ P $  and a matrix
   $ \, \Phi \in \! \lieso_t\big(\kh\big) \, $
   such that for the associated twist element  $ \F_\Phi $
   as in  \eqref{eq: Resh-twist_F-uPhgd}  we have
  $$
  U_{\!P'\!,\,\hbar}^{\R'\!,\,p}(\lieg)  \; \cong \;
  {\big(\, U_{\!P,\,\hbar}^{\hat{\R},\,p}(\lieg) \big)}^{\F_\Phi}
  $$
 \vskip3pt
   (b)\,  Let  $ \R $  and  $ \R' $  be  \textsl{straight small}
   realizations of  $ P $  and  $ P' $  respectively, with
   $ \, \rk(\R) = \rk(\R') = t\, $,  and let  $ \uRpPhg $  and
   $ U_{\!P'\!,\,\hbar}^{\R'\!,\,p}(\lieg) $  be the associated FoMpQUESAs.
   Then there exists a matrix  $ \, \Phi \in \! \lieso_t\big(\kh\big) \, $  such that for
%
%
 $ \F_\Phi $  as in  \eqref{eq: Resh-twist_F-uPhgd}  we have
  $$  U_{\!P'\!,\,\hbar}^{\R'\!,\,p}(\lieg)  \; \cong \;  {\big(\, \uRpPhg \big)}^{\F_\Phi}  $$
 In short, if  $ \, P'_s = P_s \, $,  any  \textsl{straight small}
 FoMpQUESA over  $ P' $  is isomorphic to a toral twist deformation
 of any  \textsl{straight small}  FoMpQUESA over  $ P $  of same rank.
 \vskip3pt
   (c)\,  Every  \textsl{straight small}  FoMpQUESA is
   isomorphic to some toral twist deformation of Yamane's standard
   FoMpQUESA (over  $ \, DA = P_s \, $)  of the same rank.
 \vskip3pt
   (d)\,  Similar, parallel statements hold true for the Borel FoMpQUESAs.
\end{theorem}

\pf
 \textit{(a)}\,  By  Theorem \ref{thm: twist-uRpPhg=new-uRpPhg}
 it is enough to find  $ \, \Phi \in \lieso_t(\Bbbk)  \, $  such that
 $ \, P' = P_{\scriptscriptstyle \Phi} \, $,  \,that is
 $ \; P' = P - \mathfrak{A} \, \Phi \, \mathfrak{A}^{\,\scriptscriptstyle T} \, $;
 \,but this is guaranteed by Proposition \ref{prop: realiz=twist-standard}\textit{(a)}.
 \vskip5pt
   \textit{(b)}\,  This follows directly from claim  \textit{(a)\/}
   and  Proposition \ref{prop: realiz=twist-standard}\textit{(b)}.
 \vskip5pt
   \textit{(c)}\,  This follows from  \textit{(b)},  with
   $ U_{\!P'\!,\,\hbar}^{\R'\!,\,p}(\lieg) $  the given straight small FoMpQUESA and
   $ \uRpPhg $  the standard FoMpQUESA
%
%
 over  $ \, P := DA \, $
 as in Yamane's definition.
 \vskip5pt
   \textit{(d)}\,  This is clear from definitions.
\epf

\vskip9pt

\begin{obs}  \label{obs: parameter_(P,Phi)}
 Theorems \ref{thm: twist-uRpPhg=new-uRpPhg}  and  \ref{thm: uRpPhg=twist-uhg}
 have the following interpretation.  Our FoMpQUESAs  $ \, \uRpPhg \, $
 are quantum objects depending on the multiparameter  $ P \, $;
 but when we perform onto  $ \, \uRpPhg \, $  a deformation by twist as
 in  \S \ref{twist-uPhgd},  the output
 $ \, {U_{\!P,\hskip0,7pt\hbar,\Phi}^{\R,\,p}(\hskip0,5pt\lieg)}
 := {\big(\, \uRpPhg \big)}^{\F_\Phi} \, $
 seemingly depends on  \textsl{two\/}  multiparameters, namely  $ P $
 \textsl{and\/}  $ \Phi \, $.  Thus all these
 $ {U_{\!P,\hskip0,7pt\hbar,\Phi}^{\R,\,p}(\hskip0,5pt\lieg)} $'s form a seemingly
 \textsl{richer\/}  family of ``twice-multiparametric'' formal QUESAs.
 Nevertheless,  Theorem \ref{thm: twist-uRpPhg=new-uRpPhg}  above proves
 that this family actually  \textsl{coincides\/}  with the family of all FoMpQUESAs,
 although the latter seems a priori smaller.
                                                              \par
   In short,  Theorems \ref{thm: twist-uRpPhg=new-uRpPhg}  and  \ref{thm: uRpPhg=twist-uhg}
   show the following.  The dependence of the Hopf structure of
   $ {U_{\!P,\hskip0,7pt\hbar}^{\R,\Phi}(\hskip0,5pt\lieg)} $  on the
   ``double parameter''  $ (P\,,\Phi) $  is ``split'' in the algebraic structure
   (ruled by  $ P $)  and in the coalgebraic structure  (ruled by  $ \Phi $);
   \,now  Theorems \ref{thm: twist-uRpPhg=new-uRpPhg}  and  \ref{thm: uRpPhg=twist-uhg}
   enable us to ``polarize'' this dependence so to codify it either entirely
   within the algebraic structure (while the coalgebraic one is cast into a ``canonical form''),
   so that the single multiparameter  $ P_\Phi $  is enough to describe it,
   or entirely within the coalgebraic structure (with the algebraic one being
   reduced to the ``standard'' Yamane's one), hence one multiparameter alone, say  $ \Phi_P \, $,
   is enough.
\end{obs}

\vskip7pt

\begin{rmk}  \label{rmk: split FoMpQUESAs not-stable}
 Since the subclass of  \textsl{split\/}  realizations is  \textsl{not closed\/}  under twist
 (cf.\ the end of  \S \ref{further_stability}),
 the subclass of all ``split'' FoMpQUESAs is not closed too under twist deformation;
 this is a quantum analogue of  Observations \ref{obs: parameter_(P,Phi) x MpLbA's}\textit{(b)}.
\end{rmk}

\vskip15pt

\subsection{Deformations of FoMpQUESAs by toral  polar--$ 2 $--cocycles}
\label{subsec: q2cocyc-def_FoMpQUESAs}
 {\ }
 \vskip9pt
   We discuss now suitable deformations of FoMpQUESAs by
   polar--$ 2 $--cocycles (of ``toral type''), proving that they are again FoMpQUESAs.

\vskip9pt

\begin{free text}  \label{q2cocyc-deform_uPhgd}
 {\bf Deformations of  $ \uRpPhg $  by toral  polar--$ 2 $--cocycles.}
 In  \S \ref{defs_Hopf-algs}  we introduced the notion of ``$ 2 $--cocycle''
 and the procedure of deformation by a  $ 2 $--cocycle  for a given Hopf superalgebra.
 In the setup of QUESAs, the construction has an important variation,
 leading to the notion of  ``polar--$ 2 $--cocycle''
 for which the procedure of deformation mentioned above still applies.
 This theory is developed in  \cite{GG4}  for the case of quantized universal enveloping algebras (=QUEAs),
 i.e.\ those QUESAs whose super-structure is trivial (namely, their odd part is zero).
 The task of extending the results in  \cite{GG4}
 --- stated and proved for QUEAs: see, in particular,  \S 3.3 therein ---
 to the setup of QUESAs is harmless, though technical.
 However, we do not need that: instead, we will introduce special
 ``polar--$ 2 $--cocycles''  for our FoMpQUESAs (a special class of QUESAs, indeed)
 and then we will show by direct check that the procedure of deformation by
 $ 2 $--cocycles  still work if we use instead these  \textsl{polar}--$ 2 $--cocycles.
 Thus we will be applying a general fact, but proving its validity ``on the fly''
 instead of relying on a general, theoretical proof.
\end{free text}

\vskip9pt

\begin{free text}  \label{2-cocycle-U_h(h)}
 {\bf Polar--$ 2 $--cocycles  of  $ U_{\!P,\hbar}(\lieh) \, $.}
   Let $(A,p)$ be a Cartan super-datum. Fix again
   $ \, P := {\big(\, p_{i,j} \big)}_{i, j \in I} \in M_n\big( \kh \big) \, $  of
 Cartan type with associated Cartan matrix  $ A \, $,  a realization
 $ \; \R \, := \, \big(\, \lieh \, , \Pi \, , \Pi^\vee \,\big) \; $
 of it and the (topological) Hopf algebra  $ \uRpPhg \, $,
 as in  \S \ref{sec: Cartan-data_realiz's}  and  \S \ref{subsec: FoMpQUESA's}.
 We consider special 2--cocycles of  $ \, \uRpPhg \, $,
 called ``toral'' as they are induced from the quantum torus.
 To this end, like in  \S \ref{twist-uPhgd},  we fix in  $ \lieh $  a  $ \kh $--basis
 $ \, {\big\{ H_g \big\}}_{g \in \G} \, $,
 where  $ \G $  is an index set with  $ \, |\G| = \rk(\lieh) = t \, $.
 \vskip5pt
   Like in  \S \ref{2-cocycle-deform.'s x mpmatr.'s & realiz.'s},
   we fix an antisymmetric,  $ \kh $---bilinear  map
   $ \, \chi : \lieh \times \lieh \relbar\joinrel\longrightarrow \kh \, $,
   \,that corresponds to some
   $ \, X = {\big(\, \chi_{g{}\gamma} \big)}_{g , \gamma \in \G} \in \lieso_t\big(\kh\big) \, $
   via  $ \, \chi_{g{}\gamma} = \chi(H_g\,,H_\gamma) \, $.
   We also consider the antisymmetric matrix
   $ \, \mathring{X} := {\Big(\, \mathring{\chi}_{i{}j} =
   \chi\big(\,T_i^+\,,T_j^+\big) \Big)}_{i, j \in I} \! \in \lieso_n(\kh) \, $.
   Any such map  $ \chi $  induces uniquely an antisymmetric,  $ \kh $--bilinear  map
  $$  \tilde{\chi}_{\scriptscriptstyle U} \, : \,  U_\hbar(\lieh) \times U_\hbar(\lieh) \relbar\joinrel\relbar\joinrel\relbar\joinrel\longrightarrow \kh  $$
 as follows.  By definition,  $ \, U_\hbar(\lieh) := \uRpPhh \, $  is an  $ \hbar $--adically
 complete topologically free Hopf algebra isomorphic to
 $ \, \widehat{S}_\kh(\lieh) := \widehat{\bigoplus\limits_{n \in \NN} S_\kh^{\,n}(\lieh)} \, $
 --- the  $ \hbar $--adic  completion of the symmetric algebra
 $ \, S_\kh(\lieh) = \bigoplus\limits_{n \in \NN} S_\kh^{\,n}(\lieh) \, $  ---
 hence the following makes sense:

\vskip9pt

\begin{definition}  \label{def: tildechi_U}
 We define  $ \tilde{\chi}_{\scriptscriptstyle U} $  as the unique  $ \kh $--linear
 (hence  $ \hbar $--adically  continuous) map
 $ \; U_\hbar(\lieh) \otimes U_\hbar(\lieh) \,{\buildrel \tilde{\chi}_{\scriptscriptstyle U} \over {\relbar\joinrel\longrightarrow}}\, \kh \; $
 such that (with identifications as above)
  $$  \displaylines{
   \tilde{\chi}_{\scriptscriptstyle U}(z,1) := \epsilon(z) =: \tilde{\chi}_{\scriptscriptstyle U}(1,z)  \;\; ,
   \quad  \tilde{\chi}_{\scriptscriptstyle U}(x,y) := \chi(x,y)  \quad
   \forall \;\; z \in  \widehat{S}_\kh(\lieh) \, , \; x, y \in  S_\kh^{\,1}(\lieh)  \cr
   \hfill \qquad  \tilde{\chi}_{\scriptscriptstyle U}(x,y) := 0   \quad \qquad  \forall \;\; x \in S_\k^{\,r}(\lieh) \, ,
   \; y \in S_\k^{\,s}(\lieh) \, : \, r, s \geq 1 \, , \; r+s > 2   \hfill  \diamondsuit  }  $$
\end{definition}

\vskip9pt

   By construction,  \textsl{$ \tilde{\chi}_{\scriptscriptstyle U} $  is a normalized Hochschild 2--cocycle on  $ U_\hbar(\lieh) \, $},  \,that is
  $$  \epsilon(x) \,\tilde{\chi}_{\scriptscriptstyle U}(y,z) \, - \, \tilde{\chi}_{\scriptscriptstyle U}(xy,z) \,
  + \, \tilde{\chi}_{\scriptscriptstyle U}(x,yz) \,
  - \, \tilde{\chi}_{\scriptscriptstyle U}(x,y) \, \epsilon(z) \; = \; 0
  \qquad  \forall \;\; x, y, z \in U_{\!P,\hbar}^{\R}(\lieh)  $$

\vskip5pt

   Given two linear maps
   $ \, \eta , \vartheta : U_\hbar(\lieh) \otimes U_\hbar(\lieh) \relbar\joinrel\longrightarrow \kh \, $,
   \,one defines the  {\it convolution product\/}
   $ \; \eta \ast \vartheta : {U_\hbar(\lieh)}^{\otimes 2} \relbar\joinrel\relbar\joinrel\longrightarrow \kh \; $
 by  $ \; (\eta \ast \vartheta)(x \otimes y) \, := \, \eta(x_{(1)},y_{(1)}) \, \vartheta(x_{(2)},y_{(2)}) \; $
   for all  $ \, x, y \in U_\hbar(\lieh) \, $
   --- using the coalgebra structure of
   $ \, {U_\hbar(\lieh)}^{\otimes 2}\cong {\widehat{S}_\kh(\lieh)}^{\otimes 2}
   \cong \widehat{S}_\kh(\lieh \oplus \lieh) \, $.  Then by  $ \, \eta^{\ast\,m} \, $
   we denote the  $ m $--th  power with respect to the convolution
   product of any map  $ \eta $  as above; in particular, we set  $ \, \eta^{\ast\,0} := \epsilon \otimes \epsilon \, $.
                                                              \par
   The following result, proved by direct computation, describes the powers of  $ \tilde{\chi}_{\scriptscriptstyle U} \, $:
\end{free text}

\vskip5pt

\begin{lema}  \label{lem: tildechi}
For all  $ \, H_+ ,H_- \in \lieh \, $  and  $ \, k, \ell, m \in \NN_+ \, $,  \,we have
  $$  \quad   { \tilde{\chi}_{\scriptscriptstyle U}}^{\,\ast\,m}\big(H_+^{\,k},H_-^{\,\ell}\,\big) \; = \;
   \begin{cases}
      \begin{array}{lr}
  \delta_{k,m} \, \delta_{\ell,m} \, {\big(m!\big)}^2 \, \chi(H_+,H_-)^m  &  \quad  \text{for} \;\;\;\; m \geq 1 \; ,  \quad \hskip7pt \qquad \;   \hfill  \phantom{\square}  \\
  \delta_{k,0} \, \delta_{\ell,0} \,  &  \quad  \text{for} \;\;\;\; m = 0 \; .
  \quad \hskip7pt \qquad \;   \hfill  \square
      \end{array}
   \end{cases}  $$
\end{lema}
%
%

\vskip9pt

\begin{definition}  \label{def: chi_U}
 Keep notation as above.  We define  $ \, \chi_{\scriptscriptstyle U} \, $  as the unique map from  $ \, U_\hbar(\lieh) \!\mathop{\otimes}\limits_\kh\! U_\hbar(\lieh) \, $  to  $ \khp $
 given by the exponentiation of  $ \, \hbar^{-1} \, 2^{-1} \, \tilde{\chi}_{\scriptscriptstyle U} \, $,  i.e.
 \vskip-5pt
  $$  \chi_{\scriptscriptstyle U}  \; := \;  e^{\hbar^{-1} 2^{-1} \tilde{\chi}_{\scriptscriptstyle U}}  \; = \;
  {\textstyle \sum_{m \geq 0}} \, \hbar^{-m} \, {\tilde{\chi}_{\scriptscriptstyle U}}^{\,\ast\,m} \! \Big/ 2^m\,m!   \eqno \diamondsuit  $$
\end{definition}

\vskip9pt

\begin{lema}  \label{lem: chi_U(K,K)}
 The map  $ \, \chi_{\scriptscriptstyle U} $  is a well-defined, normalized,
 $ \khp $--valued  Hopf 2--cocy\-cle for  $ \, U_\hbar(\lieh) := U_{\!P,\hbar}^{\R}(\lieh) \, $,  such that, for all
 $ H_+ \, , H_- \in \lieh \, $  and  $ \, K_\pm := e^{\,\hbar\,H_\pm} \, $,
  $$  \chi_{\scriptscriptstyle U}^{\pm 1}(H_+,H_-) \, = \, \pm \hbar^{-1} \, 2^{-1}  \chi(H_+,H_-)  \quad ,
  \qquad  \chi_{\scriptscriptstyle U}(K_+,K_-) \, = \, e^{\hbar \, 2^{-1} \chi(H_+,H_-)}  $$
\end{lema}

\pf
 This is just Lemma 5.2.5 in  \cite{GG3}.
\epf

\vskip11pt

\begin{free text}  \label{toral-cocycles-uPhgd}
 {\bf Toral  polar--$ 2 $--cocycles of  $ \uRpPhg \, $.}
 The previous construction provides, starting from  $ \chi \, $,  a normalized Hopf\/  $ 2 $--cocycle
 $ \; \chi_{\scriptscriptstyle U} : \uRpPhh \times \uRpPhh
 \relbar\joinrel\longrightarrow \khp \; $.

\vskip5pt

   \textit{We assume now that the map  $ \, \chi $  satisfies the additional requirement\/  \eqref{eq: condition-chi},  in other words we require that  $ \, \chi \in \text{\sl Alt}_{\,\kh}^{\,S}(\lieh) \, $   --- }
 notation of  \S \ref{subsec: 2-cocycle-realiz}.  The latter map canonically induces a  $ \kh $--bilinear  map
 $ \; \overline{\chi} : \overline{\lieh} \times \overline{\lieh}
\relbar\joinrel\longrightarrow \kh \; $,  \;where  $ \, \overline{\lieh} := \lieh \big/ \lies \, $  with
$ \, \lies := \textsl{Span}_\kh\big( {\{\, S_i \,\}}_{i \in I} \big)\, $,  \,given by
$$
\overline{\chi}\,\big(\, T' \! + \lies \, , T'' \! + \lies \,\big)
\, := \, \chi\big(\,T',T''\big)   \qquad  \forall \;\; T' , T'' \in \lieh
$$
   \indent   Now, replaying the construction above but with  $ \overline{\lieh} $  and
   $ \overline{\chi} $  replacing  $ \lieh $  and  $ \chi \, $,
   we construct a normalized Hopf\/  $ 2 $--cocycle
   $ \; \overline{\chi}_{\scriptscriptstyle U} \!: U_{\!P,\hbar}^{\R,\,p}\big(\,\overline{\lieh}\,\big) \times
   U_{\!P,\hbar}^{\R,\,p}\big(\,\overline{\lieh}\,\big) \relbar\joinrel\longrightarrow \khp \; $;
for the latter, the analogue of  Lemma \ref{lem: chi_U(K,K)}  holds true again.  Moreover,
  note that  $ \, U_{\!P,\hbar}^{\R,\,p}\big(\,\overline{\lieh}\,\big) \cong \widehat{S}_\kh\big(\,\overline{\lieh}\,\big) \, $,  \,and,
thanks to  \eqref{eq: condition-chi-chit},
 there exists a Hopf superalgebra epimorphism
 $ \; \pi : \uRpPhg \relbar\joinrel\relbar\joinrel\twoheadrightarrow U_{\!P,\hbar}^{\R,\,p}\big(\,\overline{\lieh}\,\big) \; $
 given by  $ \, \pi(E_i) := 0 \, $,  $ \, \pi(F_i) := 0 \, $   --- for  $ \, i \in I \, $  ---   and
 $ \, \pi(T) := (T + \lies) \, \in \, \overline{\lieh} \subseteq U_{\!P,\hbar}^\R\big(\,\overline{\lieh}\,\big) \, $
 --- for  $ \, T \in \lieh \, $.
 Then we consider
  $$  \sigma_\chi \, := \, \overline{\chi}_{\scriptscriptstyle U} \circ (\pi \times \pi) \, : \, \uRpPhg \times \uRpPhg \relbar\joinrel\relbar\joinrel\relbar\joinrel\relbar\joinrel\twoheadrightarrow \khp  $$
 which is  \textsl{automatically\/}  a normalized,  $ \khp $--valued  Hopf\/  $ 2 $--cocycle  on  $ \uRpPhg \, $.
\end{free text}

\vskip7pt

\begin{definition}   \label{def: sigma_formal}
 We call  ``\textit{polar--2--cocycles\/}  of  \textsl{toral\/}  type'', or ``\textsl{toral}  \textit{polar--$ 2 $--cocy\-cles\/}'',  of  $ \, \uRpPhg $  all normalized Hopf 2--cocycles  $ \, \sigma_\chi $  obtained
 --- from all  $ \, \chi \in \text{\sl Alt}^{\,S}_{\,\kh}(\lieh) \, $  ---   via the above construction.  We denote by  $ \, \Z_{\,2}\big( \uRpPhg \big) \, $  the set of all of them   --- which is actually independent of the multiparameter  $ P $,  indeed.
 \vskip3pt
   \textsl{N.B.:}\,  in fact,  \textsl{the  $ \sigma_\chi $'s  above are ``polar--$ 2 $--cocycles  of  $ \uRpPhg $''  in the sense of the definition given in  \cite{GG4}},  once the latter is adapted to QUESAs.   \hfill  $ \diamondsuit $
\end{definition}

\vskip7pt

\begin{free text}  \label{formulas_deform-prod}
 {\bf Formulas for the  $ \sigma_\chi $--deformed  product.}
 Let  $ \, \sigma_\chi \in \Z_{\,2}\big( \uRpPhg \big) \, $  be a toral  $ 2 $--cocycle  as above.
 Following  \S \ref{defs_Hopf-algs},  using  $ \sigma_\chi $  we introduce in  $ \uRpPhg $  a ``deformed product'',
 hereafter denoted by   $ \, \scriptstyle \dot\sigma_\chi \, $;  \,then
 $ \, X^{(n)_{\sigma_{\chi}}} = X\raise-1pt\hbox{$ \, \scriptstyle \dot\sigma_\chi $} \cdots \raise-1pt\hbox{$ \, \scriptstyle \dot\sigma_\chi $} X \, $
 will denote the  $ n $--th  power of any  $ \, X \in \uRpPhg \, $  with respect to this deformed product.
                                                \par
   Directly from definitions, sheer computation yields the following formulas,
   relating the deformed product with the old one (for all  $ \; T', T'', T \in \lieh \; $,  $ \; i \, , \, j \in I \, $,  $ \, m, n \in \NN \, $):
  $$  \displaylines{
   T' \raise-1pt\hbox{$ \, \scriptstyle \dot\sigma_\chi $} T''  \, = \;  T' \, T''  \;\;\; ,
     \;\quad  T^{(n)_{\sigma_{\chi}}} = T^n  \;\;\; ,
     \;\quad  E_i \raise-1pt\hbox{$ \, \scriptstyle \dot\sigma_\chi $} F_j  \; = \,  E_i \, F_j  \;\;\; ,
     \;\quad  F_j \raise-1pt\hbox{$ \, \scriptstyle \dot\sigma_\chi $} E_i  \; = \,  F_j \, E_i  \cr
   T \raise-1pt\hbox{$ \, \scriptstyle \dot\sigma_\chi $}  E_j  \; = \;  T \, E_j \, + \, 2^{-1} \chi\big(T,T_j^+\big) \, E_j \;\;\; ,
     \qquad  E_j \raise-1pt\hbox{$ \, \scriptstyle \dot\sigma_\chi $} T  \; = \;  E_j \, T \, + \, 2^{-1} \chi\big(T_j^+,T\big)  \, E_j  \cr
   T \raise-1pt\hbox{$ \, \scriptstyle \dot\sigma_\chi $} F_j  \; = \;  T \, F_j \, + \, 2^{-1} \chi\big(T,T_j^-\big)  \, F_j  \;\;\; ,
     \qquad  F_j \raise-1pt\hbox{$ \, \scriptstyle \dot\sigma_\chi $} T  \; = \;  F_j \, T \, + \, 2^{-1} \chi\big(T_j^-,T\big) \, F_j
%
%
  }  $$
  $$  \displaylines{
   E_{i_1} \raise-1pt\hbox{$ \, \scriptstyle \dot\sigma_\chi $} E_{i_2} \raise-1pt\hbox{$ \, \scriptstyle \dot\sigma_\chi $} \cdots \raise-1pt\hbox{$ \, \scriptstyle \dot\sigma_\chi $} E_{i_{t-1}} \raise-1pt\hbox{$ \, \scriptstyle \dot\sigma_\chi $} E_{i_t}  \; = \;  e^{\,+\hbar \, 2^{-1} \sum_{c<e} \mathring{\chi}_{\! \raise-2pt\hbox{$ \, \scriptstyle c\,e $}}} \, E_{i_1} \, E_{i_2} \cdots E_{i_{t-1}} \, E_{i_t}  \cr
   E_i^{\,(m)_{\sigma_{\chi}}}  \, = \;  {\textstyle \prod_{\ell=1}^{m-1}} \sigma_\chi\Big( e^{\,+\hbar\,\ell\,T_i^+} \! , \, e^{\,+\hbar\,T_i^+} \Big) \, E_i^{\,m}  \; = \;  E_i^{\,m}  \cr
   E_i^{\,m} \raise-1pt\hbox{$ \, \scriptstyle \dot\sigma_\chi $} E_j^{\,n}  \; = \;  \sigma_\chi\Big(
e^{\,+\hbar\,m\,T_i^+} \! , \, e^{\,+\hbar\,n\,T_j^+} \Big) \, E_i^{\,m} \, E_j^{\,n}  \; = \;  e^{\, +\hbar \,m\,n\, 2^{-1} \mathring{\chi}_{ij} } E_i^{\,m} \, E_j^{\,n}  \cr
   E_i^{\,(m)_{\sigma_{\chi}}} \raise-1pt\hbox{$ \, \scriptstyle \dot\sigma_\chi $}\, E_j \,\raise-1pt\hbox{$ \, \scriptstyle \dot\sigma_\chi $}\, E_k^{\,(n)_{\sigma_\chi}}  \, = \;
   {\textstyle \prod_{\ell=1}^{m-1}} \sigma_\chi\Big( e^{\,+\hbar\,\ell\,T_i^+} \! ,
   \, e^{\,+\hbar\,T_i^+} \Big) \cdot {\textstyle \prod_{t=1}^{n-1}} \sigma_\chi\Big( e^{\,+\hbar\,t\,T_k^+} \! , \, e^{\,+\hbar\,T_k^+} \Big) \, \cdot   \hfill  \cr
  \hfill \cdot \; \sigma_{\chi}\Big( e^{\,+\hbar\,m\,T_i^+} \! , \, e^{\,+\hbar\,T_j^+} \Big) \,
  \sigma_\chi\Big( e^{\,+\hbar\,(m\,T_i^+ + T_j^+)} \! , \, e^{\,+\hbar\,n\,T_k^+} \Big) \, E_i^{\,m} \, E_j \, E_k^{\,n}
%
%
  }  $$
  $$  \displaylines{
   F_{i_1} \raise-1pt\hbox{$ \, \scriptstyle \dot\sigma_\chi $} F_{i_2} \raise-1pt\hbox{$ \, \scriptstyle \dot\sigma_\chi $} \cdots \raise-1pt\hbox{$ \, \scriptstyle \dot\sigma_\chi $} F_{i_{t-1}} \raise-1pt\hbox{$ \, \scriptstyle \dot\sigma_\chi $} F_{i_t}  \; = \;  e^{\,-\hbar \, 2^{-1} \sum_{c<e} \mathring{\chi}_{\! \raise-2pt\hbox{$ \, \scriptstyle c\,e $}}} \, F_{i_1} \, F_{i_2} \cdots F_{i_{t-1}} \, F_{i_t}  \cr
   F_i^{\,(m)_{\sigma_\chi}}  \, = \;  {\textstyle \prod_{\ell=1}^{m-1}} \sigma_\chi^{-1}\Big( e^{\,-\hbar\,\ell\,T_i^-} \! , \,
   e^{\,-\hbar\,T_i^-} \Big) \, F_i^{\,m}  \; = \;  F_i^{\,m}  \cr
   F_i^{\,m} \raise-1pt\hbox{$ \, \scriptstyle \dot\sigma_\chi $} F_j^{\,n}  \; = \;  \sigma_\chi^{-1}\Big( e^{\,-\hbar\,m\,T_i^-} \! , \,
   e^{\,-\hbar\,n\,T_j^-} \Big) \, F_i^{\,m} \, F_j^{\,n}  \; = \;  e^{\, -\hbar \,m\,n\, 2^{-1} \mathring{\chi}_{ij}  } F_i^{\,m} \, F_j^{\,n}  \cr
   F_i^{\,(m)_{\sigma_{\chi}}} \raise-1pt\hbox{$ \, \scriptstyle \dot\sigma_\chi $}\, F_j \,\raise-1pt\hbox{$ \, \scriptstyle \dot\sigma_\chi $}\, F_k^{\,(n)_{\sigma_{\chi}}}  \, = \;
   {\textstyle \prod_{\ell=1}^{m-1}} \sigma_{\chi}^{-1}\Big( e^{\, -\hbar\,\ell\,T_i^-} \! , \, e^{\,-\hbar\,T_i^-} \Big) \cdot {\textstyle \prod_{t=1}^{n-1}} \sigma_{\chi}^{-1}\Big( e^{\,-\hbar\,t\,T_k^-} \! , \, e^{\,-\hbar\,T_k^-} \Big) \, \cdot   \hfill  \cr
   \hfill \cdot \; \sigma_{\chi}^{-1}\Big( e^{\,-\hbar\,m\,T_i^-} \! , \, e^{\,-\hbar\,T_j^-} \Big) \, \sigma_{\chi}^{-1}\Big( e^{\,-\hbar\,(m\,T_i^- + T_j^-)} \! , \, e^{\,-\hbar\,n\,T_k^-} \Big) \, F_i^{\,m} \, F_j \, F_k^{\,n}
   }  $$
 \vskip5pt
   Note also that from the identities  $ \; T^{(n)_{\sigma_{\chi}}} = T^n \; $  (for  $ \, T \in \lieh \, $  and  $ \, n \in \NN \, $)  it follows that the exponential of any toral element with respect to the deformed product  $ \raise-1pt\hbox{$ \, \scriptstyle \dot\sigma_\chi $} $  is the same as with respect to the old one.
\end{free text}

\vskip11pt

\begin{obs}
   The whole procedure of  $ 2 $--cocycle  deformation by  $ \sigma_\chi $  should apply to the scalar extension  $ \, \khp \otimes_\kh \uRpPhg \, $,  as  \textit{a priori\/}  $ \sigma_\chi $  takes values in  $ \khp $  rather than in  $ \kh \, $.  Nevertheless, the formulas above show that
%
%
  $ \uRpPhg $  is actually  \textsl{closed\/}  for the deformed product  ``$ \, \raise-1pt\hbox{$ \, \scriptstyle \dot\sigma_\chi $} $''  provided by this procedure, hence the deformation does ``restrict'' to  $ \uRpPhg $  itself, so that we eventually end up with a well-defined  $ 2 $--cocycle  deformation  $  {\big( \uRpPhg \big)}_{\!\sigma_\chi} \, $.
\end{obs}

   A first, direct consequence of these formulas is the following:

\vskip11pt

\begin{prop}  \label{prop: gen's-FoMpQUEA => def-FoMpQUEA}
%
%
 The deformed algebra  $ {\big( \uRpPhg \big)}_{\sigma_\chi} \! $  is still generated by the elements  $ E_i \, $,  $ F_i $  and  $ T $  (with  $ \, i \in I \, $  and  $ \, T \in \lieh \, $)  of  $ \, \uRpPhg \, $.
\end{prop}

\vskip7pt

\begin{free text}  \label{tor_2-coc_deform's_uPhgd}
 {\bf Deformations of  $ \uRpPhg $  by toral  polar--$ 2 $--cocycle.}
 Our key result concerns  polar--$ 2 $--cocycle  deformations by means of
 toral  polar--$ 2 $--cocycles.  In order to state it,
 we need some more notation, which we now settle.
                                                                                    \par
   Let $(A,p)$ be a Cartan super-datum,
   $ \, P := {\big(\, p_{ij} \,\big)}_{i, j \in I} \in M_n\big( \kh \big) \, $
   a multiparameter matrix of Cartan type with associated Cartan matrix $ A $
   ---  cf.\  Definition \ref{def: realization of P}  ---
   and fix a realization  $ \R = \big(\, \lieh \, , \Pi \, , \Pi^\vee \big) \, $  of  $ P $.
   Fix an antisymmetric  $ \kh $--bilinear map
   $ \, \chi \! : \lieh \times \lieh \! \longrightarrow \! \kh \, $
   enjoying  \eqref{eq: condition-chi}   --- that is,
   $ \, \chi \in \text{\sl Alt}_{\,\kh}^{\,S}(\lieh) \, $;
   \,then we associate with it the matrix
   $ \, \mathring{X} := \! {\Big( \mathring{\chi}_{i{}j} = \chi\big(\,T_i^+,T_j^+\big) \!\Big)}_{i, j \in I} \, $
   as above.  Note that
 $ \; + \chi\big(\, \text{--} \, , T_j^+ \big)  \, = \,  - \chi\big(\, \text{--} \, , T_j^- \big) \; $
 for all  $ \, j \in I \, $,  as direct consequence of  \eqref{eq: condition-chi}.
 Basing on the above, like in  \S \ref{subsec: 2-cocycle-realiz}  we define
  $$
  P_{(\chi)}  := \,  P \, + \, \mathring{X}  \, = \,  {\Big(\, p^{(\chi)}_{i{}j} :=
  \,  p_{ij} + \mathring{\chi}_{i{}j} \Big)}_{\! i, j \in I}  \;\, ,  \!\quad
   \Pi_{(\chi)}  := \,  {\Big\{\, \alpha_i^{(\chi)}  := \,
   \alpha_i \pm \chi\big(\, \text{--} \, , T_i^\pm \big)  \Big\}}_{i \in I}
   $$
\noindent
 Then, still from  \S \ref{subsec: 2-cocycle-realiz}  we know that
 $ P_{(\chi)} $  is a matrix of Cartan type   --- the same of  $ P $
 indeed ---   and
 $ \, \R_{(\chi)} \, = \, \big(\, \lieh \, , \Pi_{(\chi)} \, , \Pi^\vee \,\big) \, $
 is a realization of it.
 \vskip3pt
   We are now ready for our result on deformations by toral
   polar--$ 2 $--cocycle.
\end{free text}

\vskip7pt

\begin{theorem}  \label{thm: 2cocdef-uPhgd=new-uPhgd}
 There exists an isomorphism of topological Hopf superalgebras
  $$  {\big( \uRpPhg \big)}_{\sigma_\chi}  \; \cong \;\,  U_{\!P_{(\chi)},\,\hbar}^{\R_{(\chi)},\,p}(\lieg)  $$
 which is the identity on generators.  In short, every toral polar--2--cocycle deformation of a
 FoMpQUESA is another FoMpQUESA, whose multiparameter  $ \, P_{(\chi)} $  and realiza\-tion  $ \R_{(\chi)} $
 depend on the original  $ P $  and  $ \R \, $,  as well as on  $ \chi \, $,  as explained in  \S \ref{subsec: 2-cocycle-realiz}.
 \vskip5pt
   Similar statements hold true for the Borel FoMpQUESAs and their deformations by  $ \sigma_\chi \, $,  namely there exist isomorphisms  $ \,\; {\big( U_{\!P,\,\hbar}^{\R,\,p}(\lieb_\pm) \big)}_{\sigma_\chi} \cong \; U_{\!P_{(\chi)},\,\hbar}^{\R_{(\chi)},\,p}(\lieb_\pm) \;\, $.
\end{theorem}

\pf
 We begin by noting the following key
 \vskip7pt
   {\sl  $ \underline{\text{Fact}} $:  The generators of  $ \uRpPhg \, $,
   when thought of as elements of the deformed algebra  $ {\big( \uRpPhg \big)}_{\sigma_\chi} \, $,
   obey the defining relations of the (same name) generators of
   $ \, U_{\!P_{(\chi)},\,\hbar}^{\R_{(\chi)},\,p}(\lieg) \, $   --- with respect to the deformed product
   $ \raise-1pt\hbox{$ \, \scriptstyle \dot\sigma_\chi $} \, $.}
 \vskip7pt
   Let us assume for a moment the above  $ \underline{\text{\sl Fact\/}} $  does hold true.  Then it implies that there exists a well-defined homomorphism of topological Hopf superalgebras, say
 $ \; \ell : U_{\!P_{(\chi)},\,\hbar}^{\R_{(\chi)},\,p}(\lieg) \relbar\joinrel\relbar\joinrel\relbar\joinrel\relbar\joinrel\longrightarrow
 {\big(\, \uRpPhg \big)}_{\sigma_\chi} \; $
 given on generators by  $ \; \ell(E_i) := E_i \, $,  $ \; \ell(F_i) := F_i \, $,  $ \; \ell\big(T) := T \; $  ($ \, i \in I \, $,  $ \, T \in \lieh \, $)   --- in short, it is the identity on generators.  Moreover, thanks to  Proposition \ref{prop: gen's-FoMpQUEA => def-FoMpQUEA}  this is in fact an  \textsl{epi\/}morphism.  As an application of this result, there exists also an epimorphism of topological Hopf superalgebras
 $ \; \ell' : \uRpPhg \relbar\joinrel\relbar\joinrel\relbar\joinrel\relbar\joinrel\longrightarrow
 {\big(\, U_{\!P_{(\chi)},\,\hbar}^{\R_{(\chi)},\,p}(\lieg) \big)}_{\sigma_{-\chi}} \; $
 which again is the identity on generators   --- just replace  $ \chi $  with  $ -\chi \, $  and  $ P $  with  $ P_{(\chi)} \, $.
                                                                           \par
 Mimicking what we did for  $ \uRpPhg \, $,  \,we can construct, out of  $ \chi \, $,  a normalized Hopf  $ 2 $--cocycle  $ {\dot\sigma}_\chi $  for  $ {\big(\, U_{\!P_{(\chi)},\,\hbar}^{\R_{(\chi)},\,p}(\lieg) \big)}_{\sigma_{-\chi}} \, $;  \,then we also have a similar  $ 2 $--cocycle  $ \sigma'_\chi $  on  $ \uRpPhg $  defined as the pull-back of  $ {\dot\sigma}_\chi $  via  $ \ell' \, $,  \,and a unique, induced Hopf algebra homomorphism
 $ \; \ell'_{{\dot{\sigma}}_\chi} : {\big( \uRpPhg \big)}_{\sigma'_\chi} \!
 \relbar\joinrel\relbar\joinrel\relbar\joinrel\relbar\joinrel\longrightarrow
 {\Big( {\Big(\, U_{\!P_{(\chi)},\,\hbar}^{\R_{(\chi)},\,p}(\lieg) \Big)}_{\sigma_{-\chi}} \Big)}_{{\dot{\sigma}}_\chi} \; $
 between deformed Hopf algebras, which is once more the identity on generators.  Now, tracking the whole construction we see that  $ \, \sigma'_\chi = \sigma_\chi \, $,  so that
 $ \; {\big( \uRpPhg \big)}_{\sigma'_\chi} = {\big( \uRpPhg \big)}_{\sigma_\chi} \; $,  \,and
 $ \; {\Big( {\Big(\, U_{\!P_{(\chi)},\,\hbar}^{\R_{(\chi)},\,p}(\lieg) \Big)}_{\sigma_{-\chi}} \Big)}_{{\dot{\sigma}}_\chi} \!\! = \, U_{\!P_{(\chi)},\,\hbar}^{\R_{(\chi)},\,p}(\lieg) \; $.  Then composition gives homomorphisms
  $$  \displaylines{
   \ell'_{{\dot{\sigma}}_\chi} \!\circ \ell \, : \, U_{\!P_{(\chi)},\,\hbar}^{\R_{(\chi)},\,p}(\lieg) \relbar\joinrel\relbar\joinrel\longrightarrow
   {\big( \uRpPhg \big)}_{\sigma_\chi} \relbar\joinrel\relbar\joinrel\longrightarrow U_{\!P_{(\chi)},\,\hbar}^{\R_{(\chi)},\,p}(\lieg)  \cr
   \ell \circ \ell'_{{\dot{\sigma}}_\chi} : \, {\big( \uRpPhg \big)}_{\sigma_\chi} \relbar\joinrel\relbar\joinrel\longrightarrow
   U_{\!P_{(\chi)},\,\hbar}^{\R_{(\chi)},\,p}(\lieg) \relbar\joinrel\relbar\joinrel\longrightarrow {\big( \uRpPhg \big)}_{\sigma_\chi}  }  $$
 which (both) are the identity on generators: hence in the end we get
 $ \; \ell'_{{\dot{\sigma}}_\chi} \!\!\circ \ell \, = \, \textsl{id} \; $
 and  $ \; \ell \circ \ell'_{{\dot{\sigma}}_\chi} \! = \, \textsl{id} \; $,  \;thus in particular
 $ \, \ell \, $  is an isomorphism, \,q.e.d.
 \vskip7pt
   We are now left with the task of proving the  $ \underline{\text{\sl Fact\/}} $  above, i.e.\ a bunch of relations.
 \vskip5pt
   Indeed, most relations follow at once from the formulas in  \S \ref{formulas_deform-prod}.
   Namely, the identities  $ \; T' \raise-1pt\hbox{$ \, \scriptstyle \dot\sigma_\chi $} T'' = T' \, T'' \, $  imply
   $ \; T' \raise-1pt\hbox{$ \, \scriptstyle \dot\sigma_\chi $} T'' = T'' \raise-1pt\hbox{$ \, \scriptstyle \dot\sigma_\chi $} T' \, $
   --- for all $ \, T' , T'' \in \lieh \, $.  Also, from
 $ \,\; T \raise-1pt\hbox{$ \, \scriptstyle \dot\sigma_\chi $} E_j  \, = \,  T \, E_j \, + \, 2^{-1} \chi\big(T,T_j^+\big) \, E_j \;\, $
and
 $ \,\; E_j \raise-1pt\hbox{$ \, \scriptstyle \dot\sigma_\chi $} T  \, = \,  E_j \, T \, + \, 2^{-1} \chi\big(T_j^+,T\big) \, E_j \;\, $
   --- for all  $ \, T \in \lieh \, $  and  $ \, j \in I \, $  ---   we get (since  $ \chi $  is antisymmetric)
  $$  T \raise-1pt\hbox{$ \, \scriptstyle \dot\sigma_\chi $} E_j \, -  \, E_j \raise-1pt\hbox{$ \, \scriptstyle \dot\sigma_\chi $} T  \,\;
  = \;  \Big(\, \alpha_j(T) + 2^{-1} \big(\, \chi - \chi^{\,\scriptscriptstyle T} \big)\big(\, T \, , T_j^+ \big) \!\Big) \, E_j  \,\; = \;  +\alpha^{(\chi)}_j(T) \, E_j  $$
 \vskip3pt
\noindent
 A similar, straightforward analysis also yields
 $ \; T \raise-1pt\hbox{$ \, \scriptstyle \dot\sigma_\chi $}  F_j \, -
 \, F_j \raise-1pt\hbox{$ \, \scriptstyle \dot\sigma_\chi $} T \, = \, -\alpha^{(\chi)}_j(T) \, F_j \; $.
                                                                   \par
   The identities  $ \; E_i \raise-1pt\hbox{$ \, \scriptstyle \dot\sigma_\chi $} F_j \, = \, E_i \, F_j \; $  and  $ \; F_j \raise-1pt\hbox{$ \, \scriptstyle \dot\sigma_\chi $} E_i \, = \, F_j \, E_i \; $  in turn imply
  $$  E_i\raise-1pt\hbox{$ \, \scriptstyle \dot\sigma_\chi $} F_j \, - \, {(-1)}^{p(i)\,p(j)}
  F_j \raise-1pt\hbox{$ \, \scriptstyle \dot\sigma_\chi $} E_i  \; = \;  E_i \, F_j \, - \, {(-1)}^{p(i)\,p(j)} F_j \, E_i  \; = \;  {{\; \delta_{i,j} \big( e^{+\hbar \, T_i^+} - \, e^{-\hbar \, T_i^-} \big) \;} \over {\;\; e^{+\hbar\,p_{i{}i}/2} \, - \, e^{-\hbar\,p_{i{}i}/2} \;\;}}  $$
 Eventually, since  $ \, p_{i{}i} = p_{i{}i}^{(\chi)} \, $  by definition,
 and the exponential of toral elements with respect to  $ \raise-1pt\hbox{$ \, \scriptstyle \dot\sigma_\chi $} $
 is the same as with respect to the old product, we conclude that
  $$  E_i \raise-1pt\hbox{$ \, \scriptstyle \dot\sigma_\chi $} F_j \, - \, {(-1)}^{p(i)\,p(j)}
  \, F_j \raise-1pt\hbox{$ \, \scriptstyle \dot\sigma_\chi $} E_i  \,\; = \;\,
 \delta_{i,j} \, {{\; e_{\sigma_\chi}^{+\hbar \, T_i^+} \, -
 \, e_{\sigma_\chi}^{-\hbar \, T_i^-} \;} \over {\; e^{+\hbar\,p_{i{}i}^{(\chi)}/2} \, - \, e^{-\hbar\,p_{i{}i}^{(\chi)}/2} \;}}  $$
where  $ \, e_{\sigma_\chi}^{\,X} \, $  denotes the exponential of any  $ X $  with respect to  $ \hbox{$ \, \scriptstyle \dot\sigma_\chi $} \, $   --- which we already noticed that is the same as the exponential with respect to the original product.
 \vskip7pt
   What is less trivial is proving the other relations (for the deformed product); we do this only for the relations involving the  $ E_i $'s,  those for the  $ F_i $'s  being similar.
                                                                   \par
   We begin with ``quantum Serre relations'', i.e.\ those of the form
  $$  \sum\limits_{s=0}^{1-a_{ij}} \, {(-1)}^s {\bigg[ {{1-a_{ij}} \atop s}  \bigg]}_{\!q_i} {\big( k_{ij}^{(\chi)} \big)}^{\!s} \, E_i^{{(1-a_{ij}-s)}_{\sigma_{\chi}}} \raise-1pt\hbox{$ \, \scriptstyle \dot\sigma_\chi $} E_j \raise-1pt\hbox{$ \, \scriptstyle \dot\sigma_\chi $} E_i^{{(s)}_{\sigma_{\chi}}}  \; = \;  0  $$
 for all  $ \, i \neq j \, $  in  $ I $,  with  $ \, p(i) = \zero \, $:  \,hereafter we set
 $ \; q_{ij}^{(\chi)} \! := e^{\hbar\, p_{ij}^{(\chi)}} \! = e^{\hbar\, (p_{ij} + \mathring{\chi}_{ij} )} \; $
 and
 $ \; k_{ij}^{(\chi)} := {\big( q_{ij}^{(\chi)} \big)}^{\!+1/2} {\big( q_{ji}^{(\chi)} \big)}^{\!-1/2} \; $  ($ \, i , j \in I \, $)
 noting that  $ \, q_{ii}^{(\chi)} = q_{ii} \, $  and  $ \, q_i^{(\chi)} = e^{+\hbar\,p_{i{}i}^{(\chi)}/2} = q_i \, $  ($ \, i \in I \, $),  since  $ \, \big(P_{(\chi)}\big)_{s} = P_s = DA \, $;  moreover, we set also  $ \, \kappa_{ij} := e^{\hbar \, \mathring{\chi}_{ij}} \, $,  noting then that  $ \, q_{ij}^{(\chi)} = q_{ij}\,\kappa_{ij} \, $  ($ \, i , j \in I \, $).
                                                                   \par
   To prove the above identity, we analyze all factors in the summands separately.
 \vskip9pt
\noindent
%
%
 $ \underline{\text{\sl Claim 1}} \, $:  For  $ \, i \neq j \in I \, $  we have
  $ \; k_{ij}^{(\chi)} = \,  k_{ij} \, \kappa_{ij} \, $,
 \;hence  $ \; {\big( k_{ij}^{(\chi)} \big)}^{\!s} \! =  k_{ij}^{\,s} \, \kappa_{ij}^{\,s} \; $  for  $ \, s \in \NN \, $.
 \vskip9pt
   This follows by direct computation.  Next claim instead follows from  \S \ref{formulas_deform-prod}:
 \vskip9pt
\noindent
 $ \underline{\text{\sl Claim 2}} \, $:  Fix  $ \, i \neq j \, $  in  $ I $,  $ \, p(i) = \zero \, $,  and write  $ \, m := 1-a_{ij} \, $.  Then
  $$  E_i^{(m-s)_{\sigma_{\chi}}} \raise-1pt\hbox{$ \, \scriptstyle \dot\sigma_\chi $}
  E_j \raise-1pt\hbox{$ \, \scriptstyle \dot\sigma_\chi $} E_i^{(s)_{\sigma_{\chi}}}  \;
  = \; \sigma_\chi\big( K_i^{m-s} , K_j \big) \, \sigma_\chi\big( K_i^{m-s} K_j \, , K_i^s \big) \, E_i^{m-s} E_j \, E_i^s  $$
 \vskip3pt
   Now we evaluate the value of the toral  $ 2 $--cocycle  using the exponentials.
 \vskip9pt
\noindent
 $ \underline{\text{\sl Claim 3}} \, $:  For all  $ \, i , j \in I \, $  and  $ \, m, s, \ell \in \NN \, $,  \,we have
 \vskip7pt
\begin{enumerate}
  \item[\textit{(a)}]  $ \;\; \sigma_\chi\big( K_i^\ell , K_j \big)  \; = \;  \kappa_{ij}^{\,\ell/2} $
  \item[\textit{(b)}]  $ \;\; \sigma_\chi\big( K_i^{m-s} K_j \, , K_i^s \big)  \; = \;  \kappa_{ij}^{\,s/2} $
  \item[\textit{(c)}]  $ \;\; \sigma_\chi\big( K_i^{m-s} , K_j \big) \,
  \sigma_\chi\big( K_i^{m-s} K_j \, , K_i^s \big)  \; = \;  \kappa_{ij}^{\,(m-2s)/2} $
\end{enumerate}
 \vskip7pt
\noindent
 All assertions follow by computation using  Lemma \ref{lem: chi_U(K,K)}. Indeed, for   \textit{(a)\/}  we have
  $$  \sigma_\chi\big( K_i^\ell , K_j \big)  \; = \;  \overline{\chi}_{\scriptscriptstyle U}\big( K_i^\ell , K_j \big)  \; = \;  e^{\hbar \,
 2^{-1} \chi(\ell \, T_i^+ , T_j^+) }  \; = \;  \kappa_{ij}^{\,\ell/2}  $$
 For item  \textit{(b)},  \,Lemma \ref{lem: chi_U(K,K)}  yields
  $$  \displaylines{
   \quad   \sigma_\chi\big( K_i^{m-s} K_j \, , K_i^s \big)  \; = \;
   \overline{\chi}_{\scriptscriptstyle U}\big( K_i^{m-s} K_j \, , K_i^s \big)  \; = \;
   \overline{\chi}_{\scriptscriptstyle U}\big(\, e^{\hbar \, ((m-s) \, T_i^+ + T_j^+)} , e^{\hbar\, s\, T_i^+} \,\big)  \; =   \hfill   \cr
   \hfill   = \;  e^{\hbar \, 2^{-1} \chi  ((m-s) \, T_i^+ + T_j^+ , \, s \, T_i^+)}  \; = \;  \kappa_{ij}^{\,s/2}   \quad  }  $$
 Now, putting altogether  \textit{(a)\/}  and  \textit{(b)\/}  we eventually get  \textit{(c)}.
%
%
 \vskip11pt
   Finally,  \textsl{Claims 1},  \textsl{2\/}  and  \textsl{3\/}  altogether yield, for  $ \, m := 1 - a_{ij} \, $,
  $$  \displaylines{
   \quad   \sum\limits_{s=0}^m {(-1)}^s {\bigg[ {m \atop s} \bigg]}_{\!q_i} {\big( k_{ij}^{(\chi)} \big)}^{\!s} \, E_i^{{(m-s)}_{\sigma_{\chi}}} \raise-1pt\hbox{$ \, \scriptstyle \dot\sigma_\chi $}
   E_j \raise-1pt\hbox{$ \, \scriptstyle \dot\sigma_\chi $} E_i^{{(s)}_{\sigma_{\chi}}}  \; =   \hfill  \cr
   \quad \qquad \qquad \qquad   = \;  \sum\limits_{s=0}^m {(-1)}^k {\bigg[ {m \atop s} \bigg]}_{\!q_i} k_{ij}^{\,s} \, \kappa_{ij}^{\,s} \, \kappa_{ij}^{\,(m-2s)/2} \, E_i^{m-s} E_j \, E_i^s  \; =   \hfill  \cr
   \quad \qquad \qquad \qquad \qquad \qquad \qquad   = \;  \kappa_{ij}^{\,m/2} \, \sum\limits_{s=0}^m {(-1)}^s
   {\bigg[ {m \atop s} \bigg]}_{\!q_i} k_{ij}^{\,s} \, E_i^{m-s} E_j \, E_i^s  \; = \;  0   \hfill  }  $$
 where the last equality comes from the quantum Serre relation w.r.t.\ the old product.
 \vskip9pt
   Finally, we have to check relations  \eqref{eq: rel1}  through  \eqref{eq: rel6}.  To this end, like for the proof of  Proposition \ref{prop: deform-Yamane=FoMpQUESA}   -- where one also had to check relations, in a very similar manner ---   we observe that the relations under scrutiny involve iterated  $ q $--supercommutators,  so the key step is to check how a single  $ q $--supercommutator  change under  polar--$ 2 $--cocycle  deformation.  In this respect,  \textsl{we fix notation  $ \, {[X\,,Y]}_\mu^\sigmachi \, $  to denote the  $ \mu $---supercommutator  between  $ X $  and  $ Y $  with respect to the deformed product  ``$ \, \sigmachi \, $''}.  Then direct computation, using formulas in  \S \ref{formulas_deform-prod},  gives  (for all  $\, i , j , k \in \, I $  and for all  $ \, \nu , \mu \in \kh \, $)
  $$  \begin{aligned}
   &  \quad   {[E_i\,,E_j]}_\nu  \; = \;  E_i \, E_j - \nu \, {(-1)}^{p(i) p(j)} E_j \, E_i  \; =   \hfill  \\
   &  \quad \quad   = \;  \kappa_{ji}^{+1/2} \, E_i \sigmachi E_j - \nu \, {(-1)}^{p(i) p(j)} \kappa_{ij}^{+1/2} \, E_j \sigmachi E_i  \; =   \hfill  \\
   &  \quad \quad \quad   = \;  \kappa_{ij}^{-1/2} \, E_i \sigmachi E_j - \nu \, {(-1)}^{p(i) p(j)} \kappa_{ij}^{+1/2} \, E_j \sigmachi E_i  \; =   \hfill  \\
   &  \quad \quad \quad \quad   = \;  \kappa_{ij}^{-1/2} \, \big( E_i \sigmachi E_j - \nu \, {(-1)}^{p(i) p(j)} \kappa_{ij} \, E_j \sigmachi E_i \big)  \; = \;  \kappa_{ij}^{-1/2} \, {[E_i\,,E_j]}^\sigmachi_{\nu \,\kappa_{ij}}
      \end{aligned}  $$
 thus in short
\begin{equation}  \label{eq: cocyclebracketord2}
  {[E_i\,,E_j]}_\nu  \; = \;  \kappa_{ij}^{-1/2} \, {[E_i\,,E_j]}^\sigmachi_{\nu \,\kappa_{ij}}
\end{equation}
 Similarly, iterated  $ q $--supercommutators  in  $ \uRpPhg $  can be rewritten as suitable multiples of similar, iterated  $ q $--supercommutators  in  $ {\Big( \uRpPhg \!\Big)}_{\!\!\sigmachi} $,  giving rise to ``higher order analogues'' of  \eqref{eq: cocyclebracketord2}.  Therefore, all relations from  \eqref{eq: rel1}  to  \eqref{eq: rel6}  for  $ {\Big( \uRpPhg \!\Big)}_{\!\!\sigmachi} $  are deduced from the equivalent relations for  $ \uRpPhg $  via repeated applications of  \eqref{eq: cocyclebracketord2}  or of its higher order generalizations; indeed, the same applies to the quantum Serre relations as well, that we already disposed of independently.
                                                                         \par
   For instance, the first of these higher order analogues comes from
\begin{equation}  \label{eq: change-cocycle-bracket-ord_3}
  \begin{aligned}
     &  {\big[[E_i\,,E_j]_\nu \, , E_k\big]}_\mu  \; = \;
%
%
 \big( E_i \, E_j - \nu \, {(-1)}^{p(i)p(j)} E_j \, E_i \big) \, E_k \, -  \\
     &  \;  - {(-1)}^{(p(i)+p(j)) \, p(k)} \, \mu \, E_k \, \big( E_i \, E_j - \nu \, {(-1)}^{p(i)p(j)} E_j \, E_i \big)  \; =  \\
     &  \hskip9pt   = \,  \Big( \kappa_{ji}^{+1/2} \kappa_{kj}^{+1/2} \kappa_{ki}^{+1/2} E_i \sigmachi E_j -
 \nu \, {(-1)}^{p(i)p(j)} \kappa_{ij}^{+1/2} \kappa_{kj}^{+1/2} \kappa_{ki}^{+1/2}E_j \sigmachi E_i \Big) \sigmachi E_k \, -  \\
     &  \hskip65pt   - {(-1)}^{(p(i)+p(j)) \, p(k)} \mu \, E_k \sigmachi \Big( \kappa_{ik}^{+1/2} \kappa_{jk}^{+1/2} \kappa_{ji}^{+1/2} E_i \sigmachi E_j \, -  \\
     &  \hskip145pt   - \nu \, {(-1)}^{p(i)p(j)} \kappa_{jk}^{+1/2} \kappa_{ik}^{+1/2} \kappa_{ij}^{+1/2}E_j \sigmachi E_i \Big)  \; =  \\
     &  \hskip169pt   = \; \kappa_{ji}^{+1/2} \kappa_{kj}^{+1/2} \kappa_{ki}^{+1/2} \, {\big[
{[E_i\,,E_j]}_{\nu \, \kappa_{ij}}^\sigmachi , E_k \big]}
_{\mu \, \kappa_{jk} \, \kappa_{ik}}^\sigmachi
  \end{aligned}
\end{equation}
 which eventually yields the ``third order'' identity
\begin{equation}  \label{eq: change-cocycle-bracket-ord_3 - sum up}
  {\big[[E_i\,,E_j]_\nu \, , E_k\big]}_\mu  \; = \; \kappa_{ji}^{+1/2} \kappa_{kj}^{+1/2} \kappa_{ki}^{+1/2} \, {\big[ {[E_i\,,E_j]}_{\nu \, \kappa_{ij}}^\sigmachi , E_k \big]}
_{\mu \, \kappa_{jk} \, \kappa_{ik}}^\sigmachi
\end{equation}
   \indent   To see how these apply, the easiest case is relation  \eqref{eq: rel1}  for  $ {\Big( \uRpPhg \!\Big)}_{\!\!\sigmachi} $,  which is deduced from  \eqref{eq: rel1}  for  $ \uRpPhg $  applying  \eqref{eq: cocyclebracketord2}  as follows
  $$  0  \; = \;  {[E_i\,,E_j]}_{k_{ij}}  \, = \;  \kappa_{ij}^{-1/2} \, {[E_i\,,E_j]}^\sigmachi_{k_{ij} \, \kappa_{ij}}  \,\;\quad \Longrightarrow \quad\;\,  {[E_i\,,E_j]}^\sigmachi_{k_{ij}^{(\chi)}} \, = \; 0  $$
 where we noticed that  $ \; k_{ij}^{(\chi)} := {\big( q_{ij}^{(\chi)} \big)}^{\!+1/2} {\big( q_{ji}^{(\chi)} \big)}^{\!-1/2} = k_{ij} \, \kappa_{ij} \; $  (cf.\ Definition \ref{def: params-x-Uphgd}).
                                                              \par
   Another, less trivial example is that of  \eqref{eq: rel4}  for  $ {\Big( \uRpPhg \!\Big)}_{\!\!\sigmachi} $:  we deduce it from the parallel relation  \eqref{eq: rel4}  for  $ \uRpPhg $  by direct application of  \eqref{eq: change-cocycle-bracket-ord_3 - sum up}  in which we fix  $ \, (i\,,j\,,k) := (N-2\,,N-1\,,N) \, $,  $ \, (\nu\,,\mu) := (\nu_{N-1}\,,\nu_N) \, $  in the first addendum, and  $ \, (i,j,k) := (N-2\,,N\,,N-1) \, $,  $ \, (\nu\,,\mu) := (\nu_{N-1}\,,\nu_N) \, $  in the second addendum.
 \vskip5pt
   All other cases are entirely similar, so we leave them to the reader.
\epf

\vskip9pt

\begin{rmk}  \label{rmk: about cocycle deform.'s}
 With notation of  Theorem \ref{thm: 2cocdef-uPhgd=new-uPhgd}  above, we have
 $ \; P_{(\chi)} - P \, = \, \varLambda \; $
 for some  {\sl antisymmetric\/}  matrix  $ \, \varLambda \in \lieso_n\big(\kh\big) \, $.
 Conversely, under mild assumptions on  $ P $,
 this result can be ``reversed'' as it is shown below.
\end{rmk}

\vskip7pt

\begin{theorem}  \label{thm: double FoMpQUESAs_P-P'_mutual-deform.s}
 Let  $ \, P, P' \in M_n\big(\kh\big) \, $  be two matrices of Cartan type with the same associated Cartan matrix  $ \, A \, $.
 \vskip5pt
   (a)\,  Let  $ \, \R $  be a  \textsl{split}  realization of  $ \, P $  and  $ \, \uRpPhg $  be the associated FoMpQUESA.  Then there exists a \textsl{split}  realization  $ \check{\R}' $  of  $ \, P' $, a matrix  $ \, \mathring{X} = {\big(\, \mathring{\chi}_{i{}j} \big)}_{i, j \in I} \in \lieso_n\big(\kh\big)\, $  and a toral\/  polar--$ 2 $--cocycle  $ \, \sigma_\chi $ such that
  $$  U_{\!P'\!,\,\hbar}^{\,\R',\,p}(\lieg)  \,\; \cong \;  {\big(\, U_{\!P,\,\hbar}^{\,\R,\,p}(\lieg) \big)}_{\sigma_\chi}  $$
   \indent   In short, if  $ \, P'_s = P_s \, $  then from any  \textsl{split}  FoMpQUESA over  $ P $
   we can obtain a split FoMpQUESA (of the same rank) over  $ P' $  by a toral polar--2--cocycle deformation.
 \vskip5pt
   (b)\,  Let  $ \, \R $  be a  \textsl{split minimal}  realization of  $ P $.
   Then the FoMpQUESA  $ \, \uRpPhg $  is isomorphic to a toral polar--2--cocycle deformation of Yamane's standard double QUESA, that is
 there exists some bilinear map  $ \; \chi \in \text{\sl Alt}_{\,\kh}^{\,S}(\lieh) \; $  such that
  $$  \uRpPhg  \,\; \cong \;  {\big(\, U_{P_A,\,\hbar}^{\,\R_{st},\,p}(\lieg) \big)}_{\sigma_\chi}  $$
 where  $ \, U_{P_A,\,\hbar}^{\,\R_{st},\,p}(\lieg) \, $  is the ``quantum double'' Yamane's QUESA of  Example \ref{example-Yamane's_FoQUEASA's}\textit{(b)}.

 \vskip3pt
   (c)\,  Similar, parallel statements hold true for the Borel FoMpQUESAs.
\end{theorem}

\pf
   \textit{(a)}\,  By  Proposition \ref{prop: mutual-2-cocycle-def}\textit{(a)},  there exists
   $ \; \chi \in \text{\sl Alt}^S_{\,\kh}( \lieh ) \; $  such that
   $ \, P' = P_{(\chi)} \, $  and the triple   --- constructed as in  \S \ref{subsec: 2-cocycle-realiz}  ---
 $ \, \R' \, := \, \R_{(\chi)} \, = \, \big(\, \lieh \, , \Pi_{(\chi)} \, , \Pi^\vee \,\big) \, $
 is a split realization of  $ \, P' = P_{(\chi)} \, $.  Then  $ \; U_{\!P'\!,\hbar}^{\,\R',\,p}(\lieg) \, \cong \, {\big(\, \uRpPhg \big)}_{\sigma_\chi} \; $
 by  Theorem \ref{thm: 2cocdef-uPhgd=new-uPhgd}.
 \vskip5pt
   \textit{(b)}\,  Yamane's  $ U_{P_A,\,\hbar}^{\,\R_{st},\,p}(\lieg) $  is, in our language, nothing but the FoMpQUESA built upon a straight split realization  $ \; \R_{st} \,= \, \big(\, \lieh \, , \Pi_{st} \, , \Pi_{st}^\vee \,\big) \; $  of  $ DA \, $,  \,for which we write  $ \, \Pi_{st}^\vee = {\big\{\, T_i^\pm \,\big\}}_{i \in I} \, $  and  $ \, \Pi_{st} = {\big\{\, \alpha_i^{(st)} \,\big\}}_{i \in I} \, $.  From  Proposition \ref{prop: mutual-2-cocycle-def}\textit{(b)\/}  we have a suitable  $ \, \chi \in  \text{\sl Alt}_{\,\kh}^{\,S}(\lieh) \; $  such that the realization  $ \big(\R_{st}\big)_{(\chi)} $  obtained as deformation of  $ \R_{st} $  via  $ \chi $  coincides with  $ \R \, $;  \,then  Theorem \ref{thm: 2cocdef-uPhgd=new-uPhgd}  gives  $ \; \uRpPhg \, \cong \, {\big(\, U_{P_A,\,\hbar}^{\,\R_{st},\,p}(\lieg) \big)}_{\sigma_\chi} \, $.
\epf

\vskip3pt

\begin{free text}  \label{gener_quantum-GaGa-alg.'s}
 \textbf{Generalizations.}  The whole analysis and construction devised in this section can be also applied,  \textit{verbatim},  to the setup of QUESAs over  \textsl{affine\/}  Lie superalgebras, namely those introduced in \cite{Ya2}  and  \cite{Ya3}.  Much like in the semiclassical case  (cf.\ \S \ref{gener_GaGa-alg.'s})  one has to handle many more defining relations, keep track of their ```mutations'' under deformation, etc.: nevertheless, all this is a standard task, while the strategy stands exactly the same.  On the other hand, it is worth stressing that the key ideas underlying our recipe also apply (when suitably adapted) to a much larger framework, including concrete examples: we work on this in  \cite{GG6}.
\end{free text}

\bigskip
 \medskip

\section{Quantization and specialization: FoMpQUESAs vs. MpLSbAs}  \label{sec: FoMpQUESA's-vs-MpLSbA's}
 \vskip7pt
 \vskip1pt
   In this section we consider the the interplay of specialization
   --- performed onto quantum objects, namely our FoMpQUESAs ---
   and, conversely, of quantization   --- applied to semiclassical objects, namely our MpLSbA's.
                                                                \par
   To begin with, we shall see that every FoMpQUESA yields by
   specialization a suitable MpLSbA; conversely, for any MpLSbA
   there exists at least one quantization, in the form of a well-defined FoMpQUESA.
   Second, we shall study the interaction between the process of specialization (at  $ \, \hbar = 0 \, $)
   of any FoMpQUESA and the process of deformation
   --- either by (toral) twist or by (toral)  $ 2 $--cocycle  ---
   of that same FoMpQUEA or of the MpLSbA which is its semiclassical limit.
   In particular, we will find that, in a natural sense,
   \textsl{the two processes commute with each other}.

\medskip

\subsection{Deformation vs.\ specialization for FoMpQUESA's and MpLSbA's}
\label{subsec: def-vs.-spec-x-FoMpQUESA's&MpLSbA's} {\ }

\medskip

   In  \S \ref{sec: Hopf-setup & QUEA's}  we recalled that any QUESA defines,
   by specialization at  $ \, \hbar = 0 \, $,  a well-defined ``semiclassical limit'',
   given by a specific Lie superbialgebra   --- or more precisely by the universal
   enveloping superalgebra of the latter.  In addition, by
   Theorem \ref{thm: FoMpQUESAs-are-QUESAs}  we know that our FoMpQUESA's are
   indeed QUESA's, so they do have semiclassical limits in the previous sense.
   We shall now go and see what these limits look like, i.e.\ what Lie superalgebras we do find in this way.

\vskip11pt

\begin{free text}
 {\bf Formal MpQUESA's vs.\ MpLSbA's.}
 Let us fix a Cartan super-datum  $ \, (A\,,p) \, $,  a matrix
 $ \, P := {\big(\, p_{i,j} \big)}_{i, j \in I} \in M_n\big( \kh \big) \, $
 of Cartan type with associated Cartan matrix  $ A \, $,
 and a realization  $ \, \R := \big(\, \lieh \, , \Pi \, , \Pi^\vee \,\big) \, $  of  $ P $.
 Out of these data, we associate the FoMpQUEA  $ \uRpPhg \, $,
 as in  \S \ref{subsec: FoMpQUESA's},  and also the MpLSbA
 $ \lieg^{{\bar{\Rpicc}},p}_{\bar{\Ppicc}} $  introduced in
 \S \ref{subsec: MpLSbA's},  where
%
%
 we use the somewhat loose (yet obvious) notation such
 $ \, \, \bar{\R} := \R \; (\,\text{mod} \; \hbar \,) \, $  and
 $ \, \bar{P} := P \; (\,\text{mod} \; \hbar \,) \, $.

\vskip5pt

   We begin by pointing out that, in a nutshell, FoMpQUESA's and
   MpLSbA's are in bijection through the specialization/quantization process,
   as one might expect:
%
%
\end{free text}

\vskip9pt

\begin{theorem}  \label{thm: semicl-limit FoMpQUESA}
 With assumptions as above,  $ \uRpPhg $  is a  \textsl{QUESA}
 in the sense of  \S \ref{QUESA's},  whose semiclassical limit is
 $ U\big(\lieg^{{\bar{\Rpicc}},\,p}_{\bar{\Ppicc}}\,\big) \, $.
 In other words,  \textsl{$ \lieg^{{\bar{\Rpicc}},\,p}_{\bar{\Ppicc}} $
 is the specialization of  $ \, \uRpPhg \, $,  or   --- equivalently ---
 $ \uRpPhg $  is a quantization of  $ \lieg^{{\bar{\Rpicc}},\,p}_{\bar{\Ppicc}} \, $}.
\end{theorem}

\pf
 By  Theorem \ref{thm: FoMpQUESAs-are-QUESAs}  we know that
 $ \uRpPhg $  is a QUESA, so only have to describe its semiclassical limit.
 In fact, we aim to prove that  $ \; \uRpPhg \big/ \hbar \, \uRpPhg \, $,
 \,as a co-Poisson Hopf superalgebra, is isomorphic to
 $ U\big( \lieg^{{\bar{\Rpicc}},\,p}_{\bar{\Ppicc}} \big) \, $.
                                                                 \par
   First of all, the presentation of  $ \uRpPhg $  implies that
   $ \; \uRpPhg \Big/ \hbar \, \uRpPhg \, $  is generated by the cosets
   (modulo  $ \hbar \, \uRpPhg \, $)  of the  $ F_i $'s,  $ T $'s  and
   $ E_i $'s  ($ \, i \in I \, $,  $ \, T \in \lieh \, $);
   moreover, these cosets  $ \; \overline{X} := X \!\! \mod \hbar\,\uRpPhg \; $
   obey all relations induced modulo  $ \hbar $
   by the defining relations among the original generators  $ X $  of  $ \uRpPhg \, $.
   On the other hand, the Lie superbialgebra  $ \lieg^{{\bar{\Rpicc}},\,p}_{\bar{\Ppicc}} $
   is endowed, by construction, with a presentation (as a Lie superalgebra) by generators
   (the  $ F_i $'s,  $ T $'s  and  $ E_i $'s)  and relations,
   and explicit formulas for the value of the Lie supercobracket  $ \delta $
   on the given generators.  From this one gets an obvious presentation of
   $ U\big( \lieg^{\bar{\Rpicc}}_{\bar{\Ppicc},\,p} \big) \, $,
   where the generators are again the  $ F_i $'s,  $ T $'s  and  $ E_i $'s
   and the Poisson cobracket  $ \delta $  on them is given by explicit formulas.
                                                            \par
   Comparing the presentation of  $ U\big( \lieg^{{\bar{\Rpicc}},p}_{\bar{\Ppicc}} \big) $
   with that of  $ \; \uRpPhg \big/ \hbar \, \uRpPhg \; $
   we find that all the given relations among generators of the
   latter superalgebra do correspond to identical relations among the
   corresponding generators in the former: indeed, mapping  $ X $  to  $ \overline{X} $
   (for all  $ \, X \in \big\{ E_i \, , F_i \,\big|\, i \in I \,\big\} \cup \lieh \, $)
   turns every given relation among the  $ X $'s
   into a same-look relation among the  $ \overline{X} $'s.
   Eventually, this provides a well-defined epimorphism of Hopf superalgebras
\begin{equation}  \label{eq: iso_Ugbar-U0g}
   \hskip-11pt   \phi \, : \, U\big( \lieg^{{\bar{\Rpicc}},\,p}_{\bar{\Ppicc}} \big) \,\relbar\joinrel\relbar\joinrel\twoheadrightarrow\, \uRpPhg \Big/ \hbar \, \uRpPhg  \;\; ,  \quad
  E_i \,\mapsto\, \overline{E_i} \, ,  \;  T \,\mapsto\, \overline{T} \, ,  \;  F_i \,\mapsto\, \overline{F_i}
\end{equation}
 ($ \, i \in I \, $,  $ \, T \in \lieh \, $);  furthermore, comparing the formulas on both sides for the co-Poisson superbracket
 we see that this is also a  \textsl{co-Poisson\/}  super Hopf epimorphism.
                                                                         \par
   On the other hand, scalar restriction   --- via
   $ \; \kh \relbar\joinrel\relbar\joinrel\twoheadrightarrow \kh \big/ \hbar \, \kh \, \cong \, \Bbbk \; $
   ---   makes  $ U\big( \lieg^{\bar{\Rpicc},\,p}_{\bar{\Ppicc}} \big) $
   into a  $ \kh $--superalgebra.  Then the same argument about
   relations yields a well-defined  $ \kh $--superalgebra  epimorphism
\begin{equation*}  \label{eq: iso_Ug-Ugbar}
   \psi\,{}' : \uRpPhg \relbar\joinrel\relbar\joinrel\twoheadrightarrow
 U\big( \lieg^{{\bar{\Rpicc}},\,p}_{\bar{\Ppicc}} \big)  \;\; ,  \quad  E_i \,\mapsto\, E_i \, ,  \;  T \,\mapsto\, T \, ,  \;  F_i \,\mapsto\, F_i   \;\quad  (\, i \in I \, , \, T \in \lieh \,)
\end{equation*}
 whose kernel contains  $ \, \hbar \, \uRpPhg \, $;
 \,therefore, a  $ \Bbbk $--superalgebra  epimorphism
\begin{equation}  \label{eq: iso_U0g-Ugbar}
   \psi \, : \, \uRpPhg \Big/ \hbar\,\uRpPhg
   \,\relbar\joinrel\relbar\joinrel\twoheadrightarrow\,
   U\big( \lieg^{{\bar{\Rpicc}},\,p}_{\bar{\Ppicc}} \big)  \;\; ,  \quad
   \overline{E_i} \,\mapsto\, E_i \, ,  \; \overline{T}
   \,\mapsto\, T \, ,  \; \overline{F_i} \,\mapsto\, F_i
\end{equation}
 ($ \, i \in I \, $,  $ \, T \in \lieh \, $)
 is induced too.  Now comparing  \eqref{eq: iso_Ugbar-U0g}
 and  \eqref{eq: iso_U0g-Ugbar}  shows that
 $ \phi $  and  $ \psi $  are inverse to each other, so
 $ \psi $  is a  super \textsl{Hopf\/}  morphism too and we are done.
\epf

\medskip

\subsection{Intertwining deformation and specialization}
  \label{subsec: intrtwining-deform-spec}  {\ }
 \vskip7pt
   We shall now compare the process of deformation at the quantum level
   and at the semiclassical level, finding a very neat outcome:
   in a nutshell,  \textit{deformation (by twist or by 2--cocycle)
   \textsl{commutes}  with specialization}.

\vskip11pt

\begin{free text}
 \textbf{Intertwining twist deformation and specialization.}
 Fix again a Cartan super-datum $(A,p)$,
 a matrix  $ \, P := {\big(\, p_{i,j} \big)}_{i, j \in I} \in M_n\big( \kh \big) \, $
 of Cartan type, a realization
 $ \, \R := \big(\, \lieh \, , \Pi \, , \Pi^\vee \,\big) \, $  of it,
 and the associated FoMpQUESA  $ \uRpPhg $  and MpLSbA
 $ \lieg^{{\bar{\Rpicc}},p}_{\bar{\Ppicc}} \, $.
 In the free  $ \kh $--module  $ \lieh $  of finite rank  $ \, t \, $,
 we fix a  $ \kh $--basis  $ \, {\big\{ H_g \big\}}_{g \in \G} \, $,
 with  $ \, |\G| = \rk(\lieh) = t \, $.
 \vskip5pt
   Choose any antisymmetric matrix
   $ \; \Psi = \big( \psi_{gk} \big)_{g, k \in \G} \in \lieso_t\big(\kh\big) \; $.
   Out of it, set
  $$
  \displaylines{
   j_{\bar{\Psipicc}}  \; := \;  {\textstyle \sum_{g,k=1}^t} \,
   \overline{\psi_{gk}} \; \overline{H_g} \otimes \overline{H_k}  \,\; \in \;\,
  \overline{\lieh} \otimes \overline{\lieh}   \qquad \qquad \qquad
  \text{as in  \eqref{eq: def_twist_Lie-bialg_2nd-time}}  \cr
   \F_\Psipicc  \; := \;  \exp\Big( \hbar \; 2^{-1} \,
   {\textstyle \sum_{g,k=1}^t} \psi_{gk} \, H_g \otimes H_k \Big)
   \,\; \in \;\,  \uRpPhg \,\widehat{\otimes}\, \uRpPhg   \qquad
   \text{as in  \eqref{eq: Resh-twist_F-uPhgd}}  }  $$
 so  $ j_{\bar{\Psipicc}} $  is a (toral)  \textsl{twist\/}
 for the  \textsl{Lie superbialgebra\/}
 $ \lieg^{{\bar{\Rpicc}},p}_{\bar{\Ppicc}} $  and  $ \F_\Psipicc $
 is a (toral)  \textsl{twist\/}  for the  \textsl{Hopf superalgebra\/}
 $ \uRpPhg \, $,  we have the deformation
 $ {\big( \lieg^{{\bar{\Rpicc}},p}_{\bar{\Ppicc}} \big)}^{j_\Psipicc} $
 of  $ \lieg^{{\bar{\Rpicc}},p}_{\bar{\Ppicc}} $  by the (Lie) twist
 $ j_{\bar{\Psipicc}} $  and the deformation
 $ {\big(\,\uRpPhg\big)}^{\F_\Psipicc} $  of  $ \uRpPhg $
 \hbox{by the (Hopf) twist  $ \F_\Psipicc \, $.}
 \vskip5pt
   Again out of  $ \Psi \, $,  we define the matrix  $ P_\Psipicc $
   and its realization  $ \, \R_\Psipicc \! := \big(\, \lieh \, , \Pi \, , \Pi^\vee_\Psipicc \big) $,
   \,as in  Proposition \ref{prop: twist-realizations};  \,then we have
   $ U_{\!P_\Psipicc,\hbar}^{\,\R_\Psipicc, p}(\lieg) $  and
   $ \lieg^{\bar{\Rpicc}_\Psipicc,\,p}_{\bar{\Ppicc}_\Psipicc} \, $,
   again mutually linked by a quantization/specialization relationship.
\end{free text}

\vskip3pt

   We can now state our result, which in particular claims (in a nutshell)
   that  \textsl{``deformation by twist commutes with specialization''}.

\vskip11pt

\begin{theorem}  \label{thm: specializ twisted FoMpQUESA}
 $ \Big(\, \uRpPhg \!\Big)^{\F_\Psipicc} $  is a QUESA,
 whose semiclassical limit is isomorphic to
 $ \, U\big( {\big( \lieg^{{\bar{\Rpicc}},p}_{\bar{\Ppicc}} \big)}^{j_{\bar{\Psipicc}}} \big) \, $.
 Indeed,  $ \; \big(\, \uRpPhg \Big)^{\F_\Psipicc} \cong
 \, U_{\!P_{\,\Psipicc},\hbar}^{\,\R_\Psipicc}(\lieg) \;\, $  and
 $ \; {\big( \lieg^{{\bar{\Rpicc}},p}_{\bar{\Ppicc}} \big)}^{j_{\bar{\Psipicc}}} \cong \, \lieg_{\bar{\Ppicc}_\Psipicc,\,p}^{\bar{\Rpicc}_\Psipicc} \;\, $.
\end{theorem}

\pf
 The claim follows, as direct application, from the isomorphisms
  $$  \Big( \uRpPhg \Big)^{\F_\Psipicc}
%
%
 \,\cong\;
  U_{\!P_{\,\Psipicc},\hbar}^{\,\R_\Psipicc}(\lieg)  \;\, ,  \!\quad
   U_{\!P_{\,\Psipicc},\hbar}^{\,\R_\Psipicc}(\lieg) \Big/
   \hbar\,U_{\!P_{\,\Psipicc},\hbar}^{\,\R_\Psipicc}(\lieg)
%
%
 \;\cong\;
U_\hbar\big(\lieg_{\bar{\Ppicc}_\Psipicc}^{\bar{\Rpicc}_\Psipicc}\big)  \;\, ,  \!\quad
   \lieg_{\bar{\Ppicc}_\Psipicc}^{\bar{\Rpicc}_\Psipicc}
 \;\cong\,
%
%
   {\big( \lieg^{\bar{\Rpicc}}_{\bar{\Ppicc}} \big)}^{j_{\bar{\Psipicc}}}  $$
 which come from  Theorem \ref{thm: twist-uRpPhg=new-uRpPhg},
 Theorem \ref{thm: semicl-limit FoMpQUESA}  and
 Theorem \ref{thm: twist-liegRP=new-liegR'P'}.
\epf

\vskip9pt

\begin{rmk}
 The above result is proved by direct inspection;
 it can also be deduced as an application of a general result about
 the ``semiclassical limit'' of a twist of a QUESA, namely
 \cite[Theorem 3.1.2]{GG4}  (the result there is stated for the non-super case,
 but it extends straightforwardly to the super setup).
\end{rmk}

\vskip11pt

\begin{free text}
 \textbf{Intertwining 2--cocycle deformation and specialization.}
 Now we analyse what happens when one combines  \textsl{deformations by 2--cocycle}
 --- for both FoMpQUESAs and MpLSbA's ---   with the specialization process.
                                                                   \par
   We begin with a Cartan super-datum $(A,p)$, a matrix
   $ \, P \! := \! {\big(\, p_{i,j} \big)}_{i, j \in I} \in M_n\big( \kh \big) $  \,
   of Cartan type, a realization
   $ \, \R := \big(\, \lieh \, , \Pi \, , \Pi^\vee \,\big) \, $
   of  $ P $,  and a fixed  $ \kh $--basis  $ \, {\big\{ H_g \big\}}_{g \in \G} \, $
of  $ \lieh \, $,  with  $ \, |\G| = \rk(\lieh) = t \, $.
Then  $ \uRpPhg $  and  $ \lieg^{{\bar{\Rpicc}},p}_{\bar{\Ppicc}} $
are defined, interlocked through quantization/specialization.
 \vskip5pt
   Like in  \S \ref{2-cocycle-U_h(h)},  we fix a  $ \kh $--bilinear
   map  $ \; \chi : \lieh \times \lieh \relbar\joinrel\longrightarrow \kh \; $
   which is antisymmetric and obeys  \eqref{eq: condition-eta}.
   Taking everything modulo  $ \hbar \, $,  this  $ \chi $
   defines a similar antisymmetric,  $ \Bbbk $--bilinear  map
   $ \; \gamma := \big(\, \chi \!\mod \hbar \,\big) : \lieh_0 \times
   \lieh_0 \relbar\joinrel\longrightarrow \Bbbk \; $
   --- with  $ \, \lieh_0 := \lieh \Big/ \hbar \, \lieh \, = \,
   \overline{\lieh} \, $  ---   which obeys  \eqref{eq: condition-eta}
   again, up to replacing  ``$ \, \chi \, $''  with  ``$ \, \gamma \, $''.
   Following  \S \ref{toral-cocycles-uPhgd},  we construct out of
   $ \chi $  a  $ \khp $--valued  toral  polar--$ 2 $--cocycle
   $ \sigma_\chi : \uRpPhg \otimes \uRpPhg \relbar\joinrel\longrightarrow \khp \, $,
   \,and then the deformed Hopf superalgebra  $ \, {\big( \uRpPhg \big)}_{\sigma_\chi} \, $.
   Similarly, we construct as in  \S \ref{subsec: tor-2cocyc def's_mp-Lie_supbialg's}
   (but with  ``$ \, \gamma \, $''  replacing  ``$ \, \chi \, $'')  a toral  $ 2 $--cocycle
   $ \gamma_\lieg $  for the Lie bialgebra  $ \lieg^{{\bar{\Rpicc}},p}_{\bar{\Ppicc}} \, $,
   \,and out of it the  $ 2 $--cocycle  deformed Lie superbialgebra
   $ \, {\big( \lieg^{{\bar{\Rpicc}},p}_{\bar{\Ppicc}} \big)}_{\gamma\raisebox{-1pt}{$ {}_\lieg $}} \, $.
 \vskip3pt
   Still out of  $ \chi \, $,  we define the matrix  $ P_{(\chi)} $
   and its realization  $ \, \R_{(\chi)} := \big(\, \lieh \, , \, \Pi_{(\chi)} \, ,
   \, \Pi^\vee \,\big) \, $,  \,as in  Proposition \ref{prop: 2cocdef-realiz};
   similarly, out of  $ \gamma $  we get the matrix  $ P_{(\gamma)} $
   and its realization  $ \, \R_{(\gamma)} \, $:  then
   $ \, P_{(\gamma)} = \bar{P}_{(\chi)} \, $  and
   $ \, \R_{(\gamma)} = \bar{\R}_{(\chi)} \, $.
   Finally, attached to the latter we have
   $ \, U_{\!P_{(\chi)},\,\hbar}^{\,\R_{(\chi)},p}(\lieg) \, $  and
   $ \, \lieg^{\Rpicc_{(\gamma)},\,p}_{\Ppicc_{(\gamma)}} =
   \lieg^{\bar{\Rpicc}_{(\chi)},\,p}_{\bar{\Ppicc}_{(\chi)}} \, $,
   \,again connected via quantization/specialization.
\end{free text}

\vskip5pt

 Next result claims that  \textsl{``deformation by 2--cocycle commutes with specialization''}.

\vskip13pt

\begin{theorem}  \label{thm: specializ 2-cocycle FoMpQUESA}
 $ {\big(\, \uRpPhg \big)}_{\sigma_\chi} $  is a QUESA, with semiclassical limit  $ \, U\Big(\! {(\liegRpP)}_{\gamma\raisebox{-1pt}{$ {}_\lieg $}} \Big) \, $.
%
%
\end{theorem}

\pf
 The claim follows, as direct application, from the two isomorphisms
  $$  {\big(\, \uRpPhg \big)}_{\sigma_\chi}
%
%
 \!\cong
 U_{\!P_{(\chi)},\,\hbar}^{\,\R_{(\chi)}}(\lieg)  \;\;\; ,  \qquad
       U_{\!P_{(\chi)},\,\hbar}^{\,\R_{(\chi)}}(\lieg) \Big/ \hbar \,
   U_{\!P_{(\chi)},\,\hbar}^{\,\R_{(\chi)}}(\lieg)
%
%
 \;\cong\;
 U_\hbar\Big( \lieg_{\Ppicc_{(\gamma)}}^{\Rpicc_{(\gamma)}} \Big)   $$
 of co-Poisson Hopf superalgebras  (cf.\ Theorem \ref{thm: 2cocdef-uPhgd=new-uPhgd}  and  Theorem \ref{thm: semicl-limit FoMpQUESA})  and the isomorphism of Lie superbialgebras
   $ \; \lieg^{\Rpicc_{(\gamma)},\,p}_{\Ppicc_{(\gamma)}}
%
%
 \;\cong\,
 {\big(\liegRpP\big)}_{\gamma\raisebox{-1pt}{$ {}_\lieg $}} \; $
   --- cf.\ Theorem \ref{thm: 2-cocycle-def-MpLSbA}.
\epf

\vskip5pt

\begin{rmk}
 The previous result is proved by direct inspection; nevertheless, we can deduce it also as an application of a general result on the ``semiclassical limit'' of a polar-2--cocyle of a QUESA, namely  \cite[Theorem 3.3.72]{GG4}  (although the result there is stated for the non-super case, it does extends easily to the super setup).

\end{rmk}

\medskip

\subsection{A global synopsis}  \label{subsec: global synopsis} \

\medskip

 In this paper we studied multiparametric versions of formal QUESA's,
 as well as their semiclassical limits.  All these are presented by generators and relations,
 so their very definition highlights the relation between
 the multiparameters and the action of a fixed commutative subsuperalgebra:
 this is encoded in the notion of  \textsl{realization\/}
 of a multiparameter matrix  $ P $,  much like for Kac-Moody algebras.
 This allows us to relate the quantum objects with their semiclassical limit,
 and also multiparameter objects with standard ones:
 the latter step goes through deformation(s).
 \vskip5pt
   In conclusion we can say, a bit loosely, that:
 \vskip2pt
\begin{enumerate}
  \item[$(a)$]  multiparameters are encoded in realizations;
 \vskip2pt
  \item[$(b)$]  FoMpQUESA's are quantizations of MpLSbA's;
 \vskip2pt
  \item[$(c)$]  multiparameter objects are given by deformation of
  either the (super)algebra or the (super)coalgebra structure of (Yamane's)  \textsl{uniparameter\/}  objects.
\end{enumerate}

\vskip5pt

   Finally, we display a ``pictorial sketch'' of the global picture, as follows:
 \vskip-11pt
%
  $$  \xymatrix{
   U^{\,\R_\Psi,\,p}_{\!P_\Psi,\,\hbar}(\lieg) \, \cong {\Big( \uRpPhg \Big)}^{\F_\Psi}
     \ar@{<~>}[dd]|-(0.5){\text{Theorem \ref{thm: uRpPhg=twist-uhg}}} \!\  &  &
     \ar@{<.>}[ll]_{\hskip0pt \text{Theorem \ref{thm: semicl-limit FoMpQUESA}}}
   \,\ U\Big( \lieg^{\bar{\R}_\Psi,\,p}_{\bar{P}_\Psi} \Big) \, \cong
   \, U\Big(\! {\Big( \lieg^{\bar{\R},\,p}_{\bar{P}} \Big)}^{j_{\bar{\Psi}}} \Big)
     \ar@{<~>}[dd]|-(0.5){\text{Theorem \ref{thm: MpLSbA=twist-Yamane's}}}  \\
      &    &  \\
     \uRpPhg \  \ar@{<~>}[dd]|-(0.5){\text{Theorem \ref{thm: double FoMpQUESAs_P-P'_mutual-deform.s}}}  &  &
     \ar@{<.>}[ll]_{\hskip0pt \text{Theorem \ref{thm: semicl-limit FoMpQUESA}}}
     \  U\Big( \lieg^{\bar{\R},\,p}_{\bar{P}} \Big)
     \ar@{<~>}[dd]|-(0.5){\text{Theorem \ref{thm: 2-cocycle-def-MpLSbA}}}  \\
      &    &  \\
   {\Big( \uRpPhg \Big)}_{\!\sigma_\chi} \! \cong \,
   U^{\,\R_{(\chi)},\,p}_{\!P_{(\chi)},\,\hbar}(\lieg) \,\  &  &
   \ar@{<.>}[ll]_{\hskip-7pt \text{Theorem \ref{thm: semicl-limit FoMpQUESA}}}
    \,\  U\Big(\! {\Big( \lieg^{{\bar{\R}},\,p}_{\bar{P}} \Big)}_{\!\bar{\chi}} \,\Big) \, \cong \, U\Big( \lieg^{{\bar{\R}_{(\chi)}},\,p}_{\bar{P}_{(\chi)}} \,\Big)  }  $$
%
 In the above drawing, each horizontal arrow (with dotted shaft) denotes a ``quantization/specialization (upwards/downwards) relationship''   --- which involves the ``continuous parameter'' $ \hbar $  ---   whereas each vertical arrow (having a waving shaft) stands for a relationship ``via deformation''   --- that involves ``discrete parameters''.

\bigskip
 \medskip

\section{Polynomial multiparameter QUESA's and their deformations}
\label{sec: polynomial MpQUESA's & deform.'s}
{\ }

\vskip1pt

   Our FoMpQUESA's are a close analogue, for multiparameter depending supergroups,
   of the well-known (uniparameter) QUEA's  $ \uhg $  by Drinfeld,
   attached with finite-dimensional semisimple Lie algebras  $ \lieg \, $.
   Now, a  \textsl{polynomial\/}  version of Drinfeld's  $ \uhg $
   is also available, originally due to Jimbo and Lusztig
   --- and a multiparameter version of that is known as well, cf.\ \cite{GG2}.
   Similarly, we will now introduce a ``polynomial version'' of MpQUESA's,
   and we will shortly present its main features.

\vskip15pt

\subsection{Multiparameters of polynomial type}
\label{subsec: multiparam-polynom}
 {\ }
 \vskip9pt
 Given  $ \, n \in \NN_+ \, $  and  $ \, I := \{1,\dots,n\} \, $,
 let  $ \, \ZZ\big[\,\bx^{\pm 1}\big] := \ZZ\big[ {\big\{ x_{ij}^{\pm 1} \big\}}_{i, j \in I} \,\big] \, $
 be the ring of Laurent polynomials with coefficients in  $ \ZZ $
 in the indeterminates  $ x_{ij} $  ($ \, i, j \in I \, $),
 and let  $ \, A := {\big( a_{ij} \big)}_{i, j \in I} \, $  be an indecomposable,
 symmetrizable Cartan matrix as in  \S \ref{subsec: FoMpQUESA's}.
 Consider the quotient ring
%
%
 $ \; \Zabqpm \, := \, \ZZ\big[\,\bx^{\pm 1}\big] \bigg/
 \Big( {\big\{ x_{ij} \, x_{ji} - x_{ii}^{\,a_{ij}} \big\}}_{i, j \in I} \Big) \; $
 and denote by  $ q_{ij} $  the coset of every  $ x_{ij} $
 (for  $ \, i, j \in I \, $)  in  $ \, \Zabqpm \, $.
 This is the ring of global sections of an affine scheme over
 $ \ZZ \, $,  call it  $ \mathfrak{C}_A \, $.
                                                         \par
   From all the identities  $ \; q_{ij} \, q_{ji} = q_{ii}^{\,a_{ij}} \; $
   in  $ \Zabqpm \, $,  from the indecomposability of  $ A $  and
   from the symmetry of  $ \, D\,A \, $   --- i.e.,
   $ \, d_i\,a_{ij} = d_j\,a_{ji} \, $  for all  $ \, i, j \in I \, $  ---
   one finds that there exists  $ \, j_{\raise-2pt\hbox{$ \scriptstyle 0 $}} \in I \, $
   such that  $ \, q_{ii} = q_{j_{\raise-2pt\hbox{$ \scriptscriptstyle 0 $}}
   \, j_{\raise-2pt\hbox{$ \scriptscriptstyle 0 $}}}^{\,d_i} \, $.
   From this and the relations between the  $ q_{ij} $'s, it is easy to
   argue that  $ \mathfrak{C}_A $  is a torus, of dimension
   $ \, {{\,n\,} \choose {\,2\,}} + 1 \, $:  in particular,
   it is irreducible, hence  $ \Zabqpm $  is a domain.
%
%
                                                                   \par
   Finally, we extend the ring  $ \, \Zabqpm \, $
   using formal square roots, as follows.  First, we consider the ring extension
   $ \dotRbq $  of  $ \, \Zabqpm \, $  generated by a square root of
   $ \, q_{j_{\raise-2pt\hbox{$ \scriptscriptstyle 0 $}} \,
   j_{\raise-2pt\hbox{$ \scriptscriptstyle 0 $}}} \, $,  \,namely
  $ \; \dotRbq \, := \, \big(\, \Zabqpm \,\big)[x] \bigg/\! \Big( x^2 - \,
  q_{j_{\raise-2pt\hbox{$ \scriptscriptstyle 0 $}} \,
  j_{\raise-2pt\hbox{$ \scriptscriptstyle 0 $}}} \Big) \; $  so that
  $ \, q := q_{j_{\raise-2pt\hbox{$ \scriptscriptstyle 0 $}} \,
  j_{\raise-2pt\hbox{$ \scriptscriptstyle 0 $}}}^{\,1/2} := \,
  \overline{x} \, \in \, \dotRbq \, $.
 We still denote by  $ q_{ij} $  the images in  $ \dotRbq $
%
%
 of the ``old'' elements with same name in  $ \Zabqpm \, $.
%
%
 Note that
  $ \; \dotRbq \, \cong \, \big(\, \ZZ\big[\, {\big\{ x_{ij}^{\pm 1} \big\}}_{i, j \in I} \,\big] \,\big)[x] \bigg/ \Big(\, {\big\{ x_{ij} \, x_{ji} - x_{ii}^{\,a_{ij}} \big\}}_{i, j \in I} \, \cup \big\{ x^2 - x_{j_{\raise-2pt\hbox{$ \scriptscriptstyle 0 $}} \, j_{\raise-2pt\hbox{$ \scriptscriptstyle 0 $}}} \big\} \Big) \; $.
%
%
 In turn, we define also
  $$
  \dotRbqsq  \; := \;  \ZZ\big[ {\big\{ \xi_{ij}^{\pm 1/2} \big\}}_{i, j \in I}
  \,\big]\big[ \xi^{\pm 1/2} \big] \bigg/
   \Big( {\Big\{ \xi_{ij}^{1/2} \, \xi_{ji}^{1/2} \! -
   {\big( \xi_{i{}i}^{1/2\,} \big)}^{a_{ij}} \Big\}}_{i, j \in I} \! \cup
   \Big\{\! {\big(\xi^{1/2\,}\big)}^2 \! -
   \xi_{j_{\raise-2pt\hbox{$ \scriptscriptstyle 0 $}} \,
   j_{\raise-2pt\hbox{$ \scriptscriptstyle 0 $}}}^{\,1/2} \Big\} \Big)
   $$
 \vskip0pt
\noindent
 which is again a domain,
%
%
 and we denote by  $ q_{ij}^{\pm 1/2} $  and by  $ q^{\pm 1/2} $  the coset of
 $ \xi_{ij}^{\pm 1/2} $  and  $ \xi^{\pm 1/2} $,  respectively, in  $ \dotRbqsq \, $.
%
%
 Then set also  $ \, q_i^{\pm 1} := q^{\pm d_i} \, $  for all  $ \, i \in I \, $;
 in particular,  $ \; q_{j_{\raise-2pt\hbox{$ \scriptscriptstyle 0 $}}
 \, j_{\raise-2pt\hbox{$ \scriptscriptstyle 0 $}}}^{\,1/2} := q =
 q_{j_{\raise-2pt\hbox{$ \scriptscriptstyle 0 $}}} \; $.  Finally,
 inside the field of fractions of  $ \dotRbqsq $  we consider the subring
 $ \Rbqsq $  generated by  $ \dotRbqsq $  and the elements
 $ \, {\big(\, q_i^{+1} - q_i^{-1} \,\big)}^{-1} \, $,  for all  $ \, i \in I \, $.
                                                       \par
   At last, for later use we introduce notation
   $ \, k_{ij} := q_{ij}^{1/2} \, q_{ji}^{-1/2} \; $  for all  $ \, i \, , j \in I \, $.

\vskip21pt

\subsection{Polynomial multiparameter QUESA's}  \label{subsec: polyn-MpQUESA's}
 {\ }
 \vskip9pt
   We resume all assumptions and notation as in  \S \ref{subsec: FoMpQUESA's}
   and  \S \ref{subsec: multiparam-polynom}  above.  In particular,
   we fix the multiparameter matrix  $ \, \bq := {\big(\, q_{ij} \big)}_{i,j \in I} \, $
   and the ring  $ \Rbqsq $  associated with some fixed symmetrizable,
   indecomposable Cartan matrix  $ \, A = {\big( a_{ij} \big)}_{i,j \in I} \, $
   as in  \S \ref{subsec: FoMpQUESA's}:  we shall then say that the multiparameter
   $ \, \bq := {\big(\, q_{ij} \big)}_{i,j \in I} \, $  is of  \textsl{Cartan type  $ A \, $}.

\vskip9pt

\begin{definition}  \label{def: PolMpQUESA}
 Let  $ \, \big(\, A := {\big(\, a_{i,j} \big)}_{i, j \in I} \, , \, p \,\big) \, $
 be a fixed Cartan super-datum, and  $ \, \bq := {\big(\, q_{ij} \big)}_{i,j \in I} \, $
 a multiparameter matrix of Cartan type  $ A \, $,
 \vskip3pt
   {\it (a)}\,  We define the  {\sl polynomial multiparameter quantum
   universal super enveloping algebra}   --- in short
   {\sl polynomial MpQUESA, or simply PolMpQUESA}  ---
   with multiparameter  $ \bq $  as follows.  It is the unital, associative
   $ \Rbqsq $--superalgebra  $ \Ubqg $  generated by the elements
 $ K_i^{\pm 1} $,  $ L_i^{\pm 1} $,  $ \, E_i \, $,  $ F_i \, $  (for all  $ \, i \in I \, $),
 \,with  $ \ZZ_2 $--grading  given by setting
%
  $ \; \big| K_i^{\pm 1} \big| := \zero =: \big| L_i^{\pm 1} \big| \; $,
  $ \; |E_i| := p(i) =: |F_i| \; \big(\, \forall \; i \in I \,\big) \, $,
 \,and  \textsl{with relations
  $$  \displaylines{
   K_i^{\pm 1} L^{\pm 1}_j  \, = \,  L^{\pm 1}_j K^{\pm 1}_i \;\; ,
\qquad K_i^{\pm 1} K^{\mp 1}_i  \, = \,  1  \, = \,  L^{\pm 1}_i L^{\mp 1}_i  \cr
   K_i^{\pm 1} K^{\pm 1}_j  \, = \,  K^{\pm 1}_j K^{\pm 1}_i  \;\; ,
\qquad  L_i^{\pm 1} L_j^{\pm 1}  \, = \,  L_j^{\pm 1} L_i^{\pm 1}  \cr
   K_i \, E_j \, K_i^{-1}  \, = \,  q_{ij} \, E_j  \;\; ,  \qquad
L_i \, E_j \, L_i^{-1}  \, = \,  q_{ji}^{-1} \, E_j  \cr
   K_i \, F_j \, K_i^{-1}  \, = \,  q_{ij}^{-1} \, F_j  \;\; ,  \qquad
L_i \, F_j \, L_i^{-1}  \, = \,  q_{ji} \, F_j  \cr
   [E_i \, , F_j]  \, = \,  \delta_{i,j} \, \frac{\, K_i - L_i \,}{\, q_i^{+1} - q_i^{-1} \,}  }  $$
 as well as relations  \eqref{eq: q-Serre (E)},  \eqref{eq: q-Serre (F)},
 \eqref{eq: rel1}  and all of relations  \eqref{eq: rel2}  through
 \eqref{eq: rel6}  both for the  $ E_i $'s  and for the  $ F_i $'s  alike}.
 \vskip3pt
   {\it (b)}\,  We define the  {\sl Cartan subalgebra\/}  $ \Ubqh $  of a
   PolMpQUESA  $ \Ubqg $  as being the unital  $ \Rbqsq $--subsuperalgebra  of  $ \Ubqg $
   generated by the subset  $ \, \big\{ K_i^{\pm 1} , L_i^{\pm 1} \,\big|\, i \in I \,\big\} \, $.
 \vskip3pt
   {\it (c)}\,  We define the  {\sl positive, resp.\ negative, nilpotent subsuperalgebra\/}
   $ \Ubqnp $,  resp.\  $ \Ubqnm $,  of a PolMpQUESA  $ \Ubqg $
   to be the unital  $ \Rbqsq $--subsuperalgebra  of  $ \Ubqg $
   generated by the subset  $ \, \big\{ E_i \,\big|\, i \in I \,\big\} \, $,
   resp.\ the subset  $ \, \big\{ F_i \,\big|\, i \in I \,\big\} \, $.
 \vskip5pt
   {\it (d)}\,  We define the  {\sl positive, resp.\ negative, Borel subsuperalgebra\/}
   $ \Ubqbp $,  resp.\  $ \Ubqbm $,  of a PolMpQUESA  $ \Ubqg $
   to be the unital  $ \Rbqsq $--subsuperalgebra  of  $ \Ubqg $
   generated by the subset  $ \, \big\{ E_i \, , K_i^{\pm 1} \,\big|\, i \in I \,\big\} \, $,
   resp.\ the subset  $ \, \big\{ F_i \, , L_i^{\pm 1} \,\big|\, i \in I \,\big\} \, $.
 \vskip5pt
   {\it (e)}\,  The above are definitions for PolMpQUESA's (or their subalgebras)
   over the ``generic multiparameter''  $ \bq $  with entries in  $ \Rbqsq \, $.
   If  $ \, \overline{\bq} = {\big(\, \overline{q}_{ij} \big)}_{i , j \in I} \, $
   is any square matrix of size  $ n $  \textsl{with entries in
   $ \k^\times $  \textit{which are not roots of 1}},  we can define the
   ``specialized'' PolMpQUESA  $ U_{\overline{\bq}\,}(\lieg) $  as a  $ \k $--algebra
   which has exactly the same presentation of  $ \Ubqg $  but for the
   $ \overline{q}_{ij} $'s  replacing the  $ q_{ij} $'s.
   Similarly, we define ``specialized'' Cartan, nilpotent and
   Borel subsuperalgebras  $ U_{\overline{\bq}\,}(\lieh) \, $,
   $ U_{\overline{\bq}\,}(\lien_\pm) $  and  $ U_{\overline{\bq}\,}(\lieb_\pm) \, $.
 \vskip5pt
   \textit{$ \underline{\text{Remark}} $:}\,  The definitions in
   \textit{(a)\/}  through  \textit{(d)\/}  are given over the ground ring  $ \Rbqsq $
   that is the  \textsl{least\/}  one for which the definitions make sense.
   Another natural choice would be to pick as ground ring the field of fractions of  $ \Rbqsq \, $.
   The further results hereafter then still hold true, up to minimal adjustments.
   \hfill  $ \diamondsuit $
\end{definition}

\vskip7pt

   The key point about PolMpQUESA's is the following:

\vskip13pt

\begin{theorem}  \label{thm: Hopf-struct-of-PolMpQUESA's}
 Every PolMpQUESA  $ \Ubqg $  is a
%
%
 Hopf superalgebra over  $ \Rbqsq \, $,  \,with coproduct, counit and antipode given by (for all  $ \, i \in I \, $)
  $$  \displaylines{
   \com\big(K_i^{\pm 1}\big) \, = \, K_i^{\pm 1} \otimes K_i^{\pm 1}  \;\; ,
   \qquad  \epsilon\big(K_i^{\pm 1}\big) \, = \, 1  \;\; ,  \qquad
   \SS\big(K_i^{\pm 1}\big) \, = \, K_i^{\mp 1}  \cr
   \hskip3pt   \com\big(L_i^{\pm 1}\big) \, = \, L_i^{\pm 1} \otimes L_i^{\pm 1}  \;\; ,
   \;\qquad  \epsilon\big(L_i^{\pm 1}\big) \, = \,  1  \;\; ,
   \hskip12pt \hskip6pt \quad  \SS\big(L_i^{\pm 1}\big) \, = \,  L_i^{\mp 1}  \cr
   \com(E_i) \, = \, E_i \otimes 1 + K_i \otimes E_i  \;\; ,  \hskip9pt \quad
 \epsilon(E_i) \, = \, 0  \;\; ,  \hskip13pt \quad  \SS(E_i) \, = \, -K_i^{-1} E_i  \cr
   \hskip7pt   \com(F_i) \, = \, F_i \otimes L_i + 1 \otimes F_i  \;\; ,
   \hskip13pt \quad  \epsilon(F_i) \, = \, 0  \;\; ,
   \hskip17pt \quad  \SS(F_i) \, = \, - F_i{L_i}^{-1}  }
   $$
 and Borel subsuperalgebras  $ \Ubqbp $  and  $ \Ubqbm $  are
 \textsl{Hopf}  sub-superalgebras of  $ \Ubqg \, $.
\end{theorem}

\begin{proof}
 Let  $ \, D A \, $  the symmetrized Cartan matrix associated
 with the multiparameter  $ \bq $  (see above).
 Let  $ \mathbb{K}_n := \k\big(\{\, x_{ij} \,|\, i, j \!\in\! I , i<j \,\}\big) \, $
 be the field of rational functions in  $ {n \choose 2} $  indeterminates
 $ \, x_{ij} \, $   --- with  $ \, i < j \, $  in  $ \, I := \{1,\dots,n\} \, $  ---
 let  $ \, \mathbb{K}_n[[\hbar]] \, $  the associated ring of formal power series in
 $ \hbar \, $,  and let  $ \, \mathbb{K}_n(\!(\hbar)\!) \, $
 be the field of formal Laurent series in  $ \hbar \, $.  Now pick the matrix  $ \, P_{\underline{x}} := {\big(\, p_{ij} \big)}_{i, j \in I} \in M_n\big(\mathbb{K}_n[[\hbar]]\big) \, $  given by
  $$
  p_{ij} := d_i \, a_{ij} + x_{ij} \;\;\; \forall \; i < j \; ,
  \!\!\qquad  p_{ii} := d_i \, a_{ii} \;\;\; ,
  \!\!\qquad  p_{ij} := d_i \, a_{ij} - x_{ji} \;\;\; \forall \; i > j
  $$
 so that  $ \, P_{\underline{x}} + P_{\underline{x}}^{\,T} = 2\,D\,A \, $.
 Then, there exists a unique  $ \k $--algebra  morphism, say
 $ \, \varphi : \Rbqsq \longrightarrow \mathbb{K}_n(\!(\hbar)\!) \, $,
 given by  $ \, q_{ij} \mapsto \varphi(q_{ij}) = \exp\big(\hbar\,p_{ij}\big) \, $,
 \,which is clearly injective since the  $ q_{hk} $'s  only obey the relations
 $ \, q_{ij} \, q_{ji} = q_{ii}^{\,a_{ij}} \, $  and the  $ p_{hk} $'s
 the relations  $ \, p_{ij} + p_{ji} = 2\,d_i\,a_{ij} \, $.
                                                        \par
 Choose any realization  $ \, \R = \big(\, \lieh \, , \Pi \, , \Pi^\vee \,\big) \, $
 of  $ P_{\underline{x}} $  \textsl{over the ground ring  $ \mathbb{K}_n[[\hbar]] \, $},
 which is  \textsl{straight\/}  and  \textsl{split\/}
 --- its existence is granted by  Proposition \ref{prop: exist-realiz's}\textit{(a)}.
 Out of these data, we define the FoMpQUESA  $ \uRpPhg \, $
 \,still over the ground ring  $ \mathbb{K}_n[[\hbar]] \, $,
 \,and also its scalar extension
 $ \, \mathbb{K}_n(\!(\hbar)\!) \otimes_{\mathbb{K}_n[[\hbar]]} \uRpPhg \, $.
 Using the monomorphism  $ \, \varphi : \Rbqsq \!\relbar\joinrel\longrightarrow \mathbb{K}_n(\!(\hbar)\!) \, $
 mentioned above, by scalar restriction through it we consider
 $ \, \mathbb{K}_n(\!(\hbar)\!) \otimes_{\mathbb{K}_n[[\hbar]]} \uRpPhg \, $
 as an  $ \Rbqsq $--superalgebra  too.
 \vskip5pt
   Inside  $ \, \mathbb{K}_n(\!(\hbar)\!) \otimes_{\mathbb{K}_n[[\hbar]]} \uRpPhg \, $
   we look at the  $ \Rbqsq $--subsuperalgebra  $ \dot{U}_\bq(\lieg) $
   generated by  $ \; \Big\{\, \dot{E}_i \, , \, \dot{K}_i^{\pm 1} :=
   \exp\big( \pm \hbar \, T_i^+ \big) \, , \, \dot{L}_i^{\pm 1} :=
   \exp\big( \mp \hbar \, T_i^- \big) \, , \, \dot{F}_i \;\Big|\; i \in I \,\Big\} \, $.
   An easy calculation, along with the fact that the set
   $ \, \Pi^\vee = \big\{\, T_i^+ , T_i^- \,\big|\, i \in I \,\big\} \, $
   is linearly independent, proves that  $ \Ubqg $  is actually contained in  $ \uRpPhg \, $;
   \,moreover, the defining presentation for  $ \uRpPhg $  induces for
   $ \dot{U}_\bq(\lieg) $  the same presentation as for  $ \Ubqg \, $.
   In short,  $ \dot{U}_\bq(\lieg) $  is an isomorphic copy of  $ \Ubqg $
   embedded inside  $ \uRpPhg \, $,  \,via  $ \;\; \varPhi :
   U_\bq(\lieg) \,{\buildrel \cong \over {\relbar\joinrel\longrightarrow}}\,
   \dot{U}_\bq(\lieg) \;\; \Big(\, E_i \mapsto \dot{E}_i \, , \; K_i^{\pm 1}
   \mapsto \dot{K}_i^{\pm 1} , \, L_i^{\pm 1} \mapsto \dot{L}_i^{\pm 1} ,
   \, F_i \mapsto \dot{F}_i \,\Big) \; $.
                                                      \par
   Yet another computation yields  $ \, \Delta\big(\dot{K}_i^{\pm 1}\big) =
   \dot{K}_i^{\pm 1} \otimes \dot{K}_i^{\pm 1} \, $  and
   $ \, \Delta\big(\dot{L}_i^{\pm 1}\big) = \dot{L}_i^{\pm 1} \otimes \dot{L}_i^{\pm 1} \, $
   inside  $ \uRpPhg \, $:  together with similar formulas for the antipode
   (and for the  $ E_i $'s  and the  $ F_i $'s  as well), this implies that
   the topological Hopf structure of  $ \uRpPhg $  does restrict to a (non-topological)
   Hopf structure on  $ \dot{U}_\bq(\lieg) \, $.  The outcome is that
   $ \dot{U}_\bq(\lieg) $  is indeed a (non-topological!) Hopf superalgebra over
   $ \Rbqsq $,  hence the same holds true for  $ \Ubqg $
   via pull-back of the structure through the superalgebra isomorphism
   $ \varPhi $  mentioned above.  Finally, tracking the whole construction
   we see that this Hopf structure onto  $ \Ubqg $  is indeed described by
   the formulas in the main claim.
                                                                 \par
   The last claim about Borel subsuperalgebras then is clear.
\end{proof}

\vskip5pt

   The proof of next claim is contained in that of
   Theorem \ref{thm: Hopf-struct-of-PolMpQUESA's}  right above:

\vskip11pt

\begin{cor}  \label{cor: PolMpQUESA's-inside-FoMpQUEA's}
 Every  \textsl{polynomial}  MpQUESA embeds as a Hopf sub-superalgebra
 inside (the suitable scalar extension of) any split straight  \textsl{formal}
 MpQUESA defined by the same Cartan super-datum.   \hfill  \qed
\end{cor}

\vskip7pt

   Specializing the ``generic parameters''  $ q_{ij} $
   to specific values in  $ \k $  we get at once

\vskip11pt

\begin{cor}  \label{thm: Hopf-struct-of-special-PolMpQUESA's}
 Every  \textsl{specialized}  PolMpQUESA  $ U_{\overline{\bq}\,}(\lieg) $
 is a (non-topological) Hopf superalgebra over  $ \, \k \, $,  \,with
%
%
 Hopf structure given as in
 Theorem \ref{thm: Hopf-struct-of-PolMpQUESA's}.   \hfill  \qed
\end{cor}

\vskip7pt

   A special case in particular stands out:

\vskip11pt

\begin{definition}  \label{def: stand-PolMpQUESA}
 We define the  \textit{standard\/}  PolMpQUESA  $ U_{\check{\bq}}(\lieg) $
 as the specialized PolMpQUESA for the ``standard multiparameter''
 $ \, \check{\bq} := {\big(\, \check{q}_{ij} = q^{d_i a_{ij}} \big)}_{i, j \in I} \, $.
                                                          \par
   \textsl{N.B.:} \,  this choice corresponds to specializing the parameters
   $ \, x_{ij} = 0 \, $  in the proof above, whereas
   the parameter $ q $  stands untouched.
\end{definition}

\vskip7pt

   Finally, from  Theorem \ref{thm: triang-decomp}  we easily see that the
   triangular decompositions for formal MpQUESA's imply parallel decompositions for their polynomial
   counterparts (just because the latter embed into the former, as shown in the proof of
   Theorem \ref{thm: Hopf-struct-of-PolMpQUESA's}  above).
   The precise statement, whose proof is immediate, reads as follows:

\vskip11pt

\begin{theorem}  \label{thm: polyn-triang-decomp}
 Every PolMpQUESA  $ \Ubqg $  admits  \textsl{triangular decompositions}
  $$  \displaylines{
   \Ubqnp \otimes \Ubqh \otimes \Ubqnm  \;\; \cong \;\;  \Ubqg
   \; \cong \;  \Ubqnm \otimes \Ubqh \otimes \Ubqnp  \cr
   \Ubqbp \otimes \Ubqbm  \; \cong \;  \Ubqg  \;\; \cong \;\;
   \Ubqbm \otimes \Ubqbp  }  $$
 and similar decompositions also occur for Borel subsuperalgebras, namely
  $$
  U_\bq(\lien_\pm) \otimes \Ubqh  \;\; \cong \;\;
  U_\bq(\lieb_\pm)  \; \cong \;  \Ubqh \otimes U_\bq(\lien_\pm)
  $$
   \indent   A similar statement holds true for all
   \textsl{specialized}  PolMpQUESA's as well.   \hfill   \qed
\end{theorem}

\vskip15pt

\subsection{Deformations of polynomial MpQUESA's}  \label{subsec: deform-polyn-MpQUESA's}
 {\ }
 \vskip9pt
   We discuss now how the construction of deformations by (toral)
   twist or by (toral) polar--2--cocycles somehow ``restrict'' from formal MpQUESA's
   to  \textsl{polynomial\/}  ones   --- bearing in mind that the latter embed as (Hopf)
   sub-superalgebras of the former.
                                                                   \par
   In short, stepping from the ``formal'' setup to the ``polynomial'' one,
   things get more tricky for deformations by (toral) twist, whereas they get
   instead simpler for deformations by (toral) polar-2-cocycles.
   We present these results a bit quickly, yet the interested reader will
   certainly be able to fill in further details as needed.

\vskip13pt

\begin{free text}  \label{twist-deform x PolMpQUESA's}
 \textbf{Deformations by toral twist for PolMpQUESA's.}
 Let  $ \Ubqg $  be some PolMpQUESA, and  $ \uRpPhg $  some split,
 straight FoMpQUESA such that the former embeds inside the latter,
 as in  Corollary \ref{cor: PolMpQUESA's-inside-FoMpQUEA's}.
 Let us consider a toral twist for  $ \uRpPhg $  as in  \S \ref{twist-uPhgd},  say
 $ \; \F_\Phi \, := \, \exp\Big(\hskip1pt \hbar \,  \, 2^{-1} \,
 {\textstyle \sum_{g,k=1}^t} \phi_{gk} \, H_g \otimes H_k \Big) \; $.
 This belongs to  $ \, {\uRpPhg}^{\otimes 2} $,  but  \textsl{not\/}  to
 $ \, {\Ubqg}^{\otimes 2} $,  unless in the trivial case
 $ \, \Phi = \big( \phi_{gk} \big)_{g, k \in \G} = 0 \, $;
 \,hence in no way any such a toral twist for  $ \uRpPhg $
 can be seen as a toral twist for  $ \Ubqg \, $.  Nevertheless,
 \textsl{for suitable choices of
 $ \, \Phi = \big( \phi_{gk} \big)_{g, k \in \G} \in \lieso_t\big(\kh\big) \, $,
 \,hence of  $ \F_\Phi \, $,  the deformed coproduct  $ \Delta^\Phipicc $  of
 $ \, {\big( \uRpPhg \big)}^{\F_\Phi} $  maps  $ \Ubqg $  inside
 $ {\Ubqg}^{\otimes 2} $,  hence it is indeed a new, ``twist-deformed''
 coproduct for the Hopf algebra  $ \Ubqg \, $}.
                                                          \par
   In order to prove the above statement clearly, we note that, directly
   from the proof of  Theorem \ref{thm: twist-uRpPhg=new-uRpPhg},
   we see that the property that
   $ \, \Delta^\Phipicc\big(\Ubqg\big) \subseteq {\Ubqg}^{\otimes 2} \, $
   does hold true if and only if
  $ \; \L_{\Phi, \ell} := e^{+ \hbar \, 2^{-1}
  \sum_{g,k=1}^t \alpha_\ell(H_g) \, \phi_{gk} H_k} \in \Ubqh \; $
and similarly
  $ \; \K_{\Phi,\ell} := e^{+ \hbar \, 2^{-1} \sum_{g,k=1}^t \alpha_\ell(H_g) \,
  \phi_{kg} H_k} \in \Ubqh \; $
 --- for all  $ \, \ell \in I \, $  ---   hence if and only if
%
%
\begin{equation}   \label{eq: integr-cond-twist}
   \hskip-7pt
 2^{-1} {\textstyle\sum\limits_{g,k=1}^t} \alpha_\ell(H_g) \,
 \phi_{gk} \, H_k  \; \in \; \textsl{Span}_\ZZ\big( \Pi^\vee \big)
 \; \ni \;   2^{-1} {\textstyle\sum\limits_{g,k=1}^t} \alpha_\ell(H_g)
 \, \phi_{kg} \, H_k   \hskip1pt\quad  \big(\, \forall \; \ell \in I \,\big)
\end{equation}
 but since the matrix  $ \, \Phi = \big( \phi_{ij} \big)_{i, j \in I} \, $
 is antisymmetric, the double set of conditions above is equivalent to half
 of them, say the left-hand side conditions in  \eqref{eq: integr-cond-twist}  above.
                                                            \par
   Now,  $ \Pi^\vee $  can be completed to a  $ \kh $--basis  of  $ \lieh $
   --- by the splitness assumption ---   so let us assume from scratch that the basis
   $ \, {\{ H_k \}}_{k=1,\dots,t;} \, $  actually include  $ \Pi^\vee \, $;
   \,even more, we assume that it coincide with it (i.e., the realization
   $ \R $  is also  \textsl{minimal\/});  then the matrix
   $ \, \Phi = \big( \phi_{ij} \big)_{i, j \in I} \, $  can be written in blocks, say
 $ \, \Phi = \begin{pmatrix}
     \; \Phi_+^+  &  \Phi_+^- \;  \\
     \; \Phi_-^+  &  \Phi_-^- \;
             \end{pmatrix} \, $
 where  $ \, \Phi_\varepsilon^\eta =
 \big(\, \phi_{ij}^{\varepsilon,\eta} \,\big)_{i, j \in I} \, $  with
 $ \; \F_\Phi \, := \, \exp\Big(\hskip1pt \hbar \,  \, 2^{-1} \,
 {\textstyle \sum_{\varepsilon , \eta \in \{+,-\} }} \,
 {\textstyle \sum_{i,j=1}^n} \, \phi_{ij}^{\varepsilon,\eta} \,
 T_i^\varepsilon \otimes T_j^\eta \Big) \; $.
 Then the necessary and sufficient conditions for having
 $ \, \Delta^\Phipicc\big(\Ubqg\big) \subseteq {\Ubqg}^{\otimes 2} \, $  are
\begin{equation}   \label{eq: new-integr-cond-twist}
  {\textstyle\sum_{i=1}^n} \, 2^{-1}  \, \phi_{ij}^{\varepsilon,\eta}
  \, \alpha_\ell\big(T_i^\varepsilon\big)  \; \in \; \ZZ   \hskip45pt
  \forall \;\; \ell \, , j \in I \, , \; \varepsilon, \eta \in \{+,-\}
\end{equation}
 Finally, recalling that  $ \, \alpha_\ell\big(T_i^+\big) = p_{i\ell} \, $
 and  $ \, \alpha_\ell\big(T_i^-\big) = p_{\ell{}i} \, $,  conditions
 \eqref{eq: new-integr-cond-twist}  eventually read, in matrix form,
 as ``integrality conditions'' of the following form:
  $$
  P^T \!\cdot \Phi_+^\eta \, \in \, 2\,M_n(\ZZ) \;\; ,
  \quad  P \cdot \Phi_-^\eta \, \in \, 2\,M_n(\ZZ)   \eqno  \forall \;\; \eta \in \{+,-\}   \qquad
  $$
   \indent   When these conditions are fulfilled (and only then),
   the deformation of  $ \uRpPhg $  by the ``compatible'' toral twist
   $ \F_\Phi $  does ``restrict'' to a deformation of the polynomial MpQUESA
   $ \Ubqg $  as well, that we denote by  $ {\Ubqg}^{(\F_\Phi)} \, $.
   The new coproduct, counit and antipode in  $ {\Ubqg}^{(\F_\Phi)} $
   are described by formulas which are\,  \textit{(a)\/}  unchanged for the
   $ K_i^{\pm 1} $'s  and  $ L_i^{\pm 1} $'s,  and\,  \textit{(b)\/}
   the very same formulas as in the first part of the proof of
   Theorem \ref{thm: twist-uRpPhg=new-uRpPhg}  for the  $ E_i $'s  and
   $ F_i $'s   --- as those formulas now make sense as formulas inside
   $ {\Ubqg}^{\otimes 2} $  or  $ \Ubqg \, $.  Then the
   ``polynomial counterpart'' of  Theorem \ref{thm: twist-uRpPhg=new-uRpPhg}
   holds true, stating in particular that every ``twist deformed'' object
   $ {\Ubqg}^{(\F_\Phi)} $  issued by this procedure   --- out of a given
   polynomial MpQUESA  $ \Ubqg \, $,  through a toral twist  $ \F_\Phi $
   which is ``compatible'' in the sense explained above ---
   is yet another  \textsl{polynomial\/}  MpQUESA.  And similar
   statements hold true for Borel subalgebras too.
\end{free text}

\vskip13pt

\begin{free text}  \label{2-cocycle-deform x PolMpQUESA's}
 \textbf{Deformations by toral 2--cocycle for PolMpQUESA's.}
 Let again  $ \Ubqg $  and  $ \uRpPhg $  be a PolMpQUE\-SA and a split,
 straight and minimal FoMpQUESA, respectively, the former embedding into
 the latter as in  Corollary \ref{cor: PolMpQUESA's-inside-FoMpQUEA's}.
 \vskip5pt
   Let  $ \; \sigma_\chi : \uRpPhg \times \uRpPhg
   \relbar\joinrel\relbar\joinrel\twoheadrightarrow \khp \; $
   be a toral polar--2--cocycle of  $ \uRpPhg $   --- cf.\ \S \ref{toral-cocycles-uPhgd}  ---
   which by construction is defined starting from a  $ \kh $--bilinear  map
   $ \; \chi : \lieh \times \lieh \relbar\joinrel\relbar\joinrel\longrightarrow \kh \; $
   that vanishes (both on the left and the right) on the  $ S_i $'s  ($ \, i \in I \, $).
   Let us restrict this  $ \, \sigma_\chi \, $  to
   $ \, \Ubqg \times \Ubqg \, $:  \,by construction,
   $ \, \sigma_\chi{\Big|}_{\Ubqg \times \Ubqg} \, $  is uniquely determined by
   $ \, \chi{\Big|}_{\Ubqh \times \Ubqh} \, $.  Now,  $ \Ubqh $
   is nothing but the group algebra over  $ R_\bq $  of the free Abelian group
   $ \, \varGamma \cong \ZZ^{2n} \, $  (generated by the  $ K_i $'s  and the $ L_i $'s);
   therefore  $ \, \chi{\Big|}_{\Ubqh \times \Ubqh} \, $  is uniquely
   determined by its restriction  $ \, \eta(\chi) := \chi{\Big|}_{\varGamma \times \varGamma} \, $.
   A straightforward check shows that  $ \, \eta(\chi) \, $
   is just a standard group 2--cocycle with values in  $ R_\bq \, $,  and likewise
   $ \, \chi{\Big|}_{\Ubqh \times \Ubqh} \, $  is a  \textsl{standard\/}
   Hopf 2--cocycle for the Hopf algebra  $ \Ubqh \, $,
   with values in ground ring  $ R_\bq \, $,  and the same holds for
   $ \, \sigma_{\hat{\bq}_\chi} := \sigma_\chi{\Big|}_{\Ubqg \times \Ubqg} \, $
   --- in particular, no more ``polar'' is involved, nor there is any need to
   resort to any larger ring than the fixed ground one.  The outcome is that
   restricting to the PolMpQUESA   $ \Ubqg $  the deformation process onto the larger FoMpQUESA
   $ \uRpPhg $  via the (toral) polar--2--cocycle  $ \sigma_\chi $  yields
   the same outcome as deforming  $ \Ubqg $  by its Hopf 2--cocycle
   $ \sigma_{\hat{\bq}_\chi} $  via the standard procedure in Hopf algebra theory.
                                                                  \par
   By the way, note that (by construction) every such  $ \, \sigma_{\hat{\bq}_\chi} \, $
   vanishes on the subgroup  $ \varGamma_0 $  of  $ \varGamma $
   generated by the products  $ K_i^{+1} L_i^{-1} $  ($ \, i \in I \, $).
 \vskip5pt
   Conversely, starting from any standard group 2--cocycle
   $ \, \hat{\sigma}_{\hat\bq} \, $  on  $ \varGamma $  with values in
   $ R_\bq $  \textsl{which vanishes on the subgroup  $ \varGamma_0 \, $},
   we can uniquely extend it to a Hopf 2--cocycle of  $ \Ubqh $  and then
   even to a Hopf 2--cocycle  $ \sigma_{\hat\bq} $  of  $ \Ubqg $
   --- adapting the recipe in  \S \ref{toral-cocycles-uPhgd}.
   These Hopf 2--cocycles (in a standard sense: no ``polar'' is involved!)
   of  $ \Ubqg $  will in fact coincide with all those coming from
   restriction of toral polar--2--cocycles of  $ \uRpPhg \, $,
   hence we will call these  $ \sigma_{\hat\bq} $  ``toral'' as well.
 \vskip5pt
   To finish with, whenever we perform a deformation of  $ \Ubqg $
   by any toral 2--cocycle  $ \sigma_{\hat\bq} $  we get a genuine
   deformation of the polynomial MpQUESA  $ \Ubqg \, $,  that we denote by
   $ {\big( \Ubqg \big)}_{\sigma_{\hat\bq}} \, $:  its coalgebra structure
   is unchanged, whilst the new product is described by a straightforward (and easy), partial adaptation of formulas given in  \S \ref{formulas_deform-prod}.  Then the ``polynomial counterpart'' of  Theorem \ref{thm: 2cocdef-uPhgd=new-uPhgd}  holds true, namely there exists an isomorphism
%
%
 $ \; {\big( \Ubqg \big)}_{\sigma_{\hat\bq}} \cong U_{\bq*\hat{\bq}}(\lieg) \, $
 which is the identity on generators.  In short, every toral 2--cocycle
 deformation of a PolMpQUESA is yet another PolMpQUESA, whose new multiparameter
 $ \, \bq*\hat{\bq} \, $  is a suitable combination of the original one
 $ \bq $  and of the further parameters encoded within  $ \sigma_{\hat\bq} $
 --- details can be easily deduced again from  \S \ref{formulas_deform-prod}.  In particular, one can prove that any PolMpQUESA  $ \Ubqg $  can be obtained from the standard PolMpQUESA  $ U_{\check{\bq}}(\lieg) $  by a  $ 2 $--cocycle  deformation, by choosing the parameter  $ \hat{\bq} $  such that  $ \, \check{\bq}*\hat{\bq} = \bq \, $.
 Similar statements also hold true for the Borel PolMpQUESAs and their
 deformations by  $ \sigma_{\hat\bq} \, $.
\end{free text}

\vskip21pt

\end{document}